\documentclass[10pt,a4paper]{paper}
\usepackage[british]{babel}

\usepackage{authblk}

\usepackage{hyperref,url}

\usepackage{mathptmx}
\usepackage{amsmath}
\usepackage{courier}
\usepackage{amssymb}
\usepackage{amsthm}
\usepackage{enumerate}
\usepackage{enumitem,multicol}
\usepackage{tikz}
\usepackage{nicefrac}
\usepackage{bm}
\usepackage{algorithm}
\usepackage{algorithmicx}
\usepackage{algpseudocode}

\usepackage{thm-restate}

\usepackage{graphicx}
\usepackage{longtable}

\DeclareMathAlphabet\mathbfcal{OMS}{cmsy}{b}{n}

\theoremstyle{definition}
\newtheorem{exmp}{Example}[section]

\newtheorem{theorem}{Theorem}[section]
\newtheorem{proposition}[theorem]{Proposition}
\newtheorem{corollary}[theorem]{Corollary}
\newtheorem{lemma}[theorem]{Lemma}
\newtheorem{definition}{Definition}[section]
\newtheorem{remark}{Remark}[section]
\newtheorem*{remark*}{Remark}

\newtheorem*{claim*}{Claim}

\algrenewcommand\algorithmicrequire{\textbf{Input:}}
\algrenewcommand\algorithmicensure{\textbf{Output:}}

\newcommand{\nats}{\mathbb{N}}
\newcommand{\natswith}{\nats_{0}}
\newcommand{\reals}{\mathbb{R}}

\newcommand{\realspos}{\reals_{>0}}
\newcommand{\realsnonneg}{\reals_{\geq 0}}

\newcommand{\states}{\mathcal{X}}

\newcommand{\paths}{\Omega}

\newcommand{\power}{\mathcal{E}(\paths)}
\newcommand{\nonemptypower}{\power_{\emptyset}}

\newcommand{\processes}{\mathbb{P}}
\newcommand{\mprocesses}{\processes^{\mathrm{M}}}

\newcommand{\hmprocesses}{\processes^{\mathrm{HM}}}

\newcommand{\wprocesses}{\processes^{\mathrm{W}}}
\newcommand{\wmprocesses}{\processes^{\mathrm{WM}}}

\newcommand{\whmprocesses}{\processes^{\mathrm{WHM}}}

\newcommand{\lt}{\underline{T}}
\newcommand{\lbound}{L}

\newcommand{\gambles}{\mathcal{L}}
\newcommand{\gamblesX}{\gambles(\states)} 

\newcommand{\ind}[1]{\mathbb{I}_{#1}}

\newcommand{\rateset}{\mathcal{Q}}
\newcommand{\lrate}{\underline{Q}}

\newcommand{\asa}{\Leftrightarrow}
\newcommand{\then}{\Rightarrow}

\newcommand{\norm}[1]{\left\lVert #1 \right\rVert}
\newcommand{\abs}[1]{\left\vert #1 \right\vert}

\newcommand{\coloneqq}{:\!=}

\newcommand{\exampleend}{\hfill$\Diamond$}

\newcommand{\ictmc}{{ICTMC}}

\newcommand{\placeref}[1]{\pageref{#1}}

\newcommand{\exampleproofref}{Appendix~\ref{app:example_proofs}}

\def\presuper#1#2%
  {\mathop{}%
   \mathopen{\vphantom{#2}}^{#1}%
   \kern-\scriptspace%
   #2}

\title{Imprecise Continuous-Time Markov Chains\raggedright}

\author[*]{\vspace{0.5cm}

Thomas Krak}
\author[$\dagger$]{Jasper De Bock}
\author[$\ddagger$]{Arno Siebes}

\affil[ ]{${}^*$\texttt{\large t.e.krak@uu.nl}\\
${}^\ddagger$\texttt{\large a.p.j.m.siebes@uu.nl}\vspace{2pt}}
\affil[ ]{Utrecht University\\
Department of Information and Computing Sciences\\ Princetonplein 5, De Uithof\\
3584 CC Utrecht\\
The Netherlands}

\affil[ ]{}

\affil[$\dagger$]{\texttt{\large jasper.debock@ugent.be}\vspace{2pt}}
\affil[ ]{Ghent University\\
Department of Electronics and Information Systems\\
Technologiepark -- Zwijnaarde 914\\
9052 Zwijnaarde\\ 
Belgium}

\begin{document}

\date{}
\maketitle

\begin{abstract}
Continuous-time Markov chains are mathematical models that are used to describe the state-evolution of dynamical systems under stochastic uncertainty, and have found widespread applications in various fields. In order to make these models computationally tractable, they rely on a number of assumptions that---as is well known---may not be realistic for the domain of application; in particular, the ability to provide exact numerical parameter assessments, and the applicability of time-homogeneity and the eponymous Markov property. In this work, we extend these models to \emph{imprecise continuous-time Markov chains} (\ictmc's), which are a robust generalisation that relaxes these assumptions while remaining computationally tractable.

More technically, an \ictmc~is a \emph{set} of ``precise'' continuous-time finite-state stochastic processes, and rather than computing expected values of functions, we seek to compute \emph{lower expectations}, which are tight lower bounds on the expectations that correspond to such a set of ``precise'' models. Note that, in contrast to e.g. Bayesian methods, all the elements of such a set are treated on equal grounds; we do not consider a distribution over this set. Together with the conjugate notion of \emph{upper expectation}, the bounds that we provide can then be intuitively interpreted as providing best- and worst-case scenarios with respect to all the models in our set of stochastic processes. 

The first part of this paper develops a formalism for describing continuous-time finite-state stochastic processes that does not require the aforementioned simplifying assumptions. Next, this formalism is used to characterise \ictmc's and to investigate their properties. The concept of lower expectation is then given an alternative operator-theoretic characterisation, by means of a \emph{lower transition operator}, and the properties of this operator are investigated as well. Finally, we use this lower transition operator to derive tractable algorithms (with polynomial runtime complexity w.r.t. the maximum numerical error) for computing the lower expectation of functions that depend on the state at any finite number of time points.
\end{abstract}

\noindent\emph{Keywords:} Continuous-Time Markov Chain; Imprecise Probability; Model Uncertainty; Lower and Upper Expectation; Lower Transition Operator

\newpage
\section{Introduction}\label{sec:introduction}

Continuous-time Markov chains are mathematical models that can describe the behaviour of dynamical systems under stochastic uncertainty. In particular, they describe the stochastic evolution of such a system through a discrete state space and over a continuous time-dimension. This class of models has found widespread applications in various fields, including queueing theory~\cite{asmussen2008applied,bolch2006queueing}, mathematical finance~\cite{elliott2013default, rolski2009stochastic,sass2004optimizing}, epidemiology~\cite{ duffy1995estimation,jackson2003multistate, lemey2009reconstructing}, system reliability analysis~\cite{besnard2010approach,gokhale2004analysis, wang2007reliability}, and many others~\cite{yin2012continuous}.

In order to model a problem by means of such a continuous-time Markov chain, quite a lot of assumptions need to be satisfied.
For example, it is common practice to assume that the user is able to specify an exact value for all the parameters of the model.
A second important assumption is the Markov condition, which states that the future behaviour of the system only depends on its current state, and not on its history. Other examples are homogeneity, which assumes that the dynamics of the system are independent of time, and some technical differentiability assumptions. 
As a result of all these assumptions, continuous-time Markov chains can be described by means of simple analytic expressions.

However, we would argue that in many cases, these assumptions are not realistic and are grounded more in pragmatism than in informed consideration of the underlying system. In those cases, despite the fact that such issues are to be expected in any modelling task, we think that it is best to try and avoid these assumptions. Of course, since they are typically imposed to obtain a tractable model, relaxing these assumptions while maintaining a workable model is not straightforward. 
Nevertheless, as we will see, this can be achieved by means of \emph{imprecise continuous-time Markov chains}~\cite{Skulj:2015cq,troffaes2015using} (ICTMC's). These ICMTC's are quite similar to continuous-time Markov chains: they model the same type of dynamical systems, and they therefore have the same fields of application. However, they do not impose the many simplifying assumptions that are traditionally adopted, and are therefore far more robust. Notably, as we will show in this work, these models allow us to relax these assumptions while remaining computationally tractable.

The following provides a motivating toy example. 
It is clearly too simple to be of any practical use, but it does allow us to illustrate the simplifying assumptions that are usually adopted, and to provide a basic idea of how we intend to relax them.

\vspace{10pt}

\begin{exmp}\label{ex:health_sick_exmp}
Consider a person periodically becoming sick, and recovering after some time. If we want to model this behaviour using a continuous-time Markov chain with a binary state space $\{\text{\tt healthy, sick}\}$, we need to specify a \emph{rate parameter} for each of the two possible state-transitions: from {\tt healthy} to {\tt sick}, and from {\tt sick} to {\tt healthy}. Loosely speaking, such a rate parameter characterises how quickly the corresponding state-transition happens. Technically, it specifies the derivative of a transition probability. For example, for the transition from {\tt healthy} to {\tt sick}, the corresponding rate parameter is the derivative of $P(X_s={\tt sick}\,\vert\,X_t={\tt healthy})$ with respect to $s$, for $s=t$, where $P(X_s={\tt sick}\,\vert\,X_t={\tt healthy})$ is the probability of a person being sick at time $s$, given that he or she is healthy at time $t$. Together with the initial probabilities of a person being {\tt sick} or {\tt healthy} at time zero, these rate parameters uniquely characterise a continuous-time Markov chain, which can then be used to answer various probabilistic queries of interest.
For instance, to compute the probability that a person will be sick in ten days, given that he or she is healthy today.

In this example, the defining assumptions of a continuous-time Markov chain impose rather strict conditions on our model. 
First of all: it is necessary to provide exact values---that is, point-estimates---for the initial probabilities and for the rate parameters. If these values are incorrect, this will affect the resulting conclusions.
Secondly, in order to be able to define the rate parameters, the transition probabilities of the model need to be differentiable.
Thirdly, the Markov assumption implies that for any time points $r<t<s$,
\begin{equation*}
P(X_s={\tt sick}\,\vert\, X_t={\tt healthy}, X_r={\tt sick}) = P(X_s={\tt sick}\,\vert\, X_t={\tt healthy}),
\end{equation*}
that is, once we know the person's health state at time~$t$, his or her probability of being sick at time~$s$ does not depend on whether he or she has ever been sick before time~$t$. If it is possible to develop immunity to the disease in question, then clearly, such an assumption is not realistic. Also, it implies that the rate parameters can only depend on the current state, and not on the previous ones.
Fourthly, the rate parameters are assumed to remain constant over time, which, for example, excludes the possibility of modelling seasonal variations. This fourth condition can easily be removed by considering a continuous-time Markov chain that is not homogeneous. However, in that case, the rate parameters become time-dependent, which requires us to specify (or learn from data) even more point-estimates. Similarly, although a more complex model would be able to account for history-dependence, for example by adding more states to the model, this would vastly increase the computational complexity of working with the model. 

In the imprecise continuous-time Markov chains that we consider, these four conditions are relaxed in the following way. First of all, instead of providing a point-estimate for the initial probabilities and the rate parameters, we allow ourselves to specify a set of values. For the purpose of this example, we can take these sets to be intervals. The other three conditions are then dropped completely, provided that they remain compatible with these intervals. For example, the rate parameters do not need to be constant, but can vary in time in an arbitrary way, as long as they remain within their interval. Similarly, the rate parameters are also allowed to depend on the history of the process, that is, the value of the previous states. In fact, the rate parameters---as defined above---do not even need to exist, since we do not require differentiability either.
\exampleend\vspace{10pt}
\end{exmp}

The idea of relaxing the defining assumptions of a continuous-time Markov chain is not new. Various variations on it have been developed over the years.
For example, there are plenty of results to be found that deal with non-homogeneous continuous-time Markov chains~\cite{rindos1995exact,aalen1978empirical,johnson1989nonhomogeneous}.
Dropping the Markov condition is less common, but nevertheless definitely possible~\cite{Harlamov:1320525}. However, the approaches that drop this Markov assumption will typically replace it by some other, weaker assumption, instead of dropping it altogether. 

A common property of all of these approaches is that they still require the parameters of the model to be specified exactly. Furthermore, since these models are typically more complex, the number of parameters that needs to be specified is a lot larger than before. Therefore, in practice, specifying such a generalised model is a lot more difficult than specifying a normal, homogeneous continuous-time Markov chain.
In contrast, our approach does not introduce a large number of new parameters, but simply considers imprecise versions of the existing parameters. In particular, we provide constraints on the traditional parameters of a continuous-time Markov chain, by requiring them to belong to some specified set of candidate parameters.
In this sense, our approach can be regarded as a type of sensitivity analysis on the parameters of a continuous-time Markov chain. However, instead of simply varying the parameters of a continuous-time Markov chain, as a more traditional sensitivity analysis would do, we also consider what happens when these parameters are allowed to be time- and history-dependent. Furthermore, a sensitivity analysis is typically interested in the effect of infinitesimal parameter variations, whereas we consider the effect of variations within some freely chosen set of candidate parameters.

Our approach should also not be confused with Bayesian methods~\cite{insua2012bayesian}. Although these methods also consider parameter uncertainty, they model this uncertainty by means of a prior distribution, thereby introducing even more (hyper)parameters. Furthermore, when integrating out this prior, a Bayesian method ends up with a single `averaged' stochastic process. In contrast, our approach considers set-valued parameter assessments, and does not provide a prior distribution over these sets. Every possible combination of parameter values gives rise to a different process, and we treat all of these processes on equal grounds, without averaging them out.

Among the many extensions of continuous-time Markov chains that are able to deal with complex parameter variations, the ones that resemble our approach the most are continuous-time Markov decision processes~\cite{Xianping:2009} and continuous-time controlled Markov chains~\cite{guo2003}. Similar to what we do, these models vary the parameters of a continuous-time Markov chain, and allow these variations to depend on the history of the process. However, their variations are more restrictive, because the parameters are assumed to remain constant in between two transitions, whereas we allow them to vary in more arbitrary ways. The most important difference with our approach though, is that in these models, the parameter changes are not at all uncertain. In fact, the parameters can be chosen freely, and the goal is to control the evolution of the process in some optimal way, by tuning its parameters as the process evolves.

Having situated our topic within the related literature, let us now take a closer look at what we actually mean by an imprecise continuous-time Markov chain.
In order to formalise this concept, we turn in this work to the field of \emph{imprecise probability}~\cite{Walley:1991vk,troffaes2013:lp,augustin2013:itip}. The basic premise of this field is that, whenever it is impossible or unrealistic to specify a single probabilistic model, say $P$, it is better to instead consider a \emph{set} of probabilistic models $\mathcal{P}$, and to then draw conclusions that are robust with respect to variations in this set. In the particular case of an imprecise continuous-time Markov chain, $\mathcal{P}$ will be a set of stochastic processes. Some of these processes are homogeneous continuous-time Markov chains. However, the majority of them are not. As explained in Example~\ref{ex:health_sick_exmp}, we also consider other, more general stochastic processes, which are not required to be homogeneous, do not need to satisfy the Markov condition, and do not even need to be differentiable. 

From a practical point of view, once a probabilistic model has been formulated and its parameters have been specified, one is typically interested in computing inferences, such as the probability of some event, or the expectation of some function. For example, as in Example~\ref{ex:health_sick_exmp}, we might want to know the probability that a person will be sick in ten days, given that he or she is healthy today.
Similarly, for expectations, one might for example like to know the expected utility of some financial strategy~\cite{sass2004optimizing}, the expected time until some component in a system breaks down~\cite{besnard2010approach} or the expected speed at which the clinical symptoms of a disease develop~\cite{duffy1995estimation}. However, these inferences might depend crucially on the estimated values of the model parameters, and on the defining assumptions of the model.
It is of interest, therefore, to study the robustness of such inferences with respect to parameter changes and relaxations of the defining assumptions of the model.

In our approach, we investigate this type of robustness as follows: we simply consider the probability or expectation of interest for each of the stochastic processes in $\mathcal{P}$, and then report the corresponding lower and upper bound. Intuitively, this can be regarded as providing best- and worst-case scenarios with respect to all the models in $\mathcal{P}$.
Of course, since the stochastic processes in $\mathcal{P}$ are not required to satisfy the Markov condition, a naive optimisation method will be highly inefficient, and in most cases not even possible. However, as we will see, it is possible to develop other, more efficient methods for computing the lower and upper bounds that we are after. If the lower and upper bounds that we obtain are similar, we can conclude that the corresponding inference is robust with respect to the variations that are represented by $\mathcal{P}$. If the lower and upper bound are substantially different, then the inference is clearly sensitive to these variations, and policy may then have to be adapted accordingly.

Readers that are familiar with the literature on imprecise continuous-time Markov chains~\cite{Skulj:2015cq,troffaes2015using} should recognise the main ideas behind the approach that we have just described, but will also notice that our presentation differs from the one that is adopted in References~\cite{Skulj:2015cq} and~\cite{troffaes2015using}. Indeed, the seminal work by {\v{S}}kulj~\cite{Skulj:2015cq}, which provided the first---and so far only---theoretical study of imprecise continuous-time Markov chains, characterised these models by means of the conditional lower and upper expectations of functions that depend on a single time point. In particular, this work defined such lower and upper expectations directly, through a generalisation of the well-known differential equation characterisation of ``normal'' continuous-time Markov chains. This approach allowed the author to focus on developing algorithms for solving this generalised differential equation, thereby yielding methods for the efficient computation of lower and upper expectations of functions that depend on a single time point. These results have since been successfully applied to conduct a robust analysis of failure-rates and repair-times in power-grid networks~\cite{troffaes2015using}.

However, due to its direct characterisation of lower and upper expectations, this pioneering work left open a number of questions about which sets of stochastic processes these quantities correspond to; for instance, the question of whether or not these sets also include non-Markov processes. Furthermore, due to the focus on functions that depend on a single time point, computational methods for more general functions do not follow straightforwardly from this earlier work.

In contrast, we will in this present paper address these issues directly, by characterising imprecise continuous-time Markov chains explicitly as sets of stochastic processes. There are a number of advantages to working with such sets of processes. First of all, it removes any ambiguity as to the elements of such a set, which allows us to state exactly which assumptions the model robustifies against. Secondly, this approach allows us to prove some properties of such a set's corresponding lower and upper expectations which, in turn, allow us to derive tractable algorithms to compute lower and upper expectations of functions that depend on any finite number of time points. Finally, our approach allows us to derive algorithms that, for the special case of functions that depend on a single time point, improve upon the computational complexity of the algorithm in Reference~\cite{Skulj:2015cq}.

In summary then, our aims with the present paper are threefold. First of all, to solidify the theoretical foundations of imprecise continuous-time Markov chains. Secondly, to extend and generalise existing methods for computing the corresponding lower expectations and, as a particular case, upper expectations and lower and upper probabilities. Thirdly, to provide analytical tools that can be used for future analysis of these models, and for the development of new algorithms.

Our main contributions can be summarised as follows.
\begin{enumerate}
\item We provide a unified framework for describing finite-state stochastic processes in continuous time, using the formalism of \emph{full conditional probabilities}. This framework covers the full range of (non-)homogeneous, (non-)Markovian, and (non-)differentiable stochastic processes.
\item We use this framework to formalise \emph{imprecise continuous-time Markov chains}: sets of stochastic processes that are in a specific sense consistent with user-specified set-valued assessments of the parameters of a continuous-time Markov chain. We conduct a thorough theoretical study of the properties of these sets of processes, and of the lower expectations that correspond to them. 
\item We introduce a \emph{lower transition operator} for imprecise continuous-time Markov chains, and show that this operator satisfies convenient algebraic properties such as homogeneity, differentiability, and Markovian-like factorisation---even if the underlying set of processes does not. Furthermore, and perhaps most importantly, we show that we can use this operator to \emph{compute lower expectations} of functions that depend on the state $X_t$ at an arbitrary but finite number of time points.
\end{enumerate}

To be upfront, we would like to conclude this introduction with a cautionary remark to practitioners: the main (computational) methods that we present do not enforce the homogeneity and---in some cases also---the Markovian independence assumptions of a traditional Markov chain, but allow these to be relaxed. Therefore, if one is convinced that the (true) system that is being modelled \emph{does} satisfy these properties, then it is in general quite likely that the lower- and upper bounds that are reported by our methodology will be conservative. This is not to say that the methods that we present lead to vacuous or non-informative conclusions---see, e.g., Reference~\cite{Rottondi2017DRCN} for a successful application of our approach in telecommunication. However, tighter bounds---i.e. more informative conclusions---might then be obtainable by methodologies that do enforce these properties, provided of course that such methods are available and tractable. If one is uncertain, however, about whether these properties hold for one's system of interest, then the methods that we present are exactly applicable, and the bounds that we derive will be tight with respect to this uncertainty.

\subsection{Finding Your Way Around in This Paper}

Given the substantial length of this paper, we provide in this section some suggestions as to what readers with various interests might wish to focus on. In principle, however, the paper is written to be read in chronological order, and is organised as follows.

First, after we  introduce our notation in Section~\ref{sec:prelim} and discuss some basic mathematical concepts that will be used throughout, Section~\ref{sec:systems} discusses some crucial algebraic notions that will allow us to describe stochastic processes. Section~\ref{sec:stochastic_processes} then goes on to formally introduce stochastic processes and provides some powerful tools for describing their dynamics. 

Next, once we have all our mathematical machinery in place, we shift the focus to imprecise continuous-time Markov chains. We start in Section~\ref{sec:cont_time_markov_chains} by considering the special case of---precise---continuous-time Markov chains and then, in Section~\ref{sec:iCTMC}, we finally formalise the imprecise version that is the topic of this paper, and prove some powerful theoretical properties of this model.

The next three sections discuss computational methods: Section~\ref{sec:lowertrans} introduces a lower transition operator for imprecise continuous-time Markov chains, and in Sections~\ref{sec:connections} and~\ref{sec:funcs_multi_time_points}, we use this operator to compute lower expectations of functions that depend on the state at an arbitrary but finite number of time points. 

The last two sections provide some additional context: Section~\ref{sec:prev_work} relates and compares our results to previous work on imprecise continuous-time Markov chains, and in
Section~\ref{sec:conclusions}, we conclude the paper and provide some ideas for future work. 


Finally, the proofs of our results are gathered in an appendix, where they are organised by section and ordered by chronological appearance. This appendix also contains some additional lemmas that may be of independent interest, a basic exposition of the gambling interpretation for coherent conditional probabilities, and proofs for some of the claims in our examples.

For readers with different interests, then, we recommend to focus on the following parts of the paper. First, if one is interested in the full mathematical theory, we strongly encourage the reader to go through the paper in chronological order. However, if one has a more passing interest in the mathematical formalism, and is content with a more conceptual understanding of what we mean by an imprecise continuous-time Markov chain, then he or she may wish to start in either Section~\ref{sec:cont_time_markov_chains} or~\ref{sec:iCTMC}. Finally, readers who are mainly interested in how to compute lower expectations for ICTMC's might want to focus on Sections~\ref{sec:connections} and~\ref{sec:funcs_multi_time_points}, referring back to Section~\ref{subsec:ictmc_types}---for details about the specific types of ICTMC's that we consider---and Section~\ref{sec:connections_rate}---for the requisite details about lower transition rate operators.

Of course, skipping parts of the paper may introduce some ambiguity about the meaning and definition of the many technical concepts that we refer to throughout. In order to alleviate this as much as possible, Table~\ref{table:glossary} provides a short glossary of the most important notation and terminology.

\begin{longtable}{ l p{.57\textwidth} l }
{\bf Symbol or term} & {\bf Explanation} & {\bf Page} \\[3pt]
\endhead
$u$ & \raggedright finite, ordered sequence of time points & \placeref{notation:sequenceset} \\
$\mathcal{U}$ & \raggedright set of all finite, ordered sequences of time points & \placeref{notation:sequenceset} \\
$\mathcal{U}_{<t}$ & \raggedright set of all finite, ordered sequences of time points before time $t$ & \placeref{notation:sequenceset_ineq} \\
$\mathcal{U}_{[t,s]}$ & \raggedright set of all finite, ordered sequences of time points that partition the time interval $[t,s]$ & \placeref{notation:sequenceset_partition} \\
$\states$ & \raggedright finite state space of a stochastic process & \placeref{notation:statespace} \\
$\states_u$ & \raggedright joint state space at time points $u\in\mathcal{U}$ & \placeref{notation:statespace} \\
$\gamblesX$ & \raggedright set of all functions $f: \states\to\reals$ & \placeref{notation:functionspace} \\
$\gambles(\states_u)$ & \raggedright set of all functions $f:\states_u\to\reals$ & \placeref{notation:functionspacemulti} \\
$\norm{\cdot}$ & \raggedright maximum ($L_\infty$) norm of a function, operator, or set of matrices & \placeref{notation:norm} \\
$\mathcal{T}$ & \raggedright transition matrix system: a family of transition matrices satisfying certain properties & \placeref{def:trans_mat_system} \\
$\mathcal{T}_Q$ & \raggedright exponential transition matrix system that corresponds to the transition rate matrix $Q$ & \placeref{def:systemfromQ} \\
$\mathcal{T}_P$ & \raggedright  transition matrix system that corresponds to the stochastic process $P$ & \placeref{def:trans_matrix} \\
$\mathcal{T}^\mathbf{I}$ & \raggedright restriction of the transition matrix system $\mathcal{T}$ to the closed time interval $\mathbf{I}\subseteq\realsnonneg$ & \placeref{notation:restricted_system} \\
$\otimes$ & \raggedright concatenation operator for two restricted transition matrix systems & \placeref{def:concatenation_system} \\
well-behaved & \raggedright loosely speaking, this means ``with bounded rate-of-change''; see Definitions~\ref{def:well_behaved_trans_mat_system} and~\ref{def:well-behaved} & \placeref{def:well_behaved_trans_mat_system}, \placeref{def:well-behaved} \\
$\processes$ & \raggedright set of all continuous-time stochastic processes & \placeref{def:stoch_process} \\
$\wprocesses$ & \raggedright set of all well-behaved continuous-time stochastic processes & \placeref{def:well-behaved} \\
$\wmprocesses$ & \raggedright set of all well-behaved continuous-time Markov chains & \placeref{def:markov_property} \\
$\whmprocesses$ & \raggedright set of all well-behaved, homogeneous continuous-time Markov chains & \placeref{def:homogeneousMarkov} \\
$\overline{\partial}_+T_{t,x_u}^t$, $\overline{\partial}_-T_{t,x_u}^t$, $\overline{\partial}T_{t,x_u}^t$ & \raggedright outer partial derivatives of the transition matrix 
of a stochastic process, for time $t$ and history $x_u$ & \placeref{def:outerpartialderivatives} \\
$\mathcal{R}$ & \raggedright set of all transition rate matrices & \placeref{def:rate_matrix} \\
$\rateset$ & \raggedright set of transition rate matrices & 
\placeref{notation:rate_set} \\
$\mathcal{M}$ & \raggedright initial set of probability mass functions & \placeref{notation:rate_set} \\
$\wprocesses_{\rateset,\mathcal{M}}$, $\wmprocesses_{\rateset,\mathcal{M}}$, $\whmprocesses_{\rateset,\mathcal{M}}$ & \raggedright three different types of ICTMC's; also see Definition~\ref{def:consistent_process_set} and $\wprocesses$, $\wmprocesses$ and $\whmprocesses$  & \placeref{def:process_sets} \\
$\underline{\mathbb{E}}^\mathrm{W}_{\rateset,\mathcal{M}}$, $\underline{\mathbb{E}}^\mathrm{WM}_{\rateset,\mathcal{M}}$, $\underline{\mathbb{E}}^\mathrm{WHM}_{\rateset,\mathcal{M}}$ & \raggedright lower expectation operators of the ICTMC's $\wprocesses_{\rateset,\mathcal{M}}$, $\wmprocesses_{\rateset,\mathcal{M}}$ and $\whmprocesses_{\rateset,\mathcal{M}}$ & \placeref{def:lower_exp} \\
$\lrate$ & \raggedright lower transition rate operator
& \placeref{def:coh_low_trans_rate}
\\
$L_t^s$ & \raggedright lower transition operator from time point $t$ to $s$, corresponding to some given $\lrate$
& \placeref{def:low_trans} \\
$\underline{\mathcal{T}}_{\lrate}$ & \raggedright family of lower transition operators $L_t^s$ corresponding to some given $\lrate$ & \placeref{def:low_trans_system}
\\\\

\caption{Glossary table of important notation and terminology.}
\label{table:glossary}
\end{longtable}

\section{Preliminaries}\label{sec:prelim}

We denote the reals as $\reals$, the non-negative reals as $\realsnonneg$, the positive reals as $\realspos$ and the negative reals as $\reals_{<0}$. For any $c\in\reals$, $\reals_{\geq c}$, $\reals_{>c}$ and $\reals_{<c}$ have a similar meaning. The natural numbers are denoted by $\nats$, and we also define $\nats_0\coloneqq\nats\cup\{0\}$. The rationals will be denoted by $\mathbb{Q}$.

Infinite sequences of quantities will be denoted $\{a_i\}_{i\in\nats}$, possibly with limit statements of the form $\{a_i\}_{i\in\nats}\to c$, which should be interpreted as $\lim_{i\to\infty}a_i=c$. If the elements of such a sequence belong to a space that is endowed with an ordering relation, we may write $\{a_i\}_{i\in\nats}\to c^+$ or $\{a_i\}_{i\in\nats}\to c^-$ if the limit is approached from above or below, respectively.

When working with suprema and infima, we will sometimes use the shorthand notation $\sup\{\cdot\}<+\infty$ to mean that there exists a $c\in\reals$ such that $\sup\{\cdot\}< c$, and similarly for $\inf\{\cdot\}>-\infty$.

For any set $A$ and any subset $C$ of $A$, we use $\ind{C}$ to denote the indicator of $C$, defined for all $a\in A$ by $\ind{C}(a)=1$ if $a\in C$ and $\ind{C}(a)=0$, otherwise. If $C$ is a singleton $C=\{c\}$, we may instead write $\ind{c}\coloneqq\ind{\{c\}}$.

\subsection{Sequences of Time Points}\label{subsec:sequencesoftimepoints}

We will make extensive use of finite sequences of time points. Such a sequence is of the form $u\coloneqq t_0,t_1,\ldots,t_n$, with $n\in\natswith$ and, for all $i\in\{0,\ldots,n\}$, $t_i\in\realsnonneg$. These sequences are taken to be ordered, meaning that for all $i,j\in\{0,\ldots,n\}$ with $i<j$, it holds that $t_i\leq t_j$. Let $\mathcal{U}$ denote the set of all such finite sequences that are \emph{non-degenerate}, meaning that for all $u\in\mathcal{U}$ with $u=t_0,\ldots,t_n$, it holds that $t_i\neq t_j$ for all $i,j\in\{0,\ldots,n\}$ such that $i\neq j$. Note that this does not prohibit \emph{empty} sequences. We therefore also define $\mathcal{U}_\emptyset\coloneqq \mathcal{U}\setminus\{\emptyset\}$. \label{notation:sequenceset}

For any finite sequence $u$ of time points, let $\max u\coloneqq \max\{t_i:i\in\{0,\ldots,n\}\}$. For any time point $t\in\realsnonneg$, we then write $t>u$ if $t>\max u$, and similarly for other inequalities. If $u=\emptyset$, then $t>u$ is taken to be trivially true, regardless of the value of $t$. We use $\mathcal{U}_{<t}$ to denote the subset of $\mathcal{U}$ that consists of those sequences $u\in\mathcal{U}$ for which $u<t$, and, again, similarly for other inequalities. \label{notation:sequenceset_ineq}

Since a sequence $u\in\mathcal{U}$ is a subset of $\realsnonneg$, we can use set-theoretic notation to operate on such sequences. The result of such operations is again taken to be ordered. For example, for any $u,v\in\mathcal{U}$, we use $u\cup v$ to denote the ordered union of $u$ and $v$. Similarly, for any $s\in\realsnonneg$ and any $u\in\mathcal{U}_{<s}$ with $u=t_0,\ldots,t_n$, we use $u\cup\{s\}$ to denote the sequence $t_0,\ldots,t_n,s$.

As a special case, we consider finite sequences of time points that partition a given time interval $[t,s]$, with $t,s\in\realsnonneg$ such that $t\leq s$. Such a sequence is taken to include the end-points of this interval.  Thus, the sequence is of the form $t=t_0< t_1<\cdots< t_n=s$. We denote the set of all such sequences by $\mathcal{U}_{[t,s]}$. Since these sequences are non-degenerate, it follows that $\mathcal{U}_{[t,t]}$ consists of a single sequence $u=t$. For any $u\in\mathcal{U}_{[t,s]}$ with $u=t_0,\ldots,t_n$, we also define the sequential differences $\Delta_i\coloneqq t_i-t_{i-1}$, for all $i\in\{1,\ldots,n\}$. We then use $\sigma(u)\coloneqq \max\{\Delta_i:i\in\{1,\ldots,n\}\}$ to denote the maximum such difference. \label{notation:sequenceset_partition}

\subsection{States and Functions}\label{sec:multivar_notation}

Throughout this work, we will consider some fixed finite \emph{state space} $\states$. A generic element of this set is called a state and will be denoted by $x$. Without loss of generality, we can assume the states to be ordered, and we can then identify $\states$ with the set $\{1,\dots,\abs{\states}\}$, where $\abs{\states}$ is the number of states in $\states$. \label{notation:statespace}

We use $\gamblesX$ to denote the set of all real-valued functions on $\states$. Because $\states$ is finite, a function $f\in\gamblesX$ can be interpreted as a vector in $\reals^{\abs{\states}}$. Hence, we will in the sequel use the terms `function' and `vector' interchangeably when referring to elements of $\gamblesX$. \label{notation:functionspace}

We will often find it convenient to explicitly indicate the time point $t$ that is being considered, in which case we write $\states_t\coloneqq\states$ to denote the state space at time $t$, and $x_t$ to denote a state at time $t$. This notational trick also allows us to introduce some notation for the joint state at (multiple) explicit time points. For any finite sequence of time points $u\in\mathcal{U}$ such that $u=t_0,\ldots,t_n$, we use
\begin{equation*}
\states_u \coloneqq \prod_{t\in u}\states_t
\end{equation*}
to denote the joint state space at the time points in $u$. A joint state $x_u\in\states_u$ is a tuple $(x_{t_0},\ldots,x_{t_n})\in\states_{t_0}\times\cdots\times\states_{t_n}$ that specifies a state $x_{t_k}$ for every time point $t_k$ in $u$. Note that if $u$ only contains a single time point $t$, then we simply have that $\states_u=\states_{\{t\}}=\states_t=\states$. If $u=\emptyset$, then $x_\emptyset\in\states_\emptyset$ is a ``dummy'' placeholder, which typically leads to statements that are vacuously true. 
For any $u\in\mathcal{U}_\emptyset$, we use $\gambles(\states_u)$ to denote the set of all real-valued functions on $\states_u$. \label{notation:functionspacemulti}

\subsection{Norms and Operators}\label{sec:func_oper_norm}

For any $u\in\mathcal{U}_\emptyset$ and any $f\in\gambles(\states_u)$, let the norm $\norm{f}$ be defined as
\begin{equation*}\label{notation:norm}
\norm{f} \coloneqq \norm{f}_{\infty} \coloneqq \max\{\abs{f(x_u)}\,:\,x_u\in\states_u\}\,.
\end{equation*}
As a special case, we then have for any $f\in\gamblesX$ that $\norm{f}=\max\{\abs{f(x)}\,:\,x\in\states\}$.

Linear maps from $\gamblesX$ to $\gamblesX$ will play an important role in this work, and will be represented by matrices. Because the state space $\states$ is fixed throughout, we will always consider square, $\abs{\states}\times\abs{\states}$, real-valued matrices.
 
As such, we will for the sake of brevity simply refer to them as `matrices'. If $A$ is such a matrix, we will index its elements as $A(x,y)$ for all $x,y\in\states$, where the indexing is understood to be row-major. Furthermore, $A(x,\cdot)$ will denote the $x$-th row of $A$, and $A(\cdot,y)$ will denote its $y$-th column. The symbol $I$ will be reserved throughout to refer to the $\abs{\states}\times\abs{\states}$ identity matrix.

Because we will also be interested in non-linear maps, we consider as a more general case operators that are \emph{non-negatively homogeneous}. An operator $A$ from $\gamblesX$ to $\gamblesX$ is non-negatively homogeneous if $A(\lambda f)=\lambda \left[Af\right]$ for all $f\in\gamblesX$ and all $\lambda\in\realsnonneg$. Note that this includes matrices as a special case.

For any non-negatively homogeneous operator $A$ from $\gamblesX$ to $\gamblesX$, we consider the induced operator norm
\begin{equation}\label{eq:operatornorm}
\norm{A}\coloneqq\sup\left\{\norm{Af}\colon f\in\gamblesX,\norm{f}=1\right\}.
\end{equation}
If $A$ is a matrix, it is easily verified that then
\begin{equation}\label{eq:normofmatrix}
\norm{A}
=
\max\left\{\sum_{y\in\states}\abs{A(x,y)}\colon x\in\states\right\}.
\end{equation}
\noindent
Finally, for any set $\mathcal{A}$ of matrices, we define $\norm{\mathcal{A}}\coloneqq\sup\{\norm{A}\colon A\in\mathcal{A}\}$.

These norms satisfy the following properties; Reference~\cite{DeBock:2016} provides a proof for the non-trivial ones.

\begin{proposition}\label{prop:norm_properties}
For all $f,g\in\gamblesX$, all $A,B$ from $\gamblesX$ to $\gamblesX$ that are non-negatively homogeneous, all $\lambda\in\reals$ and all $x\in\states$, we have that
\vspace{5pt}

\begin{multicols}{2}
\begin{enumerate}[label=N\arabic*:,ref=N\arabic*]
\item
$\norm{f}\geq0$
\item
$\norm{f}=0\asa f=0$
\item
$\norm{f+g}\leq\norm{f}+\norm{g}$
\item
$\norm{\lambda f}=\abs{\lambda}\norm{f}$
\item
$\abs{f(x)}\leq\norm{f}$ \\
\item
$\norm{A}\geq0$
\item
$\norm{A}=0\asa A=0$
\item
$\norm{A+B}\leq\norm{A}+\norm{B}$
\item\label{N:homogeneous}
$\norm{\lambda A}=\abs{\lambda}\norm{A}$
\item\label{N:normAB}
$\norm{AB}\leq\norm{A}\norm{B}$
\item\label{N:normAf}
$\norm{Af}\leq\norm{A}\norm{f}$
\end{enumerate}
\end{multicols}
\end{proposition}

\section{Transition Matrix Systems}\label{sec:systems}

We provide in this section some definitions that will later be useful for characterising continuous-time Markov chains. Because we have not yet formally introduced the concept of a continuous-time Markov chain, for now, the definitions below can be taken to be purely algebraic constructs. Nevertheless, whenever possible, we will of course attempt to provide them with some intuition.

For the purposes of this section, it suffices to say that a continuous-time Markov chain is a process which at each time $t$ is in some state $x\in\states$. As time elapses, the process moves through the state space $\states$ in some stochastic fashion. We here define tools with which this stochastic behaviour can be conveniently described.

\subsection{Transition Matrices and Transition Rate Matrices}\label{sec:trans_rate_matrices}

A \emph{transition matrix} $T$ is a matrix that is row-stochastic, meaning that, for each $x\in\states$, the row $T(x,\cdot)$ is a probability mass function on $\states$.
\begin{definition}[Transition Matrix]\label{def:stoch_matrix}
A real-valued matrix $T$ is said to be a \emph{transition matrix} if
\vspace{5pt}
\begin{enumerate}[label=T\arabic*:,ref=T\arabic*]
\item\label{def:T:sumone}
$\sum_{y\in\states}T(x,y)=1$ for all $x\in\states$;\label{def:trans_matrix_is_stochastic}
\item\label{def:T:nonneg}
$T(x,y)\geq0$ for all $x,y\in\states$.
\end{enumerate}
\vspace{5pt}
\noindent We will use $\mathbb{T}$ to denote the set of all transition matrices.
\end{definition}

\begin{restatable}{proposition}{lemmacompositiontransitionmatrix}\label{lemma:compositiontransitionmatrix}
For any two transition matrices $T_1$ and $T_2$, their composition $T_1T_2$ is also a transition matrix.
\end{restatable}

The interpretation in the context of Markov chains goes as follows. The elements $T(x,y)$ of a transition matrix $T$ describe the probability of the Markov chain ending up in state $y$ at the next time point, given that it is currently in state $x$. In other words, the row $T(x,\cdot)$ contains the state-transition probabilities, conditional on currently being in state $x$. We will make this connection more explicit when we formalise continuous-time stochastic processes in Section~\ref{sec:stochastic_processes}.

For now, we note that in a continuous-time setting, this notion of ``next'' time point is less obvious than in a discrete-time setting, because the state-transition probabilities are then continuously dependent on the evolution of time. To capture this aspect, the notion of \emph{transition rate matrices}~\cite{norris1998markov} is used.
\begin{definition}[Transition Rate Matrix]\label{def:rate_matrix}
A real-valued matrix $Q$ is said to be a \emph{transition rate matrix}, or sometimes simply \emph{rate matrix}, if

\vspace{5pt}
\begin{enumerate}[label=R\arabic*:,ref=R\arabic*]
\item\label{def:Q:sumzero}
$\sum_{y\in\states}Q(x,y)=0$ for all $x\in\states$;
\item\label{def:Q:nonnegoffdiagonal}
$Q(x,y)\geq0$ for all $x,y\in\states$ such that $x\neq y$.
\end{enumerate}
\noindent
We use $\mathcal{R}$ to denote the set of all transition rate matrices. 
\vspace{5pt}
\end{definition}

The connection between transition matrices and rate matrices is perhaps best illustrated as follows. Suppose that at some time point $t$, we want to describe for any state $x$ the probability of ending up in state $y$ at some time $s\geq t$. Let $T_t^s$ denote the transition matrix that contains all these probabilities. Note first of all that it is reasonable to assume that, if time does not evolve, then the system should not change. That is, if we are in state $x$ at time $t$, then the probability of still being in state $x$ at time $s=t$, should be one. Hence, we should have $T_t^t=I$, with $I$ the identity matrix. 

A rate matrix $Q$ is then used to describe the transition matrix $T_t^{t+\Delta}$ after a small period of time, $\Delta$, has elapsed. Specifically, the scaled matrix $\Delta Q$ serves as a linear approximation of the change from $T_t^t$ to $T_t^{t+\Delta}$. The following proposition states that, for small enough $\Delta$, this linear change still results in a transition matrix.

\begin{restatable}{proposition}{propstochasticfromratematrix}
\label{prop:stochastic_from_rate_matrix}
Consider any transition rate matrix $Q\in\mathcal{R}$, and any $\Delta\in\realsnonneg$ such that $\Delta \norm{Q}\leq 1$. Then the matrix $(I+\Delta Q)$ is a transition matrix.
\end{restatable}
This also explains the terminology used; a rate matrix describes the ``rate of change'' of a (continuously) time-dependent transition matrix over a small period of time.

Of course, this notion can also be reversed; given a transition matrix $T_t^{t+\Delta}$, what is the change that it underwent compared to $T_t^t=I$? The following proposition states that such a change can always be described using a rate matrix.
\begin{restatable}{proposition}{propratefromstochasticmatrix}
\label{prop:rate_from_stochastic_matrix}
Consider any transition matrix $T$, and any $\Delta\in\realspos$. Then, the matrix $\nicefrac{1}{\Delta}(T-I)$ is a transition rate matrix.
\end{restatable}
Note that Proposition~\ref{prop:rate_from_stochastic_matrix} essentially states that the finite-difference $\nicefrac{1}{\Delta}(T_t^{t+\Delta} - T_t^t)$ is a rate matrix. Intuitively, if we now take the limit as this $\Delta$ goes to zero, this states that the derivative of a continuously time-dependent transition matrix is given by some rate matrix $Q\in\mathcal{R}$---assuming that this limit exists, of course. We will make this connection more explicit in Section~\ref{sec:stochastic_processes}.

We next introduce a function that is often seen in the context of continuous-time Markov chains: the \emph{matrix exponential}~\cite{van2006study} $e^{Q\Delta}$ of $Q \Delta$, with $Q$ a rate matrix and $\Delta\in\reals_{\geq0}$. 
There are various equivalent ways in which such a matrix exponential can be defined. We refer to~\cite{van2006study} for some examples, and will consider some specific definitions later on in this work.
For now, we restrict ourselves to stating the following well-known result.
\begin{proposition}\cite[Theorem 2.1.2]{norris1998markov}\label{prop:stochastic_from_exponential}
Consider a rate matrix $Q\in\mathcal{R}$ and any $\Delta\in\realsnonneg$. Then $e^{Q\Delta}$ is a transition matrix.
\end{proposition}

We conclude this section with some comments about \emph{sets} of rate matrices. First, note that the set of \emph{all} rate matrices, $\mathcal{R}$, is closed under finite sums and multiplication with non-negative scalars. Consider now any set $\rateset\subseteq\mathcal{R}$ of rate matrices. Then $\rateset$ is said to be \emph{non-empty} if $\rateset\neq\emptyset$ and $\rateset$ is said to be \emph{bounded} if $\norm{\rateset}<+\infty$. The following proposition provides a simple alternative characterisation of boundedness.

\begin{restatable}{proposition}{propalternativedefforbounded}
\label{prop:alternativedefforbounded}
A set of rate matrices $\rateset\subseteq\mathcal{R}$ is bounded if and only if
\begin{equation}\label{eq:alternative_bounded}
\inf\left\{Q(x,x)\colon Q\in\rateset\right\}>-\infty\text{~~for all $x\in\states$.}
\end{equation}
\end{restatable}

\subsection{Transition Matrix Systems that are Well-Behaved}\label{subsec:well_behaved_systems}

In the previous section, we used the notation $T_t^s$ to refer to a transition matrix that contains the probabilities of moving from a state at time $t$, to a state at time $s$. We now consider \emph{families} of these transition matrices. Such a family $\mathcal{T}$ specifies a transition matrix $T_t^s$ for every $t,s\in\realsnonneg$ such that $t\leq s$. 

We already explained in the previous section that it is reasonable to assume that $T_t^t=I$. If the transition matrices of a family $\mathcal{T}$ satisfy this property, and if they furthermore satisfy the \emph{semi-group} property---see Equation~\ref{eq:transmatrixproduct} below---we call this family a \emph{transition matrix system}. We will use $\mathbfcal{T}$ to refer to the set of all transition matrix systems.

\begin{definition}[Transition Matrix System]\label{def:trans_mat_system}
A \emph{transition matrix system} $\mathcal{T}$ is a family of transition matrices $T_t^s$, defined for all $t,s\in\realsnonneg$ with $t\leq s$, such that for all $t,r,s\in\realsnonneg$ with $t\leq r\leq s$, it holds that
\begin{equation}\label{eq:transmatrixproduct}
T_t^s=T_t^r T_r^s\,,
\end{equation}
and for all $t\in\realsnonneg$, $T_t^t=I$.
\end{definition}
It will turn out that there is a strong connection between transition matrix systems and continuous-time Markov chains. We will return to this in Section~\ref{sec:cont_time_markov_chains}.

In the previous section, we have seen that for any transition matrix $T$ and any $\Delta\in\realspos$, the matrix $\nicefrac{1}{\Delta}(T - I)$ is a rate matrix, and therefore, in particular, that the finite difference $\nicefrac{1}{\Delta}(T_t^{t+\Delta} - I)$ is a rate matrix. We here note that this is also the case for the term $\nicefrac{1}{\Delta}(T_{t-\Delta}^t - I)$ whenever $(t-\Delta)\geq0$.

We now consider this property in the context of a transition matrix system $\mathcal{T}$. For all $t\in\realsnonneg$ and all $\Delta\in\realspos$, such a transition matrix system specifies a transition matrix $T_t^{t+\Delta}$ and---if $(t-\Delta)\geq0$---a transition matrix $T_{t-\Delta}^t$. We now consider the behaviour of these matrices for various values of $\Delta$. In particular, we look what happens to these finite differences if we take $\Delta$ to be increasingly smaller. 

For each $\Delta\in\realspos$, due to the property that we have just recalled, there will be a rate matrix that corresponds to these finite differences. If the norm of these rate matrices never diverges to $+\infty$ as we take $\Delta$ to zero, we call the family $\mathcal{T}$ \emph{well-behaved}.

\begin{definition}[Well-Behaved Transition Matrix System]\label{def:well_behaved_trans_mat_system}
A transition matrix system $\mathcal{T}$ is called \emph{well-behaved} if 
\begin{equation}\label{eq:wellbehavedtransitionmatrixsystem}
(\forall t\in\realsnonneg)~\limsup_{\Delta\to 0^{+}}\frac{1}{\Delta}\norm{T_{t}^{t+\Delta}-I}<+\infty\,
\text{~and~}
(\forall t\in\realspos)~\limsup_{\Delta\to 0^{+}}\frac{1}{\Delta}\norm{T_{t-\Delta}^t-I}<+\infty\,.
\end{equation}
\end{definition}

Observe that this notion of well-behavedness does not imply differentiability; the limit $\lim_{\Delta\to0^+}\nicefrac{1}{\Delta}(T_t^{t+\Delta}-I)$ need not exist. Rather, it implies that the rate of change of the transition matrices in $\mathcal{T}$ is bounded at all times. In this sense, this notion of well-behavedness is similar to a kind of local Lipschitz continuity. The locality stems from the fact that, although the rate of change must be bounded for each $t$, Equation~\eqref{eq:wellbehavedtransitionmatrixsystem} does not impose that it must be uniformly bounded (at all $t$) by a single ``Lipschitz constant''. That said, we do not stress this connection any further, because we will shortly consider a more involved notion of well-behavedness for which this connection is less immediate; see Section~\ref{sec:well_behaved}.

We finally consider an important special type of transition matrix systems. We have seen in the previous section that for any $Q\in\mathcal{R}$ and any $\Delta\in\realsnonneg$, the matrix exponential $e^{Q\Delta}$ is a transition matrix. We here consider for any $Q\in\mathcal{R}$ the family $\mathcal{T}_Q$ that is generated by such transition matrices.
\begin{definition}\label{def:systemfromQ}For any rate matrix $Q\in\mathcal{R}$, we use $\mathcal{T}_Q$ to denote the family of transition matrices that is defined by
\begin{equation*}
T_t^s=e^{Q(s-t)}
\text{~~for all $t,s\in\realsnonneg$ such that $t\leq s$.}
\end{equation*}
We call this family $\mathcal{T}_Q$ the \emph{exponential transition matrix system} corresponding to $Q$.
\end{definition}

\begin{restatable}{proposition}{propsystemQ}
\label{prop:systemQ}
For any $Q\in\mathcal{R}$, $\mathcal{T}_Q$ is a well-behaved transition matrix system.
\end{restatable}

This exponential transition matrix system corresponding to a $Q\in\mathcal{R}$ will turn out to play a large role in the context of continuous-time Markov chains. We return to this in Section~\ref{sec:cont_time_markov_chains}.

\subsection{Restricted Transition Matrix Systems}\label{sec:restricted}

We finally consider yet another construct that will be useful later: the restriction of a transition matrix system $\mathcal{T}$ to a closed interval $\mathbf{I}$ in the time-line $\realsnonneg$. \label{notation:restricted_system}

By a closed interval $\mathbf{I}$, we here mean a non-empty closed subset $\mathbf{I}\subseteq\realsnonneg$ that is connected, in the sense that for any $t,s\in\mathbf{I}$ such that $t\leq s$, and any $r\in[t,s]$, it holds that $r\in\mathbf{I}$. Note that for any $c\in\realsnonneg$, $[c,+\infty)$ is such a closed interval.

For any transition matrix system $\mathcal{T}$ and any such closed interval $\mathbf{I}\subseteq\realsnonneg$, we use $\mathcal{T}^\mathbf{I}$ to denote the restriction of $\mathcal{T}$ to $\mathbf{I}$. Such a restriction is a family of transition matrices $T_t^s$ that is defined for all $t,s\in\mathbf{I}$ such that $t\leq s$.
We call such a family $\mathcal{T}^{\mathbf{I}}$ a \emph{restricted transition matrix system} on $\mathbf{I}$. The set of all restricted transition matrix systems on $\mathbf{I}$ is denoted by $\mathbfcal{T}^{\mathbf{I}}$.

\begin{restatable}{proposition}{proprestrtransmatsystemifsemigroup}
\label{prop:restr_trans_mat_system_if_semigroup}
Consider any closed interval $\mathbf{I}\subseteq\realsnonneg$, and let $\mathcal{T}^{\mathbf{I}}$ be a family of transition matrices $T_t^s$ that is defined for all $t,s\in\mathbf{I}$ with $t\leq s$. Then $\mathcal{T}^{\mathbf{I}}$ is a restricted transition matrix system on $\mathbf{I}$ if and only if, for all $t,r,s\in\mathbf{I}$ with $t\leq r\leq s$, it holds that $T_t^s = T_t^rT_r^s$ and $T_t^t=I$.
\end{restatable}

We call a restricted transition matrix system $\mathcal{T}^{\mathbf{I}}$ well-behaved if it is the restriction to $\mathbf{I}$ of a well-behaved transition matrix system.

\begin{restatable}{proposition}{propwellrestrtransmatsystemiflimsup}
\label{prop:well_restr_trans_mat_system_if_limsup}
Consider any closed interval $\mathbf{I}\subseteq\realsnonneg$, and let $\mathcal{T}^{\mathbf{I}}$ be a restricted transition matrix system on $\mathbf{I}$. Then $\mathcal{T}^{\mathbf{I}}$ is well-behaved if and only if
\begin{equation}\label{eq:wellbehavedrestrictedtransitionmatrixsystem}
(\forall t\in\mathbf{I}^+)~\limsup_{\Delta\to 0^{+}}\frac{1}{\Delta}\norm{T_{t}^{t+\Delta}-I}<+\infty\,
\text{~and~}
(\forall t\in\mathbf{I}^-)~\limsup_{\Delta\to 0^{+}}\frac{1}{\Delta}\norm{T_{t-\Delta}^t-I}<+\infty\,,
\end{equation}
where $\mathbf{I}^+\coloneqq\mathbf{I}\setminus\{\sup\mathbf{I}\}$ and $\mathbf{I}^-\coloneqq\mathbf{I}\setminus\{\min\mathbf{I}\}$.
\end{restatable}

Now, because these restricted transition matrix systems are only defined on some given closed interval, it will be useful to define a concatenation operator between two such systems. 

\begin{definition}[Concatenation Operator]\label{def:concatenation_system}
For any two closed intervals $\mathbf{I},\mathbf{J}\subseteq\realsnonneg$ such that $\max\mathbf{I}=\min\mathbf{J}$, and any two restricted transition matrix systems $\mathcal{T}^{\mathbf{I}}$ and $\mathcal{T}^{\mathbf{J}}$, the concatenation of $\mathcal{T}^{\mathbf{I}}$ and $\mathcal{T}^{\mathbf{J}}$ is denoted by $\mathcal{T}^\mathbf{I}\otimes\mathcal{T}^\mathbf{J}$, and defined as the family of transition matrices $T_t^s$ that is given by
\begin{equation*}
T_t^s \coloneqq \begin{cases}
\presuper{i}T_t^s & \text{if $t,s\in\mathbf{I}$} \\
\presuper{j}T_t^s & \text{if $t,s\in\mathbf{J}$} \\
\presuper{i}T_t^r\presuper{j}T_r^s & \text{if $t\in\mathbf{I}$ and $s\in\mathbf{J}$}
\end{cases}\text{~~~for all $t,s\in\mathbf{I}\cup\mathbf{J}$ such that $t\leq s$,}
\end{equation*}
where $r\coloneqq\max\mathbf{I}=\min\mathbf{J}$, and $\presuper{i}T_t^s$ and $\presuper{j}T_t^s$ denote the transition matrices corresponding to $\mathcal{T}^{\mathbf{I}}$ and $\mathcal{T}^{\mathbf{J}}$, respectively.
\end{definition}

\begin{restatable}{proposition}{propconcatrestrtransmatsystemsissystem}
\label{prop:concat_restr_trans_mat_systems_is_system}
Consider two closed intervals $\mathbf{I},\mathbf{J}\subseteq\realsnonneg$ such that $\max\mathbf{I}=\min\mathbf{J}$, and any two restricted transition matrix systems $\mathcal{T}^{\mathbf{I}}$ and $\mathcal{T}^{\mathbf{J}}$. Then their concatenation $\mathcal{T}^{\mathbf{I}\cup \mathbf{J}} \coloneqq \mathcal{T}^{\mathbf{I}}\otimes \mathcal{T}^{\mathbf{J}}$ is a restricted transition matrix system on $\mathbf{I}\cup\mathbf{J}$. Furthermore, if both $\mathcal{T}^{\mathbf{I}}$ and $\mathcal{T}^{\mathbf{J}}$ are well behaved, then $\mathcal{T}^{\mathbf{I}\cup\mathbf{J}}$ is also well-behaved.
\end{restatable}

\begin{exmp}\label{exmp:combinetwoexponentials}
Consider any two rate matrices $Q_1,Q_2\in\mathcal{R}$ such that $Q_1\neq Q_2$, and let $\mathcal{T}_{Q_1}$ and $\mathcal{T}_{Q_2}$ be their exponential transition matrix systems, which, as we know from Proposition~\ref{prop:systemQ}, are well-behaved. Now choose any $r\in\realsnonneg$ and define
\begin{equation*}
\mathcal{T} \coloneqq \mathcal{T}_{Q_1}^{[0,r]} \otimes \mathcal{T}_{Q_2}^{[r,+\infty)}\,.
\end{equation*}
It then follows from Proposition~\ref{prop:concat_restr_trans_mat_systems_is_system} that $\mathcal{T}$ is a well-behaved transition matrix system. Furthermore, for any $t,s\in\realsnonneg$ such that $t\leq r\leq s$, the transition matrix $T_t^s$ that corresponds to $\mathcal{T}$ is given by $T_t^s = T_t^rT_r^s = e^{Q_1(r-t)}e^{Q_2(s-r)}$.
\exampleend
\end{exmp}

We also introduce a metric $d$ between restricted transition matrix systems that are defined on the same interval $\mathbf{I}$. For any two such restricted transition matrix systems $\mathcal{T}^\mathbf{I}$ and $\mathcal{S}^\mathbf{I}$, we let
\begin{equation}\label{eq:trans_mat_system_metric}
d(\mathcal{T}^{\mathbf{I}},\mathcal{S}^{\mathbf{I}}) \coloneqq \sup\left\{\norm{T_t^s - S_t^s}\,:\,t,s\in\mathbf{I},\,t\leq s\right\},
\end{equation}
where, for all $t,s\in\mathbf{I}$, it is understood that $T_t^s$ corresponds to $\mathcal{T}^{\mathbf{I}}$ and $S_t^s$ to $\mathcal{S}^{\mathbf{I}}$. This metric allows us to state the following result.

\begin{restatable}{proposition}{lemmarestrictedtransmatsystemcauchyconverges}
\label{lemma:restricted_trans_mat_system_cauchy_converges}
Consider any interval $\mathbf{I}\subseteq\realsnonneg$ and let $d$ be the metric that is defined in Equation~\eqref{eq:trans_mat_system_metric}. The metric space $\smash{(\mathbfcal{T}^{\mathbf{I}},d)}$ is then complete.
\end{restatable}
Note that this result includes as a special case that the set $\mathbfcal{T}$ of all (unrestricted) transition matrix systems is complete. The following example illustrates how this result can be used.

\begin{exmp}\label{exmp:limit_trans_mat_system}
Consider some positive constant $c\in\realspos$ and let $\{Q_i\}_{i\in\nats_0}$ be a sequence of rate matrices such that, for all $i\in\nats$, $\norm{Q_i-Q_{i-1}}\leq c$.
We can then construct the following sequence. For $i=0$, we let $\smash{\mathcal{T}_{0}\coloneqq\mathcal{T}_{Q_0}}$, and for all $i\in\nats$, we let
\begin{equation}
\mathcal{T}_i\coloneqq
\mathcal{T}_{Q_i}^{[0,\delta_i]}\otimes \mathcal{T}_{i-1}^{[\delta_i,+\infty)}
\label{eq:def:sequenceinexample3}
\end{equation}
where, for all $i\in\nats_0$, $\delta_i\coloneqq 2^{-i}$. The resulting sequence $\smash{\{\mathcal{T}_i\}_{i\in\nats_0}}$ is then clearly a subset of $\mathbfcal{T}$ and, because of Proposition~\ref{prop:systemQ} and~\ref{prop:concat_restr_trans_mat_systems_is_system}, every transition matrix system in this sequence is well-behaved. Furthermore, as is proved in \exampleproofref, $\{\mathcal{T}_i\}_{i\in\nats_0}$ is a Cauchy sequence, which basically means that its elements become arbitrarily close to each other as the sequence progresses.

The reason why this is of interest to us is because in a complete metric space, every Cauchy sequence converges to a limit that belongs to the same space. Hence, since $\{\mathcal{T}_i\}_{i\in\nats_0}$ is Cauchy, Proposition~\ref{lemma:restricted_trans_mat_system_cauchy_converges} allows us to infer that $\{\mathcal{T}_i\}_{i\in\nats_0}$ converges to a limit $\mathcal{T}\coloneqq \lim_{i\to\infty}\mathcal{T}_i$ in $\mathbfcal{T}$.
\exampleend
\end{exmp}

As this example illustrates, Proposition~\ref{lemma:restricted_trans_mat_system_cauchy_converges} allows us to (a) establish the existence of limits of sequences of (restricted) transition matrix systems and (b) prove that these limits are restricted transition matrix systems themselves. 
In order to make this concept of a limit of transition matrix systems less abstract, we now provide, for a particular case of the sequence in Example~\ref{exmp:limit_trans_mat_system}, closed-form expressions for some of the transition matrices that correspond to its limit.

\begin{exmp}\label{exmp:limit_trans_mat_system_matrices}
Let $Q_1,Q_2\in\mathcal{R}$ be two commuting rate matrices. For example, let $Q_1\in\mathcal{R}$ be an arbitrary rate matrix and let $Q_2\coloneqq\alpha Q_1$, with $\alpha\in\realsnonneg$. 

Now let $\{Q_i\}_{i\in\nats_0}$ be defined by $Q_i\coloneqq Q_1$ if $i$ is odd and $Q_i\coloneqq Q_2$ if $i$ is even, define $\delta_i\coloneqq 2^{-i}$ for all $i\in\nats_0$, and consider the corresponding sequence of transition matrix systems $\{\mathcal{T}_i\}_{i\in\nats}$ that was defined in Example~\ref{exmp:limit_trans_mat_system}.  Since $\norm{Q_i-Q_{i-1}}=\norm{Q_1-Q_2}$ for all $i\in\nats$, the sequence $\{Q_i\}_{i\in\nats_0}$ clearly satisfies the conditions in Example~\ref{exmp:limit_trans_mat_system}---just choose $c=\norm{Q_1-Q_2}$---and therefore, as we have seen, $\{\mathcal{T}_i\}_{i\in\nats}$ converges to a limit $\mathcal{T}\coloneqq \lim_{i\to\infty}\mathcal{T}_i$ in $\mathbfcal{T}$. 

As proved in \exampleproofref, it then holds that
for any $t\in(0,1]$, the transition matrix from $0$ to $t$ that corresponds to the transition matrix system $\mathcal{T}$ is equal to
\vspace{4pt}
\begin{equation}\label{eq:ex4:def}
T_0^t=e^{Q_1\varphi_1(t)+Q_2\varphi_2(t)},
\end{equation}
with
\begin{equation}\label{eq:ex4:phi1}
\varphi_1(t)
\coloneqq
\begin{cases}
t-\nicefrac{2}{3}\delta_{i+1}
&\text{ if $\delta_{i+1}\leq t\leq\delta_{i}$ with $i$ odd}\\
\nicefrac{2}{3}\delta_{i+1}
&\text{ if $\delta_{i+1}\leq t\leq\delta_{i}$ with $i$ even} 
\end{cases}
\end{equation}
and
\begin{equation}\label{eq:ex4:phi2}
\varphi_2(t)
\coloneqq
\begin{cases}
\nicefrac{2}{3}\delta_{i+1}
&\text{ if $\delta_{i+1}\leq t\leq\delta_{i}$ with $i$ odd}\\
t-\nicefrac{2}{3}\delta_{i+1}
&\text{ if $\delta_{i+1}\leq t\leq\delta_{i}$ with $i$ even.} 
\end{cases}
\vspace{7pt}
\end{equation}
Furthermore, it can be shown that the transition matrix system $\mathcal{T}$ is well-behaved---again, see \exampleproofref~for a proof.
\exampleend
\end{exmp}

The transition matrix system $\mathcal{T}$ in our previous example was well-behaved, and was constructed as a limit of well-behaved transition matrix systems. Therefore, one might think that the former is implied by the latter. However, as our next example illustrates, this is not the case. A limit of well-behaved transition matrix systems need not be well-behaved itself.

\begin{exmp}\label{exmp:notwellbehavedMarkov}
Consider any rate matrix $Q\in\mathcal{R}$ such that $\norm{Q}=1$ and, for all $i\in\nats_0$, define $Q_i\coloneqq iQ$ and let $\mathcal{T}_i$ and $\delta_i$ be defined as in Example~\ref{exmp:limit_trans_mat_system}. Then since $\{Q_i\}_{i\in\nats_0}$ satisfies the conditions of Example~\ref{exmp:limit_trans_mat_system} with $c=1$, the sequence $\{\mathcal{T}_i\}_{i\in\nats_0}$ has a limit $\mathcal{T}\coloneqq \lim_{i\to\infty}\mathcal{T}_i$ in $\mathbfcal{T}$. 

However, despite the fact that we know from Example~\ref{exmp:limit_trans_mat_system} that each of the transition matrix systems $\mathcal{T}_i$, $i\in\nats_0$, is well-behaved, the limit $\mathcal{T}$ itself is not well-behaved; see~\exampleproofref~for a proof.
\exampleend
\end{exmp}

\section{Continuous-Time Stochastic Processes}\label{sec:stochastic_processes}

We will in this section formalise the notion of a continuous-time stochastic process. However, we do not adopt the classical---Kolmogorovian, measure-theoretic---setting, but will instead be using the framework of full conditional probabilities. Our reasons for doing so are the following.

First of all, our results on imprecise continuous-time Markov chains will be concerned with events or functions that depend on an finite number of time points only. Therefore, we do not require the use of typical measure-theoretic concepts such as $\sigma$-algebras, $\sigma$-additivity, measurability, etcetera. Instead, we will impose only the bare minimum of assumptions that are required for our results. The extra structural and continuity assumptions that are typically imposed in a measure-theoretic setting are regarded as optional.

Secondly, our approach does not suffer from some of the traditional issues with probability zero. In standard settings, conditional probabilities are usually derived from unconditional probabilities through Bayes's rule, which makes them undefined whenever the conditioning event has probability zero. Instead, we will regard conditional probabilities as primitive concepts. As a result, we can do away with some of the usual `almost surely' statements and replace them with statements that are certain.

Finally, and most importantly, in our `imprecise' setting, we will be working with a set of stochastic processes rather than with a single stochastic process. In this context, we will often need to prove that such a set contains a stochastic process that meets certain specific requirements. 
Full conditional probabilities provide a convenient framework for constructing such existence proofs, through the notion of coherence.

We start in Section~\ref{sec:cond_prob} by introducing full conditional probabilities and explaining their connection with coherence. In Section~\ref{sec:def_stochastic_processes}, we then use these concepts to formalise continuous-time stochastic processes. Section~\ref{sec:well_behaved} describes a specific subclass of stochastic processes, which we call well-behaved, and on which we will largely focus throughout this work. Finally, Section~\ref{sec:dynamics} provides some tools with which we can describe the dynamics of stochastic processes.

\subsection{Full and Coherent Conditional Probabilities}\label{sec:cond_prob}

Consider a variable $X$ that takes values $\omega$ in some non-empty---possibly infinite---\emph{outcome space} $\Omega$. The actual value of $X$ is taken to be uncertain, in the sense that it is unknown. This uncertainty may arise because $X$ is the outcome of a random experiment that has yet to be conducted, but it can also simply be a consequence of our lack of information about $X$. 
We call any subset $E$ of $\Omega$ an event, we use $\power$ to denote the set of all such events, and we let $\nonemptypower\coloneqq\power\setminus\{\emptyset\}$ be the set of all non-empty events. 
A subject's uncertainty about the value of $X$ can then be described by means of a full conditional probability~\cite{Dubins:1975ej}.

\begin{definition}[Full Conditional Probability]\label{def:cond_prob}
A full conditional probability $P$ is a real-valued map from $\power\times\nonemptypower$ to $\reals$ that satisfies the following axioms. For all $A,B\in\power$ and all \mbox{$C,D\in\nonemptypower$}:
\vspace{5pt}

\begin{enumerate}[label=F\arabic*:,ref=F\arabic*]
\item
$P(A\vert C)\geq 0$;\label{def:coh_prob_2}
\item
$P(A\vert C)=1$ if $C\subseteq A$;\label{def:coh_prob_1}
\item
$P(A\cup B\vert C)=P(A\vert C)+P(B\vert C)$ if $A\cap B=\emptyset$;\label{def:coh_prob_3}
\item
$P(A\cap D\vert C)=P(A\vert D\cap C)P(D\vert C)$ if $D\cap C\neq\emptyset$.\label{def:coh_prob_6}
\end{enumerate}
\vspace{5pt}

\noindent
For any $A\in\power$ and $C\in\nonemptypower$, we call $P(A\vert C)$ the probability of $A$ conditional on $C$. Also, for any $A\in\power$, we use the shorthand notation $P(A)\coloneqq P(A\vert\paths)$ and then call $P(A)$ the probability of $A$.
The following additional properties can easily be shown to follow from \ref{def:coh_prob_2}--\ref{def:coh_prob_3}; see Appendix~\ref{app:stoch_proc} for a proof. For all $A\in\power$ and all $C\in\nonemptypower$:
\vspace{5pt}
\begin{enumerate}[label=F\arabic*:,ref=F\arabic*]
\setcounter{enumi}{4}
\item
$0\leq P(A\vert C)\leq 1$;\label{def:coh_prob_2b}
\item
$P(A\vert C)=P(A\cap C\vert C)$;\label{def:coh_prob_7}
\item
$P(\emptyset\vert C)=0$;\label{def:coh_prob_8}
\item
$P(\Omega\vert C)=1$.\label{def:coh_prob_5}
\end{enumerate}
\vspace{2pt}
\end{definition}

Basically, \ref{def:coh_prob_2}--\ref{def:coh_prob_6} are just the standard rules of probability. However, there are four rather subtle differences with the more traditional approach. The first difference is that a full conditional probability takes conditional probabilities as its basic entities: $P(A\vert C)$ is well-defined even if $P(C)=0$. The second difference, which is related to the first, is that Bayes's rule---\ref{def:coh_prob_6}---is stated in a multiplicative form; it is not regarded as a definition of conditional probabilities, but rather as a property that connects conditional probabilities to unconditional ones. The third difference is that we consider all events, and do not restrict ourselves to some specific subset of events---such as a $\sigma$-algebra. The fourth difference, which is related to the third, is that we only require finite additivity---\ref{def:coh_prob_3}---and do not impose $\sigma$-additivity.

The `full' in full conditional probability refers to the fact that the domain of $P$ is the complete set $\power\times\nonemptypower$. At first sight, this might seem unimportant, and one might be inclined to introduce a similar definition for functions $P$ whose domain is some subset $\mathcal{C}$ of $\power\times\nonemptypower$. However, unfortunately, as our next example illustrates, such a definition would have the property that it does not guarantee the possibility of extending the function to a larger domain $\mathcal{C}^*$, with $\mathcal{C}\subseteq\mathcal{C}^*\subseteq\power\times\nonemptypower$.

\begin{exmp}\label{exmp:F1F4cannotbeextended}
Let $\Omega=\{1,2,3,4,5,6\}$ be the set of possible values for the throw of a---possibly unfair---die and let $\mathcal{C}\coloneqq\{(E_{\mathrm{o}},\Omega),(E_{\mathrm{e}},\Omega)\}\subseteq\power\times\nonemptypower$, where the events $E_{\mathrm{o}}=\{1,3,5\}$ and $E_{\mathrm{e}}=\{2,4,6\}$ correspond to an odd or even outcome of the die throw, respectively. The map $P\colon\mathcal{C}\to\reals$ that is defined by
\begin{equation*}
P(E_{\mathrm{o}})\coloneqq P(E_{\mathrm{o}}\vert\Omega)=\nicefrac{2}{3}
~\text{ and }~
P(E_{\mathrm{e}})\coloneqq P(E_{\mathrm{e}}\vert\Omega)=\nicefrac{2}{3}
\end{equation*}
then satisfies \ref{def:coh_prob_2}-\ref{def:coh_prob_6} on its domain. However, if we extend the domain by adding the trivial couple $(\Omega,\Omega)$, it becomes impossible to satisfy \ref{def:coh_prob_2}-\ref{def:coh_prob_6}, because~\ref{def:coh_prob_1} and~\ref{def:coh_prob_3} would then require that
\begin{equation*}
1=P(\Omega\vert\Omega)=P(E_{\mathrm{o}}\vert\Omega)+P(E_{\mathrm{e}}\vert\Omega)=\nicefrac{2}{3}+\nicefrac{2}{3}=\nicefrac{4}{3},
\end{equation*}
which is clearly a contradiction.
\exampleend
\end{exmp}

In order to avoid the situation in this example, that is, in order to guarantee the possibility of extending the domain of a conditional probability in a sensible way, we use the concept of coherence~\cite{berti1991coherent,DeFinetti:oT_PWtAE,regazzini1985finitely,williams1975,Williams:2007eu}.

\begin{definition}[Coherent conditional probability]\label{def:coherence}
Let $P$ be a real-valued map from $\mathcal{C}\subseteq\power\times\nonemptypower$ to $\reals$. Then $P$ is said to be a \emph{coherent conditional probability} on $\mathcal{C}$ if, for all $n\in\mathbb{N}$ and every choice of $(A_i,C_i)\in\mathcal{C}$ and $\lambda_i\in\reals$, $i\in\{1,\dots,n\}$,\footnote{Many authors replace the maximum in this expression by a supremum, and also impose an additional inequality, where the maximum---supremum---is replaced by a minimum---infimum---and where the inequality is reversed~\cite{berti2002coherent,berti1991coherent,regazzini1985finitely}. This is completely equivalent to our definition. First of all, if the maximum is replaced by a supremum, then since $n$ is finite and because, for every $i\in\{1,\dots,n\}$, $\ind{A_i}$ and $\ind{C_i}$ can only take two values---$0$ or $1$---it follows that this supremum is taken over a finite set of real numbers, which implies that it is actually a maximum. Secondly, replacing the maximum by a minimum and reversing the inequality is equivalent to replacing the $\lambda_i$ in our expression by their negation, which is clearly allowed because the coefficients $\lambda_i$ can take any arbitrary real value.}
\begin{equation*}
\max\left\{\sum_{i=1}^n\lambda_i\ind{C_i}(\omega)\bigl(P(A_i\vert C_i)-\ind{A_i}(\omega)\bigr)~\Bigg\vert~\omega\in C_0\right\}\geq0,
\end{equation*}
with $C_0\coloneqq\cup_{i=1}^nC_i$.
\end{definition}
The interested reader is invited to take a look at Appendix~\ref{app:coherence}, where we provide this abstract concept with an intuitive gambling interpretation. However, for our present purposes, this interpretation is not required. Instead, our motivation for introducing coherence stems from the following two results. First, if $\mathcal{C}=\power\times\nonemptypower$, then coherence is equivalent to the axioms of probability, that is, properties~\ref{def:coh_prob_2}--\ref{def:coh_prob_6}.
\begin{theorem}{\cite[Theorem 3]{regazzini1985finitely}}\label{theo:fullcoherent}
Let $P$ be a real-valued map from $\power\times\nonemptypower$ to $\reals$. Then $P$ is a coherent conditional probability if and only if it is a full conditional probability.
\end{theorem}

Secondly, for coherent conditional probabilities on arbitrary domains, it is always possible to extend their domain while preserving coherence.

\begin{theorem}{\cite[Theorem 4]{regazzini1985finitely}}\label{theo:largerdomain}
Let $P$ be a coherent conditional probability on $\mathcal{C}\subseteq\power\times\nonemptypower$. Then for any $\mathcal{C}\subseteq\mathcal{C}^*\subseteq\power\times\nonemptypower$, $P$ can be extended to a coherent conditional probability on $\mathcal{C}^*$.
\end{theorem}
In particular, it is therefore always possible to extend a coherent conditional probability $P$ on $\mathcal{C}$, to a coherent conditional probability on $\power\times\nonemptypower$. Due to Theorem~\ref{theo:fullcoherent}, this extension is a full conditional probability. The following makes this explicit.

\begin{restatable}{corollary}{corolcoherentextendable}
\label{corol:coherentextendable}
Let $P$ be a real-valued map from $\mathcal{C}\subseteq\power\times\nonemptypower$ to $\reals$. Then $P$ is a coherent conditional probability if and only if it can be extended to a full conditional probability.
\end{restatable}
Note, therefore, that if $P$ is a coherent conditional probability on $\mathcal{C}$, we can equivalently say that it is the restriction of a full conditional probability. Hence, any coherent conditional probability on $\mathcal{C}$ is guaranteed to satisfy properties~\ref{def:coh_prob_2}-\ref{def:coh_prob_6}. However, as was essentially already illustrated in Example~\ref{exmp:F1F4cannotbeextended}, and as our next example makes explicit, the converse is not true.

\begin{exmp}
Let $\Omega$, $\mathcal{C}$, $E_{\mathrm{o}}$, $E_{\mathrm{e}}$ and $P\colon\mathcal{C}\to\reals$ be defined as in Example~\ref{exmp:F1F4cannotbeextended}. Then as we have seen in that example, $P$ satisfies \ref{def:coh_prob_2}-\ref{def:coh_prob_6} on its domain $\mathcal{C}$. However, $P$ is not a coherent conditional probability on $\mathcal{C}$, because if it was, then according to Corollary~\ref{corol:coherentextendable}, $P$ could be extended to a full conditional probability. Since $\power\times\nonemptypower$ includes $(\Omega,\Omega)$, the argument at the end of Example~\ref{exmp:F1F4cannotbeextended} implies that this is impossible.
A similar conclusion can be reached by verifying Definition~\ref{def:coherence} directly; we leave this as an exercise.
\exampleend
\end{exmp}

\subsection{Stochastic Processes as a Special Case}\label{sec:def_stochastic_processes}

A (continuous-time) stochastic process is now simply a coherent conditional probability on a specific domain $\mathcal{C}^\mathrm{SP}$, or equivalently, the restriction of a full conditional probability to this domain $\mathcal{C}^\mathrm{SP}$---see Definition~\ref{def:stoch_process}. However, before we get to this definition, let us first provide some intuition.

Basically, a continuous-time stochastic process describes the behaviour of a system as it moves through the---finite---state space $\states$ over a continuous time line $\realsnonneg$. A single realisation of this movement is called a path or a trajectory. We are typically uncertain about the specific path that will be followed, and a stochastic process quantifies this uncertainty by means of a probabilistic model, which, in our case, will be a coherent conditional probability. These ideas are formalised as follows.

A \emph{path} $\omega$ is a function from $\realsnonneg$ to $\states$, and we denote with $\omega(t)$ the value of $\omega$ at time $t$. For any sequence of time points $u\in\mathcal{U}$ and any path $\omega$, we will write $\omega\vert_{u}$ to denote the restriction of $\omega$ to $u\subset\realsnonneg$. Using this notation, we write for any $x_u\in\states_u$ that $\omega\vert_u=x_u$ if, for all $t\in u$, it holds that $\omega(t)=x_{t}$.

The outcome space $\Omega$ of a stochastic process is a set of paths. Three commonly considered choices are to let $\Omega$ be the set of all paths, the set of all right-continuous paths~\cite{norris1998markov}, or the set of all cadlag paths (right-continuous paths with left-sided limits)~\cite{williams2000}. However, our results do not require such a specific choice. For the purposes of this paper, all that we require is that
\begin{equation}\label{eq:path_exists_for_finite_points}
(\forall u\in\mathcal{U}_\emptyset)(\forall x_u\in\states_u)(\exists \omega\in\Omega)~\omega\vert_u=x_u\,.
\end{equation}
Thus, $\Omega$ must be chosen in such a way that, for any non-empty finite sequence of time points $u\in\mathcal{U}_\emptyset$ and any state assignment $x_u\in\states_u$ on those time points, there is at least one path $\omega\in\Omega$ that agrees with $x_u$ on $u$. Essentially, 
this condition
guarantees that $\Omega$ is ``large enough to be interesting''. It serves as a nice exercise to check that the three specific sets $\Omega$ in the beginning of this paragraph each satisfy this condition.

For any set of events $\mathcal{E}\subseteq\power$, we use $\langle\mathcal{E}\rangle$ to denote the algebra that is generated by them. That is, $\langle\mathcal{E}\rangle$ is the smallest subset of $\power$ that contains all elements of $\mathcal{E}$, and that is furthermore closed under complements in $\Omega$ and finite unions, and therefore also under finite intersections. 
Furthermore, for any $t\in\realsnonneg$ and $x\in\states$, we define the elementary event
\begin{equation*}
(X_t=x)\coloneqq\{\omega\in\paths\colon\omega(t)=x\}\,,
\end{equation*}
and, for any $u\in\mathcal{U}$, we let
\begin{equation*}
\mathcal{E}_u \coloneqq \left\{
(X_t=x)
\colon
x\in\states,t\in u\cup\reals_{>u}
\right\}
\end{equation*}
be the set of elementary events whose time point is either preceded by or belongs to $u$, and we let $\mathcal{A}_u\coloneqq\langle\mathcal{E}_u\rangle$ be the algebra that is generated by this set of elementary events. 

Consider now any $u\in\mathcal{U}$. Then on the one hand, for any $A\in\mathcal{A}_u$, it clearly holds that $A\in\power$. On the other hand, for any $x_u\in\states_u$, the event
\begin{equation*}
(X_u=x_u) \coloneqq \{\omega\in\Omega\,:\,\omega\vert_u=x_u\}
\vspace{4pt}
\end{equation*}
belongs to $\nonemptypower$, because it follows from Equation~\eqref{eq:path_exists_for_finite_points} that this event is non-empty. 
Hence, for any $A\in\mathcal{A}_u$ and $x_u\in\states_u$, we find that $(A,X_u=x_u)\in\power\times\nonemptypower$. Since this is true for every $u\in\mathcal{U}$, it follows that
\begin{equation*}
\mathcal{C}^\mathrm{SP}\coloneqq\big\{
(A,X_u=x_u)
\colon
u\in\mathcal{U},~x_u\in\states_u,~A\in\mathcal{A}_u\big\}
\end{equation*}
is a subset of $\power\times\nonemptypower$. It is also worth noting that if $u=\emptyset$, then $\omega\vert_u=x_u$ is vacuously true, which implies that in that case, $(X_u=x_u)=\Omega$.

That being said, we can now finally formalise our definition of a (continuous-time) stochastic process.

\begin{definition}[Stochastic Process]\label{def:stoch_process}
A \emph{stochastic process} is a coherent conditional probability on $\mathcal{C}^\mathrm{SP}$. We denote the set of all stochastic processes by $\processes$.
\end{definition}

\begin{restatable}{corollary}{corolprocessiffrestriction}
\label{corol:processiffrestriction}
Let $P$ be a real-valued map from $\mathcal{C}^\mathrm{SP}$ to $\reals$. Then $P$ is a stochastic process if and only if it is the restriction of a full conditional probability.
\end{restatable}

There are two reason why we restrict ourselves to the domain $\mathcal{C}^\mathrm{SP}$. The most important reason is simply that all of the results in this paper can be expressed using only this domain, because they are all concerned with events or functions that depend on a finite number of time points. The second reason is that this restriction will allow us to state  uniqueness results that do not extend to larger domains; see for example~Corollary~\ref{cor:rate_has_unique_homogen_markov_process}. However, it is important to realise that our restriction of the domain does not impose any real limitations, because, as we know from Theorem~\ref{theo:largerdomain}, the domain of a coherent conditional probability---and hence also a stochastic process---can always be extended (although not uniquely). In fact, as briefly touched upon in the conclusions of this paper, it is even possible to consider extensions that are $\sigma$-additive.

Finally, we would like to point out that it is also possible to use a different---yet equivalent---definition for stochastic processes. Indeed, due to Corollary~\ref{corol:processiffrestriction}, a stochastic process can also be defined as the restriction of a full conditional probability to $\mathcal{C}^{\mathrm{SP}}$. Of course, given this realisation, one may then start to wonder why we have gone through to the trouble of introducing coherence, because this alternative definition would not require any notion of coherence. The reason why we nevertheless need coherence, is because it allows us to establish the existence of stochastic processes that have certain properties, which will often be necessary in the proofs of our results.

For instance, by means of an example, suppose that we are given an arbitrary function $p:\realspos\to [0,1]$, and that we want to know if there is a stochastic process $P$ for which, for some $x,y\in\states$,
\begin{equation}\label{eq:extendptoP}
P(X_t=y\,\vert\,X_0=x) = p(t)
~\text{ for all $t\in\realspos$.}
\end{equation}
If we had defined a stochastic process as the restriction of a full conditional probability to $\mathcal{C}^\mathrm{SP}$, without introducing coherence, then answering this question would have been entirely non-trivial, because it would essentially require us to construct a full conditional probability that coincides with $p$ on the relevant part of its domain.

In contrast, as illustrated by the following example, the introduction of coherence takes care of most of the heavy lifting in such an existence proof.
\begin{exmp}\label{exmp:coherence_constructs_process}
Let $\states$ be a state space that contains at least two states, fix two---possibly equal---states $x,y\in\states$, and consider any function $p:\realspos\to[0,1]$. The aim of this example is to prove that there is a stochastic process $P$ that satisfies Equation~\eqref{eq:extendptoP}.

The crucial step of the proof is to consider a function $\tilde{P}$ that is defined by
\begin{equation}\label{exmp:coherence:eq:probdef}
\tilde{P}(X_t=y\vert X_0=x) \coloneqq p(t)
~\text{ for all $(X_t=y,X_0=x)\in\mathcal{C}$},
\end{equation}
and to prove that this function is a coherent conditional probability on
\begin{equation*}
\mathcal{C} \coloneqq \left\{ (X_t=y,X_0=x)\colon t\in\realspos \right\},
\end{equation*}
or equivalently, that it satisfies Definition~\ref{def:coherence}.

So consider any $n\in\nats$ and, for all $i\in\{1,\ldots,n\}$, some $(X_{t_i}=y,X_0=x)\in\mathcal{C}$ and $\lambda_i\in\reals$. According to Definition~\ref{def:coherence}, we now have to show that
\begin{equation}\label{exmp:coherence:eq:coherence}
\max\left\{ \sum_{i=1}^n \lambda_i\ind{x}(\omega(0))\left(\tilde{P}(X_{t_i}=y\vert X_0=x) - \ind{y}(\omega(t_i))\right)~\Bigg\vert~ \omega\in C_0 \right\} \geq 0,
\end{equation}
with $C_0\coloneqq \bigcup_{i=1}^n (X_0=x)=(X_0=x)$. 
While doing so, we can assume without loss of generality that $i\neq j$ implies $t_i\neq t_j$, because if $t_i=t_j$ for some $i\neq j$, then we can simply add the corresponding two summands in Equation~\eqref{exmp:coherence:eq:coherence}.

Let $z\in\states$ be any state such that $z\neq y$, let $u\coloneqq(0,t_1,\ldots,t_n)\in\mathcal{U}$, and let $x_u\in\states_u$ be the unique state assignment such that $x_{0}\coloneqq x$ and
\begin{equation*}
x_{t_i}\coloneqq
\begin{cases}
y & \text{if $\lambda_i<0$} \\
z & \text{if $\lambda_i\geq0$}
\end{cases}
~~\text{ for all $i\in\{1,\ldots,n\}$.}
\end{equation*}
Furthermore, let $N_{<0}\coloneqq\{i\in\{1,\dots,n\}\colon \lambda_i<0\}$ and $N_{\geq0}\coloneqq\{i\in\{1,\dots,n\}\colon \lambda_i\geq0\}$.
Since $n$ is finite, Equation~\eqref{eq:path_exists_for_finite_points} now guarantees that there is some $\omega\in\Omega$ such that $\omega\vert_{u}=x_u$. Evaluating the sum in Equation~\eqref{exmp:coherence:eq:coherence} using this $\omega$, we find that
\begin{multline*}
 \sum_{i=1}^n \lambda_{i}\ind{x}(\omega(0))\left(\tilde{P}(X_{t_i}=y\vert X_0=x) - \ind{y}(\omega(t_i))\right)\\[-4pt]
 \begin{aligned}
 &= \sum_{i=1}^n \lambda_{i}\left(\tilde{P}(X_{t_i}=y\vert X_0=x) - \ind{y}(\omega(t_i))\right) \\
 &= \sum_{i\in N_{<0}}\lambda_{i}\left(\tilde{P}(X_{t_i}=y\vert X_0=x) - 1\right) + \sum_{i\in N_{\geq0}}\lambda_{i}\tilde{P}(X_{t_i}=y\vert X_0=x) \\
 &\geq \sum_{i\in N_{<0}}\lambda_{i}\left(\tilde{P}(X_{t_i}=y\vert X_0=x) - 1\right) = \sum_{i\in N_{<0}}\abs{\lambda_{i}}\left(1 - \tilde{P}(X_{t_i}=y\vert X_0=x)\right) \geq 0,
 \end{aligned}
\end{multline*}
where the two inequalities follow from the fact that $\tilde{P}(X_{t_i}=y\vert X_0=x)=p(t_i)\in [0,1]$\,.
Furthermore, because $\omega(0)=x$, we also have that $\omega\in (X_0=x)=C_0$. Therefore, we find that Equation~\eqref{exmp:coherence:eq:coherence} indeed holds. Hence, we conclude that $\tilde{P}$ is a coherent conditional probability on $\mathcal{C}$.

The rest of the proof is now straightforward. Since $\tilde{P}$ is a coherent conditional probability on $\mathcal{C}$, and because $\mathcal{C}$ is a subset of $\mathcal{C}^{\mathrm{SP}}$, it follows from Theorem~\ref{theo:largerdomain} that $\tilde{P}$ can be extended to a coherent conditional probability $P$ on $\mathcal{C}^{\mathrm{SP}}$, or equivalently, to a stochastic process $P$. Since this stochastic process $P$ is an extension of $\tilde{P}$, Equation~\eqref{eq:extendptoP} is now an immediate consequence of Equation~\eqref{exmp:coherence:eq:probdef}.
\exampleend
\end{exmp}

\subsection{Well-behaved stochastic processes}\label{sec:well_behaved}

Stochastic processes, as they are defined in the previous section, can behave in rather extreme ways. For instance, Example~\ref{exmp:coherence_constructs_process} tells us that for any two states $x,y\in\mathcal{X}$, there exists a stochastic process $P$ such that $P(X_t=y\vert X_0=x)=\ind{\mathbb{Q}_{>0}}(t)$, where $\ind{\mathbb{Q}_{>0}}$ is the indicator of the positive rational numbers.

In order to avoid this kind of extreme behaviour, we will require that the rate of change of a stochastic processes remains bounded. We formalise this requirement through the notion of \emph{well-behavedness}. 

\begin{definition}[Well-Behaved Stochastic Process]
\label{def:well-behaved}
A stochastic process $P\in\processes$ is said to be \emph{well-behaved} if, for any---possibly empty---time sequence $u\in\mathcal{U}$, any $x_u\in\states_u$, any $x,y\in\states$ and any $t\in\reals_{\geq0}$ such that $t>u$:
\begin{equation}\label{eq:def:well-behaved:right}
\limsup_{\Delta\to 0^{+}}\frac{1}{\Delta}\abs{P(X_{t+\Delta}=y\vert X_t=x, X_u=x_u)-\ind{x}(y)}<+\infty
\end{equation}
and, if $t\neq0$,
\begin{equation}\label{eq:def:well-behaved:left}
\limsup_{\Delta\to 0^{+}}\frac{1}{\Delta}\abs{P(X_{t}=y\vert X_{t-\Delta}=x, X_u=x_u)-\ind{x}(y)}<+\infty.
\end{equation}
The set of all well-behaved stochastic processes is denoted by $\wprocesses$.
\end{definition}

This definition of well-behavedness is related to continuity and differentiability, but stronger than the former and weaker than the latter. Our next example provides some intuition on this.

\begin{exmp}\label{exmp:well-behaved}
Let $\states$ be a state space that contains at least two states, fix two states $x,y\in\states$ such that $x\neq y$, and consider any function $p:\reals_{>0}\to[0,1]$. Then as we know from Example~\ref{exmp:coherence_constructs_process}, there is a stochastic process $P$ such that
\begin{equation*}
P(X_\Delta=y\vert X_0=x)=p(\Delta)
~\text{ for all $\Delta\in\realspos$.}
\end{equation*}
Furthermore, since $x\neq y$, it follows from~\ref{def:coh_prob_7} and \ref{def:coh_prob_8} that $P(X_0=y\vert X_0=x)=0$. We now consider two specific choices for $p$.

If we let $p(\Delta)\coloneqq\sqrt\Delta$ for $\Delta\in(0,1]$ and $p(\Delta)\coloneqq 1$ for $\Delta\geq1$, then $P(X_\Delta=y\vert X_0=x)$ is continuous on $\reals_{\geq0}$ because $P(X_0=y\vert X_0=x)=0$. However, we also find that
\begin{equation*}
\limsup_{\Delta\to0^+}\frac{1}{\Delta}\abs{P(X_{\Delta}=y\vert X_{0}=x) - \ind{x}(y)}
=
\limsup_{\Delta\to0^+}\frac{1}{\Delta}\sqrt\Delta
=+\infty
\end{equation*}
and therefore, it follows from Equation~\eqref{eq:def:well-behaved:right}---with $t=0$ and $u=\emptyset$---that $P$ is not well-behaved.

On the other hand, if we let $p(\Delta)\coloneqq \Delta\abs{\sin(\nicefrac{1}{\Delta})}$ for $\Delta\in\realspos$, we find that
\begin{equation*}
\limsup_{\Delta\to0^+}\frac{1}{\Delta}\abs{P(X_{\Delta}=y\vert X_{0}=x) - \ind{x}(y)}
=
\limsup_{\Delta\to0^+}\abs{\sin(\nicefrac{1}{\Delta})}
=1.
\end{equation*}
In this case---at least for $t=0$ and $u=\emptyset$---$P$ does exhibit the behaviour that we associate with a well-behaved stochastic process. Furthermore, as we invite the reader to check, $P(X_\Delta=y\vert X_0=x)$ is again continuous on $\reals_{\geq0}$. However, $P(X_\Delta=y\vert X_0=x)$ is not differentiable in $\Delta=0$, because $\nicefrac{1}{\Delta}P(X_\Delta=y\vert X_0=x)=\abs{\sin(\nicefrac{1}{\Delta})}$ oscillates as $\Delta$ approaches zero.
\exampleend
\end{exmp}

Definition~\ref{def:well-behaved} is also very similar to the definition of a well-behaved transition matrix system, as introduced in Section~\ref{subsec:well_behaved_systems}. Furthermore, here too, one could say that well-behavedness is intuitively related to local Lipschitz continuity. However, note that this time, this locality not only depends on the time point $t$, but also on the given history $x_u$.

\subsection{Process Dynamics}\label{sec:dynamics}

We end this section by introducing some tools to describe the behaviour of stochastic processes. Rather than work with the individual probabilities, it will be convenient to jointly consider probabilities that are related by the same conditioning event. To this end, we will next introduce the notion of transition matrices corresponding to a given stochastic process.

\begin{definition}[Corresponding Transition Matrix]\label{def:trans_matrix}
Consider any stochastic process $P\in\processes$. Then, for any $t,s\in\realsnonneg$ such that $t\leq s$, the \emph{corresponding transition matrix} $T_t^s$ is a matrix that is defined by
\begin{equation*}
T_t^s(x_t, x_s) \coloneqq P(X_s=x_s\,\vert X_t=x_t)\quad\text{for all $x_s,x_t\in\states$}\,.
\end{equation*}
We denote this family of matrices by $\mathcal{T}_P$.
\end{definition}

Because we will also want to work with conditioning events that contain more than a single time point, we furthermore introduce the following generalisation.

\begin{definition}[History-Dependent Corresponding Transition Matrix]
Let $P\in\processes$ be any stochastic process. Then, for any $t,s\in\realsnonneg$ such that $t\leq s$, any sequence of time points $u\in\mathcal{U}_{<t}$, and any state assignment $x_u\in\states_u$, the corresponding \emph{history-dependent} transition matrix $T_{t,\,x_u}^s$ is a matrix that is defined by
\begin{equation*}
T^s_{t,\,x_u}(x_t,x_s)
\coloneqq
P(X_s=x_s\vert X_t=x_t, X_u=x_u)\quad\text{for all $x_s,x_t\in\states$}\,.
\end{equation*}
For notational convenience, we allow $u$ to be empty, in which case $T_{t,\,x_u}^s=T_t^s$.
\end{definition}

The following proposition establishes some simple properties of these corresponding (history-dependent) transition matrices.

\begin{restatable}{proposition}{propstochasticprocesssimpleproperties}
\label{prop:stochasticprocess:simpleproperties}
Let $P\in\processes$ be a stochastic process.  
Then, for any $t,s\in\realsnonneg$ such that $t\leq s$, any sequence of time points $u\in\mathcal{U}_{<t}$, and any state assignment $x_u\in\states_u$, the corresponding (history dependent) transition matrix $T_{t,\,x_u}^s$ is---as its name suggests---a transition matrix, and $T_{t,\,x_u}^t=I$. Furthermore, $P$ is well-behaved if and only if, for every---possibly empty---time sequence $u\in\mathcal{U}$, any $x_u\in\states_u$ and any $t\in\reals_{\geq0}$ such that $t>u$:
\begin{equation}\label{eq:def:well-behaved:right:matrix}
\limsup_{\Delta\to 0^{+}}\frac{1}{\Delta}\norm{T_{t,\,x_u}^{t+\Delta}-I}<+\infty
\end{equation}
and, if $t\neq0$,
\begin{equation}\label{eq:def:well-behaved:left:matrix}
\limsup_{\Delta\to 0^{+}}\frac{1}{\Delta}\norm{T_{t-\Delta,x_u}^{t}-I}<+\infty.
\end{equation}
\vspace{-4pt}
\end{restatable}

\begin{remark}\label{remark:expectationT}
Note that for any $P\in\processes$, the corresponding transition matrix $T_{t, x_u}^s$ is a map from $\gamblesX$ to $\gamblesX$, that can therefore be applied to any $f\in\gamblesX$. Furthermore, for any $x_t\in\states$, we have that
\begin{align*}
\left[T_{t,x_u}^sf\right](x_t) &= \sum_{x_s\in\states}f(x_s)P(X_s=x_s\,\vert\,X_t=x_t,X_u=x_u)= \mathbb{E}_P\left[f(X_s)\,\vert\,X_t=x_t, X_u=x_u\right]\,,
\end{align*}
where the expectation is taken with respect to $P(X_s\,\vert\,X_t=x_t,X_u=x_u)$. This observation will be useful when we later focus on expectations with respect to stochastic processes; for functions $f\in\gamblesX$, their corresponding transition matrices serve as an alternative representation for the expectation operator. 
\exampleend
\end{remark}

Because a stochastic process is defined on a continuous time line, and because its corresponding transition matrices $T_{t,x_u}^s$ only describe the behaviour of this process on fixed points in time, we will furthermore require some tools to capture the dynamics of a stochastic process. That is, we will be interested in how their transition matrices change over time.

One seemingly obvious way to describe these dynamics is to use the derivatives of the transition matrices that correspond to stochastic processes. However, because we do not impose differentiability assumptions on these processes, such derivatives may not exist. We will therefore instead introduce \emph{outer partial derivatives} below. It will be instructive, however, to first consider ordinary \emph{directional partial derivatives}.

\begin{definition}[Directional Partial Derivatives]\label{def:direc_partial_deriv}
For any stochastic process $P\in\processes$, any $t\in\realsnonneg$, any sequence of time points $u\in\mathcal{U}_{<t}$, and any state assignment $x_u\in\states_u$, the \emph{right-sided partial derivative} of $T_{t,x_u}^t$ is defined by
\begin{equation*}
\partial_{+}{T_{t,\,x_u}^t}
\coloneqq
\lim_{\Delta\to 0^{+}}
\frac{1}{\Delta}
(T^{t+\Delta}_{t,\,x_u}-T^t_{t,\,x_u})
=
\lim_{\Delta\to 0^{+}}
\frac{1}{\Delta}
(T^{t+\Delta}_{t,\,x_u}-I)
\end{equation*}
and, if $t\neq0$, the \emph{left-sided partial derivative} of $T_{t,x_u}^t$ is defined by
\begin{equation*}
\partial_{-}{T_{t,\,x_u}^t}
\coloneqq
\lim_{\Delta\to 0^{+}}
\frac{1}{\Delta}
(T^{t}_{t-\Delta,\,x_u}-T^t_{t,\,x_u})
=
\lim_{\Delta\to 0^{+}}
\frac{1}{\Delta}
(T^{t}_{t-\Delta,\,x_u}-I).
\end{equation*}
If these partial derivatives exist, then because of Proposition~\ref{prop:rate_from_stochastic_matrix}, they are guaranteed to belong to the set of rate matrices  $\mathcal{R}$. If they both exist and coincide, we write $\partial{T_{t,\,x_u}^t}$ to denote their common value. If $t=0$, we let $\partial{T_{t,\,x_u}^t}\coloneqq\partial_{+}{T_{t,\,x_u}^t}$.
\end{definition}

The following example establishes that these directional partial derivatives need not exist. In particular, they need not exist even for well-behaved processes.
\begin{exmp}\label{exmp:well-behaved-no-deriv}
Let $Q_1,Q_2\in\mathcal{R}$ be two commuting rate matrices such that $Q_1\neq Q_2$---for example, let $Q_1\neq0$ be an arbitrary rate matrix and let $Q_2\coloneqq\alpha Q_1$, with $\alpha\in\realsnonneg\setminus\{1\}$---and consider a well-behaved stochastic process $P\in\wprocesses$ of which, for all $t\in(0,1]$, the transition matrix $T_0^t$ is given by Equation~\eqref{eq:ex4:def} in Example~\ref{exmp:limit_trans_mat_system_matrices}. 

For now, we simply assume that this is possible. A formal proof for the existence of such a process requires some additional machinery, and we therefore postpone it to Example~\ref{exmp:twoexamplesofMarkovchains}, where we construct a well-behaved continuous-time Markov chain that is compatible with Equation~\eqref{eq:ex4:def}.

The aim of the present example is to show that for any such process, the right-sided partial derivative $\partial_{+}{T_{0}^0}$---which corresponds to choosing $t=0$ and $u=\emptyset$ in Definition~\ref{def:direc_partial_deriv}---does not exist. The reason for this is that---as is proved in~\exampleproofref---for any $\lambda\in[\nicefrac{1}{3},\nicefrac{2}{3}]$, there is a sequence $\{\Delta_i\}_{i\in\nats_0}\to0^+$ such that
\begin{equation}\label{eq:exmp:well-behaved-no-deriv}
\lim_{i\to+\infty}
\frac{1}{\Delta_i}
(T^{\Delta_i}_{0}-I)=Q_\lambda
\end{equation}
with $Q_\lambda\coloneqq \lambda Q_1+(1-\lambda)Q_2$. The reason why this indeed implies that $\partial_{+}{T_{0}^0}$ does not exist, is because if it would exist, then Equation~\eqref{eq:exmp:well-behaved-no-deriv} would imply that $\partial_{+}{T_{0}^0}=Q_\lambda$ for all $\lambda\in[\nicefrac{1}{3},\nicefrac{2}{3}]$. The only way for this to be possible would be if $Q_1=Q_2$, but that case was excluded in the beginning of this example.
\exampleend
\end{exmp}

Observe, therefore, that the problem is essentially that the finite-difference expressions $\nicefrac{1}{\Delta}(T_{t,x_u}^{t+\Delta} - I)$ and $\nicefrac{1}{\Delta}(T_{t-\Delta,x_u}^{t} - I)$, parameterised in $\Delta$, can have multiple accumulation points as we take $\Delta$ to $0$. Therefore, it will be more convenient to instead work with what we call \emph{outer partial derivatives}. These can be seen as a kind of set-valued derivatives, containing all these accumulation points obtained as $\Delta\to0^+$.

\begin{definition}[Outer Partial Derivatives]\label{def:outerpartialderivatives}
For any stochastic process $P\in\processes$, any $t\in\realsnonneg$, any sequence of time points $u\in\mathcal{U}_{<t}$, and any state assignment $x_u\in\states_u$, the \emph{right-sided outer partial derivative} of $T_{t,x_u}^t$ is defined by
\begin{equation}
\overline{\partial}_{+}
{T^t_{t,\,x_u}}
\coloneqq
\left\{
Q\in\mathcal{R}
\colon
\left(\exists \,\{\Delta_i\}_{i\in\nats}\to0^+\,:\,
~
\lim_{i\to+\infty}
\frac{1}{\Delta_i}
(T^{t+\Delta_i}_{t,\,x_u}-I)
=Q
\right)
\right\}\label{eq:rightouterderivative}
\end{equation}
and, if $t\neq0$, the \emph{left-sided} \emph{outer partial derivative} of $T_{t,x_u}^t$ is defined by
\begin{equation}
\overline{\partial}_{-}
{T^t_{t,\,x_u}}
\coloneqq
\left\{
Q\in\mathcal{R}
\colon
\left(\exists\, \{\Delta_i\}_{i\in\nats}\to0^+\,:\,
~
\lim_{i\to+\infty}
\frac{1}{\Delta_i}
(T^{t}_{t-\Delta_i,\,x_u}-I)
=Q
\right)\label{eq:leftouterderivative}
\right\}.
\end{equation}
Furthermore, the \emph{outer partial derivative} of $T_{t,x_u}^t$ is defined as
\begin{equation*}
\overline{\partial}
{T^t_{t,\,x_u}}
\coloneqq
\overline{\partial}_{+}
{T^t_{t,\,x_u}}
\cup
\overline{\partial}_{-}
{T^t_{t,\,x_u}}
\text{ if $t>0$ and }
\overline{\partial}
{T^t_{t,\,x_u}}
\coloneqq
\overline{\partial}_{+}
{T^t_{t,\,x_u}}
\text{ if $t=0$}.
\end{equation*}\vspace{-10pt}
\end{definition}

For well-behaved processes $P\in\wprocesses$, as our next result shows, these outer partial derivatives are always non-empty, bounded and closed.

\begin{restatable}{proposition}{propboundednonemptyandclosed}
\label{prop:boundednon-emptyandclosed}
Consider any $P\in\wprocesses$. Then $\overline{\partial}_{+}
{T^t_{t,\,x_u}}$, $\overline{\partial}_{-}
{T^t_{t,\,x_u}}$ and $\overline{\partial}
{T^t_{t,\,x_u}}$ are non-empty, bounded and closed subsets of $\mathcal{R}$.
\end{restatable}

The following two examples provide this result with some intuition. Example~\ref{exmp:outerderivative} illustrates the validity of the result, while Example~\ref{exmp:emptyouterderivative} shows that the requirement that $P$ must be well-behaved is essential for the result to be true.

\begin{exmp}\label{exmp:outerderivative}
Consider again the well-behaved stochastic process $P\in\wprocesses$ from Example~\ref{exmp:well-behaved-no-deriv} of which, for all $t\in(0,1]$, the transition matrix $T_0^t$ is given by Equation~\eqref{eq:ex4:def}. 
As proved in~\exampleproofref, it holds for this particular process that
\begin{equation*}
\smash{\overline{\partial}}_+T_0^0=\big\{Q_\lambda\colon \lambda\in[\nicefrac{1}{3},\nicefrac{2}{3}]\big\},
\end{equation*}
where, for every $\lambda\in[\nicefrac{1}{3},\nicefrac{2}{3}]$, $Q_\lambda\coloneqq\lambda Q_1+(1-\lambda)Q_2$ as in Example~\ref{exmp:well-behaved-no-deriv}.
\exampleend
\end{exmp}

\begin{exmp}\label{exmp:emptyouterderivative}
Fix any rate matrix $Q\in\mathcal{R}$ such that $\norm{Q}=1$, let $\mathcal{T}$ be the transition matrix system of Example~\ref{exmp:notwellbehavedMarkov}, and consider any stochastic process $P\in\mathbb{P}$ of which the corresponding family of transition matrices $\mathcal{T}_P$ is equal to $\mathcal{T}$.

For now, we simply assume that such a process exists. A formal proof again requires some additional machinery---as in Example~\ref{exmp:well-behaved-no-deriv}---and we therefore postpone it to Example~\ref{exmp:twoexamplesofMarkovchains}, where we construct a continuous-time Markov chain whose family of transition matrices $\mathcal{T}_P$ is equal to the transition matrix system $\mathcal{T}$.

As we prove in~\exampleproofref, for such a stochastic process $P$, the right-sided outer partial derivative $\smash{\overline{\partial}}_+T_0^0$ is then empty.
\exampleend
\end{exmp}

We end this section with two additional properties of the outer partial derivatives of well-behaved stochastic processes. First, as we establish in our next result, they satisfy an $\epsilon-\delta$ expression that is similar to the limit expression of a partial derivative.

\begin{restatable}{proposition}{propouterderivativebehaveslikelimit}
\label{prop:outerderivativebehaveslikelimit}
Consider any well-behaved stochastic process $P\in\wprocesses$. Then, for any $t\in\realsnonneg$, any sequence of time points $u\in\mathcal{U}_{<t}$, any state assignment $x_u\in\states_u$, and any $\epsilon>0$, there is some $\delta>0$ such that, for all $0<\Delta<\delta$:
\begin{equation}
\label{eq:outerderivativebehaveslikelimit1}
(\exists Q\in\overline{\partial}_{+}
{T^t_{t,\,x_u}})
\norm{\frac{1}{\Delta}
(T^{t+\Delta}_{t,\,x_u}-I)-Q}<\epsilon
\end{equation}
and, if $t\neq0$,
\begin{equation}
\label{eq:outerderivativebehaveslikelimit2}
(\exists Q\in\overline{\partial}_{-}
{T^t_{t,\,x_u}})
\norm{\frac{1}{\Delta}
(T^{t}_{t-\Delta,\,x_u}-I)-Q}<\epsilon.
\end{equation}\vspace{-5pt}
\end{restatable}

Secondly, these outer partial derivatives are a proper generalisation of directional partial derivatives. In particular, if the latter exist, their values correspond exactly to the single element of the former.

\begin{restatable}{corollary}{coroloutersingleton}
\label{corol:outersingleton}
Consider any $P\in\wprocesses$. Then $\smash{\overline{\partial}_{+}
{T^t_{t,\,x_u}}}$ is a singleton if and only if $\smash{\partial_{+}
{T^t_{t,\,x_u}}}$ exists and, in that case, $\smash{\overline{\partial}_{+}
{T^t_{t,\,x_u}}}=\{\partial_{+}
{T^t_{t,\,x_u}}\}$. Analogous results hold for $\smash{\overline{\partial}_{-}
{T^t_{t,\,x_u}}}$ and $\smash{\partial_{-}
{T^t_{t,\,x_u}}}$, and for $\smash{\overline{\partial}
{T^t_{t,\,x_u}}}$ and $\smash{\partial
{T^t_{t,\,x_u}}}$.
\end{restatable}

\section{Continuous-Time Markov Chains}\label{sec:cont_time_markov_chains}

Having introduced continuous-time stochastic processes in Section~\ref{sec:stochastic_processes}, we will in this section focus on a specific class of such processes: \emph{continuous-time Markov chains}.

\begin{definition}[Markov Property, Markov Chain]\label{def:markov_property}
A stochastic process $P\in\processes$ satisfies the \emph{Markov property} if for any $t,s\in\realsnonneg$ such that $t\leq s$, any time sequence $u\in\mathcal{U}_{<t}$, any $x_u\in\states_u$, and any states $x,y\in\states$:
\begin{equation*}
P(X_s=y\,\vert\,X_t=x,X_u=x_u) = P(X_s=y\,\vert\, X_{t}=x)\,.
\end{equation*}
A stochastic process that satisfies this property is called a \emph{Markov chain}. We denote the set of all Markov chains by $\mprocesses$ and use $\wmprocesses$ to refer to the subset that only contains the well-behaved Markov chains.
\end{definition}

We already know from Proposition~\ref{prop:stochasticprocess:simpleproperties} that the transition matrices of a stochastic process---and therefore also, in particular, of a Markov chain---satisfy some simple properties. For the specific case of a Markov chain $P\in\mprocesses$, the family of transition matrices $\mathcal{T}_P$ also satisfies an additional property. In particular, for any $t,r,s\in\realsnonneg$ such that $t\leq r\leq s$, these transition matrices satisfy
\begin{equation}\label{eq:markovintermsofmatrices}
T_t^s = T_t^rT_r^s.
\end{equation}
In this context, this property is known as the \emph{Chapman-Kolmogorov equation} or the semi-group property~\cite{liggett2010continuous}. Indeed, this is the same semi-group property that we defined in Section~\ref{sec:systems} to hold for transition matrix systems $\mathcal{T}$. The following result should therefore not be surprising.

\begin{restatable}{proposition}{propMarkovhassystem}
\label{prop:Markovhassystem}
Consider a Markov chain $P\in\mprocesses$ and let $\mathcal{T}_P$ be the corresponding family of transition matrices. Then $\mathcal{T}_P$ is a transition matrix system. Furthermore, $\mathcal{T}_P$ is well-behaved if and only if $P$ is well-behaved.
\end{restatable}

At this point we know that every (well-behaved) Markov chain has a corresponding (well-behaved) transition matrix system. Our next result establishes that the converse is true as well: every (well-behaved) transition matrix system has a corresponding (well-behaved) Markov chain, and for a given initial distribution, this Markov chain is even unique.

\begin{restatable}{theorem}{theouniqueMarkovchain}
\label{theo:uniqueMarkovchain}
Let $p$ be any probability mass function on $\states$ and let $\mathcal{T}$ be a transition matrix system. Then there is a unique Markov chain $P\in\mprocesses$ such that $\mathcal{T}_P=\mathcal{T}$ and, for all $y\in\states$, $P(X_0=y)=p(y)$. Furthermore, $P$ is well-behaved if and only if $\mathcal{T}$ is well-behaved.
\end{restatable}

Hence, Markov chains---and well-behaved Markov chains in particular---are completely characterised by their transition matrices and their initial distribution. Our next example uses this result to formally establish the existence of the Markov chains that were used in Examples~\ref{exmp:well-behaved-no-deriv} and~\ref{exmp:emptyouterderivative}. Furthermore, it also illustrates that not every Markov chain is well-behaved.

\begin{exmp}\label{exmp:twoexamplesofMarkovchains}
For any transition matrix system~$\mathcal{T}$, it follows from Theorem~\ref{theo:uniqueMarkovchain}---with $p$ chosen arbitrarily---that there exists a continuous-time Markov chain $P\in\mprocesses\subseteq\processes$ such that $\mathcal{T}_P=\mathcal{T}$ and, furthermore, that $P$ is well-behaved if and only if $\mathcal{T}$ is well-behaved.

For example, for any rate matrix $Q\in\mathcal{R}$ such that $\norm{Q}=1$, if we let $\mathcal{T}$ be the transition matrix system of Example~\ref{exmp:notwellbehavedMarkov}, we find---as already claimed in Example~\ref{exmp:emptyouterderivative}---that there is a continuous-time Markov chain $P\in\mprocesses\subseteq\processes$ such that $\mathcal{T}_P=\mathcal{T}$. Furthermore, since we know from Example~\ref{exmp:notwellbehavedMarkov} that $\mathcal{T}$ is not well-behaved, it follows that $P$ is not well-behaved either.

As another example, for any two commuting rate matrices $Q_1,Q_2\in\mathcal{R}$, if we let $\mathcal{T}$ be the well-behaved transition matrix system of Example~\ref{exmp:limit_trans_mat_system_matrices}, we find---as already claimed in Example~\ref{exmp:well-behaved-no-deriv}---that there is a well-behaved continuous-time Markov chain $P\in\wmprocesses\subseteq\mprocesses\subseteq\processes$ such that, for all $t\in(0,1]$, the transition matrix $T_0^t$ is given by Equation~\eqref{eq:ex4:def} in Example~\ref{exmp:limit_trans_mat_system_matrices}.
\exampleend
\end{exmp}

As a final note, observe that not only does the Markov property simplify the conditional probabilities of a Markov chain, it also simplifies its dynamics. In particular, for any $t\in\realsnonneg$, any $u\in\mathcal{U}_{<t}$, and any $x_u\in\states_u$, it holds that $\smash{\overline{\partial}_-}T_{t,\,x_u}^t=\smash{\overline{\partial}_-}T_{t}^t$, $\smash{\overline{\partial}_+}T_{t,\,x_u}^t=\smash{\overline{\partial}_+}T_{t}^t$ and $\smash{\overline{\partial}}T_{t,\,x_u}^t=\smash{\overline{\partial}}T_{t}^t$. We now focus on a number of special cases.

\subsection{Homogeneous Markov chains}\label{sec:homogen_markov_chain}

\begin{definition}[Homogeneous Markov chain]\label{def:homogeneousMarkov}
A Markov chain $P\in\mprocesses$ is called \emph{time-homogeneous}, or simply \emph{homogeneous}, if its transition matrices $T_t^s$ do not depend on the absolute value of $t$ and $s$, but only on the time-difference $s-t$:
\begin{equation}\label{eq:homogeneousMarkov}
T_t^s=T_0^{s-t}
\text{~~for all $t,s\in\realsnonneg$ such that $t\leq s$.}
\end{equation}
We denote the set of all homogeneous Markov chains by $\hmprocesses$ and use $\whmprocesses$ to refer to the subset that consists of the well-behaved homogeneous Markov chains.
\end{definition}

Recall now from Section~\ref{sec:systems} the exponential transition matrix system $\mathcal{T}_Q$ corresponding to some $Q\in\mathcal{R}$. As we have seen, the transition matrices of such a system were defined by $T_t^s = e^{Q(s-t)}$. This family $\mathcal{T}_Q$ therefore clearly satisfies Equation~\eqref{eq:homogeneousMarkov}. Furthermore, by Proposition~\ref{prop:systemQ}, $\mathcal{T}_Q$ is well-behaved. Hence, we have the following result.

\begin{restatable}{corollary}{corratehasuniquehomogenmarkovprocess}
\label{cor:rate_has_unique_homogen_markov_process}
Consider any rate matrix $Q\in\mathcal{R}$ and let $p$ be an arbitrary probability mass function on $\states$. Then there is a unique Markov chain $P\in\mprocesses$ such that $\mathcal{T}_P=\mathcal{T}_Q$ and, for all $y\in\mathcal{X}$, $P(X_0=y)=p(y)$. Furthermore, this unique Markov chain is well-behaved and homogeneous.
\end{restatable}

Our next result strengthens this connection between well-behaved homogeneous Markov chains and exponential transition matrix systems.

\begin{restatable}{theorem}{theohomogeneoushasQ}
\label{theo:homogeneoushasQ}
For any well-behaved homogeneous Markov chain $P\in\whmprocesses$, there is a unique rate matrix $Q\in\mathcal{R}$ such that $\mathcal{T}_P=\mathcal{T}_Q$.\footnote{Although our proof for this result starts from scratch, this result is essentially well known. Our version of it should be regarded as a (re)formulation that is adapted to our terminology and notation and, in particular, to our use of coherent and/or full conditional probabilities.}
\end{restatable}

Hence, any well-behaved homogeneous Markov chain $P\in\whmprocesses$ is completely characterised by its initial distribution and a rate matrix $Q\in\mathcal{R}$. We will denote this rate matrix by $Q_P$.

The dynamic behaviour of well-behaved homogeneous Markov chains is furthermore particularly easy to describe, as shown by the next result.
\begin{restatable}{proposition}{propQissingletonderivforhomogen}
\label{prop:Q_is_singleton_deriv_for_homogen}
Consider any well-behaved homogeneous Markov chain $P\in\whmprocesses$ and let $Q_P\in\mathcal{R}$ be its corresponding rate matrix. Then $\partial T_t^t=\partial_+T_t^t=\partial_-T_t^t=Q_P$ and $\smash{\overline{\partial}}T_t^t=\smash{\overline{\partial}_+}T_t^t=\smash{\overline{\partial}_-}T_t^t=\{Q_P\}$.
\end{restatable}

\subsection{Non-homogeneous Markov chains}\label{sec:nonhomogen_markov}

In contrast to homogeneous Markov chains, a Markov chain for which Equation~\eqref{eq:homogeneousMarkov} does not hold is called---rather obviously---\emph{non-homogeneous}. While we know from Theorem~\ref{theo:homogeneoushasQ} that well-behaved homogeneous Markov chains can be characterised (up to an initial distribution) by a fixed rate matrix $Q\in\mathcal{R}$, this is not the case for well-behaved non-homogeneous Markov chains. 

Instead, such systems are typically described by a function $Q_t$ that gives for each time point $t\in\realsnonneg$ a rate matrix $Q_t\in\mathcal{R}$. For any such function $Q_t$, the existence and uniqueness of a corresponding non-homogeneous Markov chain then depend on the specific properties of $Q_t$. Rather than attempt to treat all these different cases here, we instead refer to some examples from the literature. 

Typically, some kind of continuity of $Q_t$ in terms of $t$ is assumed. The specifics of these assumptions may then depend on the intended generality of the results, computational considerations, the domain of application, and so forth. For example, Reference~\cite{aalen1978empirical} assumes that $Q_t$ is left-continuous and has bounded right-hand limits. As a stronger restriction, Reference~\cite{johnson1989nonhomogeneous} uses a collection $Q_1,\ldots,Q_n$ of commuting rate matrices, and defines $Q_t$ as a weighted linear combination of these component rate matrices wherein the weights vary continuously with $t$. In Reference~\cite{rindos1995exact}, a right-continuous and piecewise-constant $Q_t$ is used, meaning that $Q_t$ takes different values on various (half-open) intervals of $\realsnonneg$, but fixed values within those intervals.

This idea of using a time-dependent rate matrix $Q_t$ has the advantage of being rather intuitive, but it is rather difficult to formalise. Essentially, the problem with this approach is that it does not allow us to distinguish between left and right derivatives. Intuitively, $Q_t$ is supposed to be `the' derivative. However, this is impossible if $Q_t$ is discontinuous---for example in the piecewise constant case. Therefore, in our present work, instead of using a function $Q_t$, we will characterise non-homogeneous Markov chains by means of their transition matrix system and their initial distribution, making use of the results in Proposition~\ref{prop:Markovhassystem} and Theorem~\ref{theo:uniqueMarkovchain}.

One technique for constructing transition matrix systems that is particularly important for our work, and especially in our proofs, is to combine restrictions of exponential transition matrix systems to form a new transition matrix system that is, loosely speaking, piecewise constant. Example~\ref{exmp:combinetwoexponentials} provided a simple illustration of this technique. More generally, these transition matrix systems will be of the form
\begin{equation}\label{eq:nonhomogen_in_process_set_system_composition}
\mathcal{T}_{Q_0}^{[0,t_0]}\otimes \mathcal{T}_{Q_1}^{[t_0,t_1]} \otimes \cdots \otimes \mathcal{T}_{Q_n}^{[t_{n-1},t_n]} \otimes \mathcal{T}_{Q_{n+1}}^{[t_n,\infty)}.
\vspace{4pt}
\end{equation}
For example, the transition matrix systems $\mathcal{T}_i$, $i\in\nats_0$, that we defined in Equation~\eqref{eq:def:sequenceinexample3} are all of this form.
As we know from Propositions~\ref{prop:systemQ} and~\ref{prop:concat_restr_trans_mat_systems_is_system}, transition matrix systems that are of this form are always well-behaved. The following result is therefore a trivial consequence of Theorem~\ref{theo:uniqueMarkovchain}.

\begin{restatable}{proposition}{propfinitedifferentratematrixhasprocess}
\label{prop:finite_different_rate_matrix_has_process}
Let $p$ be an arbitrary probability mass function on $\states$, let $u=t_0,\ldots,t_n$ be a finite sequence of time points in $\mathcal{U}_\emptyset$, and let $Q_0,\ldots,Q_{n+1}\in\rateset$ be a collection of rate matrices. Then there is a well-behaved continuous-time Markov chain $P\in\wmprocesses$ such that $P(X_0=y)=p(y)$ for all $y\in\mathcal{X}$ and such that $\mathcal{T}_P$ is given by Equation~\eqref{eq:nonhomogen_in_process_set_system_composition}.
\end{restatable}

\section{Imprecise Continuous-Time Markov chains}
\label{sec:iCTMC}

In Sections~\ref{sec:stochastic_processes} and~\ref{sec:cont_time_markov_chains}, we formalised stochastic processes and provided ways to characterise them. We now turn to the field of \emph{imprecise probability}~\cite{augustin2013:itip,Walley:1991vk} to formalise the notion of an \emph{imprecise continuous-time Markov chain}. Basically, rather than look at a single stochastic process $P\in\processes$, we instead consider jointly some \emph{set} of processes $\mathcal{P}\subseteq\processes$. We start by looking at how such sets can be described.

\subsection{Sets of Consistent Stochastic Processes}

Recall from Section~\ref{sec:dynamics} that for a given stochastic process $P\in\processes$, its dynamics can be described by means of the outer partial derivatives $\smash{\overline{\partial}}T_{t,\,x_u}^t$ of its transition matrices, which can depend both on the time $t\in\realsnonneg$ and on the history $x_u\in\states_u$. 
 Furthermore, we also found that---at least for well-behaved processes---these outer partial derivatives are non-empty bounded sets of rate matrices. If all the partial derivatives of a process belong to the same non-empty bounded set of rate matrices $\rateset$, we call this process consistent with $\rateset$.\label{notation:rate_set}

\begin{definition}[Consistency with a set of rate matrices]\label{def:consistent_process}
Consider a non-empty bounded set of rate matrices $\rateset$ and a stochastic process $P\in\processes$. Then $P$ is said to be \emph{consistent} with $\rateset$ if
\begin{equation*}
(\forall t\in\realsnonneg)(\forall u\in\mathcal{U}_{<t})(\forall x_u\in\states_u)\,:\, \smash{\overline{\partial}}T_{t,x_u}^t \subseteq \rateset.
\vspace{4pt}
\end{equation*}
If $P$ is consistent with $\rateset$, we will write $P\sim\rateset$.
\end{definition}

Thus, when a process is consistent with a set of rate matrices $\rateset$, we know that its dynamics can always be described using the rate matrices in that set. However, we do not know which of these rate matrices $Q\in\rateset$ describe the dynamics at any given time $t\in\realsnonneg$ or for any given history $x_u\in\states_u$. 
Furthermore, consistency of a process with a set of rate matrices $\rateset$ does not tell us anything about the initial distribution of the process. 
Therefore, we also introduce the concept of consistency with a set of initial distributions~$\mathcal{M}$.\label{notation:initial_set}

\begin{definition}[Consistency with a set of initial distributions]\label{def:consistent_process_initialdistribution}
Consider any non-empty bounded set $\mathcal{M}$ of probability mass functions on $\states$ and any stochastic process $P\in\processes$. We then say that $P$ is \emph{consistent} with $\mathcal{M}$, and write $P\sim\mathcal{M}$, if $P(X_0)\in\mathcal{M}$.
\end{definition}

In the remainder of this paper, we will focus on sets of processes that are jointly consistent with some given $\rateset$ and $\mathcal{M}$. However, rather than look at the set of \emph{all} processes consistent with some $\rateset$ and $\mathcal{M}$, we will instead consider the consistent subset of some given set of processes $\mathcal{P}\subseteq\processes$.

\begin{definition}[Consistent subset of processes]\label{def:consistent_process_set}
Consider a non-empty bounded set of rate matrices $\rateset$ and a non-empty set $\mathcal{M}$ of probability mass functions on $\states$ and a set of stochastic processes $\mathcal{P}\subseteq\processes$. Then, the \emph{subset of $\mathcal{P}$ consistent with} $\rateset$ \emph{and} $\mathcal{M}$ is denoted by $\mathcal{P}_{\rateset,\mathcal{M}}$, and defined as
\begin{equation*}
\mathcal{P}_{\rateset,\mathcal{M}} \coloneqq \left\{P\in\mathcal{P}\,:\,P\sim\rateset,\,P\sim\mathcal{M}\right\}\,.
\end{equation*}
When $\mathcal{M}$ is the set of \emph{all} probability mass functions on $\states$, we will write $\mathcal{P}_{\rateset}$ for the sake of brevity.
\end{definition}
For some fixed $\mathcal{Q}$ and $\mathcal{M}$, different choices for $\mathcal{P}$ will result in different sets of consistent processes $\mathcal{P}_{\rateset,\mathcal{M}}$. Three specific choices of $\mathcal{P}$ will be particularly important in this paper because, as we will now show, they lead to three different types of imprecise continuous-time Markov chains.

\subsection{Types of Imprecise Continuous-Time Markov Chains}\label{subsec:types_ictmc}

In Sections~\ref{sec:stochastic_processes} and~\ref{sec:cont_time_markov_chains}, we introduced three sets $\wprocesses$, $\wmprocesses$ and $\whmprocesses$ of well-behaved stochastic processes with different qualitative properties. $\wprocesses$ is the set of all well-behaved stochastic processes, $\wmprocesses$ consists of the processes in $\wprocesses$ that are continuous-time Markov chains, and $\whmprocesses$ is the set of all homogeneous Markov chains that are well-behaved, which is therefore a subset of $\wmprocesses$. We now use Definition~\ref{def:consistent_process_set} to define three sets of consistent processes that also have these respective qualitative properties. 

\begin{definition}[Imprecise continuous-time Markov chain]\label{def:process_sets}
For any non-empty bounded set of rate matrices $\rateset$, and any non-empty set $\mathcal{M}$ of probability mass functions on $\states$, we define the following three sets of stochastic processes that are jointly consistent with $\rateset$ and $\mathcal{M}$:
\begin{itemize}
\item $\wprocesses_{\rateset,\mathcal{M}}$ is the consistent set of all well-behaved stochastic processes;
\item $\wmprocesses_{\rateset,\mathcal{M}}$ is the consistent set of all well-behaved Markov chains;
\item $\whmprocesses_{\rateset,\mathcal{M}}$ is the consistent set of all well-behaved homogeneous Markov chains.
\end{itemize}
We call each of these three sets 
 an \emph{imprecise continuous-time Markov chain}, and abbreviate this as \ictmc.\footnote{For the set $\wprocesses_{\rateset,\mathcal{M}}$, one might wonder why we choose to call it an imprecise \emph{Markov} chain, since it contains processes that do not satisfy the Markov property. As we will see in Section~\ref{sec:single_var_lower_exp}, this choice of terminology is motivated by that fact that the set $\wprocesses_{\rateset,\mathcal{M}}$ itself---rather than its elements---satisfies a so-called \emph{imprecise Markov property}.} 
Following Definition~\ref{def:consistent_process_set}, we will write $\wprocesses_{\rateset}$ when we take $\mathcal{M}$ to be the set of \emph{all} probability mass functions on $\states$, and similarly for $\wmprocesses_{\rateset}$ and $\whmprocesses_{\rateset}$.
\end{definition}

 Since the sets $\whmprocesses$, $\wmprocesses$ and $\wprocesses$
 are nested, it should be clear that this also true for the corresponding types of \ictmc's.

\begin{restatable}{proposition}{propmarkovsetsubsetofnonmarkovset}
\label{prop:markov_set_subset_of_nonmarkov_set}
Consider any bounded set of rate matrices $\rateset$, and any non-empty set $\mathcal{M}$ of probability mass functions on $\states$. Then,
\begin{equation*}
\whmprocesses_{\rateset,\mathcal{M}} \subseteq \wmprocesses_{\rateset,\mathcal{M}} \subseteq \wprocesses_{\rateset,\mathcal{M}}\,.
\end{equation*}
\end{restatable}

Observe furthermore that for $\whmprocesses_{\rateset,\mathcal{M}}$ and $\wmprocesses_{\rateset,\mathcal{M}}$, the extra properties of their elements allow us to simplify the notion of consistency in Definition~\ref{def:consistent_process}, which leads to the following alternative characterisations:
\vspace{3pt}
\begin{equation}\label{eq:consistentwmprocessesalternative}
\wmprocesses_{\rateset,\mathcal{M}} = \left\{P\in\wmprocesses\,:\,(\forall t\in\realsnonneg)~~ \smash{\overline{\partial}}T_t^t\subseteq\rateset,\,P(X_0)\in\mathcal{M}\right\},
\end{equation}
and
\begin{equation*}
\whmprocesses_{\rateset,\mathcal{M}} = \left\{P\in\whmprocesses\,:\,Q_P\in\rateset,\,P(X_0)\in\mathcal{M}\right\}.
\vspace{5pt}
\end{equation*}
This first equality follows from the Markov property of the elements of $\wmprocesses$, which ensures that $\smash{\overline{\partial}}T_{t,x_u}^t=\smash{\overline{\partial}}T_t^t$. 
The second equality follows from the Markov property and the homogeneity of the processes $\smash{P\in\whmprocesses}$, which, by Proposition~\ref{prop:Q_is_singleton_deriv_for_homogen}, ensures that $\smash{\overline{\partial}}T_t^t=\{Q_P\}$ for all $t\in\realsnonneg$.

The following example further illustrates the difference between the three types of \ictmc's that we consider.

\begin{exmp}\label{example:rateset_not_singleton}
Let $Q_1$ and $Q_2$ be two different transition rate matrices and consider the set $\rateset\coloneqq \{Q_1, Q_2\}$. Let furthermore $\mathcal{M}\coloneqq\{p\}$, with $p$ an arbitrary probability mass function on $\states$.

The set $\whmprocesses_{\rateset,\mathcal{M}}$ then contains exactly two stochastic processes $P_1$ and $P_2$, each of which is a well-behaved homogeneous Markov chain. They both have $p$ as their initial distribution, in the sense that $P_1(X_0)=P_2(X_0)=p$, but their transition rate matrices $Q_{P_i}$, $i\in\{1,2\}$---whose existence is guaranteed by Theorem~\ref{theo:homogeneoushasQ}---are different: $Q_{P_1}=Q_1$, and $Q_{P_2}=Q_2$. 
The transition matrix systems of the Markov chains $P_1$ and $P_2$ are given by the exponential transition matrix systems $\mathcal{T}_{Q_1}$ and $\mathcal{T}_{Q_2}$, in that order, and the transition matrices of $P_1$ and $P_2$ are therefore given by $\presuper{1}T_t^s=e^{Q_1(s-t)}$ and $\presuper{2}T_t^s=e^{Q_2(s-t)}$, respectively. 

Since $\whmprocesses_{\rateset,\mathcal{M}}$ is a subset of $\wmprocesses_{\rateset,\mathcal{M}}$, each of the two homogeneous Markov chains $P_1$ and $P_2$ belongs to $\smash{\wmprocesses_{\rateset,\mathcal{M}}}$ as well. However, $\smash{\wmprocesses_{\rateset,\mathcal{M}}}$ also contains additional processes, which are not homogeneous and therefore do not belong to $\whmprocesses_{\rateset,\mathcal{M}}$. 
For instance, for any $r>0$, it follows from Proposition~\ref{prop:nonhomogeneous_in_process_set} further on that there is a well-behaved continuous-time Markov chain $\smash{P\in\wmprocesses_{\rateset,\,\mathcal{M}}}$ that has $\smash{\mathcal{T}_{Q_1}^{[0,r]}\otimes\mathcal{T}_{Q_2}^{[r,+\infty)}}$ as its transition matrix system. This Markov chain is clearly not homogeneous because $T_0^r=e^{Q_1r}$ is different from $T_r^{2r}=e^{Q_2r}$, and therefore it does not belong to $\whmprocesses_{\rateset,\mathcal{M}}$.

Since $\wmprocesses_{\rateset,\mathcal{M}}$ is a subset of $\wprocesses_{\rateset,\mathcal{M}}$, each of the processes that we have considered so far belong to the set $\smash{\wprocesses_{\rateset,\mathcal{M}}}$ as well. However this latter set contains more complicated processes still. For instance, for any $u\in\mathcal{U}$ and any $x_u,y_u\in\states_u$ such that $x_u\neq y_u$, $\wprocesses_{\rateset,\mathcal{M}}$ will for example contain a stochastic process $P$ such that, for all $t>u$, $\smash{\overline{\partial}}T_{t,x_u}^t=\{Q_1\}$ and $\smash{\overline{\partial}}T_{t,y_u}^t=\{Q_2\}$. For all $s>t>u$, the history-dependent transition matrices $\smash{T_{t,x_u}^s}$ and $T_{t,y_u}^s$ of this process $P$ will be given by $\smash{T_{t,x_u}^s=e^{Q_1(s-t)}}$ and $\smash{T_{t,y_u}^s=e^{Q_2(s-t)}}$, which implies that this process $P$ does not satisfy the Markov property, and therefore, that it does not belong to $\wmprocesses_{\rateset,\mathcal{M}}$.
\exampleend
\end{exmp}

\begin{restatable}{proposition}{propnonhomogeneousinprocessset}
\label{prop:nonhomogeneous_in_process_set}
Consider any non-empty bounded set of rate matrices $\rateset$ and let $\mathcal{M}$ be any non-empty set of probability mass functions on $\states$. Then for any $p\in\mathcal{M}$, any ordered finite sequence of time points $u=t_0,\ldots,t_n$ in $\mathcal{U}_\emptyset$ and any collection of rate matrices $Q_0,\ldots,Q_{n+1}\in\rateset$, there is a well-behaved continuous-time Markov chain $P\in\wmprocesses_{\rateset,\,\mathcal{M}}$ such that $P(X_0=y)=p(y)$ for all $y\in\mathcal{X}$ and such that $\mathcal{T}_P$ is given by Equation~\eqref{eq:nonhomogen_in_process_set_system_composition}.
\end{restatable}

We conclude this section with some notes about closure properties of the different types of \ictmc's that we consider, which are particularly useful for existence proofs. In particular, we focus on closure properties under recombination of known elements---colloquially, the ``piecing together'' of two or more processes to construct a new process that belongs to the same \ictmc.

The example above already suggested how to do this for $\smash{\wmprocesses_{\rateset,\mathcal{M}}}$, by combining two well-behaved homogeneous Markov chains $P_1,P_2\in\whmprocesses$ to form a new process $P\in\wmprocesses$. Similarly, but more generally, for any two processes $P_1,P_2\in\wmprocesses_{\rateset,\mathcal{M}}$, we can combine their transition matrix systems $\mathcal{T}_{P_1}$ and $\mathcal{T}_{P_2}$ to construct a new transition matrix system $\smash{\mathcal{T}\coloneqq \mathcal{T}_{P_1}^{[0,r]}\otimes \mathcal{T}_{P_2}^{[r,+\infty)}}$, with $r>0$ chosen arbitrarily. Theorem~\ref{theo:uniqueMarkovchain} then guarantees that there exists a Markov chain $P$ such that $\mathcal{T}_P=\mathcal{T}$ and $P(X_0)=P_1(X_0)$. It is straightforward to verify that, because $P_1,P_2\in\wmprocesses_{\rateset,\mathcal{M}}$, also $P\in\wmprocesses_{\rateset,\mathcal{M}}$; we leave this as an exercise for the reader.

Clearly, a similar procedure is impossible for $\whmprocesses_{\rateset,\mathcal{M}}$, because the combination of two processes would make the resultant one lose the homogeneity property, which is required to be an element of $\whmprocesses_{\rateset,\mathcal{M}}$.

However, for our most general type of \ictmc, which is $\wprocesses_{\rateset,\mathcal{M}}$, it turns out to be possible to recombine elements in an even more general, history-dependent way. That is, if $\rateset$ is convex, then it is possible, for fixed time points $u\in\mathcal{U}$, to choose for every history $x_u\in\states_u$ a different process $P_{x_u}\in\wprocesses_{\rateset,\mathcal{M}}$, and to recombine these into a process $P$ that agrees with these $P_{x_u}$ conditional on the specific history $x_u$. Furthermore, the distribution on the time points $u$ can be chosen to agree with any element $\smash{P_{\emptyset}\in \wprocesses_{\rateset,\mathcal{M}}}$. The following result guarantees that this new process $P$ will again belong to $\wprocesses_{\rateset,\mathcal{M}}$.

\begin{restatable}{theorem}{theoaanelkaarplakken}
\label{theo:aanelkaarplakken}
Consider a non-empty convex set of rate matrices $\rateset\subseteq\mathcal{R}$, and any non-empty set $\mathcal{M}$ of probability mass functions on $\states$.
Fix a finite sequence of time points $u\in\mathcal{U}$. Choose any $P_\emptyset\in\wprocesses_{\rateset,\,\mathcal{M}}$ and, for all $x_u\in\states_u$, choose some $\smash{P_{x_u}\in\wprocesses_{\rateset,\,\mathcal{M}}}$. Then there is a stochastic process $\smash{P\in\wprocesses_{\rateset,\,\mathcal{M}}}$ such that, for all $u_1,u_2\subseteq u$ such that $u_1<u_2$, all $x_u\in\states_u$ and all $A\in\mathcal{A}_u$:
\begin{equation}\label{eq:theo:aanelkaarplakken:equalsfirst}
P(X_{u_2}=x_{u_2}\vert X_{u_1}=x_{u_1})=P_{\emptyset}(X_{u_2}=x_{u_2}\vert X_{u_1}=x_{u_1})
\vspace{-7pt}
\end{equation}
and
\begin{equation}\label{eq:theo:aanelkaarplakken:equalssecond}
P(A\vert X_u=x_u)=P_{x_u}(A\vert X_u=x_u).
\vspace{7pt}
\end{equation}
\end{restatable}

\subsection{Lower Expectations for \ictmc's}\label{subsec:ictmc_types}

From a practical point of view, after having specified a (precise) stochastic process, one is typically interested in the expected value of some function of interest, or the probability of some event. Similarly, in this work, our main objects of consideration will be the \emph{lower and upper expectations} and \emph{lower and upper probabilities} that correspond to the \ictmc's that we introduced in the previous section.

\begin{definition}[Lower Expectation]\label{def:lower_exp}
For any non-empty set of stochastic processes $\mathcal{P}\subseteq\processes$, the \emph{(conditional) lower expectation with respect to $\mathcal{P}$} is defined as
\begin{equation}\label{eq:genericlowerexpectation}
\underline{\mathbb{E}}[\cdot\,\vert\,\cdot] \coloneqq \inf\left\{\mathbb{E}_P[\cdot\,\vert\,\cdot]\,:\,P\in\mathcal{P}\right\},
\end{equation}
where $\mathbb{E}_P[\cdot\,\vert\,\cdot]$ denotes the (conditional) expectation taken with respect to $P$.

In particular, for any non-empty bounded set of rate matrices $\rateset$ and any non-empty set $\mathcal{M}$ of probability mass functions on $\states$, we let
\begin{equation}\label{eq:lowerexp3}
\underline{\mathbb{E}}_{\rateset,\mathcal{M}}^{\mathrm{W}}[\cdot\,\vert\,\cdot] \coloneqq \inf\left\{\mathbb{E}_P[\cdot\,\vert\,\cdot]\,:\,P\in\mathbb{P}_{\rateset,\mathcal{M}}^{\mathrm{W}}\right\},
\end{equation}
and similarly for $\smash{\underline{\mathbb{E}}_{\rateset,\mathcal{M}}^{\mathrm{WM}}}$ and $\smash{\underline{\mathbb{E}}_{\rateset,\mathcal{M}}^{\mathrm{WHM}}}$.
If $\mathcal{M}$ is the set of \emph{all} probability mass functions on $\states$, then as in Definition~\ref{def:consistent_process_set}, we will write $\underline{\mathbb{E}}_{\rateset}^{\mathrm{W}}$ instead of $\underline{\mathbb{E}}_{\rateset,\mathcal{M}}^{\mathrm{W}}$, and similarly for $\underline{\mathbb{E}}_{\rateset}^{\mathrm{WM}}$ and $\underline{\mathbb{E}}_{\rateset}^{\mathrm{WHM}}$.
\end{definition}

Upper expectations can be defined analogously, simply by replacing the infimum by a supremum: $\smash{\overline{\mathbb{E}}[\cdot\,\vert\,\cdot]\coloneqq \sup\{\mathbb{E}_P[\cdot\,\vert\,\cdot]\,:\,P\in\mathcal{P}\}}$.
However, there is no need to study these upper expectations separately, because they are in one-to-one correspondence to lower expectations through the conjugacy relation $\overline{\mathbb{E}}[\cdot\,\vert\,\cdot]=-\underline{\mathbb{E}}[-\cdot\,\vert\,\cdot]$. 
Lower and upper probabilities also have analogous definitions. However, these too do not need to be studied separately, because they correspond to special cases of lower and upper expectations. For example, for any $u\in\mathcal{U}$, $x_u\in\states_u$ and $A\in\mathcal{A}_u$, we have that
\begin{align}
\underline{P}(A\vert X_u=x_u)
\coloneqq&\inf\{P(A\vert X_u=x_u)\colon P\in\mathcal{P}\}\notag\\
=&\inf\{\mathbb{E}_P[\ind{A}\vert X_u=x_u]\colon P\in\mathcal{P}\}
=\underline{\mathbb{E}}[\ind{A}\vert X_u=x_u],
\label{eq:lowerprobaslowerexp}
\end{align}
and similarly, we also have that $\overline{P}(A\vert X_u=x_u)=\overline{\mathbb{E}}[\ind{A}\vert X_u=x_u]$, which, due to conjugacy, implies that $\overline{P}(A\vert X_u=x_u)=-\underline{\mathbb{E}}[-\ind{A}\vert X_u=x_u]$. Hence, from a computational point of view, it clearly suffices to focus on lower expectations, which is what we will do in the remainder of this work.

In particular, one of the main aims of this paper is to provide methods to compute lower expectations for the different types of \ictmc's that we have introduced in Section~\ref{subsec:types_ictmc}; however, as we will see at the end of this section, doing this for the set $\whmprocesses_{\rateset,\mathcal{M}}$ is particularly difficult. We will therefore largely focus on performing these computations for the ICTMCs $\wprocesses_{\rateset,\mathcal{M}}$ and $\wmprocesses_{\rateset,\mathcal{M}}$. We start by giving some useful properties of the three types of lower expectations that we are interested in.

First of all, it can be shown that if $\rateset$ is a non-empty, bounded, convex and closed set of rate matrices, then the conditional lower expectations $\underline{\mathbb{E}}_{\rateset}^\mathrm{W}$, $\underline{\mathbb{E}}_{\rateset}^\mathrm{WM}$ and $\underline{\mathbb{E}}_{\rateset}^\mathrm{WHM}$ are actually minima---rather than infima---because they are always reached by some element of their corresponding \ictmc. However, the proof of this claim is rather involved and requires some technical machinery that is outside of the scope of this paper; we intend to publish these and related results separately in future work.

Moving on, we know from Section~\ref{subsec:types_ictmc} that the sets $\whmprocesses_{\rateset,\mathcal{M}}$, $\wmprocesses_{\rateset,\mathcal{M}}$ and $\wprocesses_{\rateset,\mathcal{M}}$ are nested subsets of each other. As an immediate consequence, their corresponding lower expectations provide (lower) bounds for each other.
\begin{restatable}{proposition}{proplowerexpmarkovboundedbynonmarkov}
\label{prop:lower_exp_markov_bounded_by_nonmarkov}
Consider any non-empty bounded set of rate matrices $\rateset$, and any non-empty set $\mathcal{M}$ of probability mass functions on $\states$. Then,
\begin{equation*}
\underline{\mathbb{E}}_{\rateset,\mathcal{M}}^\mathrm{W}[\cdot\,\vert\,\cdot] \leq
\underline{\mathbb{E}}_{\rateset,\mathcal{M}}^\mathrm{WM}[\cdot\,\vert\,\cdot] \leq
\underline{\mathbb{E}}_{\rateset,\mathcal{M}}^\mathrm{WHM}[\cdot\,\vert\,\cdot]\,.
\end{equation*}
\end{restatable}

Generally speaking, the inequalities in this proposition can be---and often are---strict; for the first inequality, we will illustrate this further on in this section, in Figure~\ref{fig:homogeneousCounterExample}, whereas for the second inequality, this can be seen by comparing Examples~\ref{exmp:num_multivar_func_nonmarkov} and \ref{exmp:num_counterexample_markov} in Section~\ref{sec:decomposition}.

For now, we first provide a useful property of the lower expectation $\smash{\underline{\mathbb{E}}_{\rateset,\mathcal{M}}^\mathrm{W}}$ that corresponds to $\wprocesses_{\rateset,\mathcal{M}}$. To this end, consider a function $f\in\gambles(\states_{u\cup v\cup w})$ defined on the union $u\cup v\cup w$ of three finite sets of time points $u,v,w\in\mathcal{U}$ such that $u<v<w$. It then follows from the basic properties of expectations that for any stochastic process $P$, the corresponding expectation of $f$, conditional on $X_u$, decomposes as follows:
\begin{equation}\label{eq:lawofiteratedexpectation}
\mathbb{E}_P[f(X_u,X_v,X_w)\,\vert\,X_u] = \mathbb{E}_P\bigl[\mathbb{E}_P[f(X_u,X_v,X_w)\,\vert\,X_u,X_v]\,\big\vert\,X_u\bigr].
\end{equation}
This equality is well-known, and is called the \emph{law of iterated expectation}. Rather remarkably, if $\rateset$ is convex, then the lower expectation $\underline{\mathbb{E}}^{\mathrm{W}}_{\rateset,\,\mathcal{M}}$ satisfies a similar property. The proof for this so-called law of iterated \emph{lower} expectation is based on Theorem~\ref{theo:aanelkaarplakken}.

\begin{restatable}{theorem}{theoremdecompositionmultivar}
\label{theorem:decomposition_multivar}
Let $\rateset$ be an arbitrary non-empty, bounded, and convex set of rate matrices, and consider any non-empty set $\mathcal{M}$ of probability mass functions on $\states$. Then for any $u,v,w\in\mathcal{U}$ such that $u<v<w$ and any $f\in\gambles(\states_{u\cup v\cup w})$:
\begin{equation}\label{eq:lower_exp_factorizes}
\underline{\mathbb{E}}^{\mathrm{W}}_{\rateset,\,\mathcal{M}}\left[f(X_u,X_v,X_w)\,\vert\,X_u\right] = \underline{\mathbb{E}}^{\mathrm{W}}_{\rateset,\,\mathcal{M}}\Bigl[\underline{\mathbb{E}}^{\mathrm{W}}_{\rateset,\,\mathcal{M}}\left[f(X_u,X_v,X_w)\,\vert\,X_u,X_v\right] \Big\vert\,X_u\Bigr]\,. 
\end{equation}
\end{restatable}
The reason why this result is useful, is because it essentially allows us to compute lower expectations recursively. In particular, instead of computing a lower expectation for all time points simultaneously, we can focus on each of the time points separately, and eliminate them one by one. We will revisit this idea in Section~\ref{sec:funcs_multi_time_points}, where we will use it to develop efficient algorithms.

For now, we would like to draw attention to the fact that Theorem~\ref{theorem:decomposition_multivar} relies heavily on the fact that the individual elements of $\smash{\wprocesses_{\rateset,\mathcal{M}}}$ do not need to satisfy the Markov property, in the sense that they can have history-dependent probabilities. Therefore, unfortunately, the result in Theorem~\ref{theorem:decomposition_multivar} does not extend to $\smash{\underline{\mathbb{E}}_{\rateset,\mathcal{M}}^\mathrm{WM}}$ or $\smash{\underline{\mathbb{E}}_{\rateset,\mathcal{M}}^\mathrm{WHM}}$. Since Theorem~\ref{theorem:decomposition_multivar} simplifies considerably the process of computing the lower expectation $\underline{\mathbb{E}}_{\rateset,\mathcal{M}}^\mathrm{W}$, this suggests that the qualitative differences between the three types of \ictmc's that we consider make some of their corresponding lower expectations harder to compute than others. In fact, we will show in Section~\ref{sec:funcs_multi_time_points} that $\wprocesses_{\rateset\,,\mathcal{M}}$ is in some sense the easiest set to do this for, exactly because of Theorem~\ref{theorem:decomposition_multivar}.

Furthermore, rather ironically, computing lower expectations for the set $\whmprocesses_{\rateset\,,\mathcal{M}}$---which, intuitively, is the simplest of our three types of \ictmc's---seems to be much harder than for the sets $\smash{\wmprocesses_{\rateset\,,\mathcal{M}}}$ or $\smash{\wprocesses_{\rateset\,,\mathcal{M}}}$.
The problem is, essentially, that while homogeneous Markov chains are easy to work with numerically, the set $\whmprocesses_{\rateset\,,\mathcal{M}}$ does not provide enough ``degrees of freedom'' to easily solve the optimisation problem that is involved in computing $\underline{\mathbb{E}}^\mathrm{WHM}_{\rateset\,,\mathcal{M}}$.
The following serves as an illustration.

Suppose that we want to compute the lower expectation $\underline{\mathbb{E}}_{\rateset,\,\mathcal{M}}^{\mathrm{WHM}}[f(X_t)\,\vert\,X_0=x_0]$ of some function $f\in\gamblesX$ at time $t$, conditional on the information that the state $X_0$ at time $0$ takes the value $x_0\in\states_0$. It then follows from Definition~\ref{def:lower_exp},  Remark~\ref{remark:expectationT} and Proposition~\ref{prop:Q_is_singleton_deriv_for_homogen} that
\begin{align*}
\underline{\mathbb{E}}_{\rateset,\,\mathcal{M}}^{\mathrm{WHM}}[f(X_t)\,\vert\,X_0=x_0] &= \inf\left\{ \mathbb{E}_P[f(X_t)\,\vert\,X_0=x_0]\,:\,P\in\whmprocesses_{\rateset,\,\mathcal{M}} \right\} \\
 &= \inf\left\{ [e^{Q_P t}f](x_0)\,:\,P\in\whmprocesses_{\rateset,\,\mathcal{M}} \right\} \\
 &= \inf\left\{ [e^{Q t}f](x_0)\,:\,Q\in\rateset\right\}.
 \vspace{6pt}
\end{align*}
Therefore, computing $\smash{\underline{\mathbb{E}}_{\rateset,\,\mathcal{M}}^{\mathrm{WHM}}[f(X_t)\,\vert\,X_0=x_0]}$ is at its core a non-linear, constrained optimisation problem over the set $\rateset$, where the non-linearity stems from the term $e^{Qt}$, and the specific form of the constraints depends on the choice of $\rateset$.  The following example illustrates the non-linearity of this type of optimisation problem in a simple case.

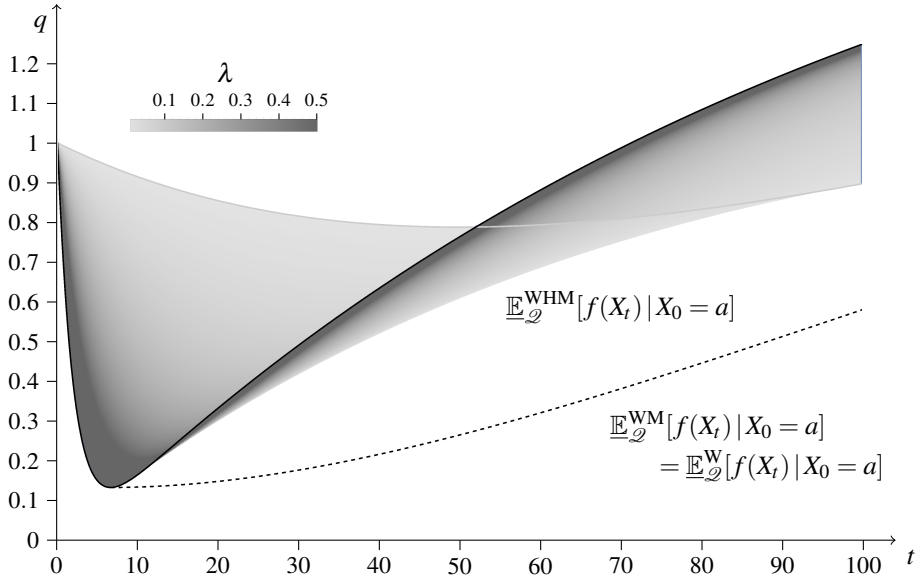
\begin{figure}[t]
\begin{tikzpicture}
      \node[anchor=south west,inner sep=0] (plot) at (0,0) {\includegraphics[width=106mm]{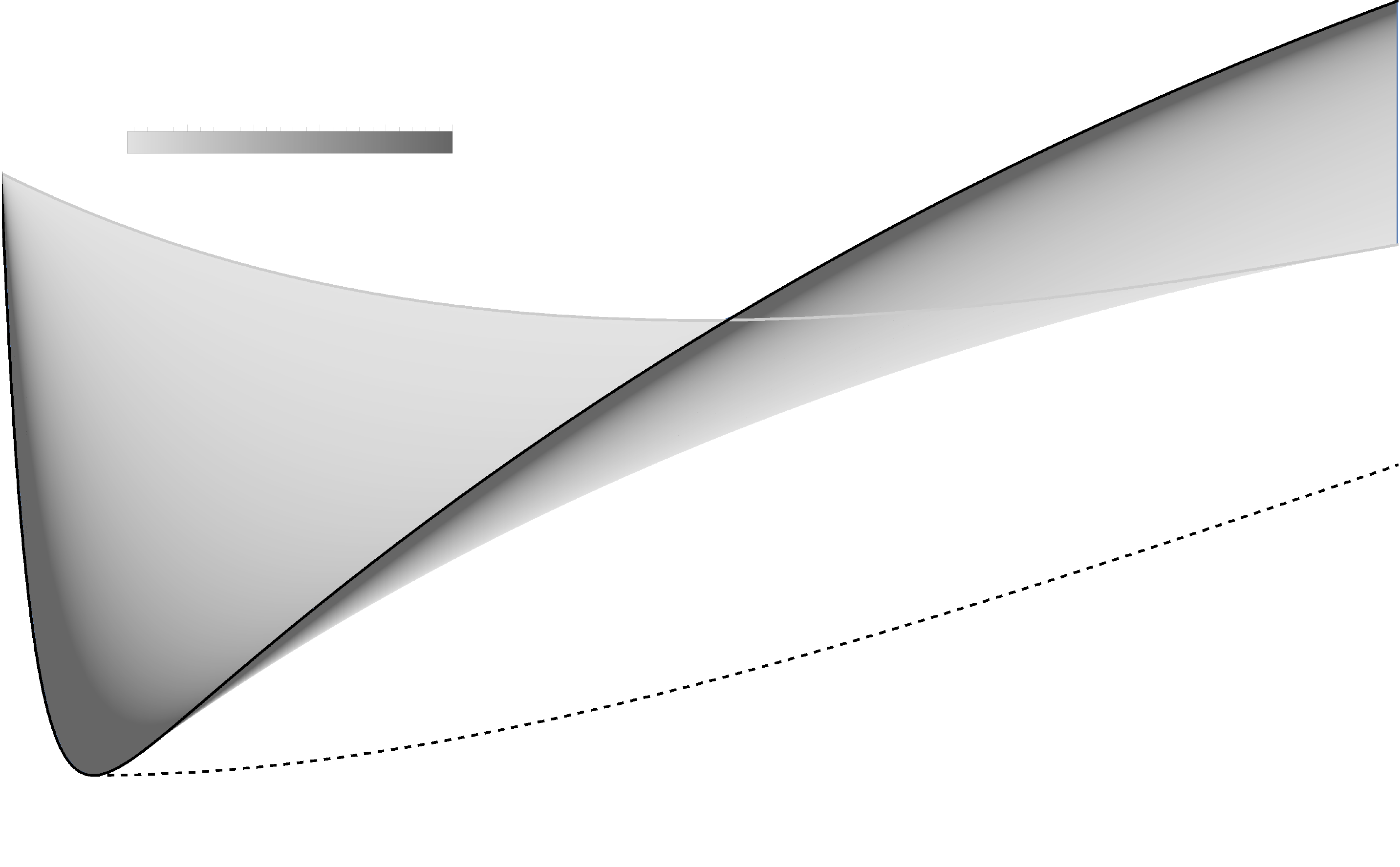}};
      \begin{scope}[x={(plot.south east)},y={(plot.north west)}]
        \draw[->] (0,0) -- (0,1.08) ;
        \node[left=0pt] at (0,1.04) {$q$};
        \draw[->] (0,0) -- (1.04,0) ;
        \node[below] at (1.06,0) {$t$}; 
        \draw (0,0) -- (0,-0.02);
        \node[below=3pt] at (0,0) {\small$0$};
        \draw (0.1,0) -- (0.1,-0.02);
        \node[below=3pt] at (0.1,0) {\small$10$};
        \draw (0.2,0) -- (0.2,-0.02);
        \node[below=3pt] at (0.2,0) {\small$20$};
        \draw (0.3,0) -- (0.3,-0.02);
        \node[below=3pt] at (0.3,0) {\small$30$};
        \draw (0.4,0) -- (0.4,-0.02);
        \node[below=3pt] at (0.4,0) {\small$40$}; 
        \draw (0.5,0) -- (0.5,-0.02);
        \node[below=3pt] at (0.5,0) {\small$50$}; 
        \draw (0.6,0) -- (0.6,-0.02);
        \node[below=3pt] at (0.6,0) {\small$60$}; 
        \draw (0.7,0) -- (0.7,-0.02);
        \node[below=3pt] at (0.7,0) {\small$70$}; 
        \draw (0.8,0) -- (0.8,-0.02);
        \node[below=3pt] at (0.8,0) {\small$80$}; 
        \draw (0.9,0) -- (0.9,-0.02);
        \node[below=3pt] at (0.9,0) {\small$90$}; 
        \draw (1,0) -- (1,-0.02);
        \node[below=3pt] at (1,0) {\small$100$}; 

        \draw (0,0) -- (-0.012,0);
        \node[left=3pt] at (0,0) {\small$0$};
        \draw (0,0.08) -- (-0.012,0.08);
        \node[left=3pt] at (0,0.08) {\small$0.1$};
        \draw (0,0.16) -- (-0.012,0.16);
        \node[left=3pt] at (0,0.16) {\small$0.2$};
        \draw (0,0.24) -- (-0.012,0.24);
        \node[left=3pt] at (0,0.24) {\small$0.3$};
        \draw (0,0.32) -- (-0.012,0.32);
        \node[left=3pt] at (0,0.32) {\small$0.4$};
        \draw (0,0.4) -- (-0.012,0.4);
        \node[left=3pt] at (0,0.4) {\small$0.5$};
        \draw (0,0.48) -- (-0.012,0.48);
        \node[left=3pt] at (0,0.48) {\small$0.6$};
        \draw (0,0.56) -- (-0.012,0.56);
        \node[left=3pt] at (0,0.56) {\small$0.7$};
        \draw (0,0.64) -- (-0.012,0.64);
        \node[left=3pt] at (0,0.64) {\small$0.8$};
        \draw (0,0.72) -- (-0.012,0.72);
        \node[left=3pt] at (0,0.72) {\small$0.9$};
        \draw (0,0.8) -- (-0.012,0.8);
        \node[left=3pt] at (0,0.8) {\small$1$};
        \draw (0,0.88) -- (-0.012,0.88);
        \node[left=3pt] at (0,0.88) {\small$1.1$};
        \draw (0,0.96) -- (-0.012,0.96);
        \node[left=3pt] at (0,0.96) {\small$1.2$};

		\node at (0.21,0.943) {$\lambda$};
        \draw (0.134,0.848) -- (0.134,0.86);
        \node at (0.134,0.887) {\scriptsize$0.1$};
        \draw (0.1812,0.848) -- (0.1812,0.86);
        \node at (0.1812,0.887) {\scriptsize$0.2$};
        \draw (0.2284,0.848) -- (0.2284,0.86);
        \node at (0.2284,0.887) {\scriptsize$0.3$};
        \draw (0.2756,0.848) -- (0.2756,0.86);
        \node at (0.2756,0.887) {\scriptsize$0.4$};
        \draw (0.3224,0.848) -- (0.3224,0.86);
        \node at (0.3224,0.887) {\scriptsize$0.5$};


        \node at (0.7,0.47) {$\underline{\mathbb{E}}_\rateset^\mathrm{WHM}[f(X_t)\,\vert\,X_0=a]$};
        \node at (0.82,0.23) {$\underline{\mathbb{E}}_\rateset^\mathrm{WM}[f(X_t)\,\vert\,X_0=a]$};
        \node at (0.87,0.15) {$~~~~=\underline{\mathbb{E}}_\rateset^\mathrm{W}[f(X_t)\,\vert\,X_0=a]$};
      \end{scope}
    \end{tikzpicture}
\caption{Plot of the induced set of expected values $\mathbb{E}_P[f(X_t)\,\vert\,X_0=a]$, for time points $t\in[0,100]$, corresponding to all $P\in\whmprocesses_{\rateset}$ obtained as we vary $\lambda\in[0.01,0.5]$. Observe that the lower expectation with respect to this set is initially reached by choosing $\lambda=0.5$. This changes around the time point $t\approx 6.6$, after which the minimising value of $\lambda$ becomes a (changing) internal point of the interval $[0.01,0.5]$. The dashed line corresponds to the lower expectation with respect to the sets $\wmprocesses_\rateset$ and $\wprocesses_\rateset$, which also include non-homogeneous Markov chains and, for $\wprocesses_\rateset$, even more general processes. Note that the lower expectations with respect to $\whmprocesses_{\rateset}$, $\wmprocesses_{\rateset}$ and $\wprocesses_{\rateset}$ are all equal for $t<6.6$. However, as time evolves further, the lower expectation with respect to $\whmprocesses_{\rateset}$ starts to diverge from the other two.}
\label{fig:homogeneousCounterExample}
\end{figure}

\begin{exmp}\label{ex:homogeneousexample}
Consider an ordered ternary state space $\states\coloneqq \{a,b,c\}$,  let $f\in\gamblesX$ be defined as $f\coloneqq [1~0~2]^\top$---in the sense that $f(a)\coloneqq1$, $f(b)\coloneqq0$ and $f(c)\coloneqq2$---and consider the set of transition rate matrices
\vspace{4pt}
\begin{equation*}
\rateset \coloneqq \left\{ \left[\begin{array}{rrr}
-\lambda & \lambda & 0 \\
0 & -0.01 & 0.01 \\
0 & 0 & 0
\end{array}\right] \,:\,\lambda\in[0.01,0.5]\right\}.
\vspace{4pt}
\end{equation*}
Every process $P$ in the imprecise Markov chain $\whmprocesses_{\rateset,\,\mathcal{M}}$ is then a homogeneous Markov chain of which the unique transition rate matrix $Q$ is completely determined by some $\lambda$ in $[0.01,0.5]$. Furthermore, as we know from Remark~\ref{remark:expectationT} and Theorem~\ref{theo:homogeneoushasQ}, the conditional expectation $\mathbb{E}_P[f(X_t)\vert X_0=a]$ that corresponds to this homogeneous Markov chain is equal to $[e^{Q t}f](a)$. Due to this equality, it is a matter of applying some basic---yet cumbersome---algebra to find that 
\begin{align*}
\mathbb{E}_P[f(X_t)\vert X_0=a]
=
\begin{cases}
\displaystyle
2+e^{-\lambda t}+2\frac{e^{-\lambda t}-100\lambda e^{-\nicefrac{t}{100}}}{100\lambda-1}&\text{ if $\lambda\in(0.01,0.5]$}\\[8pt]
\displaystyle
2-e^{-\nicefrac{t}{100}}-\frac{1}{50}te^{-\nicefrac{t}{100}}&\text{ if $\lambda=0.01$.}
\end{cases}\\[-10pt]
\end{align*}
Obtaining the value of $\underline{\mathbb{E}}_{\rateset,\,\mathcal{M}}^{\mathrm{WHM}}[f(X_t)\,\vert\,X_0=a]$ now corresponds to minimising this expression as $\lambda$ ranges over the interval $[0.01,0.5]$. Figure~\ref{fig:homogeneousCounterExample} illustrates that, even in this simple ternary case, this minimisation problem is already non-trivial, because---depending on the value of $t$---the minimum is not guaranteed to be obtained for one of the end points of $[0.01,0.5]$, but may 
only be achieved by an internal point. 

Nevertheless, in this simple case, $\underline{\mathbb{E}}_{\rateset,\,\mathcal{M}}^{\mathrm{WHM}}[f(X_t)\,\vert\,X_0=a]$ can of course be closely approximated by taking a very fine discretisation of the interval $[0.01, 0.5]$, and performing an exhaustive search through this discretised parameter space to estimate the value of the lower expectation and the range of (precise) expectations that can be obtained within the set; this leads to the solution that is depicted in Figure~\ref{fig:homogeneousCounterExample}. 

Figure~\ref{fig:homogeneousCounterExample} also depicts $\underline{\mathbb{E}}_{\rateset,\,\mathcal{M}}^{\mathrm{WM}}[f(X_t)\,\vert\,X_0=a]$ and $\underline{\mathbb{E}}_{\rateset,\,\mathcal{M}}^{\mathrm{W}}[f(X_t)\,\vert\,X_0=a]$, which happen to coincide. The reason why they coincide in this case, and the method by which we have computed them, will be explained in Section~\ref{sec:connections}. For now, it suffices to notice that for large enough values of $t$, their common value differs from $\underline{\mathbb{E}}_{\rateset,\,\mathcal{M}}^{\mathrm{WHM}}[f(X_t)\,\vert\,X_0=a]$, thereby illustrating that the first inequality in Proposition~\ref{prop:lower_exp_markov_bounded_by_nonmarkov} can indeed be strict.
\exampleend
\end{exmp}

While this example shows that for sufficiently ``nice'' sets $\rateset$ and in low dimensions, a ``discretise and exhaustive-search'' method allows us to numerically solve the type of non-linear optimisation problems that are associated with $\smash{\underline{\mathbb{E}}^\mathrm{WHM}_{\rateset\,,\mathcal{M}}}$, the computational complexity of such an approach will in general quickly explode as the size of $\states$ increases.\footnote{In the sense that in general, the dimensionality of the parameter-space grows quadratically in the number of states---since the number of elements of a rate matrix grows quadratically in this number, and we are considering sets of rate matrices---and hence the size of any discretisation of this parameter space (for a given precision) scales exponentially in $\abs{\states}^2$. Clearly, performing an exhaustive search through such a discretised parameter-space will quickly become infeasible.}
Since we are not aware of other computational methods, this suggests that computing lower expectations for \ictmc's that are of the type $\smash{\whmprocesses_{\rateset,\,\mathcal{M}}}$ is typically very difficult, and it will therefore often be necessary to resort to approximation methods. For the particular case where $\rateset$ corresponds to an interval matrix, such methods have been explored in, for example, References \cite{Goldsztejn2014} and~\cite{oppenheimer1988}.

In the remainder of this paper, we will not consider \ictmc's that are of the type $\smash{\whmprocesses_{\rateset,\,\mathcal{M}}}$. Instead, we will focus on the lower expectations $\smash{\underline{\mathbb{E}}_{\rateset,\mathcal{M}}^\mathrm{WM}}$ and $\smash{\underline{\mathbb{E}}_{\rateset,\mathcal{M}}^\mathrm{W}}$ that correspond to \ictmc's that are of the type $\smash{\wmprocesses_{\rateset,\,\mathcal{M}}}$ or $\smash{\wprocesses_{\rateset,\,\mathcal{M}}}$, and we will develop efficient methods for computing them. In order to do that, we start by introducing the notion of a lower transition operator.

\section{Towards Lower Transition Operators for \ictmc's}
\label{sec:lowertrans}

As explained in Remark~\ref{remark:expectationT}, the transition matrix $T_t^s$ of a stochastic process $P$ serves as an alternative representation of the expectation operator $\mathbb{E}_P[f(X_s)\,\vert\,X_t]$, in the sense that
\begin{equation}\label{eq:transitionmatrixasexpectation}
[T_t^sf](x_t)=\mathbb{E}_P[f(X_s)\,\vert\,X_t=x_t]
\text{ for all $f\in\gamblesX$ and $x_t\in\states$.}
\end{equation}
Furthermore, if $P$ is a well-behaved homogeneous Markov chain, then as explained in Section~\ref{sec:homogen_markov_chain}, $T_t^s$ is completely determined by a unique transition rate matrix $Q$, in the sense that $T_t^s=e^{Q(s-t)}$.

We will in this section introduce a generalisation of these transition matrices and transition rate matrices, called \emph{lower transition operators} and \emph{lower transition rate operators}, respectively. Furthermore, and most importantly, we will introduce the notion of a \emph{corresponding lower transition operator}, which, much like in the precise homogeneous case, will be completely determined by some given lower transition rate operator. 

Further on in this paper, we will then  show that this type of lower transition operator serves as an alternative representation for the lower expectations of imprecise continuous-time Markov chains, thereby establishing an analogy with Equation~\eqref{eq:transitionmatrixasexpectation}. However, for now, in this section, we focus on introducing the relevant concepts, and on deriving this operator of interest. We end in Section~\ref{sec:properties_lower_trans} by showing that this operator satisfies a number of convenient properties.

\subsection{Lower Transition Operators}\label{subsec:lowertrans_rate}

The central concept that we will be interested in throughout this section is that of a lower transition operator $\lt$.

\begin{definition}[Lower Transition Operator]\label{def:coh_low_trans}
A map $\lt$ from $\gamblesX$ to $\gamblesX$ is called a \emph{lower transition operator} if, for all $f,g\in\gamblesX$, all $\lambda\in\realsnonneg$, and all $x\in\states$:
\begin{enumerate}[label=LT\arabic*:,ref=LT\arabic*]
\item
$\left[\lt\,f\right](x)\geq\min\left\{f(y)\,\colon\,y\in\states\right\}$; \label{LT:bounded_min}
\item
$\left[\lt(f+g)\right](x)\geq \left[\lt\,f\right](x)+\left[\lt\,g\right](x)$; \label{LT:super_additive}
\item
$\left[\lt(\lambda f)\right](x)=\lambda\left[\lt\,f\right](x)$. \label{LT:homo}
\end{enumerate}
\noindent We will use $\underline{\mathbb{T}}$ to denote the set of all lower transition operators.

Such lower transition operators furthermore satisfy the following properties---see Reference~\cite{DeBock:2016} for a proof. For any lower transition operator $\lt$, any $f,f_1,f_2\in\gamblesX$, any $\mu\in\reals$ and any two non-negatively homogeneous operators $A,B$ from $\gamblesX$ to $\gamblesX$:
\begin{enumerate}[label=LT\arabic*:,ref=LT\arabic*,start=4]
\item
$\norm{\lt} \leq 1$; \label{LT:norm_at_most_one}
\item
$f_1\geq f_2~\then~\lt f_1\geq\lt f_2$;\label{LT:monotonicity}
\item
$\lt(f+\mu)=\lt(f)+\mu$;\label{LT:constantadditivity}
\item
$\norm{\lt A - \lt B} \leq \norm{A - B}$. \label{LT:differencenorm}
\end{enumerate}
\vspace{0pt}
\end{definition}

\begin{restatable}{proposition}{lemmacompositioncoherence}
\label{lemma:compositioncoherence}
For any two lower transition operators $\lt,\underline{S}\in\underline{\mathbb{T}}$, their composition $\lt\,\underline{S}$ is again a lower transition operator.
\end{restatable}

The first thing to note is that any transition matrix $T$ will also satisfy properties~\ref{LT:bounded_min}-\ref{LT:homo}, and hence is also a lower transition operator. It is therefore clear that lower transition operators are a generalisation of transition matrices.

A first way to motivate this specific generalisation is to note the following. For any lower transition operator $\lt$ and any $x\in\states$, consider the map $\lt_x:\gamblesX\to\reals$, defined for all $f\in\gamblesX$ as
\begin{equation}\label{eq:lowerprevisionfromlt}
\lt_xf \coloneqq \left[\lt f\right](x)\,.
\end{equation}
Due to properties~\ref{LT:bounded_min}-\ref{LT:homo}, it then follows that $\lt_x$ is a map from $\gamblesX$ to $\reals$ that is super-additive, non-negatively homogeneous, and bounded below by the minimum operator. By definition~\cite[Definition~2.3.3]{Walley:1991vk}, that means that $\lt_x$ is a \emph{coherent lower prevision} on $\gamblesX$. Note that the term ``prevision'' here is a synonym for ``expectation''---be it with a different interpretation attached to it---so this essentially states that $\lt_x$ is a ``coherent lower expectation'' on $\gamblesX$.

For the reader that is unfamiliar with this notion of coherent lower previsions, this is perhaps best clarified as follows. Consider some arbitrary set $\mathcal{M}$ of probability mass functions on $\states$, and consider a map $\underline{\mathbb{E}}:\gamblesX\to\reals$, defined for all $f\in\gamblesX$ as
\begin{equation*}
\underline{\mathbb{E}}f \coloneqq \inf\left\{\sum_{x\in\states} p(x)f(x)\,:\,p\in\mathcal{M}\right\}\,.
\end{equation*}
We call this map $\underline{\mathbb{E}}$ the \emph{lower envelope} of $\mathcal{M}$, and it should be clear that this map computes a lower expectation with respect to $\mathcal{M}$. Furthermore, note that for any $p\in\mathcal{M}$, the quantity $\sum_{x\in\states}p(x)f(x)$ is bounded from below by $\min\{f(y):y\in\states\}$, and since this is true for any $p\in\mathcal{M}$, it follows that $\underline{\mathbb{E}}f$ is also bounded below by this minimum. Similarly, it is easily verified that $\underline{\mathbb{E}}$ is super-additive and non-negatively homogeneous, due to the properties of the $\inf$ operator. Hence, by the definition cited above, the lower expectation operator $\underline{\mathbb{E}}$ that corresponds to $\mathcal{M}$ is a coherent lower prevision on $\gamblesX$. 

Our consideration of the coherent lower prevision $\lt_x$  above is essentially the same idea, but \emph{without} considering explicitly any link to some set of probability mass functions $\mathcal{M}_x$. Suffice it to say that, for any coherent lower prevision $\lt_x$, such a set $\mathcal{M}_x$ will always exist~\cite[Section 10.2]{Huber:1981ch}. It should furthermore be noted that the qualifier ``coherent'' here has connections to the notion of coherence that appeared in Section~\ref{sec:cond_prob}, 
and that it can be given a direct gambling interpretation that resembles the one in Appendix~\ref{app:coherence}; we refer to~\cite{troffaes2013:lp,Walley:1991vk} for further discussion on this.

In any case, for our present purposes, it suffices to realise that, because $\lt_x$ is a coherent lower prevision for any $x\in\states$, the lower transition operator $\lt$ can be seen as a vector of coherent lower previsions. Therefore, and because each of these coherent lower previsions $\lt_x$ has some set $\mathcal{M}_x$ of probability mass functions of which $\lt_x$ computes the lower envelope, we can combine these sets of probability mass functions to form a set of transition matrices
\begin{equation*}
\mathbb{T}_{\lt}=\{T\in\mathbb{T}\colon (\forall x\in\states)~T(x,\cdot)\in\mathcal{M}_x\}
\end{equation*}
of which $\lt$ is the lower envelope, in the sense that
\begin{equation*}
[\lt f](x) = \inf\{ [Tf](x)\,:\, T\in\mathbb{T}_{\lt} \}~
\text{ for all $f\in\gamblesX$ and $x\in\states$.}
\end{equation*}
Due to the correspondence between transition matrices and conditional expectation operators---see Remark~\ref{remark:expectationT}---it should therefore be clear that lower transition operators are an intuitive starting point to try and find alternative characterisations for the lower expectations that correspond to an \ictmc. What remains is to find the specific lower transition operator $\lt$ whose set of dominating transition matrices $\mathbb{T}_{\lt}$ corresponds to the set of transition matrices that is induced by a given \ictmc; the remainder of this section will provide the machinery required to do this.

We conclude with the following result about the set $\underline{\mathbb{T}}$ of all lower transition operators, which states that this set is a complete metric space with respect to our usual norm.

\begin{restatable}{proposition}{lemmacompletemetricspace}
\label{lemma:completemetricspace}
The metric space $(\underline{\mathbb{T}},d)$ is complete with respect to the metric $d$ that is induced by our usual norm $\norm{\cdot}$.
\end{restatable}

\subsection{Lower Transition Rate Operators}\label{sec:connections_rate}

We next focus on the generalisation of transition rate matrices $Q$ to lower transition rate operators $\lrate$, as follows.
\begin{definition}[Lower Transition Rate Operator]\label{def:coh_low_trans_rate}
A map $\lrate$ from $\gamblesX$ to $\gamblesX$ is called a \emph{lower transition rate operator} if, for all $f,g\in\gamblesX$, all $\lambda\in\realsnonneg$, all constant functions $\mu\in\gamblesX$, and all $x\in\states$:

\begin{enumerate}[label=LR\arabic*:,ref=LR\arabic*]
\item\label{LR:constantzero}
$\left[\lrate\mu\right](x)=0$;
\item\label{LR:nondiagpos}
$\left[\lrate\ind{y}\right](x)\geq0$ for all $y\in\states$ such that $x\neq y$;
\item\label{LR:subadditive}
$\left[\lrate(f+g)\right](x)\geq\left[\lrate f\right](x)+\left[\lrate g\right](x)$;
\item\label{LR:homo}
$\left[\lrate(\lambda f)\right](x)= \lambda\left[\lrate f\right](x)$.
\end{enumerate}
Such lower transition rate operators furthermore satisfy the following properties---see Reference~\cite{DeBock:2016} for a proof. For any lower transition rate operator $\lrate$ and any two non-negatively homogeneous operators $A,B$ from $\gamblesX$ to $\gamblesX$:
\begin{enumerate}[label=LR\arabic*:,ref=LR\arabic*,start=5]
\item
$\norm{\lrate}\leq 2\max\big\{\abs{[\lrate\ind{x}](x)}\colon x\in\states\big\} < +\infty$ \label{LR:normlratefinite};
\item
$\norm{\lrate A - \lrate B} \leq 2\norm{\lrate}\norm{A - B}.$ \label{LR:differenceofnorm}
\end{enumerate}
\vspace{0pt}
\end{definition}

Note that properties~\ref{LR:constantzero} and~\ref{LR:nondiagpos} essentially preserve properties~\ref{def:Q:sumzero} and~\ref{def:Q:nonnegoffdiagonal} from Definition~\ref{def:rate_matrix}. The main difference lies in the fact that a rate matrix $Q$ is a linear map, whereas properties~\ref{LR:subadditive} and~\ref{LR:homo} merely require that a lower transition rate operator $\lrate$ is super-additive and non-negatively homogeneous. Therefore, every rate matrix is clearly a lower transition rate operator, with the latter concept providing a generalisation of the former.

A first reason why this specific generalisation is of interest, is because it preserves the relation between transition matrices and rate matrices that was established in Propositions~\ref{prop:stochastic_from_rate_matrix} and~\ref{prop:rate_from_stochastic_matrix}. Indeed, the following two results generalise these relations to our current setting.

\begin{proposition}[Reference {\cite[Proposition 5]{DeBock:2016}}]\label{lemma:normQsmallenough}
Consider any lower transition rate operator $\lrate$, and any $\Delta\in\realsnonneg$ such that $\Delta\norm{\lrate}\leq 1$. Then the operator $(I+\Delta\lrate)$ is a lower transition operator.
\end{proposition}

\begin{proposition}[Reference {\cite[Proposition 6]{DeBock:2016}}]\label{lemma:lower_trans_to_lower_rate}
Consider any lower transition operator $\lt$, and any $\Delta\in\realspos$. Then the operator $\nicefrac{1}{\Delta}(\lt - I)$ is a lower transition rate operator.
\end{proposition}

Now, in our discussion of lower transition operators in Section~\ref{subsec:lowertrans_rate}, we mentioned that lower transition operators can be interpreted as the lower envelope of a set of transition matrices. As we are about to show, there is a similar connection between lower transition rate operators and sets of rate matrices.

So consider any non-empty bounded set $\rateset\subseteq\mathcal{R}$ of rate matrices. Then for any $f\in\gamblesX$, if we let
\begin{equation}\label{eq:correspondinglowertrans}
[\lrate f](x)\coloneqq\inf\{[Qf](x)\colon Q\in\rateset\}
\text{ for all $x\in\states$},
\end{equation}
the resulting function $\lrate f$ is again an element of $\gamblesX$,\footnote{Since $\rateset$ is bounded,~\ref{N:normAf} implies that, for all $Q\in\rateset$, $\norm{Qf}\leq\norm{Q}\norm{f}\leq\norm{\rateset}\norm{f}<+\infty$. Therefore, and since $\rateset$ is non-empty, the components of $\lrate f$ are bounded below by $-\norm{\rateset}\norm{f}$, which implies that $\lrate f$ is a real-valued function on $\states$.}
and therefore, $\lrate$ is a map from $\gamblesX$ to $\gamblesX$. We call this operator $\lrate$, as defined by Equation~\eqref{eq:correspondinglowertrans}, the \emph{lower envelope} of $\rateset$. It is a matter of straightforward verification to see that $\lrate$ is a lower transition rate operator.

\begin{restatable}{proposition}{proplowerenvelopeislowertrans}
\label{prop:lowerenvelopeislowertrans}
For any non-empty bounded set $\rateset\subseteq\mathcal{R}$ of rate matrices, the corresponding operator $\lrate\colon\gamblesX\to\gamblesX$, as defined by Equation~\eqref{eq:correspondinglowertrans}, is a lower transition rate operator.
\end{restatable}

\noindent
Inspired by this result, we will also refer to the lower envelope of $\rateset$ as the \emph{lower transition rate operator that corresponds to $\rateset$}.  
However, this correspondence is not one-to-one. As the following example establishes, different non-empty bounded sets of rate matrices may have the same corresponding lower transition rate operator.

\begin{exmp}\label{example:different_sets_same_lower_rate}
For the sake of simplicity, we assume that the state space $\states$ has only two elements, which allows us to work with $2\times 2$ matrices. Consider now two rate matrices
\begin{equation*}
A\coloneqq\left[\begin{array}{rr}-1 & 1 \\2 & -2\end{array}\right]\quad\text{and}\quad
B\coloneqq\left[\begin{array}{rr}-3 & 3 \\1 & -1\end{array}\right]\,,
\vspace{7pt}
\end{equation*}
let $C\coloneqq \nicefrac{1}{2}(A+B)$ be their convex mixture, which is clearly also a rate matrix, and use these matrices to define the sets $\rateset_1\coloneqq\{A,B\}$ and $\rateset_2\coloneqq\{A,B,C\}$. Then clearly, $\rateset_1$ and $\rateset_2$ are two different, non-empty and bounded sets of rate matrices. Let $\lrate_1$ and $\lrate_2$ be the lower transition rate operators that correspond to $\rateset_1$ and $\rateset_2$, respectively. Then as we prove in~\exampleproofref, these lower transition rate operators are identical, i.e. $\lrate_1=\lrate_2$.
\exampleend
\end{exmp}
Of course, this should not really be surprising---Example~\ref{example:different_sets_same_lower_rate} essentially establishes that different sets can have the same infimum. However, it does lead to a natural question: what do these different sets have in common?

Therefore, we next consider some fixed lower transition rate operator $\lrate$.
All the non-empty bounded sets $\rateset$ of rate matrices that have $\lrate$ as their lower envelope then share a common property: they consist of rate matrices $Q$ that dominate $\lrate$, in the sense that $Qf\geq\lrate f$ for all $f\in\gamblesX$. Therefore, each of these sets $\rateset$ is contained in the following set of dominating rate matrices:
\begin{equation}\label{eq:dominatingratematrices}
\rateset_{\lrate}\coloneqq
\left\{
Q\in\mathcal{R}
\colon
Qf\geq\lrate f\text{ for all $f\in\gamblesX$}
\right\}.
\end{equation}
As our next result shows, this set $\rateset_{\lrate}$ is non-empty and bounded, and has $\lrate$ as its lower envelope. Even stronger, the infimum in Equation~\eqref{eq:correspondinglowertrans} is reached---can be replaced by a minimum.

\begin{restatable}{proposition}{propdominatingnonemptybounded}
\label{prop:dominating_nonempty_bounded}
Consider a lower transition rate operator $\lrate$ and let $\rateset_{\lrate}$ be the corresponding set of dominating rate matrices, as defined by Equation~\eqref{eq:dominatingratematrices}. Then $\rateset_{\lrate}$ is non-empty and bounded and, for all $f\in\gamblesX$, there is some $Q\in\rateset_{\lrate}$ such that $\lrate f=Qf$.
\end{restatable}

\noindent
Because of this result, and since---as discussed above---every non-empty bounded set of rate matrices that has $\lrate$ as its lower envelope is a subset of $\rateset_{\lrate}$, it follows that $\rateset_{\lrate}$ is the largest non-empty bounded set of rate matrices that has $\lrate$ as its lower envelope.
Furthermore, as we show in Proposition~\ref{prop:dominatingproperties} below, this set $\rateset_{\lrate}$ is also closed and convex, and has what we call \emph{separately specified rows}. 

Intuitively, we say that a set $\rateset$ of rate matrices has separately specified rows if it is closed under taking arbitrary combinations of rows from its elements. More formally, if for every $x\in\states$ we let $\rateset_x\coloneqq\{Q(x,\cdot):Q\in\rateset\}$ denote the set of $x$-rows of the matrices in $\rateset$, then we say that $\rateset$ has separately specified rows if $\rateset$ contains every matrix $Q$ that can be constructed by selecting, for all $x\in\states$, an arbitrary row $Q(x,\cdot)$ from $\rateset_x$.

\begin{definition}\label{def:separatelyspecifiedrows}
A set of rate matrices $\rateset\subseteq\mathcal{R}$ has separately specified rows if
\begin{equation*}
\rateset=\left\{
Q\in\mathcal{R}
\colon
(\forall x\in\states)~Q(x,\cdot)\in\rateset_x\right\},
\end{equation*}
where, for every $x\in\states$, $\rateset_x\coloneqq\{Q(x,\cdot)\colon Q\in\rateset\}$ is some given set of rows from which the $x$-row $Q(x,\cdot)$ of the rate matrices $Q$ in $\rateset$ is selected, independently of the other rows.\footnote{This concept is related to the separately specified local models that are often used in credal networks ---see, e.g., Reference~\cite{daRocha:2002uf}---and, for readers that are familiar with the literature on this latter subject, can be regarded as the continuous-time analogue of this notion.}
\end{definition}

\begin{restatable}{proposition}{propdominatingproperties}
\label{prop:dominatingproperties}
Consider a lower transition rate operator $\lrate$ and let $\rateset_{\lrate}$ be the corresponding set of dominating rate matrices, as defined by Equation~\eqref{eq:dominatingratematrices}. Then $\rateset_{\lrate}$ is closed and convex, and has separately specified rows.
\end{restatable}

\noindent
These additional properties characterise $\rateset_{\lrate}$ completely, in the sense that no other set satisfies them.

\begin{restatable}{proposition}{propdominatinguniquecharacterization}
\label{prop:dominating_unique_characterization}
Consider any non-empty, bounded, closed and convex set of rate matrices $\rateset\subseteq\mathcal{R}$ with separately specified rows that has $\lrate$ as its lower envelope. Then $\rateset=\rateset_{\lrate}$.
\end{restatable}

\begin{exmp}\label{ex:dominatingset}
Let $\rateset_1$ and $\rateset_2$ be constructed as in Example~\ref{example:different_sets_same_lower_rate}, and let $\lrate\coloneqq\lrate_1=\lrate_2$ be their common lower transition rate operator. As we are about to show, the corresponding set of dominating rate matrices $\rateset_{\lrate}$ is then equal to
\begin{equation}\label{eq:ex:dominatingset:globaldef}
\rateset^* \coloneqq \{Q\in\mathcal{R}\,:\,(\forall x\in\states)\, Q(x,\cdot)\in\rateset_x\},
\end{equation}
where, for all $x\in\states$, $\rateset_x$ is given by
\begin{equation}\label{eq:ex:dominatingset:localdef}
\rateset_x \coloneqq \left\{\lambda A(x,\cdot)+(1-\lambda)B(x,\cdot)\,:\,\lambda\in[0,1]\right\}.
\end{equation}

First of all, as we shown in~\exampleproofref, $\lrate$ is also the lower transition operator corresponding to $\rateset^*$. Furthermore, $\rateset^*$ is clearly non-empty, bounded, closed, convex, has separately specified rows and, as we have just mentioned, has $\lrate$ as its lower transition rate operator. Therefore, by Proposition~\ref{prop:dominating_unique_characterization}, it follows that $\rateset^*=\rateset_{\lrate}$.
\exampleend
\end{exmp}

We conclude from all of this that non-empty bounded sets of rate matrices are more informative than lower transition rate operators, in the following sense. Different non-empty bounded sets of rate matrices $\rateset$ may have the same lower transition rate operator $\lrate$ and therefore, in general, knowledge of $\lrate$ does not suffice to reconstruct $\rateset$; we can only reconstruct an outer approximation $\rateset_{\lrate}$, which is guaranteed to include $\rateset$. This changes if, besides non-empty and bounded, $\rateset$ is also closed and convex and has separately specified rows. In that case, $\lrate$ serves as an alternative representation for $\rateset$ because, since $\rateset=\rateset_{\lrate}$, we can use $\lrate$ to reconstruct $\rateset$. In other words: there is a one-to-one correspondence between lower transition rate operators and non-empty, bounded, closed and convex sets of rate matrices that have separately specified rows.

\subsection{Corresponding Lower Transition Operators}

Proposition~\ref{lemma:normQsmallenough} already established that we can fairly easily construct a lower transition operator from a given lower transition rate operator $\lrate$: if $\Delta\geq0$ is sufficiently small, then $I+\Delta\lrate$ will be a lower transition operator. In this section, we construct a somewhat more complicated lower transition operator from a given lower transition rate operator, and it is this specific lower transition operator on which we will focus for the remainder of this work. In particular, we will introduce the \emph{lower transition operator corresponding to a given lower transition rate operator}.

To this end, we will assume here that we are given some arbitrary lower transition rate operator $\lrate$, and any two time points $t,s\in\realsnonneg$ such that $t\leq s$. For any $u\in\mathcal{U}_{[t,s]}$ such that $u=t_0,\ldots,t_n$, we then define the auxiliary operator
\begin{equation}\label{eq:aux_lower_trans}
\Phi_u\coloneqq\prod_{i=1}^n(I+\Delta_i\lrate)\,,
\end{equation}
where, as in Section~\ref{subsec:sequencesoftimepoints}, for every $i\in\{1,\ldots,n\}$, $\Delta_i= t_i-t_{i-1}$ denotes the difference between two consecutive time points in $u$, and $\sigma(u)\coloneqq \max\{\Delta_i:i\in\{1,\ldots,n\}\}$ is the maximum such difference. Clearly, if $\sigma(u)$ is small enough, Proposition~\ref{lemma:normQsmallenough} guarantees that each of the terms $I+\Delta_i\lrate$ is a lower transition operator, and it then follows from Proposition~\ref{lemma:compositioncoherence} that $\Phi_u$---since it is a composition of lower transition operators---is also a lower transition operator. The so-called \emph{lower transition operator corresponding to} $\lrate$ will be defined below as the limit of these lower transition operators $\Phi_u$, obtained as we take $u$ to be an increasingly finer partition of the interval $[t,s]$.

However, before we can do that, we first need to establish that this limit indeed exists. To this end, we start by providing a bound on the distance between two operators $\Phi_u$ and $\Phi_{u*}$.

\begin{restatable}{proposition}{propdifferencebetweenu}
\label{prop:differencebetweenu}
Consider any $t,s\in\realsnonneg$ with $t\leq s$, any $\delta\in\realspos$ such that $\delta\norm{\lrate}\leq1$, and any $u,u^*\in\mathcal{U}_{[t,s]}$ such that $\sigma(u)\leq\delta$ and $\sigma(u^*)\leq\delta$. Let $C\coloneqq s-t$. Then
\begin{equation*}
\norm{\Phi_u-\Phi_{u^*}}\leq 2\delta C\norm{\lrate}^2\,.
\end{equation*}\\[-27pt]
\end{restatable}

Note, therefore, that the distance $\norm{\Phi_u - \Phi_{u^*}}$ vanishes as we make $\sigma(u)$ and $\sigma(u^*)$ smaller and smaller. This allows us to state the following result.

\begin{restatable}{corollary}{corolcauchy}
\label{corol:cauchy}
For every sequence $\{u_i\}_{i\in\nats}$ in $\mathcal{U}_{[t,s]}$ such that $\lim_{i\to\infty}\sigma(u_i)=0$, the corresponding sequence $\{\Phi_{u_i}\}_{i\in\nats}$ is Cauchy.
\end{restatable}

Since we already know that, for partitions $u$ that are sufficiently fine, $\Phi_u$ is a lower transition operator, Proposition~\ref{lemma:completemetricspace} now implies that this Cauchy sequence converges to a limit, and that this limit is again a lower transition operator.

\begin{restatable}{corollary}{corollimitexistsandiscoherent}
\label{corol:limitexistsandiscoherent}
For every sequence $\{u_i\}_{i\in\nats}$ in $\mathcal{U}_{[t,s]}$ such that $\lim_{i\to\infty}\sigma(u_i)=0$, the corresponding sequence $\{\Phi_{u_i}\}_{i\in\nats}$ converges to a lower transition operator.
\end{restatable}

Finally, as our next result establishes, this limit is unique, in the sense that it is independent of the choice of $\{u_i\}_{i\in\nats}$.

\begin{restatable}{theorem}{theoconvergencelowerbound}
\label{theo:convergencelowerbound}
For any $t,s\in\realsnonneg$ such that $t\leq s$ and any lower transition rate operator $\lrate$, there is a unique lower transition operator $\lt\in\underline{\mathbb{T}}$ such that 
\begin{equation}\label{eq:theo:convergencelowerbound}
(\forall\epsilon>0)\,
(\exists\delta>0)\,
(\forall u\in\mathcal{U}_{[t,s]}\colon\sigma(u)\leq\delta)~\norm{\lt - \Phi_u}\leq\epsilon.
\end{equation}
\end{restatable}

Note that the $\epsilon-\delta$ expression in Theorem~\ref{theo:convergencelowerbound} is a limit statement. Specifically, it is a limit of operators $\Phi_{u}$ corresponding to increasingly finer partitions $u$ of the interval $[t,s]$. In the sequel, whenever such a unique limit exists and equals some lower transition operator $\lt$, we will denote it as
\begin{equation}\label{eq:net_limit_lower_trans}
\lim_{\sigma(u)\to0}\left\{\Phi_u\,\colon\,u\in\mathcal{U}_{[t,s]}\right\} = \lt\,.
\end{equation}
Here, the notation is understood to indicate that the limit of these operators $\Phi_{u}$ is independent of the exact choice of $\{u_i\}_{i\in\nats}$ in $\mathcal{U}_{[t,s]}$, so long as $\lim_{i\to\infty}\sigma(u_i)=0$.

We are now ready to define the \emph{lower transition operator corresponding to $\lrate$}, which is the operator in which we will be interested for the remainder of this work.

\begin{definition}[Corresponding Lower Transition Operator]\label{def:low_trans}
Consider any $t,s\in\realsnonneg$ such that $t\leq s$ and let $\lrate$ be an arbitrary lower transition rate operator. The \emph{corresponding lower transition operator} $\lbound_t^s$ is a map from $\gamblesX$ to $\gamblesX$, defined by
\begin{equation*}
\lbound_t^s\coloneqq\lim_{\sigma(u)\to0}\left\{ \Phi_u\,\colon\,u\in\mathcal{U}_{[t,s]}\right\},
\end{equation*}
where the limit is understood as in Equation~\eqref{eq:net_limit_lower_trans}.
\end{definition}

\subsection{Properties of Corresponding Lower Transition Operators}\label{sec:properties_lower_trans}

We will next establish that this operator $L_t^s$ satisfies a number of convenient properties. In particular, we will focus on the family $\underline{\mathcal{T}}_{\lrate}$ of lower transition operators corresponding to a given lower transition rate operator $\lrate$.

\begin{definition}[Lower Transition Operator System]\label{def:low_trans_system}
Let $\lrate$ be an arbitrary lower transition rate operator. Then, the \emph{lower transition operator system} corresponding to $\lrate$ is the family $\underline{\mathcal{T}}_{\lrate}$ of lower transition operators $L_t^s$ corresponding to $\lrate$, defined for all $t,s\in\realsnonneg$, with $t\leq s$, as in Definition~\ref{def:low_trans}.
\end{definition}

Our first result is that this family $\smash{\underline{\mathcal{T}}_{\lrate}}$ satisfies the same semi-group property that was found to hold for the transition matrix system $\mathcal{T}_P$ of a Markov chain $P\in\mprocesses$ in Section~\ref{sec:cont_time_markov_chains}.

\begin{restatable}{proposition}{proplowertranssystemissystem}
\label{prop:lower_trans_system_is_system}
Let $\lrate$ be an arbitrary lower transition rate operator, and let $\underline{\mathcal{T}}_{\lrate}$ be the corresponding lower transition operator system. Then, for all $t,r,s\in\realsnonneg$ such that $t\leq r\leq s$, it holds that
\begin{equation*}
L_t^s = L_t^rL_r^s.
\vspace{6pt}
\end{equation*}
Furthermore, for all $t\in\realsnonneg$, we have that $L_t^t=I$.
\end{restatable}

Our next result is that this family $\underline{\mathcal{T}}_{\lrate}$ is time-homogeneous:

\begin{restatable}{proposition}{proplowertransitionishomogeneous}
\label{prop:lower_transition_is_homogeneous}
Let $\lrate$ be an arbitrary lower transition rate operator, and let $\smash{\underline{\mathcal{T}}_{\lrate}}$ be the corresponding lower transition operator system. Then, for all $t,s\in\realsnonneg$ such that $t\leq s$, we have that $L_t^s=L_0^{s-t}$.
\end{restatable}

Finally, we find that the derivatives of these lower transition operators always exist, and that they furthermore satisfy the following simple equalities.

\begin{restatable}{proposition}{proplowertransitionhasderiv}
\label{prop:lower_transition_has_deriv}
Let $\lrate$ be an arbitrary lower transition rate operator, and let $\underline{\mathcal{T}}_{\lrate}$ be the corresponding lower transition operator system. Then, for all $t,s\in\realsnonneg$ such that $t\leq s$, it holds that\footnote{If $0=t<s$, the derivative with respect to $t$ is taken to be a right derivative. If $t=s$, the derivative with respect to $s$ is taken to be a right derivative and the derivative with respect to $t$ is taken to be a left derivative (or becomes meaningless if $t=0$).}
\begin{equation*}
\frac{\partial}{\partial t}\lbound_t^s=-\lrate\lbound_t^s\,\quad\text{and}\quad\frac{\partial}{\partial s}\lbound_t^s=\lrate\lbound_t^s.\vspace{7pt}
\end{equation*}
\end{restatable}
We would like to point out here that the derivatives in this result are not taken pointwise, but are taken with respect to the operator norm. For example, for $t=0$ and $s>0$, Proposition~\ref{prop:lower_transition_has_deriv} does not state that
\vspace{3pt}
\begin{equation}\label{eq:pointwisedifferential}
\frac{d}{d s}\lbound_0^sf=\lrate\lbound_0^sf
~\text{ for all $f\in\gamblesX$,}
\end{equation}
but rather that
\begin{equation}\label{eq:uniformdifferential}
\lim_{\Delta\to0}
\norm{\frac{L_0^{s+\Delta}-L_0^s}{\Delta}-\lrate L_0^s}=0.\vspace{9pt}
\end{equation}
Of these two statements, the latter is the strongest one, in the sense that it trivially implies the former. Hence, although from an intuitive point of view, the reader may wish to interpret the results in Proposition~\ref{prop:lower_transition_has_deriv} as in Equation~\eqref{eq:pointwisedifferential}---which would be correct---one should keep in mind that from a technical point of view, the result is in fact stronger, and is intended to be read as in Equation~\eqref{eq:uniformdifferential}.

It is also worth noting that, as a consequence of Propositions~\ref{prop:lower_transition_has_deriv} and~\ref{prop:lower_trans_system_is_system}, the operator $L_t^s$ satisfies the following differential equation:
\begin{equation*}
\frac{\partial}{\partial s}L_t^s=\lrate L_t^s\,,\quad\quad L_t^t=I\,.
\end{equation*}
Observe, therefore, the strong correspondence between the operator $L_t^s$ corresponding to some $\lrate$, and the matrix exponential $e^{Q(s-t)}$ of a rate matrix $Q\in\mathcal{R}$. In particular, as is very well known~\cite[Equation 4.4]{van2006study}, this matrix exponential is the unique solution of the differential equation
\begin{equation*}
\frac{\partial}{\partial s}e^{Q(s-t)}=Qe^{Q(s-t)}\,,\quad\quad e^{Q(t-t)}=I\,.
\end{equation*}
Now, recall that any rate matrix $Q$ is also a lower transition rate operator. It then follows from the above that the lower transition operator $L_t^s$ that corresponds to this $\lrate=Q$ is given by $\smash{L_t^s=e^{Q(s-t)}}$. 
Hence, more generally, $L_t^s$ can be regarded as a generalised version of the matrix exponential of a transition rate matrix, and---with some slight abuse of terminology---can be considered to be the `matrix exponential' of the lower transition rate operator $\lrate$.

Another closely related observation is that, if $\lrate=Q$ for some $Q\in\mathcal{R}$, then the family $\smash{\underline{\mathcal{T}}_{\lrate}}$ is equal to $\mathcal{T}_Q$, which is the exponential transition matrix system from Definition~\ref{def:systemfromQ} that, by Corollary~\ref{cor:rate_has_unique_homogen_markov_process}, is known to correspond to a well-behaved homogeneous Markov chain $P\in\whmprocesses$. 
Interestingly, then, the family $\underline{\mathcal{T}}_{\lrate}$ maintains the convenient properties of differentiability, time-homogeneity, and ``Markovian-like'' factorisation, when instead of some rate matrix $Q$ we replace it by a lower transition rate operator~$\lrate$.

Mathematical niceties aside, we are of course not really interested in the trivial case where $\lrate=Q$. Instead, we wish to use the lower transition operator $L_t^s$ to compute lower expectations for imprecise continuous-time Markov chains. We will show in the next section that this is indeed possible.

\section{Connecting \ictmc's and Lower Transition Operators}\label{sec:connections}

As we know from Section~\ref{subsec:lowertrans_rate}, lower transition operators are essentially just lower envelopes of transition matrices. Combined with the fact that transition matrices are a convenient tool for representing and computing expectations in a Markov chain, it seems intuitive to expect that, similarly, lower transition operators can be used to represent and compute lower expectations in an imprecise Markov chain. We will show in this section that this is indeed the case. 

In particular, we establish in this section that for \ictmc's that are of the type $\smash{\wmprocesses_{\rateset,\mathcal{M}}}$ or $\smash{\wprocesses_{\rateset,\mathcal{M}}}$, with $\rateset$ having separately specified rows, we can use the lower transition operator $L_t^s$ to represent and compute conditional lower expectations of functions $f(X_s)$ that depend on the state $X_s$ at a single time point $s$ in the future. The treatment of more general functions is deferred to Section~\ref{sec:funcs_multi_time_points}.

\subsection{Lower Transition Operators as a Representational Tool}\label{sec:single_var_lower_exp}

In order to establish a connection between the operator $L_t^s$ and the lower expectations that correspond to an \ictmc, it is important to realise that the latter is derived from a set $\rateset$ of transition rate matrices---as in Definition~\ref{def:process_sets}---whereas the former is derived from a lower transition rate operator $\lrate$---as in Definition~\ref{def:low_trans}. Therefore, we clearly need to start by creating a link between $\rateset$ and $\lrate$. Fortunately, we have already seen in Section~\ref{sec:connections_rate} that there is a strong connection between sets of rate matrices $\rateset$ and lower transition rate operators $\lrate$. In particular, any set $\rateset$ has a corresponding lower transition rate operator $\lrate$, which computes the lower envelope with respect to $\rateset$, as in Equation~\eqref{eq:correspondinglowertrans}. It is exactly this connection between sets of transition rate matrices and lower transition rate operators that we will use here to establish a connection between the operator $L_t^s$ and the lower expectations that correspond to an \ictmc.

To start with, as the following result shows, for any lower transition rate operator $\lrate$, the corresponding lower transition operator $L_t^s$ provides a lower bound on the conditional expectations $\mathbb{E}_P[f(X_s)\,\vert\,X_t=x_t,X_u=x_u]$ of any well-behaved stochastic process $\smash{P\in\wprocesses_\rateset}$ that is consistent with a set of rate matrices $\rateset$ that has $\lrate$ as its the lower envelope.

\begin{restatable}{proposition}{theoremnonmarkovsinglevarlowerbounded}
\label{theorem:nonmarkov_single_var_lower_bounded}
Consider a non-empty bounded set of rate matrices $\rateset$ whose corresponding lower transition rate operator is $\lrate$, and let $\smash{\underline{\mathcal{T}}_{\lrate}}$ be the corresponding lower transition operator system. Then, for any $\smash{P\in\wprocesses_\rateset}$, any $t,s\in\realsnonneg$ such that $t\leq s$, any $u\in\mathcal{U}_{<t}$, any $x_t\in\states$ and $x_u\in\states_u$, and any $f\in\gamblesX$:
\begin{equation*}
 \mathbb{E}_P[f(X_s)\,\vert\,X_t=x_t,X_u=x_u]\geq[L_{t}^s f](x_t).
\end{equation*}
\end{restatable}

Notice that this result is stated for stochastic processes $P$ in $\smash{\wprocesses_{\rateset}}$, whose initial distributions $P(X_0)$ are not required to belong to some given set of initial distributions $\mathcal{M}$. However, of course, since $\wprocesses_{\rateset,\,\mathcal{M}}$ is a clearly a subset of $\wprocesses_{\rateset}$, the same result also holds for any choice of such $\mathcal{M}$.

Our next result establishes that the bound in Proposition~\ref{theorem:nonmarkov_single_var_lower_bounded} is tight if $\rateset$ has separately specified rows. Specifically, we show that $L_t^sf$ can then be approximated to arbitrary precision by carefully choosing a Markov process $P$ from the set $\wmprocesses_{\rateset,\,\mathcal{M}}$.

\begin{restatable}{proposition}{theoremlowermarkovboundistight}
\label{theorem:lower_markov_bound_is_tight}
Let $\mathcal{M}$ be a non-empty set of probability mass functions on $\states$, let $\rateset$ be a non-empty bounded set of rate matrices that has separately specified rows, with corresponding lower transition rate operator $\lrate$, and let $\smash{\underline{\mathcal{T}}_{\lrate}}$ be the corresponding lower transition operator system. Then for all $t,s\in\realsnonneg$ such that $t\leq s$, all $f\in\gamblesX$, and all $\epsilon\in\realspos$, there is a well-behaved Markov chain $P\in\wmprocesses_{\rateset,\,\mathcal{M}}$ such that
\begin{equation*}
\abs{\mathbb{E}_P[f(X_s)\,\vert\,X_t=x_t]-[\lbound_t^sf](x_t)} < \epsilon
~\text{ for all $x_t\in\states$.}
\end{equation*}
\end{restatable}

Together, Propositions~\ref{theorem:nonmarkov_single_var_lower_bounded} and~\ref{theorem:lower_markov_bound_is_tight} establish a strong connection between the operator $L_t^s$ and the lower expectations that correspond to $\smash{\wmprocesses_{\rateset,\mathcal{M}}}$ or $\smash{\wprocesses_{\rateset,\mathcal{M}}}$. In particular, for $\rateset$ with separately specified rows, and for functions $f(X_s)$ that depend on the state $X_s$ at a single time point $s$ in the future, these three objects end up being identical.

\begin{restatable}{corollary}{corloweroperatorisinfimum}
\label{cor:lower_operator_is_infimum}
Let $\mathcal{M}$ be a non-empty set of probability mass functions on $\states$, let $\rateset$ be a non-empty bounded set of rate matrices that has separately specified rows, with corresponding lower transition rate operator $\lrate$, and let $\underline{\mathcal{T}}_{\lrate}$ be the corresponding lower transition operator system. Then, for all $t,s\in\realsnonneg$ such that $t\leq s$, all $u\in\mathcal{U}_{<t}$, $x_u\in\states_u$ and $x_t\in\states$, and all $f\in\gamblesX$:
\begin{align*}
\underline{\mathbb{E}}^{\mathrm{W}}_{\,\rateset,\,\mathcal{M}}[f(X_s)\,\vert\,X_t=x_t,X_u=x_u]=\underline{\mathbb{E}}^{\mathrm{WM}}_{\,\rateset,\,\mathcal{M}}[f(X_s)\,\vert\,X_t=x_t,X_u=x_u] =\left[L_t^sf\right](x_t).
\end{align*}\\[-25pt]
\end{restatable}

Hence, we find that there is indeed a correspondence between the operator $L_t^s$ and the lower expectations that correspond to \ictmc's.

This result also helps to clarify why we choose to call $\smash{\wprocesses_{\rateset,\mathcal{M}}}$ an imprecise Markov chain, despite the fact that it contains processes that do not satisfy the Markov property. In order to see that, observe that Corollary~\ref{cor:lower_operator_is_infimum} holds for \emph{all} histories $x_u\in\states_u$ and \emph{any} sequence of time points $u\in\mathcal{U}_{<t}$. Therefore, and because the definition of $L_t^s$ does not depend on this choice of $u$ and $x_u$, it follows that for $\rateset$ with separately specified rows:
\vspace{3pt}
\begin{equation}\label{eq:impreciseMarkov1}
\underline{\mathbb{E}}_{\rateset,\mathcal{M}}^\mathrm{WM}[f(X_s)\,\vert\,X_t=x_t,X_u=x_u] = \underline{\mathbb{E}}_{\rateset,\mathcal{M}}^\mathrm{WM}[f(X_s)\,\vert\,X_t=x_t]\,\,\,
\end{equation}
and
\begin{equation}\label{eq:impreciseMarkov2}
\underline{\mathbb{E}}_{\rateset,\mathcal{M}}^\mathrm{W}[f(X_s)\,\vert\,X_t=x_t,X_u=x_u] = \underline{\mathbb{E}}_{\rateset,\mathcal{M}}^\mathrm{W}[f(X_s)\,\vert\,X_t=x_t].
\vspace{5pt}
\end{equation}
In other words, the conditional lower expectations $\underline{\mathbb{E}}_{\rateset,\mathcal{M}}^\mathrm{WM}$ and $\underline{\mathbb{E}}_{\rateset,\mathcal{M}}^\mathrm{W}$ satisfy an \emph{imprecise Markov property}: conditional on the state at time $t$, the lower expectation of a function $f(X_s)$ at a future time point $s$ is functionally independent of the states at time points $u$ that precede $t$.
For the lower expectation $\smash{\underline{\mathbb{E}}_{\rateset,\mathcal{M}}^\mathrm{WM}}$, this is of course to be expected: since $\underline{\mathbb{E}}_{\rateset,\mathcal{M}}^\mathrm{WM}$ is the lower envelope of the set of Markov processes $\smash{\wmprocesses_{\rateset,\mathcal{M}}}$, it is not surprising that this lower envelope itself satisfies a Markov property as well. In fact, for this reason, Equation~\eqref{eq:impreciseMarkov1} is clearly also true if $\rateset$ does not have separately specified rows. The most important message here though is that $\smash{\underline{\mathbb{E}}_{\rateset,\mathcal{M}}^\mathrm{W}}$ also satisfies such an imprecise Markov property. In this case, this result is far from trivial, because the individual processes in $\smash{\wprocesses_{\rateset,\mathcal{M}}}$ are not required to---and usually do not---satisfy a Markov property. It remains an open question at this point whether Equation~\eqref{eq:impreciseMarkov2} also holds if $\rateset$ does not have separately specified rows.

That being said, Corollary~\ref{cor:lower_operator_is_infimum} also establishes that the correspondence between $\smash{L_t^s}$ and \ictmc's is not one-to-one. For starters, $L_t^s$ computes the lower expectation for two different sets of processes: $\wmprocesses_\rateset$ and $\wprocesses_\rateset$. Furthermore, we know from Section~\ref{sec:connections_rate} that different sets $\rateset_1$ and $\rateset_2$ may have the same corresponding lower transition rate operator $\lrate$. Hence, whenever this is the case, $L_t^s$ will---assuming the conditions in Corollary~\ref{cor:lower_operator_is_infimum} are met by both $\rateset_1$ and $\rateset_2$---compute the lower expectation with respect to the sets of stochastic processes $\wmprocesses_{\rateset_1}$, $\wmprocesses_{\rateset_2}$, $\wprocesses_{\rateset_1}$ and $\wprocesses_{\rateset_2}$.

A particularly interesting special case corresponds to the situation where $\rateset_1$ is a non-empty bounded set of rate matrices $\rateset$ that has separately specified rows, and $\rateset_2$ is its closed convex hull, which, because of Proposition~\ref{prop:dominating_unique_characterization}, is equal to $\rateset_{\lrate}$, where $\lrate$ is the lower transition rate operator that corresponds to $\rateset$. The two sets of transition rate matrices $\rateset$ and $\rateset_{\lrate}$ then clearly (i) have the same lower corresponding lower transition rate operator $\lrate$ and (ii) satisfy the conditions in Corollary~\ref{cor:lower_operator_is_infimum}. Therefore, it follows from the preceding argument that the resulting lower expectations are identical and, in particular, that
\begin{equation*}
 \underline{\mathbb{E}}_{\,\rateset}^{\mathrm{WM}}[f(X_s)\,\vert\,X_{t}=x_t,X_u=x_u] = \underline{\mathbb{E}}_{\,\rateset_{\lrate}}^{\mathrm{W}}[f(X_s)\,\vert\,X_t=x_t,X_u=x_u]=[L_{t}^sf](x_t),
\end{equation*}
which in turn immediately implies that for any set of stochastic processes $\mathcal{P}$ such that $\smash{\wmprocesses_\rateset \subseteq \mathcal{P} \subseteq \wprocesses_{\rateset_{\lrate}}}$:
\begin{equation}\label{eq:EequalsLformathcalP}
 \underline{\mathbb{E}}[f(X_s)\,\vert\,X_{t}=x_t,X_u=x_u] =[L_{t}^sf](x_t),
\end{equation}
where $\underline{\mathbb{E}}$ is the lower expectation with respect to $\mathcal{P}$, as defined in Equation~\eqref{eq:genericlowerexpectation}.

A common feature of these sets of stochastic processes $\mathcal{P}$, is that each of their elements $P$ is well-behaved and consistent with $\rateset_{\lrate}$. An obvious question, then, is whether this feature is necessary in order for Equation~\eqref{eq:EequalsLformathcalP} to hold. The following result establishes that this is indeed the case. 

\begin{restatable}{theorem}{theodominatingrateprocessesmaxset}
\label{theo:dominating_rate_processes_max_set}
Let $\lrate$ be an arbitrary lower transition rate operator, with $\rateset_{\lrate}$ its set of dominating rate matrices, and let $\smash{\underline{\mathcal{T}}_{\lrate}}$ be the corresponding lower transition operator system. Then the largest set of stochastic processes $\mathcal{P}$ for which the corresponding conditional lower expectation operator $\underline{\mathbb{E}}[\cdot\,\vert\,\cdot]$---as defined in Equation~\eqref{eq:genericlowerexpectation}---satisfies
\begin{equation*}
\underline{\mathbb{E}}[f(X_s)\,\vert\,X_t=x_t,X_u=x_u]=[L_t^sf](x_t)
\end{equation*}
for all $t,s\in\realsnonneg$ such that $t\leq s$, all $u\in\mathcal{U}_{<t}$, all $x_t\in\states$ and $x_u\in\states_u$, and every $f\in\gamblesX$, is the set $\wprocesses_{\rateset_{\lrate}}$.
\end{restatable}

We regard this result as a vindication for our choice to focus on \emph{well-behaved} stochastic processes---instead of more restricted ones, such as, say, continuous or differentiable stochastic processes. Since our aim here is to use $L_t^s$ as a representational and computational tool for lower expectations, it follows from this result that in order to be able to do this, it is indeed necessary to impose this minimal property of well-behavedness.

\subsection{Lower Transition Operators as a Computational Tool}\label{subsec:compute_single_var}

An important consequence of the fact that the lower expectations $\underline{\mathbb{E}}_{\,\rateset,\mathcal{M}}^{\mathrm{WM}}$ and $\underline{\mathbb{E}}_{\,\rateset,\mathcal{M}}^{\mathrm{W}}$ can be conveniently represented by the operator $L_t^s$---at least for functions of the state at a single time point in the future---is that we can focus our computational efforts on evaluating this operator $L_t^s$, thereby abstracting away all the technicalities of dealing with lower expectations with respect to sets of stochastic process. Of course, in practice, we are only interested in a finite precision approximation of $L_t^sf$, and we will therefore focus on computing $L_t^sf$ within some guaranteed $\epsilon$-bound.

The construction of the operator $L_t^s$ in Section~\ref{sec:lowertrans} already suggests how we can do this; namely, by using a finite-precision approximation of $L_t^s$ using the auxiliary operator $\Phi_u\coloneqq \prod_{i=1}^n(I+\Delta_i\lrate)$. Recall from Section~\ref{sec:lowertrans} that the approximation of $L_t^s$ by $\Phi_u$ becomes better as we take $u\in\mathcal{U}_{[t,s]}$ to be an increasingly finer partition of the interval $[t,s]$. The following result tells us exactly how fine this partition needs to be for a specific function $f\in\gamblesX$, in order to guarantee an $\epsilon$-error bound on $L_t^sf$.

\begin{restatable}{proposition}{propapproximationerrorbound}
\label{prop:approximation_error_bound}
Let $\lrate$ be a lower transition rate operator, choose any $t,s\in\realsnonneg$ such that $t\leq s$, and let $L_t^s$ be the lower transition operator corresponding to $\lrate$. Then for any $f\in\gamblesX$ and $\epsilon\in\realspos$, if we choose any $n\in\nats$ such that
\begin{equation*}
n \geq\max\left\{
(s-t)\norm{\lrate},
\frac{1}{2\epsilon}(s-t)^2\norm{\lrate}^2\norm{f}_\mathrm{v}
\right\},
\end{equation*}
with $\norm{f}_\mathrm{v}\coloneqq\max f-\min f$,
we are guaranteed that
\begin{equation*}
\norm{L_t^sf - \prod_{i=1}^n(I + \Delta\lrate)f} \leq \epsilon,
\end{equation*}
with $\Delta\coloneqq \nicefrac{(s-t)}{n}$.
\end{restatable}

Simply put, this result tells us that if we can compute $\lrate g$ for all $g\in\gamblesX$, then we can also approximate the quantity $L_t^sf$ to arbitrary precision, for any given $f\in\gamblesX$. Therefore, and because of our results in Section~\ref{sec:single_var_lower_exp}, if $\lrate$ is the lower transition rate operator that corresponds to a given set $\rateset$ that is non-empty, bounded and has separately specified rows, the non-linear optimization problem of computing $\underline{\mathbb{E}}_\rateset^\mathrm{WM}[f(X_s)\vert X_t=x_t,X_u=x_u]$ or $\underline{\mathbb{E}}_\rateset^\mathrm{W}[f(X_s)\vert X_t=x_t,X_u=x_u]$ reduces to the problem of computing $\lrate g_i$ for $n$ different functions $g_i\in\gamblesX$, with $n\in\nats$ as in Proposition~\ref{prop:approximation_error_bound}.

In the remainder of this section, we will numerically illustrate this method of computing the conditional lower expectation of a given function $f\in\gamblesX$. In order to make this illustration less abstract, we consider the context of the simple disease model that was put forward in Example~\ref{ex:health_sick_exmp} in Section~\ref{sec:introduction}. To this end, we first construct a parameter set $\rateset$ that we will use in the examples to come.

\begin{exmp}\label{exmp:example_rateset_simple_model}
Consider again the binary-state disease model from Example~\ref{ex:health_sick_exmp}, modelling a person periodically becoming sick and recovering after some time. The state space here is of the form $\states=\{\text{{\tt healthy}},\text{{\tt sick}}\}$, and we wish to specify numerical values for the rate at which the two possible transitions occur; that is, the transitions from {\tt healthy} to {\tt sick}, and from {\tt sick} to {\tt healthy}.

If we were using a precise homogeneous Markov chain $P\in\whmprocesses$, it would suffice to select a single rate matrix $Q\in\mathcal{R}$, which, in this binary case, would be of the form
\begin{equation*}
Q = \left[ \begin{array}{rr}
-a & a \\
b & -b
\end{array}\right] \text{ for some } a,b\in\reals_{\geq0}.
\end{equation*}
The parameter $a$ here is the rate at which a healthy person becomes sick.
Technically, this means that if a person is healthy at time $t$, the probability that he or she will be sick at time $t+\Delta$, for small $\Delta$, is very close to $\Delta a$. More intuitively, if we take the time unit to be one week, it means that he or she will, on average, become sick after $\nicefrac{1}{a}$ weeks. The parameter $b$ is the rate at which a sick person becomes healthy again, and has a similar interpretation. In other words, if we would for example say that $a=\nicefrac{1}{52}$ and $b=1$, then this would mean that---on average---we expect that a healthy person will stay healthy for about one year---52 weeks---and that a sick person will remain sick for one week.

In practice, assessing the exact values of $a$ and $b$ may be difficult, especially if we want them to be reliable. In those cases, instead of modelling our problem by means of a precise continuous-time Markov chain, we can consider an \ictmc. We then no longer need to assess a single transition rate matrix $Q$, but only need to specify a set of rate matrices $\rateset$ to which we think $Q$ might belong. Suppose for example that we feel confident in saying that the average time for a healthy person to become sick lies somewhere in between four months and a year, and that the average time for a sick person to recover is situated somewhere in between half a week and two weeks. The corresponding set of transition rate matrices is then given by
\vspace{3pt}
\begin{equation}\label{eq:num_example_rateset_params}
\rateset \coloneqq \left\{\left[\begin{array}{rr}
-a & a \\
b & -b
\end{array}\right]\,:\,a\in\left[\frac{1}{52},\frac{3}{52}\right], b\in\left[\frac{1}{2},2\right]\right\},
\vspace{3pt}
\end{equation}
with the time unit again being one week. Note that $\rateset$ here is clearly non-empty, bounded, closed, convex, and has separately specified rows. As such, this set will satisfy all the preconditions necessary to be used in the examples that follow.
\exampleend
\end{exmp}

As mentioned above, for a given set $\rateset$, Proposition~\ref{prop:approximation_error_bound} allows us to reduce the problem of computing conditional lower expectations to the problem of evaluating the lower transition operator $\lrate$ of $\rateset$, as defined in Equation~\eqref{eq:correspondinglowertrans}. In theory, the latter could still be a difficult problem, since $\rateset$ can be very complicated---we only require that it is non-empty and has separately specified rows. However, in practice, $\rateset$ will typically be described by means of linear constraints, and computing $\lrate g$ is then a standard constrained linear optimisation problem, which is easily solved by means of linear programming techniques. 
We will henceforth assume that the optimisation problem in Equation~\eqref{eq:correspondinglowertrans} is solvable. However, as we have just explained, the complexity of this problem will depend on the exact form of $\rateset$.

For the set $\rateset$ of Example~\ref{exmp:example_rateset_simple_model}, as defined by Equation~\eqref{eq:num_example_rateset_params}, the constraints that we are dealing with are simple box-constraints. Furthermore, because the state space $\states$ contains only two states in this specific example model, computing $\lrate g$ is particularly easy, regardless of the choice of $g\in\gamblesX$. This is made explicit in the following example, where we also compute the quantity $\norm{\lrate}$ that is required for the numerical computations that follow.

\begin{exmp}\label{exmp:numerical_lrate}
Consider the set $\rateset$ of Example~\ref{exmp:example_rateset_simple_model}, as defined by Equation~\eqref{eq:num_example_rateset_params}, let $\lrate$ be the corresponding lower transition rate operator, and consider any $g\in\gamblesX$. Furthermore, for the sake of notational convenience, let us abbreviate {\tt healthy} as {\tt h} and {\tt sick} as {\tt s}. It then follows from Equation~\eqref{eq:correspondinglowertrans} that
\begin{align*}
[\lrate g](\text{{\tt h}})
\coloneqq
\inf\{[Qg](\text{{\tt h}})\,:\,Q\in\rateset\}
=&
\inf\left\{\left[\begin{array}{rr}-a & a \end{array}\right]\left[\begin{array}{l} g(\text{{\tt h}}) \\ g(\text{{\tt s}}) \end{array}\right]\,:\,a\in\left[\frac{1}{52}, \frac{3}{52} \right]\right\} \\[2pt]
=&\inf\left\{a\big(g(\text{{\tt s}})-g(\text{{\tt h}})\big)\,:\,a\in\left[\frac{1}{52}, \frac{3}{52} \right]\right\},
\vspace{-6pt}
\end{align*}
and therefore, that
\begin{equation}\label{eq:lratehealthexample:h}
[\lrate g](\text{{\tt h}})
=
\begin{cases}
\nicefrac{1}{52}\big(g(\text{{\tt s}})-g(\text{{\tt h}})\big)
&\text{ if $g(\text{{\tt s}})\geq g(\text{{\tt h}})$}\\
\nicefrac{3}{52}\big(g(\text{{\tt s}})-g(\text{{\tt h}})\big)
&\text{ if $g(\text{{\tt s}})\leq g(\text{{\tt h}})$}.
\end{cases}
\end{equation}
In a completely analogous way, we also find that
\begin{equation}\label{eq:lratehealthexample:s}
[\lrate g](\text{{\tt s}})
=
\begin{cases}
\nicefrac{1}{2}\big(g(\text{{\tt h}})-g(\text{{\tt s}})\big)
&\text{ if $g(\text{{\tt s}})\leq g(\text{{\tt h}})$}\\
2\big(g(\text{{\tt h}})-g(\text{{\tt s}})\big)
&\text{ if $g(\text{{\tt s}})\geq g(\text{{\tt h}})$}.
\end{cases}
\end{equation}

Using these observations, we can now compute $\norm{\lrate}$. If $g\in\gamblesX$ is such that $\norm{g}=1$, then clearly, it follows from Equation~\eqref{eq:lratehealthexample:h} that the highest possible value for $\abs{[\lrate g]({\tt h})}$ is $\nicefrac{3}{52}$---for $g({\tt s})=-1$ and $g({\tt h})=1$---and similarly, using Equation~\eqref{eq:lratehealthexample:s}, we find that the highest possible value for $\abs{[\lrate g]({\tt s})}$ is $4$---for $g({\tt s})=1$ and $g({\tt h})=-1$. Since
\vspace{-2pt}
\begin{align*}
\norm{\lrate} &= \sup\left\{ \norm{\lrate g}\,:\,g\in\gamblesX,\, \norm{g}=1 \right\} \\
 &= \sup\left\{ \max\left\{\abs{\left[\lrate g\right](x)}\,:\,x\in\states\right\}\,:\,g\in\gamblesX,\, \norm{g}=1 \right\},
\end{align*}\\[-4pt]
this implies that $\norm{\lrate}=4$.
\exampleend
\end{exmp}

We now have everything in place to demonstrate the use of Proposition~\ref{prop:approximation_error_bound} for numerically computing conditional lower expectations. 

\begin{exmp}\label{exmp:single_time_numerical}
Consider again the set of rate matrices $\rateset$ from Example~\ref{exmp:example_rateset_simple_model}, which, as we have seen there, expresses that the rate $a$ at which a healthy person becomes sick belongs to the interval $[\nicefrac{1}{52},\nicefrac{3}{52}]$, and the rate $b$ at which a sick person becomes healthy belongs to the interval $[\nicefrac{1}{2},2]$.

Depending on the extra assumptions that we impose, we can now consider three different \ictmc's. If we believe that the rates $a$ and $b$ remain constant over time, and interpret the intervals $[\nicefrac{1}{52},\nicefrac{3}{52}]$ and $[\nicefrac{1}{2},2]$ as bounds on these unknown but constant rates, then we can model our problem by means of the \ictmc~$\mathbb{P}_{\rateset,\mathcal{M}}^{\mathrm{WHM}}$. If, on the other hand, we think that the rates $a$ and $b$ might vary within these bounds as time progresses, then we should use the \ictmc's $\mathbb{P}_{\rateset,\mathcal{M}}^{\mathrm{WM}}$ or $\mathbb{P}_{\rateset,\mathcal{M}}^{\mathrm{W}}$ instead, where the former assumes that the time-dependent rates $a$ and $b$ cannot be influenced by the value of the past states of the process, whereas the latter does not make such an assumption. The choice of $\mathcal{M}$ does not matter, because it will not effect any of the computations in this example; we therefore drop it from our notation.

Suppose now that we are interested in the lower probability that a person is sick one week from now, given that he or she is currently also sick. That is, we want to compute the conditional lower probability $\underline{P}(X_1 = \text{{\tt sick}}\,\vert\,X_0=\text{{\tt sick}})$, which---using the abbreviation of Example~\ref{exmp:numerical_lrate}---we will denote by $\underline{P}(X_1 = \text{{\tt s}}\,\vert\,X_0=\text{{\tt s}})$. Then as we know from Equation~\eqref{eq:lowerprobaslowerexp}, this is equivalent to computing $\underline{\mathbb{E}}[\ind{\text{{\tt s}}}(X_1)\,\vert\,X_0=\text{{\tt s}}]$.

The result of this computation will, generally speaking, depend on our choice of \ictmc. In this example, we consider $\smash{\mathbb{P}_{\rateset}^{\mathrm{WM}}}$ or $\smash{\mathbb{P}_{\rateset}^{\mathrm{W}}}$. The choice between these two does not make any difference, because it follows from Corollary~\ref{cor:lower_operator_is_infimum} that\footnote{Actually, in this particular---binary---case, it can be shown that $\mathbb{P}_{\rateset}^{\mathrm{WHM}}$ would lead to the exact same result. However, in general, as we have seen in Example~\ref{ex:homogeneousexample}, this will not be the case.}
\begin{equation*}
\underline{\mathbb{E}}_{\rateset}^\mathrm{W}[\ind{\text{{\tt s}}}(X_1)\,\vert\,X_0=\text{{\tt s}}]
=
\underline{\mathbb{E}}_{\rateset}^\mathrm{WM}[\ind{\text{{\tt s}}}(X_1)\,\vert\,X_0=\text{{\tt s}}]=[L_0^1\ind{\text{{\tt s}}}](\text{\tt s}).
\end{equation*}
Hence, our problem has now been reduced to the task of computing $[L_0^1\ind{\text{{\tt s}}}](\text{\tt s})$.

Proposition~\ref{prop:approximation_error_bound} tells us that we can approximate this quantity by using a fine enough partition of the time interval $[0,1]$. We will be using a maximum error for this computation of $\epsilon\coloneqq 10^{-3}$. Furthermore, in this case, the length of the time interval is equal to $(s-t)=1$, the difference between the maximum and minimum of $\ind{\text{s}}$ is $\norm{\ind{\text{s}}}_\mathrm{v}=1$, and due to our choice of $\rateset$, we have $\norm{\lrate} = 4$, as in Example~\ref{exmp:numerical_lrate}. Hence, if we subdivide the time interval into $n$ steps, with
\begin{equation*}
n=\max\left\{
(s-t)\norm{\lrate},
\frac{1}{2\epsilon}(s-t)^2\norm{\lrate}^2\norm{f}_\mathrm{v}
\right\}
=\max\left\{4,8000\right\}=8000,
\end{equation*}
resulting in a step size of $\Delta=\nicefrac{(s-t)}{n}=1.25\times 10^{-4}$, then Proposition~\ref{prop:approximation_error_bound} guarantees that we can approximate $[L_0^1\ind{\text{{\tt s}}}]$---with a maximum error of $\epsilon$---by
\begin{equation*}
\prod_{i=1}^n(I + \Delta\lrate)\ind{\text{{\tt s}}} = \prod_{i=1}^{8000}(I + \Delta\lrate)\ind{\text{{\tt s}}} = \left(\prod_{i=1}^{7999}(I + \Delta\lrate)\right)\left((I+\Delta\lrate)\ind{\text{{\tt s}}}\right)\,.
\end{equation*}
We start by computing the right-most factor on the right-hand side of this equation. Using Equations~\eqref{eq:lratehealthexample:h} and~\eqref{eq:lratehealthexample:s} in Example~\ref{exmp:numerical_lrate}, we find that
\begin{equation*}
[\lrate\ind{\text{{\tt s}}}](x)
=
\begin{cases}
\nicefrac{1}{52}
&\text{ if $x=\text{\tt h}$}\\
-2
&\text{ if $x=\text{\tt s}$}
\end{cases}
~~~\text{for all $x\in\states$},
\end{equation*}
and therefore, it follows that
\begin{align*}
g_1(x)\coloneqq[(I + \Delta\lrate)\ind{\text{{\tt s}}}](x)
&= [\ind{\text{{\tt s}}}](x) + \Delta[\lrate\ind{\text{{\tt s}}}](x)\\
&=
\begin{cases}
2.4038\times 10^{-5}
&\text{ if $x=\text{\tt h}$}\\
0.9975
&\text{ if $x=\text{\tt s}$}
\end{cases}
~~~\text{for all $x\in\states$.}\\
&
\end{align*}
Having completed this first step, we now have that
\begin{equation*}
\prod_{i=1}^n(I + \Delta\lrate)\ind{\text{{\tt s}}} = \prod_{i=1}^{7999}(I + \Delta\lrate)g_1 = \left(\prod_{i=1}^{7998}(I + \Delta\lrate)\right)\left((I+\Delta\lrate)g_1\right)\,.
\end{equation*}
At this point, we simply keep on repeating this process, computing $g_2\coloneqq(I+\Delta\lrate)g_1$, then $g_3\coloneqq(I+\Delta\lrate)g_2$, and so on. In this way, after $j$ steps, we find that
\begin{equation*}
\prod_{i=1}^n(I + \Delta\lrate)\ind{\text{{\tt s}}} = \prod_{i=1}^{8000-j}(I + \Delta\lrate)g_j\
\end{equation*}
and after completing all $n$ steps, we eventually find that
\begin{equation*}
\prod_{i=1}^n(I + \Delta\lrate)\ind{\text{{\tt s}}} = g_{n}
=
\begin{cases}
0.0083
&\text{ if $x=\text{\tt h}$}\\
0.1410
&\text{ if $x=\text{\tt s}$}
\end{cases}
~~~\text{for all $x\in\states$,}
\end{equation*}
from which we conclude that
\begin{align*}
\underline{P}_{\rateset}^{\mathrm{W}}(X_1 = \text{{\tt s}}\,\vert\,X_0=\text{{\tt s}}) = \underline{P}_{\rateset}^{\mathrm{WM}}(X_1 = \text{{\tt s}}\,\vert\,X_0=\text{{\tt s}})
&=
\left[L_0^1\ind{\text{{\tt s}}}\right](\text{{\tt s}})\\
&= g_{n}(\text{{\tt s}}) \pm \epsilon = 0.141 \pm 0.001,
\end{align*}
where the third equality\footnote{Formally, we should write this equality as $[L_0^1\ind{\text{{\tt s}}}](\text{{\tt s}})\in [g_n(\text{{\tt s}})-\epsilon, g_n(\text{{\tt s}})+\epsilon]$ or, completely equivalently, as $g_n(\text{{\tt s}}) \in \bigl[[L_0^1\ind{\text{{\tt s}}}](\text{{\tt s}}) - \epsilon, [L_0^1\ind{\text{{\tt s}}}](\text{{\tt s}}) + \epsilon\bigr]$. However, we find the shorthand $[L_0^1\ind{\text{{\tt s}}}](\text{{\tt s}}) = g_n(\text{{\tt s}})\pm \epsilon$ equally clear. Furthermore, we will later consider similar statements for functions, which we can write much more succinctly with this minor abuse of notation---with the understanding that the above inclusion should then be taken point-wise.} follows from the error bound that is guaranteed by Proposition~\ref{prop:approximation_error_bound}.
\exampleend
\end{exmp}

This iterative procedure for computing $L_t^sf$, as illustrated in the above example, is outlined in Algorithm~\ref{alg:compute_singlevar}. The algorithm first finds the number of steps required to reach the given precision $\epsilon$ (Line 2), and computes from this the corresponding step size $\Delta$ (Line 3). Starting with the function $g_0\coloneqq f$ (Line 4), the algorithm iteratively computes the function $g_i\coloneqq (I+\Delta\lrate)g_{(i-1)}$ (Line 6). After repeating this for $n$ steps (Line 5), the returned function $g_n$ (Line 8) corresponds to $L_t^sf\pm\epsilon$, due to Proposition~\ref{prop:approximation_error_bound}. The algorithm takes for granted that $\smash{\norm{\lrate}}$ is known and/or can be derived from $\lrate$; if this is not the case, then the algorithm should be adapted accordingly, by replacing the norm $\smash{\norm{\lrate}}$ with an upper bound, such as the one that is provided in \ref{LR:normlratefinite}.

\begin{algorithm}[htb]
  \caption{Numerically compute $L_t^sf$ for any $f\in\gamblesX$.}
    \label{alg:compute_singlevar}
  \begin{algorithmic}[1]
  \vspace{4pt}
    \Require{A lower transition rate operator $\lrate$, two time points $t,s\in\realsnonneg$ such that $t\leq s$, a function $f\in\gamblesX$ and a maximum numerical error $\epsilon\in\realspos$.}\vspace{4pt}    
\Ensure{A function $L_t^sf\pm \epsilon$ in $\gamblesX$.}\vspace{-5pt}
    \Statex
    \Function{ComputeLf}{$\lrate, t,s, f,\epsilon$}\vspace{4pt}
      \State $n\gets \left\lceil\max\left\{\nicefrac{\left((s-t)^2\norm{\lrate}^2\norm{f}_\mathrm{v}\right)}{2\epsilon},\, (s-t)\norm{\lrate}\right\}\right\rceil$\vspace{3pt}
		\State $\Delta\gets \nicefrac{(s-t)}{n}$
		\State $g_0 \gets f$
		\For{$i\in\{1,\ldots,n\}$}
		\State $g_i\gets g_{(i-1)} + \Delta\lrate g_{(i-1)}$
		\EndFor
      \State \Return{$g_n$}
    \EndFunction
    \vspace{4pt}
  \end{algorithmic}
\end{algorithm}

\section{A General Framework for Computing Lower Expectations}\label{sec:funcs_multi_time_points}

Having shown in the previous Section~\ref{sec:single_var_lower_exp} that the operator $L_t^s$ allows us to compute conditional lower expectations for functions $f\in\gamblesX$ that are defined on the state at a single future time point, we will now turn our attention to functions defined on multiple time points. 

We will start by considering conditional lower expectations, where the conditioning is done with respect to the states $x_u$ at a finite sequence of time points $u$. In this context, we distinguish between two different classes of functions. 
We begin in Section~\ref{sec:function_single_future_multiple_past} by considering functions $f\in\gambles(\states_{u\cup s})$ and lower expectations of the form $\underline{\mathbb{E}}[f(X_u,X_s)\,\vert\,X_u=x_u]$, with $s$ a single future time point. Thus, although the function $f$ depends on multiple time points, all but one of these time points coincide with the time points on which we are conditioning. This is then generalised in Section~\ref{sec:decomposition} to functions $f\in\gambles(\states_{u\cup v})$ and lower expectations $\underline{\mathbb{E}}[f(X_u,X_v)\,\vert\,X_u=x_u]$, where $v$ is now a finite sequence of future time points. We end in Section~\ref{sec:marginal_lower_exp} by showing how to compute unconditional lower expectations of the form $\underline{\mathbb{E}}[f(X_u)]$, where $u$ is again a finite sequence of time points.

\subsection{Multi-Variable Functions on a Single Future Time Point}\label{sec:function_single_future_multiple_past}

Consider a non-empty sequence of time points $u=t_0,\ldots,t_n\in\mathcal{U}_\emptyset$ and a single future time point $s>u$. Let $f\in\gambles(\states_{u\cup s})$ be a function that depends on the states at these time points. Then for any stochastic process $P$ and any history $x_u\in\states_u$, the corresponding conditional expectation of such a function is well known to satisfy the following simple property:
\vspace{-2pt}
\begin{equation*}
\mathbb{E}_P[f(X_u,X_s)\,\vert\,X_u=x_u] = \mathbb{E}_P[f(x_u,X_s)\,\vert\,X_u=x_u].
\end{equation*}
Therefore, and because a lower expectation is an infimum over expectations, we find that also
\vspace{-2pt}
\begin{equation}\label{eq:fixxu}
\underline{\mathbb{E}}[f(X_u,X_s)\,\vert\,X_u=x_u] = \underline{\mathbb{E}}[f(x_u,X_s)\,\vert\,X_u=x_u].
\end{equation}\\[-18pt]

While writing down these equations, we have implicitly introduced a notational convention that should be intuitively clear: we interpret the function $f(x_u,X_s)$ as the restriction of $f(X_u,X_s)$ to the state $\states_s$, for the fixed state assignment $x_u$. A more formal---but completely equivalent---way of doing this, is to identify $f(x_u,X_s)$ with a function $f_{x_u}\in\gambles(\states_s)=\gamblesX$, defined by $f_{x_u}(x_s)\coloneqq f(x_u,x_s)$ for all $x_s\in\states$.
With this notation, the equation above then turns into
\begin{equation*}
\underline{\mathbb{E}}[f(X_u,X_s)\,\vert\,X_u=x_u] = \underline{\mathbb{E}}[f_{x_u}(X_s)\,\vert\,X_u=x_u],
\end{equation*}\\[-14pt]
and in this way, it becomes clear that the problem of computing $\underline{\mathbb{E}}[f(X_u,X_s)\,\vert\,X_u=x_u]$ is completely equivalent to the problem that we have considered in the previous section, which is to compute lower expectations of functions that depend on the state at a single future time point.
In particular, under the conditions of Corollary~\ref{cor:lower_operator_is_infimum}, we find that
\begin{equation*}
\underline{\mathbb{E}}_{\rateset,\,\mathcal{M}}^{\mathrm{W}}[f(X_u,X_s)\,\vert\,X_u=x_u] = \underline{\mathbb{E}}_{\rateset,\,\mathcal{M}}^{\mathrm{WM}}[f(X_u,X_s)\,\vert\,X_u=x_u]=\left[L_{t_n}^sf_{x_u}\right](x_{t_n}).
\end{equation*}
However, note that the quantity $[L_{t_n}^sf_{x_u}](x_{t_n})$ still depends on the \emph{entire} state assignment $x_u\in\states_u$, and not just on the state $x_{t_n}$.

In order to unify our notation, we therefore stipulate the following convention. For any $s\in\realspos$ and any $u\in\mathcal{U}_{<s}$ with $u\neq\emptyset$ such that $u=t_0,\ldots,t_n$, we allow $L_{t_n}^s$ to be applied to any $f\in\gambles(\states_{u\cup s})$, by applying it to the restriction of $f$ to the \emph{latest} time point at which it is defined---the time point $s$, in this case. Because this restriction depends on the state assignment $x_u$ at the other time points, the result is a function $[L_{t_n}^sf]\in\gambles(\states_u)$. In short, we stipulate for any $f\in\gambles(\states_{u\cup s})$ that
\begin{equation}\label{eq:applyLtolargerfunctions}
[L_{t_n}^sf](x_u) \coloneqq [L_{t_n}^sf_{x_u}](x_{t_n}) \coloneqq [L_{t_n}^sf(x_u,X_s)](x_{t_n})
~~\text{for all $x_u\in\states_u$}.
\end{equation}

Using this notational convention, the following corollary formalises the fact that our previous results also apply to functions in $\gambles(\states_{u\cup s})$.

\begin{restatable}{corollary}{corinfworksforsinglefuturevar}
\label{cor:inf_works_for_single_future_var}
Let $\mathcal{M}$ be an arbitrary non-empty set of probability mass functions on $\states$, let $\rateset$ be an arbitrary non-empty bounded set of rate matrices that has separately specified rows, with corresponding lower transition rate operator $\lrate$, and let $\underline{\mathcal{T}}_{\lrate}$ be the corresponding lower transition operator system. Then, for all $s\in\realspos$, all $u\in\mathcal{U}_{<s}$ such that $u\neq\emptyset$ and all $f\in\gambles(\states_{u\cup s})$:
\begin{equation*}
\underline{\mathbb{E}}^\mathrm{W}_{\,\rateset,\,\mathcal{M}}[f(X_u,X_s)\,\vert\,X_u]=\underline{\mathbb{E}}^{\mathrm{WM}}_{\,\rateset,\,\mathcal{M}}[f(X_u,X_s)\,\vert\,X_u]=\left[L_{t_n}^s f\right](X_u).\vspace{3pt}
\end{equation*}
\end{restatable}

Now, since $L_{t_n}^sf$ must be read as a separate application of the operator $L_{t_n}^s$ to every $f_{x_u}$, for $x_u\in\states_u$, it follows that we can compute $L_{t_n}^sf$ by applying Algorithm~\ref{alg:compute_singlevar} multiple times.
The method for this computation is outlined in Algorithm~\ref{alg:compute_singlevar_multiple_past}, which can be read as follows. 

The algorithm starts by allocating space for a new function $g\in\gambles(\states_u)$ (Line 2), which will be the result of the algorithm. Then, for each state assignment $x_u\in\states_u$ (Line 3), it computes the value of $g(x_u)$ (Lines 4-6), as follows. First, we take the restriction of $f\in\gambles(\states_{u\cup s})$ to the state $\states_s$, for a specific state assignment $x_u$ (Line 4). Since the resulting function $f_{x_u}\coloneqq f(x_u,X_s)$ belongs to $\gamblesX$, we can then approximate $L_{t_n}^sf_{x_u}$ using Algorithm~\ref{alg:compute_singlevar} (Line 5), where $t_n$ is the last time point in $u$. The resulting function $\smash{\hat{f}_{x_u}=L_{t_n}^sf_{x_u}}\pm\epsilon$ is then a function in $\gambles(\states_{t_n})$, which, as we know from Equation~\eqref{eq:applyLtolargerfunctions}, can be evaluated in $x_{t_n}$---the restriction of $x_u$ to the time point $t_n$---to obtain the value $g(x_u)=[L_{t_n}^sf](x_u)\pm\epsilon$ of $g$ in $x_u$ (Line 6). Therefore, up to a maximum error of $\epsilon$, the returned function $g$ (Line 8) will be equal to $L_{t_n}^sf$.

\begin{algorithm}[htb]
  \caption{Numerically compute $L_{t_n}^sf$ for any $f\in\gambles(\states_{u\cup s})$.}
    \label{alg:compute_singlevar_multiple_past}
  \begin{algorithmic}[1]
  \vspace{4pt}
    \Require{A lower transition rate operator $\lrate$, a time point $s\in\realsnonneg$, a sequence of time points $u\in\mathcal{U}_{<s}$ with $u=t_0,\ldots,t_n$, a function $f\in\gambles(\states_{u\cup s})$, and a maximum numerical error $\epsilon\in\realspos$.}
    \vspace{4pt}  
\Ensure{A function $L_{t_n}^sf\pm \epsilon$ in $\gambles(\states_u)$.}
\vspace{-5pt}
    \Statex
    \Function{ComputeLuf}{$\lrate, u,s, f,\epsilon$}
      \State $g\gets g\in\gambles(\states_u)$ \Comment{Allocate space for $g\in\gambles(\states_u)$}
		\For{$x_u\in\states_u$} \Comment{Compute $g(x_u)$ for all $x_u\in\states_u$} 
		\State $f_{x_u}\gets f(x_u,X_s)$ 
		\State $\hat{f}_{x_u} \gets$ {\textsc{ComputeLf}($\lrate, t_n,s,f_{x_u},\epsilon$)} \Comment{Run Agorithm~\ref{alg:compute_singlevar}}
		\State $g(x_u)\gets \hat{f}_{x_u}(x_{t_n})$ 
		\EndFor
      \State \Return{$g$}
    \EndFunction
    \vspace{4pt}
  \end{algorithmic}
\end{algorithm}

\subsection{Multi-Variable Functions on Multiple Future Time Points}\label{sec:decomposition}

We next consider functions $f\in\gambles(\states_{u\cup v})$, where $u=t_0,\ldots,t_n$ is a sequence of time points in a process's history, on which we intend to condition, and $v=s_0,\ldots,s_m$ is a sequence of time points in a process's future; hence, we have that $s_0>t_n$. For functions of this kind, it---surprisingly---turns out that computing the conditional lower expectation with respect to $\wprocesses_{\rateset,\,\mathcal{M}}$ is easier than for $\wmprocesses_{\rateset,\,\mathcal{M}}$. We start in this section by showing how to do this for the former set, that is, how to compute the conditional lower expectation
\begin{equation}\label{eq:lowerexpmultipletimepoints}
\underline{\mathbb{E}}_{\rateset,\,\mathcal{M}}^{\mathrm{W}}[f(X_u,X_v)\,\vert\,X_u],
\end{equation}
and then go on to provide a counter example to illustrate that, unfortunately, this method does not work for the latter set, that is, for lower expectations with respect to $\wmprocesses_{\rateset,\,\mathcal{M}}$.

As we have seen in the previous sections, we can use the operator $L_t^s$ to compute lower expectations of functions that depend on the state at a single future time point, provided that $\rateset$ is non-empty, bounded, and has separately specified rows. We here show that if $\rateset$ is additionally convex, then we can do the same for functions that depend on multiple future time points. 
The reason why this is the case is Theorem~\ref{theorem:decomposition_multivar}. In particular, because of that theorem, if $\rateset$ is convex, we can decompose the lower expectation in Equation~\eqref{eq:lowerexpmultipletimepoints} as follows:
\begin{equation}\label{eq:nested_lower_exp_single_step}
\underline{\mathbb{E}}_{\rateset,\,\mathcal{M}}^{\mathrm{W}}[f(X_u,X_v)\,\vert\,X_u] = \underline{\mathbb{E}}_{\rateset,\,\mathcal{M}}^{\mathrm{W}}\bigl[ \underline{\mathbb{E}}_{\rateset,\,\mathcal{M}}^{\mathrm{W}}[f(X_u,X_v)\,\vert\,X_u,X_{v\setminus\{s_m\}}] \,\big\vert\,X_u\bigr]\,,
\end{equation}
where $s_m$ is the last time point in $v=s_1,\dots,s_m$. The essential feature of this decomposition is that the inner lower expectation on the right-hand side of the equality is conditioned on all the time points in $u$ and $v\setminus\{s_m\}$, or equivalently, that this lower expectation is taken with respect to a single future time point $s_m$. Therefore, it follows from the discussion in Section~\ref{sec:function_single_future_multiple_past}, and from Corollary~\ref{cor:inf_works_for_single_future_var} in particular, that
\begin{equation*}
\underline{\mathbb{E}}_{\rateset,\,\mathcal{M}}^{\mathrm{W}}[f(X_u,X_v)\,\vert\,X_u,X_{v\setminus\{s_m\}}]=\big[L_{s_{m-1}}^{s_m}f\big]\left(X_u,X_{v\setminus\{s_m\}}\right),
\end{equation*}
which, by substituting the inner lower expectation in Equation~\eqref{eq:nested_lower_exp_single_step}, implies that
\begin{equation*}
\underline{\mathbb{E}}_{\rateset,\,\mathcal{M}}^{\mathrm{W}}[f(X_u,X_v)\,\vert\,X_u] = \underline{\mathbb{E}}_{\rateset,\,\mathcal{M}}^{\mathrm{W}}\left[ \big[L_{s_{m-1}}^{s_m}f\big]\left(X_u,X_{v\setminus\{s_m\}}\right) \,\big\vert\,X_u\right].
\end{equation*}
In this way, our original problem, which was to compute $\underline{\mathbb{E}}_{\rateset,\,\mathcal{M}}^{\mathrm{W}}[f(X_u,X_v)\,\vert\,X_u]$, has been reduced to a completely analogous---but smaller-sized---problem; the only difference is that $f$ has been replaced by a new function, which no longer depends on $X_u$ and $X_{v}$, but only on $X_u$ and $X_{v\setminus s_m}$.

We can now simply apply this step over and over again, by repeatedly removing the last remaining time point $s_{m-1}, s_{m-2},\ldots,s_{0}$ and, at each step, replacing the inner lower expectation with the operator $L_{s_{m-2}}^{s_{m-1}},L_{s_{m-3}}^{s_{m-2}},\ldots,L_{t_n}^{s_0}$. Because $v$ is finite, this process eventually stops, and we then obtain the following result.

\begin{restatable}{corollary}{corcompositionlowertrans}
\label{cor:composition_lower_trans}
Let $\mathcal{M}$ be a non-empty set of probability mass functions on $\states$, let $\rateset$ be a non-empty, bounded and convex set of rate matrices that has separately specified rows, with lower transition rate operator $\lrate$, and let $\smash{\underline{\mathcal{T}}_{\lrate}}$ be the corresponding lower transition operator system. Then, for any $u=t_0,\ldots,t_n$ and $v={s_0,\ldots,s_m}$ in $\mathcal{U}_\emptyset$ such that $u<v$, and any $f\in\gambles(\states_{u\cup v})$:
\vspace{2pt}
\begin{equation}\label{eq:composition_lower_trans}
\underline{\mathbb{E}}^{\mathrm{W}}_{\,\rateset,\,\mathcal{M}}[f(X_u,X_v)\,\vert\,X_u]=L_{t_n}^{s_0}L_{s_0}^{s_1}\cdots L_{s_{m-1}}^{s_m}f.\\[5pt]
\end{equation}
\end{restatable}

The following example numerically illustrates the use of Corollary~\ref{cor:composition_lower_trans} for computing the conditional lower expectation of a function that depends on the state at multiple time points.

\begin{exmp}\label{exmp:num_multivar_func_nonmarkov}
Consider again the set of rate matrices $\rateset$ from Example~\ref{exmp:example_rateset_simple_model}, and suppose that we are interested in the lower probability $\underline{P}(X_1 = X_2\,\vert\, X_0)$. That is, given that we start in some initial state $X_0$, we want to know the lower probability that $X_1$ and $X_2$ will have identical values. 
As we know from Equation~\eqref{eq:lowerprobaslowerexp}, this lower probability is given by
\begin{equation*}
\underline{P}(X_1 = X_2\,\vert\, X_0) = \underline{\mathbb{E}}[
\ind{X_1=X_2}(X_1,X_2)\,\vert\,X_0],\vspace{4pt}
\end{equation*}
with $\ind{X_1=X_2}$ the indicator of the event $X_1=X_2$, as defined by
\begin{equation*}
\ind{X_1=X_2}(x_1,x_2) \coloneqq
\begin{cases}
1 & \text{if $x_1 = x_2$} \\
0 & \text{if $x_1 \neq x_2$}
\end{cases}
~~~\text{ for all $x_1,x_2\in\states$.}
\end{equation*}
As explained in Example~\ref{exmp:single_time_numerical}, the value of such a lower expectation can depend on the particular type of {\ictmc} that we choose to use, and in fact, in this particular case, as we will see in Example~\ref{exmp:num_counterexample_markov}, this choice does indeed make a difference. 

For now, in this example, we use the \ictmc~$\smash{\mathbb{P}_{\rateset}^{\mathrm{W}}}$. The aim is therefore to compute the conditional lower expectation $\underline{\mathbb{E}}_{\rateset}^{\mathrm{W}}[\ind{X_1=X_2}(X_1,X_2)\,\vert\,X_0]$. Technically speaking, we cannot do this by means of Corollary~\ref{cor:composition_lower_trans}, because $\ind{X_1=X_2}$ does not formally depend on $X_0$. However, in practice, this is of course not a problem, because we can trivially regard $\ind{X_1=X_2}$ as a function of $X_0$, $X_1$ and $X_2$ whose value remains constant on $X_0$. In order to formalise this, we introduce the function $f(X_0,X_1,X_2)$, defined by
\begin{equation}\label{eq:fforexample}
f(x_0,x_1,x_2)\coloneqq \ind{X_1=X_2}(x_1,x_2)
=
\begin{cases}
1 & \text{if $x_1 = x_2$} \\
0 & \text{if $x_1 \neq x_2$}
\end{cases}
~~~\text{ for all $x_0,x_1,x_2\in\states$.}
\end{equation}
Since $f(X_0,X_1,X_2)$ is clearly identical to $\ind{X_1=X_2}(X_1,X_2)$, our problem now consists of computing $\underline{\mathbb{E}}_{\rateset}^{\mathrm{W}}[f(X_0,X_1,X_2)\,\vert\,X_0]$, which, because of Corollary~\ref{cor:composition_lower_trans}, is given by
\begin{equation}\label{eq:num_example_composition}
\underline{\mathbb{E}}_{\rateset}^\mathrm{W}[f(X_0,X_1,X_2)\,\vert\,X_0] = \left[L_0^1L_1^2f\right](X_0).
\end{equation}
In order to compute the left hand side of this equality, we will resolve the composition of operators $L_0^1L_1^2$ by starting from the latest time point, and working back to the earliest time point. Hence, we start by looking at $L_1^2f$. 

The function $f$ depends on multiple---in this case three---time points, and we will therefore use Algorithm~\ref{alg:compute_singlevar_multiple_past} to compute $L_1^2f$, which consists in applying Algorithm~\ref{alg:compute_singlevar} four times. In particular, for every $x_0\in\states=\{{\tt h},{\tt s}\}$ and $x_1\in\states=\{{\tt h},{\tt s}\}$, we use Algorithm~\ref{alg:compute_singlevar} to compute
\begin{align*}
g(x_0,x_1)
&\coloneqq
[L_1^2f](x_0,x_1)\\[-7pt]
&\coloneqq
[L_1^2f(x_0,x_1,X_2)](x_1)
=
\begin{cases}
\left[L_1^2\ind{{\tt h}}\right](\text{{\tt h}})=0.956\pm\epsilon
&\text{ if $x_1={\tt h}$}\\
\left[L_1^2\ind{{\tt s}}\right](\text{{\tt s}})=0.141\pm\epsilon
&\text{ if $x_1={\tt s}$}
\end{cases}
\end{align*}
up to some desired maximum error $\epsilon$, which we have here chosen to be $\epsilon=10^{-3}$. Note that the value of $\left[L_1^2\ind{{\tt s}}\right](\text{{\tt s}})$ is identical to the value that we had obtained for $\left[L_0^1\ind{{\tt s}}\right](\text{{\tt s}})$ in Example~\ref{exmp:single_time_numerical}; this is to be expected, as it follows from the time-homogeneity property in Proposition~\ref{prop:lower_transition_is_homogeneous}.

Next, we apply the operator $L_0^1$ to the function $L_1^2f$, or equivalently, to the function~$g$. The approach is completely similar. Again, as in Algorithm~\ref{alg:compute_singlevar_multiple_past}, we simply apply Algorithm~\ref{alg:compute_singlevar} multiple times. In particular, for every $x_0\in\states=\{{\tt h},{\tt s}\}$, we use Algorithm~\ref{alg:compute_singlevar} to compute
\begin{equation}\label{eq:secondpass}
[L_0^1L_1^2f](x_0)
=[L_0^1g](x_0)
\coloneqq
[L_0^1 g(x_0,X_1)](x_0)
=
\begin{cases}
0.920\pm2\epsilon
&\text{ if $x_0={\tt h}$}\\
0.453\pm2\epsilon
&\text{ if $x_0={\tt s}$},
\end{cases}
\end{equation}
where the maximum error is now $2\epsilon=2\cdot10^{-3}$, which is the sum of our previous error bound on $L_1^2f$, which was $\epsilon$, and the maximum extra error that may have arisen while computing Equation~\eqref{eq:secondpass}; the additivity of the two error bounds is a result of~\ref{LT:monotonicity} and~\ref{LT:constantadditivity} and the fact that $L_t^s$ is a lower transition operator.

Finally, by combining Equations~\eqref{eq:num_example_composition} and~\eqref{eq:secondpass}, we conclude that
\begin{align*}
\underline{P}_{\rateset}^{\mathrm{W}}(X_1 = X_2\,\vert\,X_0=\text{{\tt h}})=0.920\pm0.002~~\text{and}~~\underline{P}_{\rateset}^{\mathrm{W}}(X_1 = X_2\,\vert\,X_0=\text{{\tt s}})=0.453\pm0.002.
\end{align*}
\exampleend
\end{exmp}

A general method for computing the right hand side of Equation~\eqref{eq:composition_lower_trans} is outlined in Algorithm~\ref{alg:compute_multivar}. This algorithm starts by constructing a buffer $w\coloneqq v=s_0,\ldots,s_m$ of the time points for which the operators $L_{t_n}^{s_0},L_{s_1}^{s_2},\ldots,L_{s_{m-1}}^{s_m}$ have not yet been resolved, which initially corresponds to just $v$ (Line 2). We rename the function $f$ to $g_m$ for indexing purposes (Line 3). The algorithm then iteratively resolves the composition of operators in a backward fashion, working from the last time point in $v$ back to the first time point in $v$ (Line 4). One step here amounts to removing the current latest remaining time point $s_i$ from the buffer $w$ (Line 5), and then computing $L_{s_{i-1}}^{s_i}g_i$ using Algorithm~\ref{alg:compute_singlevar_multiple_past}. The outcome of this computation is a function $g_{(i-1)}\in\gambles(\states_{u\cup w})$.
After all time points $s_m,s_{m-1},\ldots,s_0$ have been resolved, it follows from Proposition~\ref{prop:approximation_error_bound} that the returned function $g_{(-1)}$ (Line 8) is equal to $L_{t_n}^{s_0}L_{s_0}^{s_1}\cdots L_{s_{m-1}}^{s_m}f\pm\epsilon$.

Note that the error bound of $\pm\epsilon$ holds because we compute each step using a maximum error of $\nicefrac{\epsilon}{m+1}$ (Line 6). That is, the finite-precision approximation error scales linearly in the number of time points that have to be computed. The reason why the accumulated error bounds can simply be added, is because $L_{s_{i-1}}^{s_i}$ is a lower transition operator and therefore satisfies~\ref{LT:monotonicity} and~\ref{LT:constantadditivity}. In particular, using these two properties, it is easily verified that for any $\mu>0$, $L_{s_{i-1}}^{s_i}(g_i\pm\mu)=L_{s_{i-1}}^{s_i}g_i\pm\mu$, where, in this case, we choose $\mu=\nicefrac{\epsilon}{m+1}$.

\begin{algorithm}[htb]
  \caption{Numerically compute $L_{t_n}^{s_0}L_{s_0}^{s_1}\cdots L_{s_{m-1}}^{s_m}f$ for any $f\in\gambles(\states_{u\cup v})$.}
    \label{alg:compute_multivar}
  \begin{algorithmic}[1]
  \vspace{4pt}
    \Require{A lower transition rate operator $\lrate$, two sequences of time points $u=t_0,\ldots,t_n$ and $v=s_0,\ldots,s_n$ in $\mathcal{U}_\emptyset$ such that $u<v$, a function $f\in\gambles(\states_{u\cup v})$, and a maximum numerical error $\epsilon\in\realspos$.}\vspace{4pt}
\Ensure{A function $L_{t_n}^{s_0}L_{s_0}^{s_1}\cdots L_{s_{m-1}}^{s_m}f\pm\epsilon$ in $\gambles(\states_u)$.}
\vspace{-5pt}
	 \Statex
    \Function{ComputeLuvf}{$\lrate, u,v, f,\epsilon$}
		\State $w \gets v$ 
		\State $g_{m}\gets f$\vspace{2pt} 
		\For{$i\in\{m,m-1,\ldots,0\}$}	\Comment{Iterate backward over $v=s_0,\ldots,s_m$}
			\State $w \gets w\setminus s_i$ \Comment{Remove last point of $w$}
			\State $g_{(i-1)}\gets ${\textsc{ComputeLuf}}($\lrate,u\cup w,s_i,g_i,\nicefrac{\epsilon}{m+1}$)  \Comment{Run Algorithm~\ref{alg:compute_singlevar_multiple_past}}\vspace{3pt}
		\EndFor
		\State \Return{$g_{(-1)}$}
		\vspace{2pt}
    \EndFunction
    \vspace{4pt}
  \end{algorithmic}
\end{algorithm}

As we already claimed in the beginning of this section, we cannot use Corollary~\ref{cor:composition_lower_trans}---nor, therefore, Algorithm~\ref{alg:compute_multivar}---to compute lower expectations of functions $f\in\gambles(\states_{u\cup v})$ with respect to a set of Markov chains $\wmprocesses_{\rateset,\mathcal{M}}$. In the remainder of this section, we provide this claim with some intuition and illustrate it by means of a numerical counterexample.

The core of the problem is that for a function $f\in\gambles(\states_{u\cup s})$, the conditional lower expectation $\underline{\mathbb{E}}_\rateset^\mathrm{WM}[f(X_u,X_s)\vert X_u]$ is not necessarily obtained by a single Markov process. 
That is, for every $\epsilon>0$ and every $x_u\in\states_u$, it follows from Equation~\eqref{eq:genericlowerexpectation} that there is a Markov chain $P_{x_u}\in\wmprocesses_{\rateset,\mathcal{M}}$ such that
\begin{equation*}
\underline{\mathbb{E}}_\rateset^\mathrm{WM}[f(X_u,X_s)\vert X_u=x_u]
\leq
\mathbb{E}_{P_{x_u}}[f(X_u,X_s)\vert X_u=x_u]\leq\underline{\mathbb{E}}_\rateset^\mathrm{WM}[f(X_u,X_s)\vert X_u=x_u]+\epsilon,
\end{equation*}
but there may \emph{not} be a single Markov chain $P\in\wmprocesses_{\rateset,\mathcal{M}}$ that does this simultaneously for \emph{all} histories $x_u\in\states_u$.

For example, if we were to assume that
\begin{equation*}
\underline{\mathbb{E}}_{\rateset,\,\mathcal{M}}^{\mathrm{WM}}[f(X_u,X_t,X_s)\,\vert\,X_u] = \underline{\mathbb{E}}_{\rateset,\,\mathcal{M}}^{\mathrm{WM}}\bigl[ \underline{\mathbb{E}}_{\rateset,\,\mathcal{M}}^{\mathrm{WM}}[f(X_u,X_t,X_s)\,\vert\,X_u,X_t] \,\big\vert\,X_u\bigr],
\end{equation*}
thereby mimicking Equation~\eqref{eq:nested_lower_exp_single_step} for $v=\{t,s\}$, we would essentially be assuming that the individual Markov chains $P_{x_u,\,x_t}\in\wmprocesses_{\rateset,\mathcal{M}}$ for which the conditional lower expectations
\begin{equation*}
\underline{\mathbb{E}}_{\rateset,\,\mathcal{M}}^{\mathrm{WM}}[f(X_u,X_t,X_s)\,\vert\,X_u=x_u,X_t=x_t]\vspace{5pt}
\end{equation*}
are obtained---or rather, are approximated from above up to some arbitrary $\epsilon$---can be combined into a single new Markov chain $P\in\wmprocesses_{\rateset,\mathcal{M}}$ such that
\begin{equation}\label{eq:Markovfailure}
P(X_s\vert X_u=x_u,X_t=x_t)=P_{x_u,x_t}(X_s\vert X_u=x_u,X_t=x_t)=P_{x_u,x_t}(X_s\vert X_t=x_t)
\end{equation}
for all $x_u\in\states_u$ and $x_t\in\states$. However, in most cases, this will not be possible. The issue is that for a stochastic process $P$ that satisfies Equation~\eqref{eq:Markovfailure}, the Markov property will---generally speaking---fail, because the right hand side of the expression still depends on $x_u$.

All of this changes if we replace $\wmprocesses_{\rateset,\mathcal{M}}$ with $\wprocesses_{\rateset,\mathcal{M}}$. In that case, we are no longer forced to obey the Markov property, and it then does become possible to combine different stochastic processes in a history-dependent way, as in Equation~\eqref{eq:Markovfailure}---using Theorem~\ref{theo:aanelkaarplakken}. Essentially, this is the main reason why we were able to decompose the lower expectations $\smash{\underline{\mathbb{E}}_{\rateset,\,\mathcal{M}}^{\mathrm{W}}}$---as we did in Theorem~\ref{theorem:decomposition_multivar} and Equation~\eqref{eq:nested_lower_exp_single_step}---and ultimately, why we could use the operator $L_t^s$ to compute them.\footnote{For readers that are closely familiar with the theory of credal networks~\cite{Cozman:2000ug}, it might be helpful to compare the difference between $\wmprocesses_{\rateset,\mathcal{M}}$ and $\wprocesses_{\rateset,\mathcal{M}}$ with the difference between credal networks under epistemic irrelevance~\cite{de2015credal,deCooman:2010gd} and credal networks under strong independence~\cite{Antonucci:2014ty}. There too, dropping the restrictions that are imposed by the Markov property---as is done for credal networks under epistemic irrelevance---allows for lower expectations to be decomposed and, as such, for efficient recursive computational methods to be developed.}

In the example below, we reconsider the problem of Example~\ref{exmp:num_multivar_func_nonmarkov}, but for a different {\ictmc}: we now use $\wmprocesses_{\rateset,\mathcal{M}}$ instead of $\wprocesses_{\rateset,\mathcal{M}}$. The result ends up being different, and thereby illustrates---by means of a counterexample---that Algorithm~\ref{alg:compute_multivar} is not in general applicable when working with sets of Markov chains such as $\wmprocesses_{\rateset,\mathcal{M}}$.

\begin{exmp} \label{exmp:num_counterexample_markov}
Consider again the set of rate matrices $\rateset$ from Example~\ref{exmp:example_rateset_simple_model} and the problem of computing $\underline{P}(X_1 = X_2\,\vert\, X_0)$. However, this time, we will compute this quantity with respect to the {\ictmc} $\wmprocesses_{\rateset}$ instead of $\wprocesses_{\rateset}$. Using an argument that is analogous to that in Example~\ref{exmp:example_rateset_simple_model}, we find that this problem is equivalent to computing $\underline{\mathbb{E}}_{\rateset}^{\mathrm{WM}}[f(X_0,X_1,X_2)\,\vert\,X_0]$, with $f$ as in Equation~\eqref{eq:fforexample}. 

The issue, however, is that we do not have a general method for computing such lower expectations for sets of Markov chains. Fortunately, in the very specific case of a binary state space $\mathcal{X}$---as we use here in this example---and for functions $f$ that depend on a small number of states $X_t$---only three in this case---it is possible to numerically solve the optimisation problem given by the right-hand side of
\begin{equation*}
\underline{\mathbb{E}}_{\rateset}^{\mathrm{WM}}[f(X_0,X_1,X_2)\,\vert\,X_0] = \inf\left\{ \mathbb{E}_P[f(X_1,X_2)\,\vert\,X_0]\,:\,P\in\wmprocesses_\rateset\right\}
\end{equation*}
by combining brute-force optimisation techniques with some clever tricks. We will not report on this method here, because it only works in very specific cases such as the one we consider here, and does not scale well to larger problems. For our present purposes, it suffices to know that this method yields the following results:
\begin{align*}
\underline{P}_{\rateset}^{\mathrm{WM}}(X_1 = X_2\,\vert\,X_0=\text{{\tt h}}) \approx 0.939~~\text{and}~~\underline{P}_{\rateset}^{\mathrm{WM}}(X_1 = X_2\,\vert\,X_0=\text{{\tt s}}) \approx 0.467.
\end{align*}

Comparing to the results for $\wprocesses_\rateset$ in Example~\ref{exmp:num_multivar_func_nonmarkov}, we see that the obtained lower probabilities are indeed different for these two sets of processes. Furthermore, as guaranteed by Proposition~\ref{prop:lower_exp_markov_bounded_by_nonmarkov}, the lower probabilities with respect to $\wprocesses_\rateset$---which, as we know from Section~\ref{subsec:ictmc_types}, is just a special type of lower expectation---provides a lower bound on the lower probabilities with respect to $\wmprocesses_\rateset$.
\exampleend
\end{exmp}

We end with some remarks about the tractability of the algorithms that we have presented. In particular, observe that Algorithm~\ref{alg:compute_singlevar_multiple_past} and~\ref{alg:compute_multivar} have a runtime complexity that is exponential in the number of time points. For example, for the computation of the lower expectation of a function $f\in\gambles(\states_{u\cup s})$, Algorithm~\ref{alg:compute_singlevar_multiple_past} requires a separate execution of Algorithm~\ref{alg:compute_singlevar} for every $x_{u}\in\states_{u}$. If we write $u=t_0,\ldots,t_n$ and let $\abs{\states}$ denote the number of states in $\states$, then, clearly, this requires $\abs{\states}^{n+1}$ executions of Algorithm~\ref{alg:compute_singlevar}. Similarly, for the computation of the lower expectation of a function $f\in\gambles(\states_{u\cup v})$, with $v=s_0,\ldots,s_m$, Algorithm~\ref{alg:compute_multivar} requires a number of executions of Algorithm~\ref{alg:compute_singlevar} that is of the order $\smash{\abs{\states}^{n+m+2}}$.
Of course, this should not really be surprising---even the simple enumeration of the different values of $f$ takes this many steps, so we cannot in general expect the computation of its lower expectation to be of a lower complexity. What is important though---but beyond the scope of this paper---is that if $f$ has some specific structure that reduces the complexity of its specification, then the complexity of our algorithms will scale down accordingly; we discuss these and other avenues for future research in our conclusions.

\subsection{Unconditional Lower Expectations}\label{sec:marginal_lower_exp}

Having shown in the previous sections how to compute conditional lower expectations, we now shift our attention to unconditional ones. In particular, we consider a function $f\in\gambles(\states_u)$, with $u\in\mathcal{U}_{\emptyset}$, and are interested in computing $\underline{\mathbb{E}}[f(X_u)]$. 

As we will see, the difference between unconditional and conditional lower expectations is mainly related to the role of the set of initial distributions $\mathcal{M}$. In fact, as an astute reader may have noticed, all our previous results about computing conditional lower expectations with the operator $L_t^s$ did not depend on the particular choice of this set $\mathcal{M}$.
In contrast, and rather unsurprisingly, unconditional lower expectations do depend on this choice.

In particular, the choice of $\mathcal{M}$ will influence our computations through the corresponding lower expectation operator $\underline{\mathbb{E}}_{\mathcal{M}}$, defined by
\begin{equation}\label{eq:initiallowerexp}
\underline{\mathbb{E}}_{\mathcal{M}}[f]
\coloneqq
\inf
\left\{\sum_{x\in\states}p(x)f(x)\colon p\in\mathcal{M}\right\}
~~\text{for all $f\in\gamblesX$.}
\end{equation}
In principle, evaluating this operator can be difficult, because the only restriction we impose on the set $\mathcal{M}$ is that it should be non-empty. However, in practice, $\mathcal{M}$ will typically have a rather simple  structure, and computing $\underline{\mathbb{E}}_\mathcal{M}[f]$ will then be straightforward. For example, if $\mathcal{M}$ consists of a finite number of probability mass functions---or is equal to their convex hull---then evaluating $\underline{\mathbb{E}}_\mathcal{M}[f]$ is just a matter of minimising over that finite set of probability mass functions.
Another typical situation is when $\mathcal{M}$ is specified by means of linear constraints. In that case, $\underline{\mathbb{E}}_\mathcal{M}[f]$ can be computed by means of linear programming techniques. We will proceed under the assumption that the optimisation problem in Equation~\eqref{eq:initiallowerexp} is solvable.

Regardless of the type of {\ictmc} that we consider, $\underline{\mathbb{E}}_{\mathcal{M}}$ allows us to compute the lower expectation of functions that depend on the initial state $X_0$.

\begin{restatable}{proposition}{propcomputeinitial}
\label{prop:computeinitial}
Let $\mathcal{M}$ be a non-empty set of probability mass functions on $\states$ and let $\rateset$ be a non-empty bounded set of rate matrices. Then for all $f\in\gamblesX$:
\begin{equation*}
\underline{\mathbb{E}}^\mathrm{W}_{\rateset,\mathcal{M}}[f(X_0)]
=
\underline{\mathbb{E}}^\mathrm{WM}_{\rateset,\mathcal{M}}[f(X_0)]
=
\underline{\mathbb{E}}^\mathrm{WHM}_{\rateset,\mathcal{M}}[f(X_0)]
=
\underline{\mathbb{E}}_{\mathcal{M}}[f].\vspace{4pt}
\end{equation*}
\end{restatable}

For functions that depend on some non-initial state $X_s$, with $s>0$, the operator $\underline{\mathbb{E}}_{\mathcal{M}}$ still allows us to compute their lower expectation, by combining it with Algorithm~\ref{alg:compute_singlevar}. However, this only works for {\ictmc}'s that are of the type $\wprocesses_{\rateset,\mathcal{M}}$ or $\wmprocesses_{\rateset,\mathcal{M}}$.

\begin{restatable}{proposition}{propcomputemarginal}
\label{prop:computemarginal}
Let $\mathcal{M}$ be a non-empty set of probability mass functions on $\states$ and let $\rateset$ be a non-empty bounded set of rate matrices that has separately specified rows, with lower transition rate operator $\lrate$. Then for any $s\in\realsnonneg$ and any $f\in\gamblesX$:
\begin{equation*}
\underline{\mathbb{E}}^\mathrm{W}_{\rateset,\mathcal{M}}[f(X_s)]
=
\underline{\mathbb{E}}^\mathrm{WM}_{\rateset,\mathcal{M}}[f(X_s)]
=
\underline{\mathbb{E}}_{\mathcal{M}}[L_0^s f].\vspace{5pt}
\end{equation*}
\end{restatable}

For {\ictmc}'s that are of the type $\wprocesses_{\rateset,\mathcal{M}}$, this result generalises to functions that depend on multiple time points. Computing the lower expectation of such a function can then be done by combining $\underline{\mathbb{E}}_{\mathcal{M}}$ with Algorithm~\ref{alg:compute_multivar}.

\begin{restatable}{proposition}{propcomputeunconditional}
\label{prop:computeunconditional}
Let $\mathcal{M}$ be a non-empty set of probability mass functions on $\states$ and let $\rateset$ be a non-empty, bounded, convex set of rate matrices that has separately specified rows, with lower transition rate operator $\lrate$. Then for any $u=t_0,\dots,t_n$ in $\mathcal{U}_{\emptyset}$ such that $t_0=0$, and any $f\in\gambles(\states_u)$:
\begin{equation*}
\underline{\mathbb{E}}_{\rateset,\,\mathcal{M}}^\mathrm{W}[f(X_u)]
=
\underline{\mathbb{E}}_{\mathcal{M}}[L_{t_0}^{t_1}L_{t_1}^{t_2}\cdots L_{t_{n-1}}^{t_n}f].\vspace{5pt}
\end{equation*}
\end{restatable}
\noindent
In this result, the requirement $t_0=0$ is a purely formal one, and imposes no actual restrictions. Whenever $0\notin u$, it suffices to first replace $f(X_u)$ by its trivial extension $f^*(X_0,X_u)$ to $\states_{0\cup u}$, defined by $f^*(x_0,x_u)\coloneqq f(x_u)$ for all $x_0\in\states$ and $x_u\in\states_u$. Since $f^*(X_0,X_u)$ is then identical to $f(X_u)$, this addition of $X_0$ to the domain is merely a notational trick, and will clearly not influence the result of the computations.

\section{Relation to Previous Work}\label{sec:prev_work}

To the best of our knowledge, the concept of an imprecise continuous-time Markov chain was first introduced in the literature by the seminal work of {\v{S}}kulj~\cite{Skulj:2015cq}. There, $\lrate$ was defined as the lower envelope of a non-empty, bounded, closed, and convex set of rate matrices $\rateset$ with separately specified rows, and the concept of an imprecise continuous-time Markov chain was then introduced in two different ways.

In the main body of his paper, {\v{S}kulj} characterises this concept in the following way. For any $f\in\gamblesX$, he considers a time-dependent function $f_s$, which---using our notation---is defined as $f_s\coloneqq T_0^s f$, and then requires that $f_0=f$,
\begin{equation}\label{eq:constraintsDamjan}
\lrate f_s
\leq
D^+f_s\coloneqq
\liminf_{\Delta\to0^+}
\frac{f_{s+\Delta}-f_s}{\Delta}
~~~\text{and}~~~
\limsup_{\Delta\to0^+}
\frac{f_{s+\Delta}-f_s}{\Delta}
=\hspace{-1.5pt}\colon D_+f_s
\leq\overline{Q} f_s,
\end{equation}
where $\overline{Q}$ is the upper envelope of $\rateset$, or equivalently,  the conjugate upper transition rate operator of $\lrate$, defined by $\overline{Q}g\coloneqq-\lrate(-g)$ for all $g\in\gamblesX$, and where $D^+f_s$ and $D_-f_s$ are called Dini derivatives. In this way, he generalises the usual differential equation characterisation of a precise homogeneous Markov chain, where the Dini derivatives would be replaced by derivatives, where the operator $\lrate$ would be replaced by a transition rate matrix $Q$, and where the inequalities would become equalities. 

Since this generalisation considers inequalities, the solution $f_s$ is not uniquely determined, and the problem is therefore to find lower and upper bounds for $f_s$. {\v{S}kulj} solves this problem, by showing that $\underline{f}_s\leq f_s\leq\overline{f}_s$, where $\underline{f}_s$ and $\overline{f}_s$ are the unique solutions of the differential equations
\begin{equation}\label{eq:boundsdifferentialDamjan}
\frac{d \underline{f}_{\,s}}{d s}=\lrate\,\underline{f}_{\,s}
~~~\text{and}~~~
\frac{d \overline{f}_{\,s}}{d s}=\overline{Q}\,\overline{f}_{\,s},
\end{equation}
with $\underline{f}_0=\overline{f}_0=f$. In practice, it suffices to restrict attention to either one of these two bounds, because they are related by conjugacy, in the sense that $\smash{\underline{f}_0=-\overline{f}_0}$ implies that $\smash{\underline{f}_s=-\overline{f}_s}$.

Using this characterisation, {\v{S}kulj} focusses on computing $\overline{f}_s$. Since, for all $x\in\states$, $\overline{f}_s(x)$ is a tight upper bound on $f_s(x)\coloneqq [T_0^sf](x)=\mathbb{E}[f(X_s)\vert X_0=x]$, this can be interpreted as computing a conditional upper expectation. In particular, for all $x\in\states$, we can interpret $\overline{f}_s(x)$ as
\begin{equation*}
\overline{f}_s(x)
=
\overline{\mathbb{E}}[f(X_s)\vert X_0=x]=-\underline{\mathbb{E}}[-f(X_s)\vert X_0=x]
=-\underline{[-f]}_s(x),
\end{equation*}
where the final two equalities follow from conjugacy.

The above characterisation leaves some ambiguity about the set of stochastic processes $\mathcal{P}$ with respect to which these lower and upper expectations are taken. 
However, as mentioned above, {\v{S}kulj} also provides another definition, which in his paper preceeds the rigorous characterisation in terms of lower- and upper bounds that we have summarised above. There, he defines an imprecise continuous-time Markov chain as a random process $X\colon\reals\to\states$ whose transition rate matrix is an unspecified function $Q_s\in\rateset$ \cite[Definition~1]{Skulj:2015cq}. However, he does not provide a definition for $Q_s$---as being a derivative or some other representation of the infinitesimal behaviour of the process' transition probabilities---nor for the notion of a random process, and in that sense, this definition is rather informal.

On the one hand, this other definition of an imprecise continuous-time Markov chain clearly suggests that $\mathcal{P}$ is the set of all random processes whose time-dependent transition rate matrix $Q_s$ takes values in $\rateset$. Since $Q_s$ is typically defined as a derivative, this suggests that $T_0^s$ should be differentiable in some way---or, at the very least, that $Q_s$ should be integrable. Furthermore, the fact that $Q_s$ is indexed only by the time $s$ suggests that $Q_s$ cannot be history-dependent---thereby implying that the corresponding random process satisfies the Markov property. On the other hand, {\smash{\v{S}kulj}}'s main definition suggests a different set of processes $\mathcal{P}$, because Equation~\eqref{eq:constraintsDamjan} does not impose a Markov condition, and because the use of the Dini derivatives suggests that $T_0^s$ is not required to be differentiable. In summary then, {\smash{\v{S}kulj}} clearly intends for $\mathcal{P}$ to be a set of stochastic processes of which the infinitesimal differences are compatible with $\rateset$, but leaves open the question of whether or not these processes are Markov chains, and whether or not their transition operators $T_0^s$ should be differentiable.

Our approach, in contrast, removes this ambiguity, by making it formally explicit which set of stochastic processes $\mathcal{P}$ is being considered. Interestingly, but perhaps unsurprisingly, our results end up being very closely related to the work of {\v{S}kulj}.
In particular, by comparing Equations~\eqref{eq:pointwisedifferential} and~\eqref{eq:boundsdifferentialDamjan}, we see that $\underline{f}_s=L_0^sf$, and therefore, it follows from Corollary~\ref{cor:lower_operator_is_infimum} that
\begin{equation*}
\underline{f}_s(x)
=
\underline{\mathbb{E}}_{\rateset,\mathcal{M}}^{\mathrm{W}}[f(X_s)\vert X_0=x]
=
\underline{\mathbb{E}}_{\rateset,\mathcal{M}}^{\mathrm{WM}}[f(X_s)\vert X_0=x],
\end{equation*}
thereby providing the function $\underline{f}_s$ with a clear interpretation in terms of sets of stochastic processes. Due to this connection, many of the results in Section~\ref{sec:lowertrans} can be interpreted as properties of $\underline{f}_s$. For example, Definition~\ref{def:low_trans} provides an alternative characterisation for $\underline{f}_s$, and Proposition~\ref{prop:lower_transition_has_deriv} shows that the defining differential equation for $\underline{f}_s$---Equation~\eqref{eq:boundsdifferentialDamjan}---can be stated uniformly, that is, with respect to the operator norm instead of pointwise.

Besides the differences that are related to the role of the set of processes $\mathcal{P}$, an other difference between our work and that of {\smash{\v{S}kulj}} is the type of functions that are considered. Due to its focus on $\overline{f}_s$ (and, by conjugacy, $\underline{f}_s$), the work of {\v{S}kulj} applies to functions $f$ that depend on the state $X_s$ at a single time point $s$. In contrast, our work also considers functions that depend on the state at any finite number of time points. In this more general context, as we have seen, it actually becomes crucial to be aware of the exact set of stochastic processes with respect to which the lower expectation is taken, because this choice influences the value of the lower expectation, as well as whether or not it can be computed efficiently.

A final difference between our work and that of {\v{S}kulj} are the computational methods that are presented. 
In both cases, these methods are concerned with approximating $\smash{\underline{f}_s=L_0^sf}$ up to some desired finite precision $\epsilon$. However, the particular approximation methods differ. {\v{S}kulj} considers three distinct methods: a uniform grid discretisation, an adaptive grid discretisation, and a combination of both. We consider only one method, which corresponds to Algorithm~\ref{alg:compute_singlevar}.

{\v{S}kulj}'s uniform grid discretisation method works similarly to our Algorithm~\ref{alg:compute_singlevar}, in that it considers a uniform partition $u=t_0,t_1,\ldots,t_n$ of the time interval $[0,s]$ and, at each time point $t_i$, requires solving a local optimisation problem. In our case, this local computation consists in evaluating $\smash{(I+\Delta\lrate)g=g+\Delta\lrate g}$ for some $g\in\gamblesX$, with $\Delta=\nicefrac{s}{n}$. In the uniform grid discretisation method of {\v{S}kulj}, this local computation consists in evaluating $e^{\Delta Q}g$, where $Q\in\rateset$ needs to be selected in such a way that $Qg=\lrate g$. 

It turns out that among these two methods, our Algorithm~\ref{alg:compute_singlevar} is much more efficient, for the following two reasons. First, it does not require the computation of a matrix exponential, which, for large state spaces, can be computationally demanding. Secondly, the number of steps $n$ that we require is much smaller. For instance, for the numerical example presented at the end of~\cite[Section 4.1]{Skulj:2015cq}, {\v{S}}kulj finds that his partition requires more than $n=6000$ steps. In contrast, Algorithm~\ref{alg:compute_singlevar} requires only $n=50$ steps for that particular example. Similarly, in our Example~\ref{exmp:single_time_numerical} we found that we required $n=8000$ steps. Applying the method from~\cite{Skulj:2015cq}, by comparison, requires that $n\approx 2.3\times 10^9$. 
These observations confirm empirical findings in Reference~\cite{troffaes2015using}, where the discretisation method of Algorithm~\ref{alg:compute_singlevar} was applied in a practical problem. The theoretical guarantee that is provided by Corollary~\ref{prop:approximation_error_bound} was not available at the time, but the authors of Reference~\cite{troffaes2015using} did observe that the number of steps required to obtain empirical convergence was much lower than {\v{S}kulj}'s theoretical upper bound in Reference~\cite{Skulj:2015cq}.

In order to avoid having to use a too fine grid, as his uniform discretisation method tends to require, \v{S}kulj also considers another computational method, which he calls adaptive grid discretisation. In this method, the size of the discretisation steps is not forced to be constant, but is instead varied in order to try and reduce the total number of steps required. This second method is computationally more efficient, but has the disadvantage that it can only be applied to certain cases. Therefore, \v{S}kulj also considers a third method, which consists in combining his first two methods in such a way as to profit from their respective advantages. 

Since the limit in Definition~\ref{def:low_trans} does not require the partition $u$ to be uniform, our computational methods could, in principle, consider non-uniform `adaptive grid' discretisations as well, thereby mimicking \v{S}kulj's second and third method. However, since we did not investigate the---theoretical and empirical---efficiency of such an approach, we are unable to compare it with the methods of \v{S}kulj. As discussed in our conclusions below, we consider this to be a promising avenue for future research.

\section{Conclusions \& Future Work}\label{sec:conclusions}

In this work, we formalised the concept of \emph{imprecise continuous-time Markov chains}; a generalisation of continuous-time Markov chains that robustifies this model class with respect to often-made simplifying assumptions. 
In particular, it allows one to relax the requirement of having to specify exact values---point-estimates---for the numerical parameters of a continuous-time Markov chain, and furthermore relaxes the assumptions of time-homogeneity and Markovian probabilistic independence. Since these assumptions are often made pragmatically to ensure a tractable model class, it should not be obvious that this generalisation results in a practically workable model class. Nevertheless, we were still able to derive polynomial runtime-complexity algorithms for computing quantities of interest.

In the remainder of this section, given the length of this paper, we start by providing a brief summary of our most important technical contributions. We then finally close by discussing some future lines of research that we consider to be promising.

The starting point of this work was to formally define \emph{continuous-time stochastic processes} within the framework of full conditional probabilities, using the notion of coherent conditional probabilities. After restricting ourselves to \emph{well-behaved} stochastic processes, we then introduced a way to characterise the dynamics of such stochastic processes by means of their \emph{outer partial derivatives}---sets of \emph{rate matrices}, which we have shown to be non-empty, bounded and closed, and which describe the infinitesimal state-transition rates of these processes. In particular, these outer partial derivatives allow us to describe this behaviour without imposing differentiability assumptions on the conditional state-transition probabilities. Next, by imposing the Markov property on such stochastic processes, we obtained continuous-time Markov chains as a special case. Finally, by additionally assuming the time-homogeneity of the state-transition probabilities, we have re-derived, within the framework of full conditional probabilities, the well known homogeneous continuous-time Markov chains, which are those stochastic processes whose matrix of conditional state-transition probabilities $T_t^s$ is given by the matrix exponential $e^{Q(s-t)}$, for some rate matrix $Q$.

Using these definitions, we have defined \emph{imprecise continuous-time Markov chains}, which we abbreviate as {\ictmc}'s, as sets of well-behaved stochastic processes whose dynamics are \emph{consistent} with a given set of rate matrices $\rateset$ and a given set of initial distributions $\mathcal{M}$. In particular, an \ictmc~is a set of stochastic processes whose outer partial derivatives are contained within $\rateset$ and whose initial distribution belongs to $\mathcal{M}$. We distinguished between three different {\ictmc}'s, some of which impose  additional conditions. The {\ictmc} $\smash{\wprocesses_{\rateset,\mathcal{M}}}$ imposes no additional conditions. The {\ictmc} $\smash{\wmprocesses_{\rateset,\mathcal{M}}}$ restricts attention to Markov processes, and the {\ictmc} $\smash{\whmprocesses_{\rateset,\mathcal{M}}}$ only considers Markov processes that are time-homogeneous. We have investigated the closure-properties of these three different {\ictmc}'s and, in particular, have studied their closure under specific types of recombinations of their elements.

The \emph{lower expectation} of an {\ictmc} was defined as the infimum of the expectations that correspond to the different stochastic processes in the \ictmc, that is, the highest lower bound on these expectations. Upper expectations and lower and upper probabilities were defined similarly, but since they correspond to special cases, we focussed on lower expectations. The difficulty of computing such lower expectations was shown to vary between the sets $\smash{\wprocesses_{\rateset,\mathcal{M}}}$, $\smash{\wmprocesses_{\rateset,\mathcal{M}}}$, and $\smash{\whmprocesses_{\rateset,\mathcal{M}}}$. For example, we have seen that the law of iterated lower expectation applies only to the set $\wprocesses_{\rateset,\mathcal{M}}$.

In order to make the computation of lower expectations tractable, we introduced the notion of a lower transition rate operator $\lrate$, and used it to define the corresponding \emph{lower transition operator} $L_t^s$. 
First, we showed that $\lrate$ can be regarded as the lower envelope of a non-empty bounded set of rate matrices $\rateset$, and that there is a one-to-one correspondence between lower transition rate operators $\lrate$ and sets of rate matrices $\rateset$ that are non-empty, bounded, closed and convex, and have \emph{separately specified rows}. Next, we showed that the operator $L_t^s$ satisfies various convenient algebraic properties, including time-homogeneity, differentiability and the semi-group property, and that it can be regarded as a generalisation of the well-known matrix exponential. Finally, we established that the operator $L_t^s$ corresponds to a uniform---rather than pointwise---solution of the differential equation~\eqref{eq:pointwisedifferential}, which was previously proposed as a definition for the conditional lower expectation $\underline{\mathbb{E}}[f(X_s)\vert X_t]$ of an {\ictmc} in Reference~\cite{Skulj:2015cq}---thereby bypassing the role of sets of stochastic processes.

For the \ictmc's $\smash{\wprocesses_{\rateset,\mathcal{M}}}$ and $\smash{\wmprocesses_{\rateset,\mathcal{M}}}$, we have shown that the corresponding lower expectations satisfy a property that is analogous to the definition in Reference~\cite{Skulj:2015cq}. In particular, if the lower transition rate operator $\lrate$ is the lower envelope of $\rateset$, and if $\rateset$ has separately specified rows, then $\underline{\mathbb{E}}[f(X_s)\vert X_t, X_u]$ is equal to $[L_t^sf](X_t)$, where the lower expectation $\underline{\mathbb{E}}[\cdot\vert\cdot]$ can be taken with respect to either $\wprocesses_{\rateset,\mathcal{M}}$ or $\wmprocesses_{\rateset,\mathcal{M}}$. Moreover, we have seen that the largest set of stochastic processes for which this is the case, is the set of all well-behaved stochastic processes that are consistent with $\rateset_{\lrate}$, where $\rateset_{\lrate}$ is the unique largest set of rate matrices that dominates $\lrate$.

Using this connection between the operator $L_t^s$ and the lower expectations of $\smash{\wprocesses_{\rateset,\mathcal{M}}}$ and $\smash{\wmprocesses_{\rateset,\mathcal{M}}}$, we then went on to develop algorithms for the numerical computation of these lower expectations, for functions that depend on the state at a finite number of time points. For $\smash{\wmprocesses_{\rateset,\mathcal{M}}}$, we showed that this is---in general---only possible if all but one of these time points is situated in the past. For $\smash{\wprocesses_{\rateset,\mathcal{M}}}$, such a restriction was not necessary, and we presented algorithms that are able to deal with functions that depend on the state at multiple future time points, provided that $\rateset$ is convex. Our algorithms apply to conditional as well as unconditional lower expectations, and are able to compute these quantities within a guaranteed $\epsilon$-error bound.

In summary, then, we have provided in this work a generalisation of continuous-time Markov chains to imprecise continuous-time Markov chains, which robustifies these models with respect to both numerical parameter assessments and the simplifying assumptions of time-homogeneity, Markovian probabilistic independence and differentiability of the state-transition probabilities. Notably, our lower transition operator provides a convenient analytical tool to represent such a model's lower expectations, and these quantities can be tractably computed using the algorithms that we presented. Therefore, this work can be seen as providing the required rigorous foundations for working with \ictmc's, both theoretically and practically. As such, we expect that it will provide plenty of avenues for future research.

On a technical level, a first important line of future research would be to extend our results to functions that are allowed to depend on the state $X_t$ at all time points $t\in\realsnonneg$---rather than only a finite number of time points. Being able to compute the lower and upper expectation of such functions would allow for the study of some practically important aspects that our current framework cannot tackle yet, such as the time until the system reaches some desired state, the average time that is spent in a state, etcetera. 
Our framework is well-suited to be extended to such a setting, because the coherence of the stochastic processes that we consider ensures that their domain can be extended to include events that depend on the state at all time points. Furthermore---should this be desired---the results in Reference~\cite{berti2002coherent} allow for such an extension to be established in a $\sigma$-additive way. 

A second technical line of future research would be to drop the assumption that $\states$ should be finite, and to exend our framework to the case where $\states$ is countably---or perhaps even uncountably---infinite.

Thirdly, still on the technical level, it would be interesting to investigate whether the requirement that $\rateset$ should be convex---which we currently need in order to guarantee that Algorithm~\ref{alg:compute_multivar} yields a correct result---is really necessary. In fact, we conjecture that it is not. Basically, the reason why Algorithm~\ref{alg:compute_multivar} currently requires $\rateset$ to be convex, is because it relies on Theorem~\ref{theorem:decomposition_multivar}, the proof of which in turn relies on Theorem~\ref{theo:aanelkaarplakken}. While we believe that the convexity of $\rateset$ is necessary in order for Theorem~\ref{theo:aanelkaarplakken} to hold, we think that this is not the case for Theorem~\ref{theorem:decomposition_multivar}, and we think that it should be possible to provide the latter with an alternative proof that does not require $\rateset$ to be convex. 

Yet another interesting line of future research would be to generalise the models that we consider by allowing the set of rate matrices $\rateset$ to be time-dependent, that is, to consider a set-valued function $\rateset_t$. As one anonymous reviewer pointed out, this could be used for example to model the fact that over time, we become less and less certain about what rate matrix to use. The subtlety will then lie in choosing the kind of continuity assumptions that should be imposed on such a function $\rateset_t$. If it is piecewise-constant, then we expect that our methods should translate fairly naturally to such a generalised model. For more complex time-variations, however, technical and computational difficulties may need to be overcome. 

As far as the computational aspects of our work are concerned, there is also still quite a lot of room for future work. In particular, we believe that it should be possible to improve upon the efficiency of our methods for computing $L_t^s$. On the one hand, we did not yet study the effect of allowing for non-uniform discretisations---as \v{S}kulj does in his adaptive discretisation method~\cite{Skulj:2015cq}. On the other hand, our approach---as well as that of \v{S}kulj---does not yet exploit the fact that $L_t^s=L_0^{s-t}$ is guaranteed to converge to a limit as $s-t$ approaches infinity~\cite{DeBock:2016}; for large values of $s-t$, this feature should surely make it easier to compute $L_t^s$, whereas our---and \v{S}kulj's---methods become less efficient as $s-t$ increases. Some initial results in this direction can already be found in Reference~\cite{Krak:2017}.

Finally, on an algorithmic level, a practically important future line of research would be to develop specialized versions of our algorithms, which are tailored to deal with specific functions. The reason why we consider this to be important is because, for functions that depend on the state at a finite number of time points, the runtime-complexity of our algorithms is currently exponential in this number of time points. This is not problematic---and is in fact to be expected---because for general functions, even simply specifying the function has such a complexity. However, in many practical applications, there will typically be some underlying structure of the functions that we are interested in. For instance, for time points $v=s_0,\ldots,s_m$, there may be functions $g_i\in\gambles(\states_{s_i})$, $i\in\{0,\ldots,m\}$, such that $f\in\gambles(\states_v)$ is of the form $f(x_v) = \sum_{i=0}^m g_i(x_{s_i})$ or $f(x_v) = \prod_{i=0}^m g_i(x_{s_i})$.
In the discrete time case, that is, for imprecise discrete time Markov chains, it has been shown that for these specific types of functions, inference algorithms such as Algorithm~\ref{alg:compute_multivar} simplify considerably, to the extent that their runtime complexity becomes linear in the number of time points~\cite{de2015credal,Lopatatzidis2015}. We believe that it is possible---and even relatively straightforward---to extend these techniques to the continuous time framework of this paper. Furthermore, since similar techniques also lie at the heart of existing algorithms for imprecise discrete-time \emph{hidden} Markov chains~\cite{Benavoli:2011ei,DeBock:2014ts}, we believe that this line of research could also serve as a first step towards the development of imprecise continuous-time \emph{hidden} Markov chains---where the states are observed indirectly through noisy outputs---and in particular, towards the development of efficient algorithms for problems such as filtering and smoothing. In fact, some first steps in this direction have already been taken in Reference~\cite{Krak:2017}.

\subsection*{Acknowledgements}

The authors wish to express their sincere gratitude to three anonymous reviewers who, despite the substantial length of this manuscript, provided timely, valuable and detailed feedback. We also thank the area editor, Marco Zaffalon, for finding these reviewers; surely, given the length of the paper, this must not have been easy. Jasper De Bock is a Postdoctoral Fellow of the Research Foundation - Flanders (FWO) and wishes to acknowledge its financial support; he would also like to thank Alexander Erreygers for some stimulating discussions about the computational aspects of imprecise continuous-time Markov chains.

\bibliographystyle{plain} 
\bibliography{general}

\appendix

\section{Proofs and Lemmas for the results in Section~\ref{sec:systems}}\label{app:systems}

\lemmacompositiontransitionmatrix*
\begin{proof}
Simply check each of the properties.
\end{proof}

\propstochasticfromratematrix*
\begin{proof}
\ref{def:T:sumone} follows from \ref{def:Q:sumzero}: for all $x\in\states$, \ref{def:Q:sumzero} implies that
\begin{equation*}
\sum_{y\in\states} [I + \Delta Q](x,y) = \sum_{y\in\states}I(x,y) + \Delta \sum_{y\in\states}Q(x,y) = 1+\Delta 0=1.
\end{equation*}
\ref{def:T:nonneg} follows from~\ref{def:Q:nonnegoffdiagonal} and because $0\leq \Delta\norm{Q} \leq 1$: for all $x,y\in\states$ such that $x\neq y$, $0\leq\Delta\norm{Q} \leq 1$ implies that $[I+\Delta Q](x,x)=1+\Delta Q(x,x)\geq 1-\Delta\norm{Q}\geq0$, and \ref{def:Q:nonnegoffdiagonal} and $\Delta\geq0$ imply that $[I+\Delta Q](x,y)=\Delta Q(x,y)\geq0$.
\end{proof}

\propratefromstochasticmatrix*
\begin{proof}
This proof is analogous to that of Proposition~\ref{prop:stochastic_from_rate_matrix}; simply verify each of the properties in Definition~\ref{def:rate_matrix}.
\end{proof}

\propalternativedefforbounded*
\begin{proof}
We start by proving that Equation~\eqref{eq:alternative_bounded} implies $\norm{\rateset}<+\infty$. 
 To this end, assume that Equation~\eqref{eq:alternative_bounded} holds. Since $\states$ is finite, it then follows that
\begin{align*}
\norm{\rateset} = \sup\{\norm{Q}\,:\,Q\in\rateset\}
 &= \sup\left\{\max\left\{\sum_{y\in\states}\abs{Q(x,y)}\,:\,x\in\states\right\}\,:\,Q\in\rateset\right\} \\
 &= \max\left\{\sup\left\{\sum_{y\in\states}\abs{Q(x,y)}\,:\,Q\in\rateset\right\} \,:\,x\in\states\right\}.
\end{align*}Hence, there is some $x'\in\states$ such that
\begin{align*}
\norm{\rateset} = \sup_{Q\in\rateset}\sum_{y\in\states}\abs{Q(x',y)}
&=\sup_{Q\in\rateset}\left(2\abs{Q(x',x')}\right)
 = 2\sup_{Q\in\rateset}\abs{Q(x',x')}\\
 &= -2\inf_{Q\in\rateset}\left(-\abs{Q(x',x')}\right)
 = -2\inf_{Q\in\rateset}Q(x',x')< +\infty\,,
\end{align*}
where the second and last equality follows from Definition~\ref{def:rate_matrix}, and the final inequality follows from Equation~\eqref{eq:alternative_bounded}.

We next show that $\norm{\rateset}<+\infty$ implies Equation~\eqref{eq:alternative_bounded}. To this end, consider any set of rate matrices $\rateset\subseteq\mathcal{R}$ such that $\norm{\rateset}<+\infty$, and assume \emph{ex absurdo} that Equation~\eqref{eq:alternative_bounded} is not true. Then clearly, there is some $x'\in\states$ and $Q'\in\rateset$ such that $Q'(x,x)<-\norm{\rateset}$. However, this implies that $\norm{\rateset}\geq\norm{Q'}\geq\abs{Q'(x,x)}>\norm{\rateset}$, which is a contradiction. Hence, it follows that Equation~\eqref{eq:alternative_bounded} must be true.
\end{proof}

\begin{lemma}{\cite[Theorem 2.1.1]{norris1998markov}}\label{lemma:deriv_exponential_trans}
For any $Q\in\mathcal{R}$, we have that
\begin{equation*}
{\frac{d}{d \Delta}e^{Q\Delta}}\big\vert_{\Delta=0} \coloneqq \lim_{\Delta\to 0}\frac{1}{\Delta}(e^{Q\Delta}-I)= Q.
\end{equation*}
\end{lemma}

\propsystemQ*
\begin{proof}
We start by showing that $\mathcal{T}_Q$ is a transition matrix system. Because of Proposition~\ref{prop:stochastic_from_exponential}, $\mathcal{T}_Q$ is clearly a family of transition matrices. Consider now any $t,r,s\in\realsnonneg$ such that $t\leq r\leq s$. It then follows from the definition of $\mathcal{T}_Q$ and~\cite[Theorem 2.1.1]{norris1998markov} that $T_t^s=T_t^rT_r^s$, and $T_t^t=I$. Because the $t,r,s$ are arbitrary, it follows from Definition~\ref{def:trans_mat_system} that $\mathcal{T}_Q$ is a transition matrix system.

To prove that $\mathcal{T}_Q$ is well-behaved, note that for any $t\in\realsnonneg$, because of Definition~\ref{def:systemfromQ},
\begin{align*}
\limsup_{\Delta\to 0^+}\frac{1}{\Delta}\norm{T_t^{t+\Delta}-I} 
&=\limsup_{\Delta\to 0^+}\norm{\frac{1}{\Delta}(T_t^{t+\Delta}-I)-Q+Q}\\
&\leq\limsup_{\Delta\to 0^+}\norm{\frac{1}{\Delta}(T_t^{t+\Delta}-I)-Q}+\norm{Q}  
= \norm{Q} < \infty\,,
\end{align*}
where the first inequality follows from Proposition~\ref{prop:norm_properties}, the second equality follows from Lemma~\ref{lemma:deriv_exponential_trans} and the final inequality follows from the fact that $Q$ is real-valued. Because this holds for any $t\in\realsnonneg$, the first condition in Equation~\eqref{eq:wellbehavedtransitionmatrixsystem} is satisfied. A similar argument shows that also the second condition is satisfied for all $t\in\realspos$, and hence $\mathcal{T}_Q$ is well-behaved.
\end{proof}

\proprestrtransmatsystemifsemigroup*
\begin{proof}
If $\mathcal{T}^{\mathbf{I}}$ is a restricted transition matrix system, then, by definition, it is the restriction to $\mathbf{I}$ of some transition matrix system $\mathcal{T}$. Therefore, the `only if' part of this result follows trivially from Definition~\ref{def:trans_mat_system}.

For the `if' part, we need to prove that for any family of transition matrices $\mathcal{T}^{\mathbf{I}}$ such that, for all $t,r,s\in\mathbf{I}$ with $t\leq r\leq s$, it holds that $T_t^s = T_t^rT_r^s$ and $T_t^t=I$, there is a transition matrix system $\mathcal{T}$ that coincides with $\mathcal{T}^{\mathbf{I}}$ on $\mathbf{I}$. In order to prove this, it suffices to show that the unique family of transition matrices $\mathcal{T}$ that coincides with $\mathcal{T}^{\mathbf{I}}$ on $\mathbf{I}$ and that is otherwise defined by
\begin{equation}\label{eq:prop:restr_trans_mat_system_if_semigroup}
T_t^s
\coloneqq
\begin{cases}
I &\text{ if $s<\min \mathbf{I}$}\\
T_{\min\textbf{I}}^s &\text{ if $t<\min\textbf{I}$ and $s\in\textbf{I}$}\\
T_{\min\textbf{I}}^{\sup\textbf{I}} &\text{ if $t<\min\textbf{I}$ and $\sup\textbf{I}<s$}\\
T_t^{\sup\textbf{I}}
&\text{ if $t\in\textbf{I}$ and $\sup\textbf{I}<s$}\\
I &\text{ if $\sup\mathbf{I}<t$}
\end{cases}
~~~\text{ for $t,s\in\realsnonneg$ with $t\leq s$ and $[t,s]\not\subseteq\mathbf{I}$,}
\end{equation}
is a transition matrix system. This is a matter of straightforward verification.
\end{proof}

\propwellrestrtransmatsystemiflimsup*
\begin{proof}
If $\mathcal{T}^{\mathbf{I}}$ is well-behaved, then, by definition, it is the restriction to $\mathbf{I}$ of a well-behaved transition matrix system $\mathcal{T}$. Therefore, the `only if' part of this result follows trivially from Definition~\ref{def:well_behaved_trans_mat_system}.

For the `if' part, we need to show that for any restricted transition matrix system $\mathcal{T}^\mathbf{I}$ on $\mathbf{I}$ that satisfies Equation~\eqref{eq:wellbehavedrestrictedtransitionmatrixsystem}, there is a well-behaved transition matrix system $\mathcal{T}$ that coincides with $\mathcal{T}^\mathbf{I}$ on $\mathbf{I}$. Let $\mathcal{T}$ be constructed as in the proof of Proposition~\ref{prop:restr_trans_mat_system_if_semigroup}. Then, as explained in that proof, $\mathcal{T}$ is a transition matrix system that coincides with $\mathcal{T}^{\mathbf{I}}$ on $\mathbf{I}$. Therefore, it suffices to prove that $\mathcal{T}$ is well-behaved. We start by proving the first part of Equation~\eqref{eq:wellbehavedtransitionmatrixsystem}. So consider any $t\in\realsnonneg$. If $t\in\mathbf{I}^+$, then the desired inequality follows from Equation~\eqref{eq:wellbehavedrestrictedtransitionmatrixsystem}. If $t\notin\mathbf{I}^+$, then either $t<\min\mathbf{I}$ or $t\geq\sup\mathbf{I}$, and therefore, for sufficiently small $\Delta>0$, it follows from Equation~\eqref{eq:prop:restr_trans_mat_system_if_semigroup} that $T_t^{t+\Delta}=I$, thereby making the desired inequality trivially true. The second part of Equation~\eqref{eq:wellbehavedtransitionmatrixsystem} can be proved similarly.
\end{proof}

\propconcatrestrtransmatsystemsissystem*
\begin{proof}
For all $t,s\in\mathbf{I}\cup\mathbf{J}$ such that $t\leq s$, it follows from Proposition~\ref{lemma:compositiontransitionmatrix} that the matrix $T_t^s$ that corresponds to $\mathcal{T}^{\mathbf{I}\cup\mathbf{J}}$ is a transition matrix. Furthermore, for all $t\in\mathbf{I}\cup\mathbf{J}$, we have that either $t\in\mathbf{I}$ or $t\in\mathbf{J}$. In either case, we have that $T_t^t=I$, because either $T_t^t=\presuper{i}T_t^t=I$ or $T_t^t=\presuper{j}T_t^t=I$, with $\presuper{i}T_t^t$ and $\presuper{j}T_t^t$ corresponding to $\mathcal{T}^{\mathbf{I}}$ and $\mathcal{T}^{\mathbf{J}}$, respectively. Next, we show that for all $t,q,s\in\mathbf{I}\cup\mathbf{J}$ with $t\leq q\leq s$, it holds that
\begin{equation*}
T_t^s = T_t^qT_q^s\,.
\end{equation*}
If both $t,s\in\mathbf{I}$ or if both $t,s\in\mathbf{J}$, this clearly holds. Therefore, we may assume that $t\in\mathbf{I}$ and $s\in\mathbf{J}$. Suppose furthermore that $q\in\mathbf{I}$. Then, from the definition of the concatenation operator $\otimes$, we have that $T_q^s=T_q^rT_r^s$, with $r=\max\mathbf{I}=\min\mathbf{J}$. Because $t,q,r\in\mathbf{I}$, we know that $T_t^qT_q^r=T_t^r$, and hence, by the definition of the concatenation operator,
\begin{equation*}
T_t^qT_q^s = T_t^qT_q^rT_r^s = T_t^rT_r^s = T_t^s\,.
\end{equation*}
An exactly analogous argument proves the case for $q\in\mathbf{J}$. Therefore, it follows from Proposition~\ref{prop:restr_trans_mat_system_if_semigroup} that $\mathcal{T}^{\mathbf{I}\cup\mathbf{J}}$ is a restricted transition matrix system.

It remains to prove that if $\mathcal{T}^{\mathbf{I}}$ and $\mathcal{T}^{\mathbf{J}}$ are both well-behaved, then $\mathcal{T}^{\mathbf{I}\cup\mathbf{J}}$ is also well-behaved. Due to proposition~\ref{prop:well_restr_trans_mat_system_if_limsup}, it suffices to prove Equation~\eqref{eq:wellbehavedrestrictedtransitionmatrixsystem}. We only prove the left part of this equation, that is, we prove that
\begin{equation*}
\left(\forall t\in(\mathbf{I}\cup\mathbf{J})^+\right)~\limsup_{\Delta\to0^+}\norm{\frac{1}{\Delta}\left(T_t^{t+\Delta} - I\right)} < +\infty.
\end{equation*}
The proof for the right part of Equation~\eqref{eq:wellbehavedrestrictedtransitionmatrixsystem} is completely analogous. So consider any $t\in(\mathbf{I}\cup\mathbf{J})^+$. Since $\sup\mathbf{I}=\max\mathbf{I}=\min\mathbf{J}$, it follows that
\begin{equation*}
(\mathbf{I}\cup\mathbf{J})^+
\coloneqq
(\mathbf{I}\cup\mathbf{J})\setminus\{\sup(\mathbf{I}\cup\mathbf{J})\}=(\mathbf{I}\cup\mathbf{J})\setminus\{\sup\mathbf{J}\}
=(\mathbf{I}\setminus\sup\mathbf{I})\cup(\mathbf{J}\setminus\sup\mathbf{J})=\mathbf{I}^+\cup\mathbf{J}^+.
\end{equation*}
Therefore, without loss of generality, we may assume that $t\in\mathbf{I}^+$. The desired result now follows by applying Proposition~\ref{prop:well_restr_trans_mat_system_if_limsup} to the well-behaved restricted transition matrix system $\mathcal{T}^\mathbf{I}$.
\end{proof}

The following lemma states a very useful norm inequality that we will use repeatedly in the proofs in this work. This result states that the distance between two composed transition matrices $T_1T_2\cdots T_n$ and $S_1S_2\cdots S_n$ is bounded from above by the sum of the distances $\norm{T_i - S_i}$ of their component transition matrices.

\begin{lemma}\label{lemma:recursive}
Let $T_1,\ldots,T_n$ and $S_1,\ldots,S_n$ be two finite sequences of transition matrices. Then
\begin{equation*}
\norm{\prod_{i=1}^nT_i - \prod_{i=1}^nS_i} \leq \sum_{i=1}^n \norm{T_i - S_i}.
\end{equation*}
\end{lemma}
\begin{proof}
This is a special case of Lemma~\ref{lemma:recursive_lower_trans}, which states a more general version. We have included this version separately because Lemma~\ref{lemma:recursive_lower_trans} uses concepts that, on a chronological reading of this paper, would be undefined at this point.
\end{proof}

\lemmarestrictedtransmatsystemcauchyconverges*
\begin{proof}
Consider any sequence $\{\mathcal{T}_i^{\mathbf{I}}\}_{i\in\nats}$ in $\mathbfcal{T}^{\mathbf{I}}$ that is Cauchy. We will prove that this sequence converges to a limit that belongs to $\mathbfcal{T}^{\mathbf{I}}$. For all $i\in\nats$ and any $t,s\in\mathbf{I}$ such that $t\leq s$, we will use $\presuper{i}T_t^s$ to denote the transition matrix that corresponds to $\mathcal{T}_i^{\mathbf{I}}$. 

Since $\{\mathcal{T}_i^{\mathbf{I}}\}_{i\in\nats}$ is Cauchy, 
it follows from Equation~\eqref{eq:trans_mat_system_metric} that
\begin{equation}\label{eq:lemma:restricted_trans_mat_system_cauchy_converges:consequenceofCauchy}
(\forall\epsilon\in\realspos)\,(\exists n_\epsilon\in\nats)\,(\forall k,\ell > n_\epsilon)\,(\forall t,s\in\mathbf{I}\colon t\leq s)
~\norm{\presuper{k}T_t^s-\presuper{\ell}T_t^s} < \epsilon.
\end{equation}
Clearly, for any $t,s\in\mathbf{I}$ such that $t\leq s$, this implies that the sequence $\{\presuper{i}T_t^s\}_{i\in\nats}$ is Cauchy. Since the set of all transition matrices---of the same finite dimension---is trivially complete, this implies that the sequence $\{\presuper{i}T_t^s\}_{i\in\nats}$ has a limit $T_t^s$, and that this limit is a transition matrix. We use $\mathcal{T}^{\mathbf{I}}$ to denote the family of transition matrices that consists of these limits. 

Fix any $t,r,s\in \mathbf{I}$ such that $t\leq r\leq s$. Then for any $i\in\nats$, because $\mathcal{T}_i^{\mathbf{I}}$ is a restricted transition matrix system, we know  that $\presuper{i}T_t^t=I$ and $\presuper{i}T_t^s=\presuper{i}T_t^r\presuper{i}T_r^s$, which implies that $\norm{T_t^t-I}=\norm{T_t^t-\presuper{i}T_t^t}$ and, due to Lemma~\ref{lemma:recursive}, that
\begin{align*}
\norm{T_t^s-T_t^rT_r^s}
&\leq
\norm{T_t^s-\presuper{i}T_t^s}+\norm{\presuper{i}T_t^r\presuper{i}T_r^s-T_t^rT_r^s}\\
&\leq
\norm{T_t^s-\presuper{i}T_t^s}
+
\norm{\presuper{i}T_t^r-T_t^r}
+
\norm{\presuper{i}T_r^s-T_r^s}.
\end{align*}
Since we also know that $\lim_{i\to+\infty}\presuper{i}T_t^t=T_t^t$, $\lim_{i\to+\infty}\presuper{i}T_t^s=T_t^s$, $\lim_{i\to+\infty}\presuper{i}T_t^r=T_t^r$ and $\lim_{i\to+\infty}\presuper{i}T_r^s=T_r^s$, this implies that $\norm{T_t^t-I}=0$ and $\norm{T_t^s-T_t^rT_r^s}=0$, or equivalently, that $T_t^t=I$ and $T_t^s=T_t^rT_r^s$. Since this is true for any $t,r,s\in I$ such that $t\leq r\leq s$, and because we already know that that the family $\mathcal{T}^{\mathbf{I}}$ consists of transition matrices, it follows from Proposition~\ref{prop:restr_trans_mat_system_if_semigroup} that $\mathcal{T}^{\mathbf{I}}$ is a restricted transition matrix system.
In the remainder of this proof, we will show that $\mathcal{T}^{\mathbf{I}}=\lim_{i\to\infty}\mathcal{T}_i^{\mathbf{I}}$. 

Fix any $\epsilon>0$ and consider the corresponding $n_\epsilon\in\nats$ whose existence is guaranteed by Equation~\eqref{eq:lemma:restricted_trans_mat_system_cauchy_converges:consequenceofCauchy}. Fix any $k>n_\epsilon$. For any $t,s\in\mathbf{I}$ such that $t\leq s$, it then follows from Equation~\eqref{eq:lemma:restricted_trans_mat_system_cauchy_converges:consequenceofCauchy} that, for all $\ell>n_\epsilon$:
\begin{equation*}
\norm{\presuper{k}T_t^s-T_t^s}
\leq
\norm{\presuper{k}T_t^s-\presuper{\ell}T_t^s}
+
\norm{\presuper{\ell}T_t^s-T_t^s} < \epsilon+\norm{\presuper{\ell}T_t^s-T_t^s}.
\end{equation*}
Since $\lim_{\ell\to+\infty}\presuper{\ell}T_t^s=T_t^s$, this implies that $\norm{\presuper{k}T_t^s-T_t^s}\leq\epsilon$. Since this is true for all $t,s\in\mathbf{I}$ such that $t\leq s$, it follows from Equation~\eqref{eq:trans_mat_system_metric} that $d(\mathcal{T}_k^{\mathbf{I}},\mathcal{T}^{\mathbf{I}})\leq\epsilon$. Since $\epsilon>0$ was arbitrary, we conclude that
\begin{equation*}
(\forall\epsilon\in\realspos)\,(\exists n_\epsilon\in\nats)\,(\forall k > n_\epsilon)~d(\mathcal{T}_k^{\mathbf{I}},\mathcal{T}^{\mathbf{I}})\leq\epsilon,
\end{equation*}
which implies that $\mathcal{T}^{\mathbf{I}}=\lim_{i\to\infty}\mathcal{T}_i^{\mathbf{I}}$.
\end{proof}

\begin{lemma}\label{lemma:linearpartofexponential}
Consider any $Q\in\mathcal{R}$ and any $\Delta\geq0$. Then,
\begin{equation*}
\norm{e^{Q\Delta}-(I+\Delta Q)}\leq
\Delta^2\norm{Q}^2.
\end{equation*}
\end{lemma}
\begin{proof}
This is a special case of Lemma~\ref{lemma:quadraticboundonL}, which states a more general version. We have included this version separately because Lemma~\ref{lemma:quadraticboundonL} uses concepts that, on a chronological reading of this paper, would be undefined at this point.
\end{proof}

\begin{lemma}\label{lemma:linearboundonexponential}
Consider any $Q\in\mathcal{R}$ and any $\Delta\geq0$. Then,
\begin{equation*}
\norm{e^{Q\Delta}-I}\leq
\Delta\norm{Q}.
\end{equation*}
\end{lemma}
\begin{proof}
This is a special case of Lemma~\ref{lemma:linearboundonL}, which states a more general version. We have included this version separately because Lemma~\ref{lemma:linearboundonL} uses concepts that, on a chronological reading of this paper, would be undefined at this point.
\end{proof}

\begin{lemma}\label{lemma:differencebetweenexponentials}
If $Q_1,Q_2\in\mathcal{R}$ and $\Delta\in\realsnonneg$, then $\norm{e^{Q_1\Delta}-e^{Q_2\Delta}}\leq\Delta\norm{Q_1-Q_2}$.
\end{lemma}
\begin{proof}
Consider any $n\in\nats$. It then follows from Lemma~\ref{lemma:recursive} that
\begin{equation*}
\norm{e^{Q_1\Delta}-e^{Q_2\Delta}}
=\norm{\prod_{k=1}^n e^{Q_1\frac{\Delta}{n}}-\prod_{k=1}^n e^{Q_2\frac{\Delta}{n}}}
\leq\sum_{k=1}^n\norm{e^{Q_1\frac{\Delta}{n}}-e^{Q_2\frac{\Delta}{n}}}
=
n\norm{e^{Q_1\frac{\Delta}{n}}-e^{Q_2\frac{\Delta}{n}}}.
\end{equation*}
which, since we know from Lemma~\ref{lemma:linearpartofexponential} that
\begin{align*}
\norm{e^{Q_1\frac{\Delta}{n}}-e^{Q_2\frac{\Delta}{n}}}
&\leq
\norm{e^{Q_1\frac{\Delta}{n}}-(I+\frac{\Delta}{n}Q_1)}
+
\norm{\frac{\Delta}{n}(Q_1-Q_2)}
+\norm{(I+\frac{\Delta}{n}Q_2)-e^{Q_2\frac{\Delta}{n}}}\\
&\leq\frac{\Delta^2}{n^2}\norm{Q_1}^2
+
\frac{\Delta}{n}\norm{Q_1-Q_2}
+
\frac{\Delta^2}{n^2}\norm{Q_2}^2,
\end{align*}
implies that
\begin{equation*}
\norm{e^{Q_1\Delta}-e^{Q_2\Delta}}
\leq
\frac{\Delta^2}{n}\norm{Q_1}^2
+
\Delta\norm{Q_1-Q_2}
+
\frac{\Delta^2}{n}\norm{Q_2}^2.
\end{equation*}
The result now follows by taking the limit for $n$ going to infinity.
\end{proof}

\section{Proofs for the results in Section~\ref{sec:stochastic_processes}}\label{app:stoch_proc}

\begin{proof}[Proof of \ref{def:coh_prob_2b}-\ref{def:coh_prob_5}]
Consider any $A\in\power$ and $C\in\nonemptypower$. It then follows from \ref{def:coh_prob_1} and \ref{def:coh_prob_3} that
\vspace{-7pt}
\begin{equation}\label{eq:extracohprop1}
P(A\vert C)
=P(A\cup C\vert C)-P(C\setminus A\vert C)
=1-P(C\setminus A\vert C)
\end{equation}\\[-20pt]
and
\begin{equation}\label{eq:extracohprop2}
P(A\cap C\vert C)
=P(C\vert C)-P(C\setminus A\vert C)
=1-P(C\setminus A\vert C).
\vspace{6pt}
\end{equation}
\ref{def:coh_prob_2b} follows from Equation~\eqref{eq:extracohprop1} and~\ref{def:coh_prob_2}. \ref{def:coh_prob_7} follows from Equations~\eqref{eq:extracohprop1} and~\eqref{eq:extracohprop2}. \ref{def:coh_prob_8} follows from \ref{def:coh_prob_3}, by letting $B\coloneqq\emptyset$. \ref{def:coh_prob_5} follows trivially from \ref{def:coh_prob_1}.
\end{proof}

\corolcoherentextendable*
\begin{proof}
First assume that $P$ can be extended to a full conditional probability $P^*$. Theorem~\ref{theo:fullcoherent} then implies that $P^*$ is a coherent conditional probability, and therefore, since $P$ is the restriction of $P^*$ to $\mathcal{C}$, it clearly follows from Definition~\ref{def:coherence} that $P$ is a coherent conditional probability.

Conversely, if $P$ is a coherent conditional probability on $\mathcal{C}$, it follows from Theorem~\ref{theo:largerdomain} that $P$ can be extended to a coherent conditional probability $P^*$ on $\power\times\nonemptypower$, which, because of Theorem~\ref{theo:fullcoherent}, is a full conditional probability.
\end{proof}

\corolprocessiffrestriction*
\begin{proof}
Trivial consequence of Corollary~\ref{corol:coherentextendable}.
\end{proof}

\propstochasticprocesssimpleproperties*
\begin{proof}
The first part of the statement follows trivially from Corollary~\ref{corol:processiffrestriction} and Definitions~\ref{def:stoch_matrix} and \ref{def:cond_prob}. The second part is an immediate consequence of Definition~\ref{def:well-behaved} and Equation~\eqref{eq:normofmatrix}.
\end{proof}

\propboundednonemptyandclosed*
\begin{proof}
We only give the proof for $\smash{\overline{\partial}_{+}
{T^t_{t,\,x_u}}}$. The proof for $\smash{\overline{\partial}_{-}
{T^t_{t,\,x_u}}}$ is completely analogous. The proof for $\smash{\overline{\partial}
{T^t_{t,\,x_u}}}$ then follows trivially because a union of two bounded, non-empty and closed sets is always bounded, non-empty and closed itself.

We start by establishing the boundedness of $\smash{\overline{\partial}_{+}
{T^t_{t,\,x_u}}}$. Since $P$ is well-behaved, it follows from Proposition~\ref{prop:stochasticprocess:simpleproperties} that there is some $B>0$ and $\delta>0$ such that
\begin{equation}\label{eq:boundedbyB}
(\forall 0<\Delta<\delta)
~
\norm{\frac{1}{\Delta}
(T^{t+\Delta}_{t,\,x_u}-I)}
=
\frac{1}{\Delta}
\norm{
(T^{t+\Delta}_{t,\,x_u}-I)}
\leq B.
\end{equation}
Consider now any $Q\in\smash{\overline{\partial}_{+}
{T^t_{t,\,x_u}}}$. Because of Equation~\eqref{eq:rightouterderivative}, $Q$ is the limit of a sequence of matrices $\{Q_k\}_{k\in\nats}$, defined by
\begin{equation}\label{eq:sequenceofQsinproof}
Q_k\coloneqq\frac{1}{\Delta_k}
(T^{t+\Delta_k}_{t,\,x_u}-I)
\text{~~for all $k\in\nats$}.
\end{equation}
Because of Equation~\eqref{eq:boundedbyB}, the norms $\norm{Q_k}$ of these matrices are eventually (for large enough $k$) bounded above by $B$. Hence, it follows that $\norm{Q}\leq B$. Since this is true for any $Q\in\smash{\overline{\partial}_{+}
{T^t_{t,\,x_u}}}$, we find that $\smash{\overline{\partial}_{+}
{T^t_{t,\,x_u}}}$ is bounded.

In order to prove that $\smash{\overline{\partial}_{+}
{T^t_{t,\,x_u}}}$ is non-empty, we consider any sequence $\{\Delta_k\}_{k\in\nats}\to0^+$. The corresponding sequence of matrices $\{Q_k\}_{k\in\nats}$, as defined by Equation~\eqref{eq:sequenceofQsinproof}, is then bounded because $P$ is well-behaved---see Proposition~\ref{prop:stochasticprocess:simpleproperties}---and therefore, it follows from the Bolzano-Weierstrass theorem that it has a convergent subsequence $\{Q_{k_i}\}_{i\in\nats}$ of which we denote the limit by $Q^*$. Hence, we have found a sequence $\{\Delta_{k_i}\}_{i\in\nats}\to0^+$ such that $\{Q_{k_i}\}_{i\in\nats}\to Q^*$.
Since we know from Lemma~\ref{prop:rate_from_stochastic_matrix} that each of the matrices in $\{Q_{k_i}\}_{i\in\nats}$ is a rate matrix, the limit $Q^*$ is also a rate matrix, which therefore clearly belongs to $\smash{\overline{\partial}_{+}
{T^t_{t,\,x_u}}}$.

We end by showing that $\smash{\overline{\partial}_{+}
{T^t_{t,\,x_u}}}$ is closed, or equivalently, that for any converging sequence $\{Q^*_k\}_{k\in\nats}$, of rate matrices in $\smash{\overline{\partial}_{+}
{T^t_{t,\,x_u}}}$, the limit point $Q^*\coloneqq\lim_{k\to+\infty}Q^*_k$ is again an element of $\smash{\overline{\partial}_{+}
{T^t_{t,\,x_u}}}$. The argument goes as follows. First, since each of the rate matrices $Q^*_k$ belongs to the bounded set $\smash{\overline{\partial}_{+}
{T^t_{t,\,x_u}}}$, their limit $Q^*$ is a (real-valued) rate matrix. Next, for any $k\in\nats$, since $Q_k^*\in\smash{\overline{\partial}_{+}
{T^t_{t,\,x_u}}}$, it follows from Equation~\eqref{eq:rightouterderivative} that there is some $0<\Delta_k<\nicefrac{1}{k}$ such that $\norm{Q_k-Q^*_k}\leq\nicefrac{1}{k}$, with $Q_k$ defined as in Equation~\eqref{eq:sequenceofQsinproof}.
Consider now the sequences $\{Q_k\}_{k\in\nats}$ and $\{\Delta_k\}_{k\in\nats}$. Then on the one hand, we find that 
\begin{align*}
0
\leq
\limsup_{k\to+\infty}\norm{Q^*-Q_k}
&\leq
\limsup_{k\to+\infty}\norm{Q^*-Q^*_k}+\limsup_{k\to+\infty}\norm{Q^*_k-Q_k}\\
&=\limsup_{k\to+\infty}\norm{Q^*-Q^*_k}+\lim_{k\to+\infty}\nicefrac{1}{k}
=\limsup_{k\to+\infty}\norm{Q^*-Q^*_k}
=0,
\end{align*}
which implies that sequence $\{Q_k\}_{k\in\nats}$ converges to $Q^*$. On the other hand, we have that $\lim_{k\to+\infty}\Delta_k=0$. Hence, because of Definition~\ref{def:outerpartialderivatives}, it follows that $Q^*\in\smash{\overline{\partial}_{+}
{T^t_{t,\,x_u}}}$.
\end{proof}

\propouterderivativebehaveslikelimit*
\begin{proof}
Fix any $\epsilon>0$.
Assume \emph{ex absurdo} that
\begin{equation*}
(\forall\delta>0)(\exists0<\Delta<\delta)
(\forall Q\in\overline{\partial}_{+}
{T^t_{t,\,x_u}})
\norm{\frac{1}{\Delta}
(T^{t+\Delta}_{t,\,x_u}-I)-Q}\geq\epsilon.
\end{equation*}
Clearly, this implies the existence of a sequence $\{\Delta_k\}_{k\in\nats}\to0^+$ such that
\begin{equation}\label{eq:boundednonelement}
\norm{Q_k-Q}\geq\epsilon
\text{~~for all $k\in\nats$ and all $Q\in\overline{\partial}_{+}
{T^t_{t,\,x_u}}$},
\end{equation}
with $Q_k$ defined as in Equation~\eqref{eq:sequenceofQsinproof}. As we know from the proof of Proposition~\ref{prop:boundednon-emptyandclosed}, the sequence $\{Q_k\}_{k\in\nats}$ has a convergent subsequence $\{Q_{k_i}\}_{i\in\nats}$ of which the limit $Q^*$ belongs to $\overline{\partial}_{+}
{T^t_{t,\,x_u}}$. On the one hand, since $\lim_{i\to+\infty}Q_{k_i}=Q^*$, we now have that $\lim_{i\to+\infty}\norm{Q_{k_i}-Q^*}=0$. On the other hand, since $Q^*\in\overline{\partial}_{+}
{T^t_{t,\,x_u}}$, it follows from Equation~\eqref{eq:boundednonelement} that $\lim_{i\to+\infty}\norm{Q_{k_i}-Q^*}\geq\epsilon>0$. From this contradiction, it follows that there must be some $\delta_1>0$ such that Equation~\eqref{eq:outerderivativebehaveslikelimit1} holds for all $0<\Delta<\delta_1$. Similarly, using a completely analogous argument, we infer that if $t\neq0$, there must be some $\delta_2>0$ such that Equation~\eqref{eq:outerderivativebehaveslikelimit2} holds for all $0<\Delta<\delta_2$. Now let $\delta\coloneqq\min\{\delta_1,\delta_2\}$ if $t\neq0$ and let $\delta\coloneqq\delta_1$ if $t=0$.
\end{proof}

\coroloutersingleton*
\begin{proof}
This follows trivially from Proposition~\ref{prop:outerderivativebehaveslikelimit}.
\end{proof}

\section{Proofs and Lemmas for the results in Section~\ref{sec:cont_time_markov_chains}}

\propMarkovhassystem*
\begin{proof}
Consider any Markov chain $P\in\mprocesses$, with $\mathcal{T}_P$ its corresponding family of transition matrices. Then, because $P$ is a stochastic process, it follows from Proposition~\ref{prop:stochasticprocess:simpleproperties} that $T_t^t=I$ for all $t\in\realsnonneg$. 

Consider now any $t,r,s\in\realsnonneg$ with $t\leq r\leq s$. We need to show that $T_t^s=T_t^rT_r^s$. If $t=r$, we have that $T_t^s=T_t^rT_r^s=IT_r^s=T_t^s$, and hence the result follows trivially. Similarly, the claim is trivial for $r=s$. Hence, it remains to show that the claim holds for $t < r < s$. It follows from Definition~\ref{def:markov_property} that for all $x_t,x_r,x_s\in\states$,
\begin{equation*}
P(X_s=x_s\,\vert\,X_r=x_r,X_t=x_t) = P(X_s=x_s\,\vert\,X_r=x_r)\,.
\end{equation*}
Furthermore, because $P$ is a stochastic process, it follows from Corollary~\ref{corol:processiffrestriction} that $P$ satisfies~\ref{def:coh_prob_6} and~\ref{def:coh_prob_3} on its domain. From~\ref{def:coh_prob_6}, we infer that
\begin{align*}
P(X_s=x_s,X_r=x_r\,\vert\,X_t=x_t) &= P(X_s=x_s\,\vert\,X_r=x_r,X_t=x_t)P(X_r=x_r\,\vert\,X_t=x_t) \\
 &= P(X_s=x_s\,\vert\,X_r=x_r)P(X_r=x_r\,\vert\,X_t=x_t)\,,
\end{align*}
where the second equality used the Markov property. From~\ref{def:coh_prob_3}, we infer that
\begin{equation*}
P(X_s=x_s\,\vert\,X_t=x_t) = \sum_{x_r\in\states} P(X_s=x_s,X_r=x_r\,\vert\,X_t=x_t)\,.
\end{equation*}
From the definition of $\mathcal{T}_P$ in Definition~\ref{def:trans_matrix}, it now follows that, for any $x_t,x_s\in\states$,
\begin{align*}
T_t^s(x_t,x_s) &= P(X_s=x_s\,\vert\,X_t=x_t) = \sum_{x_r\in\states} P(X_s=x_s,X_r=x_r\,\vert\,X_t=x_t) \\
 &= \sum_{x_r\in\states} P(X_s=x_s\,\vert\,X_r=x_r)P(X_r=x_r\,\vert\,X_t=x_t) = \sum_{x_r\in\states} T_t^r(x_t,x_r) T_r^s(x_r,x_s)\,,
\end{align*}
and hence, by the rules of matrix multiplication, we find that $T_t^s=T_t^rT_r^s$. Therefore, and because the $t,r,s\in\realsnonneg$ were arbitrary, $\mathcal{T}_P$ is a transition matrix system.

The fact that $\mathcal{T}_P$ is well-behaved if and only if $P$ is well-behaved follows immediately from Definition~\ref{def:well_behaved_trans_mat_system} and Proposition~\ref{prop:stochasticprocess:simpleproperties} because the Markov property implies that $T_{t,\,x_u}^{t+\Delta}=T_{t}^{t+\Delta}$ and $T_{t-\Delta,\,x_u}^{t}=T_{t-\Delta}^{t}$.
\end{proof}

\theouniqueMarkovchain*
\begin{proof}
Let
\begin{multline*}
\mathcal{C}\coloneqq\{
(X_s=y,X_u=x_u)\in\mathcal{C}^{\mathrm{SP}}
\colon 
u\in\mathcal{U}_\emptyset,~s>u,~x_u\in\states_u,~y\in\states
\}\\
\cup
\{
(X_0=y,X_\emptyset=x_\emptyset)\in\mathcal{C}^{\mathrm{SP}}\colon y\in\states
\}
\end{multline*}
and consider a real-valued function $\tilde{P}$ on $\mathcal{C}$ that is defined by 
\begin{equation}\label{eq:theo:uniqueMarkovchain:prob_func}
\tilde{P}(X_s=y\vert X_u=x_u)
\coloneqq
\begin{cases}
p(y)&\text{~if $u=\emptyset$}\\
T_{\max u}^s(x_{\max u},y)&\text{~otherwise}
\end{cases}
\text{~~~for all $(X_s=y,X_u=x_u)\in\mathcal{C}$.}
\end{equation}

We first prove that $\tilde{P}$ is a coherent conditional probability on $\mathcal{C}$. So consider any $n\in\nats$ and, for all $i\in\{1,\dots,n\}$, choose $(A_i,C_i)=(X_{s_i}=y_i,X_{u_i}=x_{u_i})\in\mathcal{C}$ and $\lambda_i\in\reals$. We need to show that
\begin{equation}\label{eq:theo:uniqueMarkovchain:coh1}
\max\left\{\sum_{i=1}^n\lambda_i\ind{C_i}(\omega)\bigl(\tilde{P}(A_i\vert C_i)-\ind{A_i}(\omega)\bigr)~\Bigg\vert~\omega\in C_0\right\}\geq0,
\end{equation}
with $C_0\coloneqq\cup_{i=1}^nC_i$.
Since every sequence $u_i$ is finite, there is some finite set $w=\{w_0,w_1,\dots,w_m\}\subset\reals_{\geq0}$ of time points, with $m\in\nats$, such that $0=w_0<w_1<\dots<w_m$ and, for all $i\in\{1,\dots,n\}$, $u_i\subseteq w$ and $s_i\in w$.
Let $P_w$ be the restriction of $\tilde{P}$ to $\mathcal{C}_w$, with $\mathcal{C}_w$ defined as in Lemma~\ref{lemma:simplechaincoherence}. Then since $P_w$ clearly satisfies the conditions of Lemma~\ref{lemma:simplechaincoherence}, it follows that $P_w$ is a coherent conditional probability. Because of Theorem~\ref{theo:largerdomain}, this implies that $P_w$ can be extended to a coherent conditional probability $\tilde{P}_w$ on $\power\times\nonemptypower$, which, because of Theorem~\ref{theo:fullcoherent}, is also a full conditional probability. Since $\tilde{P}_w$ is a coherent conditional probability, it now follows from Definition~\ref{def:coherence} that
\begin{equation}\label{eq:theo:uniqueMarkovchain:coh2}
\max\left\{\sum_{i=1}^n\lambda_i\ind{C_i}(\omega)\bigl(\tilde{P}_w(A_i\vert C_i)-\ind{A_i}(\omega)\bigr)~\Bigg\vert~\omega\in C_0\right\}\geq0.
\end{equation}
By comparing Equations~\eqref{eq:theo:uniqueMarkovchain:coh1} and~\eqref{eq:theo:uniqueMarkovchain:coh2}, we see that in order to prove that $\tilde{P}$ is coherent, it suffices to show that, for all $i\in\{1,\dots,n\}$, $\tilde{P}_w(A_i\vert C_i)=\tilde{P}(A_i\vert C_i)$. So fix any $i\in\{1,\dots,n\}$. If $u_i=\emptyset$, then $s_i=0=w_0$ and therefore $(A_i,C_i)\in\mathcal{C}_w$, which implies that $\tilde{P}_w(A_i\vert C_i)=P_w(A_i\vert C_i)=\tilde{P}(A_i\vert C_i)$. If $u_i\neq\emptyset$, then since $u_i\subseteq w$, $s_i\in w$ and $s_i>u_i$, it follows from Lemma~\ref{lemma:simplechainextend} that $\tilde{P}_w(A_i\vert C_i)=\tilde{P}(A_i\vert C_i)$. Hence, $\tilde{P}$ is a coherent conditional probability on $\mathcal{C}$.

Therefore, due to Theorem~\ref{theo:largerdomain}, and because $\mathcal{C}\subseteq\mathcal{C}^\mathrm{SP}$, $\tilde{P}$ can be extended to a coherent conditional probability $P$ on $\mathcal{C}^\mathrm{SP}$, which, according to Definition~\ref{def:stoch_process}, is a stochastic process. Due to Equation~\eqref{eq:theo:uniqueMarkovchain:prob_func}, this implies that $P$ is a Markov chain such that $\mathcal{T}_P=\mathcal{T}$ and, for all $y\in\states$, $P(X_0=y)=p(y)$. Lemma~\ref{lemma:samepandTissameP} implies that this Markov chain is unique and, since $\mathcal{T}_P=\mathcal{T}$, Proposition~\ref{prop:Markovhassystem} implies that $\mathcal{T}$ is well-behaved if and only if $P$ is well-behaved.
\end{proof}

\begin{lemma}\label{lemma:simplechaincoherence}
Let $w=\{w_0,w_1,\dots,w_m\}\subset\reals_{\geq0}$ be a finite set of time points, with $m\in\nats_0$, such that $w_0<w_1<\dots<w_m$.
Let $P_w$ be a real-valued function on
\begin{equation*}
\mathcal{C}_w\coloneqq
\left\{
(X_{w_j}=y,X_u=x_u)
\colon 
j\in\{0,\dots,m\},~
u=\{w_0,\dots,w_{j-1}\},~
y\in\states,~
x_u\in\states_u
\right\}
\end{equation*}
such that, for any $j\in\{0,\dots,m\}$, $u=\{w_0,\dots,w_{j-1}\}$ and $x_u\in\states_u$, $P_w(X_{w_j}=x\,\vert X_u=x_u)$, as a function of $x\in\states$, is a probability mass function on $\states$. Then $P_w$ is a coherent conditional probability.
\end{lemma}
\begin{proof}
We provide a proof by induction. Assume that this statement is true for any $m'<m$---this is trivially the case for $m=0$. We will show that this implies that it is also true for $m$.

Consider any $n\in\nats$ and, for all $i\in\{1,\dots,n\}$, choose $(A_i,C_i)\in\mathcal{C}_w$ and $\lambda_i\in\reals$. We need to show that
\begin{equation}\label{eq:lemma:simplechaincoherence:TB}
\max\left\{\sum_{i=1}^n\lambda_i\ind{C_i}(\omega)\bigl(P_w(A_i\vert C_i)-\ind{A_i}(\omega)\bigr)~\Bigg\vert~\omega\in C_0\right\}\geq0,
\end{equation}
with $C_0\coloneqq\cup_{i=1}^nC_i$. 

For any $i\in\{1,\dots,n\}$, since $(A_i,C_i)\in\mathcal{C}_w$, there is some $j_i\in\{0,\dots,m\}$ and, for all $\ell\in\{0,\dots,j_i\}$, some $z_{\ell,i}\in\states$ such that
\begin{equation*}
A_i=(X_{w_{j_i}}=z_{j_i,i})
\text{~~and~~}
C_i=(X_{w_{0}}=z_{0,i}, \dots, X_{w_{j_i-1}}=z_{j_i-1,i}).
\end{equation*}
Let $S=\left\{i\in\{1,\dots,n\}\colon j_i<m\right\}$. If $S\neq\emptyset$, then by the induction hypothesis, we know that
\begin{equation*}
\max\left\{\sum_{i\in S}\lambda_i\ind{C_i}(\omega)\bigl(P_w(A_i\vert C_i)-\ind{A_i}(\omega)\bigr)~\Bigg\vert~\omega\in C_0^*\right\}\geq0,
\end{equation*}
with $C_0^*\coloneqq\cup_{i\in S}C_i$. It follows that there is some $\omega^*\in C_0^*\subseteq C_0$ such that
\begin{equation}\label{eq:lemma:simplechaincoherence:supremumreached}
\sum_{i\in S}\lambda_i\ind{C_i}(\omega^*)\bigl(P_w(A_i\vert C_i)-\ind{A_i}(\omega^*)\bigr)\geq0.
\end{equation}
If $S=\emptyset$, then let $\omega^*$ be any element of $C_0$. Equation~\eqref{eq:lemma:simplechaincoherence:supremumreached} is then trivially satisfied. Hence, in all cases, we have found some $\omega^*\in C_0$ that satisfies Equation~\eqref{eq:lemma:simplechaincoherence:supremumreached}.

Let $C^*\coloneqq\cap_{1\leq \ell<m}(X_{w_\ell}=\omega^*(w_\ell))$ and $S^*\coloneqq\{i\in\{1,\dots,n\}\colon C_i=C^*\}$. Then by the assumptions of this lemma, there is some probability mass function $p$ on $\states$ such that, for all $x\in\states$, $P_w(X_{w_m}=x\vert C^*)=p(x)$. 
For all $x\in\states$, let $\lambda_x\coloneqq\sum_{\{i\in S^*\colon z_{m,i}=x\}}\lambda_i$.
Since $p$ is a probability mass function, it then follows that
\begin{equation*}
\sum_{i\in S^*}\lambda_i P_w(A_i\vert C^*)
=
\sum_{x\in\states}
\lambda_x p(x)
\geq
\sum_{x\in\states} \left(\min_{y\in\states}\lambda_y\right) p(x)
=
\left(\min_{y\in\states}\lambda_y\right)\sum_{x\in\states}p(x)
=\min_{y\in\states}\lambda_y.
\end{equation*}
Now let $y^*$ be any element of $\states$ such that $\min_{y\in\states}\lambda_y=\lambda_{y^*}$ (since $\states$ is finite, this is always possible). Let $\omega^{**}$ be any path in $\Omega$ such that $\omega^{**}\in C^*$ and $\omega^{**}(w_m)=y^*$; Equation~\eqref{eq:path_exists_for_finite_points} guarantees that this $\omega^{**}\in\Omega$ exists. Then
\begin{equation*}
\sum_{i\in S^*}\lambda_i\bigl(P_w(A_i\vert C^*)-\ind{A_i}(\omega^{**})\bigr)
\geq
\min_{y\in\states}\lambda_y
-\sum_{i\in S^*}\lambda_i\ind{A_i}(\omega^{**})
=\lambda_{y^*}-\lambda_{y^*}=0,
\end{equation*}
where the first equality holds because, for every $i\in S^*$, $A_i=(X_{w_m}=z_{m,i})$.

Let $S^{**}\coloneqq\{1,\dots,n\}\setminus(S\cup S^*)$. Since $\omega^{**}\in C^*$, we find that $\ind{C_i}(\omega^{**})=\ind{C_i}(\omega^{*})$ and $\ind{A_i}(\omega^{**})=\ind{A_i}(\omega^{*})$ for all $i\in S$, that $\ind{C_i}(\omega^{**})=1$ for all $i\in S^{*}$, and that $\ind{C_i}(\omega^{**})=0$ for all $i\in S^{**}$. Hence, it follows from Equation~\eqref{eq:lemma:simplechaincoherence:supremumreached} that
\begin{equation*}
\sum_{i=1}^n\lambda_i\ind{C_i}(\omega^{**})\bigl(P_w(A_i\vert C_i)-\ind{A_i}(\omega^{**})\bigr)
\geq
\sum_{i\in S^*}\lambda_i\bigl(P_w(A_i\vert C^*)-\ind{A_i}(\omega^{**})\bigr).
\end{equation*}
By combining this inequality with the previous one, we find that in order to show that Equation~\eqref{eq:lemma:simplechaincoherence:TB} holds, it suffices to prove that $\omega^{**}\in C_0$. 

In order to prove this, it suffices to notice that the question of whether or not a path $\omega\in\Omega$ belongs to $C_0$, only depends on the values $\omega(t)$ of $\omega$ at time points $t\in\{w_0,\dots,w_{m-1}\}$. Indeed, since we infer from $\omega^{**}\in C^*$ that the value of $\omega^*$ and $\omega^{**}$ at these time points is the same, and because $\omega^*\in C_0$, this implies that $\omega^{**}\in C_0$.
\end{proof}

\begin{lemma}\label{lemma:simplechainextend}
Let $w=\{w_0,w_1,\dots,w_m\}\subset\reals_{\geq0}$ be a finite set of time points, with $m\in\nats_0$, such that $w_0<w_1<\dots<w_m$. Let $\mathcal{T}$ be a transition matrix system and let $\tilde{P}_w$ be any full conditional probability such that for all $j\in\{1,\dots,m\}$ and $x_{w_{\ell}}\in\states$, $\ell\in\{0,\dots,j\}$:
\begin{equation*}
\tilde{P}_w(X_{w_j}=x_{w_j}\vert X_{w_0}=x_{w_0},\dots,X_{w_{j-1}}=x_{w_{j-1}})=T_{w_{j-1}}^{w_j}(x_{w_{j-1}},x_{w_j}).
\end{equation*}
Then for any $s\in w$ and $u\subseteq w$ such that $s>u$ and $u\neq\emptyset$, any $y\in\states$ and any $x_u\in\states_u$, we have that
\begin{equation*}
\tilde{P}_w(X_s=y\vert X_u=x_u)
=T_{\max u}^s(x_{\max u},y).
\end{equation*}
\end{lemma}
\begin{proof}
Since $\emptyset\neq u\subseteq w$, $s\in w$ and $s>u$, it follows that there is some $j\in\{1,\dots,m\}$ such that $s=w_j$ and $u\subseteq\{w_0,\dots,w_{j-1}\}$.

We provide a proof by induction. If $s=w_1$, then $u=\{w_0\}$, and therefore, the result follows trivially from the assumptions in this lemma. Assume now that the result is true for $s=w_j$, with $1\leq j<m$. We will prove that this implies that it is also true for $s=w_{j+1}$. We consider two cases: $\max u=w_j$ and $\max u<w_j$.

If $\max u=w_j$, then with $v\coloneqq\{w_0,\dots,w_j\}\setminus u$:
\begin{align*}
\tilde{P}_w(X_{w_{j+1}}=y\vert X_u=x_u)
&=\sum_{z_{v}\in\states_{v}}\tilde{P}_w(X_{w_{j+1}}=y, X_{v}=z_{v}\vert X_u=x_u)\\
&=\sum_{z_{v}\in\states_{v}}
\tilde{P}_w(X_{w_{j+1}}=y\vert X_u=x_u, X_{v}=z_{v})
\tilde{P}_w(X_{v}=z_{v}\vert X_u=x_u)\\
&=\sum_{z_{v}\in\states_{v}}
T_{\max u}^{w_{j+1}}(x_{\max u},y)
\tilde{P}_w(X_{v}=z_{v}\vert X_u=x_u)\\[-1mm]
&\quad\quad\quad\quad\quad\quad
=T_{\max u}^{w_{j+1}}(x_{\max u},y)
\sum_{z_{v}\in\states_{v}}
\tilde{P}_w(X_{v}=z_{v}\vert X_u=x_u)\\[-1mm]
&\quad\quad\quad\quad\quad\quad\quad\quad\quad\quad\quad\quad\quad\quad\quad\quad~~~
=T_{\max u}^{w_{j+1}}(x_{\max u},y),
\end{align*}
where the first equality follows from~\ref{def:coh_prob_3}, the second equality follows from \ref{def:coh_prob_6}, the third equality follows from the assumptions in this lemma and the fact that $\max u=w_j$, and the last equality follows from \ref{def:coh_prob_3} and~\ref{def:coh_prob_5}.

If $\max u<w_j$, then with $v\coloneqq\{w_0,\dots,w_{j-1}\}\setminus u$:
\begin{align*}
&\tilde{P}_w(X_{w_{j+1}}=y\vert X_u=x_u)\\[1,5mm]
&=\sum_{z_{w_j}\in\states}
\sum_{z_{v}\in\states_{v}}
\tilde{P}_w(X_{w_{j+1}}=y, X_{w_j}=z_{w_j}, X_v=z_v\vert X_u=x_u)\\
&=\sum_{z_{w_j}\in\states}
\sum_{z_{v}\in\states_{v}}
\tilde{P}_w(X_{w_{j+1}}=y\vert X_u=x_u, X_{w_j}=z_{w_j}, X_v=z_v)\\[-4mm]
&\quad\quad\quad\quad\quad\quad\quad\quad\quad\quad\quad~\,
\tilde{P}_w(X_v=z_v\vert X_u=x_u, X_{w_j}=z_{w_j})
\tilde{P}_w(X_{w_j}=z_{w_j}\vert X_u=x_u)\\[4mm]
&=\sum_{z_{w_j}\in\states}
\sum_{z_{v}\in\states_{v}}
T_{w_j}^{w_{j+1}}(z_{w_j},y)
\tilde{P}_w(X_v=z_v\vert X_u=x_u, X_{w_j}=z_{w_j})
\tilde{P}_w(X_{w_j}=z_{w_j}\vert X_u=x_u)\\
&=\sum_{z_{w_j}\in\states}
\sum_{z_{v}\in\states_{v}}
T_{w_j}^{w_{j+1}}(z_{w_j},y)
\tilde{P}_w(X_v=z_v\vert X_u=x_u, X_{w_j}=z_{w_j})
T_{\max u}^{w_{j}}(x_{\max u},z_{w_j})\\
&=\sum_{z_{w_j}\in\states}
T_{w_j}^{w_{j+1}}(z_{w_j},y)
T_{\max u}^{w_{j}}(x_{\max u},z_{w_j})
\sum_{z_{v}\in\states_{v}}
\tilde{P}_w(X_v=z_v\vert X_u=x_u, X_{w_j}=z_{w_j})
\\
&=\sum_{z_{w_j}\in\states}
T_{w_j}^{w_{j+1}}(z_{w_j},y)
T_{\max u}^{w_{j}}(x_{\max u},z_{w_j})
=T_{\max u}^{w_{j+1}}(x_{\max u},y),
\end{align*}
where the first equality follows from~\ref{def:coh_prob_3}, the second equality follows from \ref{def:coh_prob_6}, the third equality follows from the assumptions in this lemma, the fourth equality follows from the induction hypothesis, the sixth equality follows from \ref{def:coh_prob_3} and~\ref{def:coh_prob_5}, and the last equality follows from Equation~\eqref{eq:transmatrixproduct}.
\end{proof}

\begin{lemma}\label{lemma:samepandTissameP}
Consider two Markov chains $P_1,P_2\in\mprocesses$ such that $\mathcal{T}_{P_1}=\mathcal{T}_{P_2}$ and, for all $y\in\states$, $P_1(X_0=y)=P_2(X_0=y)$. Then $P_1=P_2$.
\end{lemma}
\begin{proof}
Let $\mathcal{T}\coloneqq\mathcal{T}_{P_1}=\mathcal{T}_{P_2}$ be the common transition matrix system of $P_1$ and $P_2$ and let $p$ be their common initial probability mass function, as defined by $p(y)\coloneqq P_1(X_0=y)=P_2(X_0=y)$ for all $y\in\states$. Let $\tilde{P}$ be a real-valued function on $\mathcal{C}$, with $\mathcal{C}$ and $\tilde{P}$ defined as in the proof of Theorem~\ref{theo:uniqueMarkovchain}. It then follows from Definition~\ref{def:markov_property} that the restriction of $P_1$ and $P_2$ to $\mathcal{C}$ is equal to $\tilde{P}$. Furthermore, for any $s>0$, $y\in\states$ and $j\in\{1,2\}$, we find that
\begin{align*}
P_j(X_s=y)
=\sum_{x\in\states}P_j(X_s=y, X_0=x)
&=\sum_{x\in\states}P_j(X_s=y\vert X_0=x)P_j(X_0=x)\\
&=\sum_{x\in\states}\tilde{P}(X_s=y\vert X_0=x)\tilde{P}(X_0=x).
\end{align*}
Hence, the restrictions of $P_1$ and $P_2$ to
\begin{align*}
\mathcal{C}^*
\coloneqq&\mathcal{C}\cup
\{
(X_s=y,X_\emptyset=x_\emptyset)
\colon 
s\in\reals_{>0},~y\in\states
\}\\
=&\{
(X_s=y,X_u=x_u)
\colon 
u\in\mathcal{U},~s\in\reals_{\geq0},~s>u,~x_u\in\states_u,~y\in\states
\}
\end{align*}
are identical. We denote this common restriction by $\tilde{P}^*$.

Consider now any $(A,X_u=x_u)\in\mathcal{C}^\mathrm{SP}$. Then since $A\in\mathcal{A}_u$, there is some finite set of time points $v=\{v_1,v_2,\dots,v_n\}\subseteq\reals_{\geq0}$, with $n\in\nats$, such that $\max u<v_1<v_2<\dots<v_n$, and some set $S\subseteq\states_{u\cup v}$ such that $A=\cup_{z_{u\cup v}\in S}(X_{u\cup v}=z_{u\cup v})$. 
Let $S_v\coloneqq\{z_v\in\states_v\colon (x_u,z_v)\in S\}$.
For any $j\in\{1,2\}$, we then find that
\begin{align*}
P_j(A\vert X_u=x_u)
&=\sum_{z_{u\cup v}\in S}
P_j(X_{u\cup v}=z_{u\cup v}\vert X_u=x_u)\\
&=\sum_{z_{v}\in S_v}
P_j(X_{v_1}=z_{v_1}, X_{v_2}=z_{v_2}, \dots, X_{v_n}=z_{v_n}\vert X_u=x_u)\\[-1mm]
&=\sum_{z_{v}\in S_v}
\,\,
\prod_{i=1}^n
\,\,
P_j(X_{v_i}=z_{v_i}\vert X_u=x_u, X_{v_1}=z_{v_1}, \dots, X_{v_{i-1}}=z_{v_{i-1}})\\
&=\sum_{z_{v}\in S_v}
\,\,
\prod_{i=1}^n
\,\,
\tilde{P}^*(X_{v_i}=z_{v_i}\vert X_u=x_u, X_{v_1}=z_{v_1}, \dots, X_{v_{i-1}}=z_{v_{i-1}}),
\end{align*}
which implies that $P_1(A\vert X_u=x_u)=P_2(A\vert X_u=x_u)$. Since this is the case for any $(A,X_u=x_u)\in\mathcal{C}^\mathrm{SP}$, it follows that $P_1=P_2$.
\end{proof}

\corratehasuniquehomogenmarkovprocess*
\begin{proof}
Since we know from Proposition~\ref{prop:systemQ} that $\mathcal{T}_Q$ is a well-behaved transition matrix system, it follows from Theorem~\ref{theo:uniqueMarkovchain} that there is a unique Markov chain $P\in\mprocesses$ such that $\mathcal{T}_P=\mathcal{T}_Q$ and, for all $y\in\mathcal{X}$, $P(X_0=y)=p(y)$, and that this Markov chain is furthermore well-behaved. Since it---trivially---follows from Definition~\ref{def:systemfromQ} that $\mathcal{T}_Q$ satisfies Equation~\eqref{eq:homogeneousMarkov}, Definition~\ref{def:homogeneousMarkov} implies that $P$ is homogeneous.
\end{proof}

\theohomogeneoushasQ*
\begin{proof}
Because of Proposition~\ref{prop:boundednon-emptyandclosed}, we know that $\overline{\partial}_{+}
{T^0_{0}}$ is a non-empty bounded set of rate matrices, which implies that there is some real $B>0$ such that $\norm{Q'}\leq B$ for all $Q'\in\overline{\partial}_{+}
{T^0_{0}}$. Let $Q$ be any element of $\overline{\partial}_{+}
{T^0_{0}}$.

Fix any $c\geq0$, $\epsilon>0$ and $\delta>0$. 
It then follows from Proposition~\ref{prop:outerderivativebehaveslikelimit} and~\ref{N:homogeneous} that there is some $\delta^*>0$ such that
\begin{equation}
\label{eq:homogeneoushasQ1}
(\forall 0<\Delta^*<\delta^*)
~
(\exists Q^*\in\overline{\partial}_{+}
{T^0_{0}})
~
\norm{T_0^{\Delta^*}-(I+\Delta^*Q^*)}<\Delta^*\epsilon.
\end{equation}
Furthermore, because of Equation~\eqref{eq:rightouterderivative} and~\ref{N:homogeneous}, there is some $0<\Delta<\min\{\delta,\delta^*\}$ such that
\begin{equation}
\label{eq:homogeneoushasQ2}
\norm{T^{\Delta}_{0}-(I+\Delta Q)}<\Delta\epsilon.
\end{equation}
If we now define $n\coloneqq\lfloor\nicefrac{c}{\Delta}\rfloor$ and $d\coloneqq c-n\Delta$, then $n\Delta\leq c<(n+1)\Delta$ and therefore also $0\leq d<\Delta$. Because of Proposition~\ref{prop:Markovhassystem}, Equation~\eqref{eq:transmatrixproduct} and Definition~\ref{def:homogeneousMarkov}, we know that
\begin{equation*}
T_0^c=\left(
\prod_{j=1}^{n}
T_{(j-1)\Delta}^{j\Delta}
\right)
T_{n\Delta}^c
=\left(T_0^{\Delta}\right)^{n}
T_0^{d}
\end{equation*}
and therefore, it follows from Lemma~\ref{lemma:recursive} that
\begin{equation}
\label{eq:homogeneoushasQ3}
\norm{
	e^{Qc}-T_0^c
}
=
\norm{
\left(T_0^{\Delta}\right)^{n}
T_0^{d}
-
\left(
e^{Q\Delta}
\right)^{n}
e^{Qd}
}
\leq
n\norm{T_0^{\Delta}-e^{Q\Delta}}
+\norm{T_0^{d}-e^{Qd}}.
\end{equation}
From Equation~\eqref{eq:homogeneoushasQ2} and Lemma~\ref{lemma:linearpartofexponential}, we infer that
\begin{equation}
\label{eq:homogeneoushasQ4}
\norm{T_0^{\Delta}-e^{Q\Delta}}
\leq
\norm{T_0^{\Delta}-(I+\Delta Q)}
+
\norm{(I+\Delta Q)-e^{Q\Delta}}
\leq
\Delta\epsilon
+
\Delta^2\norm{Q}^2.
\end{equation}
Since $d<\Delta<\delta^*$, we infer from Equation~\eqref{eq:homogeneoushasQ1} that there is some $Q^*\in\overline{\partial}_{+}
{T^0_{0}}$ such that $\norm{T_0^{d}-(I+d Q^*)}<d\epsilon$. Hence, also using Lemma~\ref{lemma:linearpartofexponential}, we find that
\begin{align}
\norm{T_0^{d}-e^{Qd}}
&\leq
\norm{T_0^{d}-(I+d Q^*)}
+
\norm{(I+d Q^*)-(I+d Q)}
+
\norm{(I+d Q)-e^{Qd}}\notag\\
&\leq
d\epsilon+d\norm{Q^*-Q}
+d^2\norm{Q}^2
\leq
d\epsilon+d\norm{Q^*}
+d\norm{Q}
+d^2\norm{Q}^2.\label{eq:homogeneoushasQ5}
\end{align}
By combining Equations~\eqref{eq:homogeneoushasQ3}, \eqref{eq:homogeneoushasQ4} and~\eqref{eq:homogeneoushasQ5}, it follows that
\begin{equation*}
\norm{
	e^{Qc}-T_0^c
}
\leq
n\Delta\epsilon
+
n\Delta^2\norm{Q}^2
+
d\epsilon
+d\norm{Q^*}
+d\norm{Q}
+d^2\norm{Q}^2.
\end{equation*}
Taking into account that $\norm{Q}\leq B$, $\norm{Q^*}\leq B$, $n\Delta\leq c$ and $d<\Delta<\delta$, this implies that
\begin{equation*}
\norm{
	e^{Qc}-T_0^c
}
\leq
c\epsilon
+
c\delta B^2
+
\delta\epsilon
+2\delta B
+\delta^2 B^2.
\end{equation*}
Since this is true for any $\epsilon>0$ and $\delta>0$, it follows that $\norm{e^{Qc}-T_0^c}\leq0$, which implies that $T_0^c=e^{Qc}$. Since this is true for all $c\geq0$, it follows from Definition~\ref{def:homogeneousMarkov} that
\begin{equation}\label{eq:homogeneoushasQ6}
T_t^s=T_0^{s-t}=e^{Q(s-t)}
\text{~~for all $0\leq t\leq s$,}
\end{equation}
or equivalently, that $\mathcal{T}_P=\mathcal{T}_Q$.

Finally, we prove that $Q$ is unique. Assume \emph{ex absurdo} that this is not the case, or equivalently, that there are rate matrices $Q_1$ and $Q_2$, with $Q_1\neq Q_2$, such that $\mathcal{T}_P=\mathcal{T}_{Q_1}$ and $\mathcal{T}_P=\mathcal{T}_{Q_2}$. For all $\Delta>0$, we then have that $T_0^\Delta=e^{Q_1\Delta}=e^{Q_2\Delta}$, and therefore, it follows from Lemma~\ref{lemma:deriv_exponential_trans} that $\partial_+T_0^0=Q_1$ and $\partial_+T_0^0=Q_2$, which implies that $Q_1=Q_2$. From this contradiction, it follows that $Q$ is indeed unique.
\end{proof}

\propQissingletonderivforhomogen*
\begin{proof}
The result about the partial derivatives is an immediate consequence of Lemma~\ref{lemma:deriv_exponential_trans} and the fact that $T^{t+\Delta}_{t}=T^{t}_{t-\Delta}=e^{Q\Delta}$. The result about the outer partial derivatives then follows from Corollary~\ref{corol:outersingleton}.
\end{proof}

\section{Proofs for the results in Section~\ref{sec:iCTMC}}

\propmarkovsetsubsetofnonmarkovset*
\begin{proof}This is immediate from Definitions~\ref{def:consistent_process_set} and~\ref{def:process_sets} and the fact that $\whmprocesses\subseteq\wmprocesses\subseteq\wprocesses$.
\end{proof}

\propnonhomogeneousinprocessset*
\begin{proof}
Proposition~\ref{prop:finite_different_rate_matrix_has_process} implies the existence of a process $P\in\wmprocesses$ such that $P(X_0=y)$ for all $y\in\mathcal{X}$ and such that $\mathcal{T}_P$ satisfies Equation~\eqref{eq:nonhomogen_in_process_set_system_composition}. It remains to show that $P\in\wmprocesses_{\rateset,\,\mathcal{M}}$. Because of Equation~\eqref{eq:consistentwmprocessesalternative}, and since we already know that $p\in\mathcal{M}$, this means that we have to show that $\smash{\overline{\partial}}T_t^t\subseteq\rateset$ for all $t\in\realsnonneg$. To this end, consider any $t\in\realsnonneg$.

We consider several cases. If $t<t_0$, then $\smash{\overline{\partial}}T_t^t$ corresponds to $\mathcal{T}_{Q_0}^{[0,t_0]}$, and it then follows from Proposition~\ref{prop:Q_is_singleton_deriv_for_homogen} that $\smash{\overline{\partial}}T_t^t=\{Q_0\}\subseteq\rateset$. If $t>t_n$, then $\smash{\overline{\partial}}T_t^t$ corresponds to $\smash{\mathcal{T}_{Q_{n+1}}^{[t_n,\infty)}}$, in which case $\smash{\overline{\partial}}T_t^t=\{Q_{n+1}\}\subseteq\rateset$. Similarly, if there is some $i\in\{1,\ldots,n\}$ such that $t\in(t_{i-1},t_i)$, then $\smash{\overline{\partial}}T_t^t$ corresponds to $\smash{\mathcal{T}_{Q_i}^{[t_{i-1},t_i]}}$, and therefore $\smash{\overline{\partial}}T_t^t=\{Q_i\}\subseteq\rateset$. 
The only remaining case is when $t=t_i\neq0$ for some $i\in\{0,\ldots,n\}$. In this case, we have that $\smash{\overline{\partial}}_+T_t^t=\{Q_{i+1}\}$ and, if $t\neq0$, that $\smash{\overline{\partial}}_-T_t^t=\{Q_i\}$, and therefore, it follows from Definition~\ref{def:direc_partial_deriv} that $\smash{\overline{\partial}}T_t^t\subseteq\rateset$.
\end{proof}

\theoaanelkaarplakken*
\begin{proof}
This proof is rather lengthy, and consists of two parts. First, we will show that there exists a stochastic process $P$ that satisfies Equations~\eqref{eq:theo:aanelkaarplakken:equalsfirst} and~\eqref{eq:theo:aanelkaarplakken:equalssecond}, by constructing it as the extension of a coherent conditional probability on a set of events $\mathcal{C}\subset\mathcal{C}^\mathrm{SP}$. Next, we will finish the proof by showing that $P\in\wprocesses_{\rateset,\,\mathcal{M}}$, as desired.

Let $\mathcal{C}\coloneqq\mathcal{C}_\emptyset\cup(\bigcup_{x_u\in\states_u}\mathcal{C}_{x_u})$, with
\begin{multline}\label{eq:theo:aanelkaarplakken:firstpartofdomain}
\mathcal{C}_\emptyset\coloneqq
\{(A,X_v=x_v)\in\mathcal{C}^{\mathrm{SP}}\colon v\in\mathcal{U}_{< \max u}\text{~and~}\\A\in\left\langle
\left\{
(X_t=x)
\colon
x\in\states,t\in[0,\max u]
\right\}
\right\rangle\}
\end{multline}
and, for all $x_u\in\states_u$,
\begin{equation}\label{eq:theo:aanelkaarplakken:secondpartofdomain}
\mathcal{C}_{x_u}\coloneqq\{
(A,X_v=x_v)\in\mathcal{C}^\mathrm{SP}
\colon
u\subseteq v\in\mathcal{U},\,
x_{v\setminus u}\in\states_{v\setminus u},\,
 A\in\mathcal{A}_{u\cup(v\setminus[0,\max u])}
\}
\end{equation}
Consider a real-valued function $\tilde{P}$ on $\mathcal{C}$ that is defined, for all $(A,X_v=x_v)\in\mathcal{C}$, by
\begin{equation}\label{eq:theo:aanelkaarplakken:defPtilde}
\tilde{P}(A\vert X_v=x_v)
\coloneqq
\begin{cases}
P_\emptyset(A\vert X_v=x_v)&\text{~if $(A,X_v=x_v)\in\mathcal{C}_\emptyset$}\\
P_{x_u}(A\vert 
X_{u\cup(v\setminus[0,\max u])}=x_{u\cup(v\setminus[0,\max u])})&\text{~if $(A,X_v=x_v)\in\mathcal{C}_{x_u}$}\,.
\end{cases}
\end{equation}

We first prove that $\tilde{P}$ is a coherent conditional probability on $\mathcal{C}$. So consider any $n\in\nats$ and, for all $i\in\{1,\dots,n\}$, choose $(A_i,C_i)\in\mathcal{C}$ and $\lambda_i\in\reals$. We need to show that
\begin{equation}\label{eq:theo:aanelkaarplakken:coh}
\max\left\{\sum_{i=1}^n\lambda_i\ind{C_i}(\omega)\bigl(\tilde{P}(A_i\vert C_i)-\ind{A_i}(\omega)\bigr)~\Bigg\vert~\omega\in C_0\right\}\geq0,
\end{equation}
with $C_0\coloneqq\cup_{i=1}^nC_i$.

Let $S^*\coloneqq\left\{i\in\{1,\dots,n\}\colon(A_i,C_i)\in\mathcal{C}_\emptyset\right\}$. We first consider the case $S^*\neq\emptyset$. Then since $P_\emptyset$ is a stochastic process, it follows from Equation~\eqref{eq:theo:aanelkaarplakken:defPtilde} and Definitions~\ref{def:stoch_process} and~\ref{def:coherence} that
\begin{equation*}
\max\left\{\sum_{i\in S^*}\lambda_i\ind{C_i}(\omega)\bigl(\tilde{P}(A_i\vert C_i)-\ind{A_i}(\omega)\bigr)~\Bigg\vert~\omega\in C^*\right\}\geq0,
\end{equation*}
with $C^*\coloneqq\cup_{i\in S^*}C_i$. Therefore, there is some $\omega^*\in C^*$ such that
\begin{equation}\label{eq:theo:aanelkaarplakken:geqfirstpart}
\sum_{i\in S^*}\lambda_i\ind{C_i}(\omega^*)\bigl(\tilde{P}(A_i\vert C_i)-\ind{A_i}(\omega^*)\bigr)\geq0.
\end{equation}
If $S^*=\emptyset$, we let $\omega^*$ be any element of $C_0$ (this is always possible, because $C_0\neq\emptyset$). Clearly, this path $\omega^*$ will then also satisfy Equation~\eqref{eq:theo:aanelkaarplakken:geqfirstpart}---because the left-hand side is a sum over an empty set and therefore zero.

Now let $x_u^*\in\states_u$ be defined by $x_u^*\coloneqq \omega^*\vert_u$.
Then for all $i\in\{1,\dots,n\}$ such that $(A_i,C_i)\in\mathcal{C}_{x_u^*}$, we know from Equation~\eqref{eq:theo:aanelkaarplakken:secondpartofdomain} that there are $u\subseteq v_i\in\mathcal{U}$ and $x_{v_i\setminus u}\in\states_{v_i\setminus u}$ such that
\begin{equation}\label{eq:theo:aanelkaarplakken:Cisplit}
C_i=(X_u=x_u^*)\cap(X_{v_i\setminus u}=x_{v_i\setminus u})=C_i^*\cap C_i^{**},\vspace{2mm}
\end{equation}
with $C_i^{*}\coloneqq
(X_{(v_i\setminus u)\cap [0,\max u]}=x_{(v_i\setminus u)\cap [0,\max u]})$, and
\begin{equation}\label{eq:theo:aanelkaarplakken:Cistarstar}
C_i^{**}\coloneqq(X_u=x_u^*)\cap(X_{v_i\setminus [0,\max u]}=x_{v_i\setminus [0,\max u]})\,.
\end{equation}
Using this notation, we define
\begin{equation}\label{eq:theo:aanelkaarplakken:Sstarstardef}
S^{**}\coloneqq
\{
i\in\{1,\dots,n\}
\colon
(A_i,C_i)\in\mathcal{C}_{x_u^*}
\text{~and~}
\ind{C_i^*}(\omega^*)=1
\}.
\end{equation}

We first consider the case $S^{**}\neq\emptyset$. Since $P_{x_u^*}$ is a stochastic process, it then follows from Definitions~\ref{def:stoch_process} and~\ref{def:coherence} that
\begin{equation*}
\max\left\{\sum_{i\in S^{**}}\lambda_i\ind{C_i^{**}}(\omega)\bigl(P_{x_u^*}(A_i\vert C_i^{**})-\ind{A_i}(\omega)\bigr)~\Bigg\vert~\omega\in C^{**}\right\}\geq0,
\end{equation*}
with $C^{**}\coloneqq\cup_{i\in S^{**}}C_i^{**}$. Because of Equation~\eqref{eq:theo:aanelkaarplakken:defPtilde}, this implies that
\begin{equation*}
\max\left\{\sum_{i\in S^{**}}\lambda_i\ind{C_i^{**}}(\omega)\bigl(\tilde{P}(A_i\vert C_i)-\ind{A_i}(\omega)\bigr)~\Bigg\vert~\omega\in C^{**}\right\}\geq0,
\end{equation*}
which allows us to infer that there is some $\omega^{**}\in C^{**}$ such that
\begin{equation}\label{eq:theo:aanelkaarplakken:geqsecondpart}
\sum_{i\in S^{**}}\lambda_i\ind{C_i^{**}}(\omega^{**})\bigl(\tilde{P}(A_i\vert C_i)-\ind{A_i}(\omega^{**})\bigr)\geq0.
\end{equation}
Furthermore, since $\omega^{**}\in C^{**}$, Equation~\eqref{eq:theo:aanelkaarplakken:Cistarstar} implies that
\begin{equation}\label{eq:theo:aanelkaarplakken:starstarniceonu}
\omega^{**}\vert_u=x_u^* = \omega^{*}\vert_u\,.
\end{equation}
If $S^{**}=\emptyset$, we let $\omega^{**}=\omega^{*}$. Clearly, also in this case, $\omega^{**}$ satisfies Equations~\eqref{eq:theo:aanelkaarplakken:geqsecondpart} and~\eqref{eq:theo:aanelkaarplakken:starstarniceonu}.

For any $i\in\{1,\ldots,n\}$, because $(A_i,C_i)\in\mathcal{C}$, there exists some finite sequence of time points $w_{C_i}\in\mathcal{U}$ such that $C_i$ only depends on the time points in $w_{C_i}$. Furthermore, it follows from Equations~\eqref{eq:theo:aanelkaarplakken:firstpartofdomain} and~\eqref{eq:theo:aanelkaarplakken:secondpartofdomain} that $A_i$ is an element of some algebra $\mathcal{A}$ that is generated by a set of events that only depend on a finite number of time points. Therefore, there is also some finite sequence of time points $w_{A_i}\in\mathcal{U}$, such that $A_i$ only depends on the time points in $w_{A_i}$. If we now let $w_i\coloneqq w_{A_i}\cup w_{C_i}$, then $(A_i,C_i)$ only depends on the (finite) sequence of time points $w_i$.

Because this holds for any $i\in\{1,\ldots,n\}$, this implies the existence of some finite sequence $w\in\mathcal{U}$ such that $w_i\subseteq w$ for all $i\in\{1,\ldots,n\}$.

Now let $\omega^{***}\in\Omega$ be any path such that, for all $s\in w$,
\begin{equation*}
\omega^{***}(s)\coloneqq
\begin{cases}
\omega^{*}(s) & \text{if $s<\max u$}\\
\omega^{**}(s) & \text{if $s\geq \max u$}
\end{cases}
\end{equation*}
Equation~\eqref{eq:path_exists_for_finite_points} guarantees that this $\omega^{***}\in\Omega$ exists. Furthermore, because of Equation~\eqref{eq:theo:aanelkaarplakken:starstarniceonu}, we know that, for all $s\in w$,
\begin{equation}\label{eq:theo:aanelkaarplakken:triplestarpartone}
\omega^{***}(s)=\omega^*(s)
\text{~~if $s\in [0,\max u]$}\vspace{-3mm}
\end{equation}
and
\begin{equation}\label{eq:theo:aanelkaarplakken:triplestarparttwo}
\omega^{***}(s)=\omega^{**}(s)
\text{~~if $s\in u\cup [\max u,+\infty)$}
\vspace{2mm}
\end{equation}
and therefore, it follows from Equation~\eqref{eq:theo:aanelkaarplakken:Cisplit} that
\begin{equation}\label{eq:theo:aanelkaarplakken:triplestarequivalence}
\omega^{***}\in C_i
\Leftrightarrow
(\omega^{***}\in C_i^*
\text{~and~}
\omega^{***}\in C_i^{**})
\Leftrightarrow
(\omega^{*}\in C_i^*
\text{~and~}
\omega^{**}\in C_i^{**})
\end{equation}
for all $i\in\{1,\dots,n\}$ such that $(A_i,C_i)\in\mathcal{C}_{x_u^*}$.

Next, for any $i\in S^*$, we infer from Equation~\eqref{eq:theo:aanelkaarplakken:firstpartofdomain} that the value of $\ind{A_i}(\omega^{***})$ and $\ind{C_i}(\omega^{***})$ is completely determined by $\omega^{***}(t)$, $t\in(w\cap[0,\max u])$. Therefore, it follows from Equations~\eqref{eq:theo:aanelkaarplakken:geqfirstpart} and ~\eqref{eq:theo:aanelkaarplakken:triplestarpartone} that 
\begin{equation}\label{eq:theo:aanelkaarplakken:geqfirstparttriplestar}
\sum_{i\in S^*}\lambda_i\ind{C_i}(\omega^{***})\bigl(\tilde{P}(A_i\vert C_i)-\ind{A_i}(\omega^{***})\bigr)
\geq0.
\end{equation}

Similarly, for any $i\in S^{**}$, 
Equations~\eqref{eq:theo:aanelkaarplakken:triplestarequivalence} and~\eqref{eq:theo:aanelkaarplakken:Sstarstardef} imply that $\ind{C_i}(\omega^{***})=\ind{C_i^{**}}(\omega^{**})$,
and Equations~\eqref{eq:theo:aanelkaarplakken:secondpartofdomain} and~\eqref{eq:theo:aanelkaarplakken:triplestarparttwo} imply that $\ind{A_i}(\omega^{***})=\ind{A_i}(\omega^{**})$. Therefore, it follows from Equation~\eqref{eq:theo:aanelkaarplakken:geqsecondpart} that
\begin{equation}\label{eq:theo:aanelkaarplakken:geqsecondparttriplestar}
\sum_{i\in S^{**}}\lambda_i\ind{C_i}(\omega^{***})\bigl(\tilde{P}(A_i\vert C_i)-\ind{A_i}(\omega^{***})\bigr)
\geq0.
\end{equation}

Consider now any $i\in\{1,\dots,n\}$ such that $i\notin S^*$ and $i\notin S^{**}$. Since $i\notin S^*$, there is some $x_u\in\states_u$ such that $(A_i,C_i)\in\mathcal{C}_{x_u}$. If $x_u= x_u^*$, then since $i\notin S^{**}$, it follows from Equation~\eqref{eq:theo:aanelkaarplakken:Sstarstardef} that $\ind{C_i^*}(\omega^*)=0$, and therefore, Equation~\eqref{eq:theo:aanelkaarplakken:triplestarequivalence} implies that $\ind{C_i}(\omega^{***})=0$. 
If $x_u\neq x_u^*$, then $(X_u=x_u)\cap(X_u=x_u^*)=\emptyset$, and therefore, since $(A_i,C_i)\in\mathcal{C}_{x_u}$ implies that $C_i\subseteq (X_u=x_u)$, it follows that $C_i\cap (X_u=x_u^*)=\emptyset$. Since it follows from Equations~\eqref{eq:theo:aanelkaarplakken:starstarniceonu} and~\eqref{eq:theo:aanelkaarplakken:triplestarparttwo} that $\omega^{***}(t)=x_t^*$ for all $t\in u$, this implies that $\omega^{***}\notin C_i$, and therefore, we find that $\ind{C_i}(\omega^{***})=0$.
Hence, in all cases, we find that $\ind{C_i}(\omega^{***})=0$. Since this is true for any $i\in\{1,\dots,n\}$ such that $i\notin S^*$ and $i\notin S^{**}$, it follows from Equations~\eqref{eq:theo:aanelkaarplakken:geqfirstparttriplestar} and~\eqref{eq:theo:aanelkaarplakken:geqsecondparttriplestar} that
\begin{equation}\label{eq:theo:aanelkaarplakken:geqtotal}
\sum_{i=1}^n\lambda_i\ind{C_i}(\omega^{***})\bigl(\tilde{P}(A_i\vert C_i)-\ind{A_i}(\omega^{***})\bigr)\geq0.
\end{equation}

We will now prove that $\omega^{***}\in C_0$. We consider two cases: $S^*\neq\emptyset$ and $S^*=\emptyset$. First assume that $S^*\neq\emptyset$. In that case, we have that $\omega^*\in C^*$, which implies that there is some $i\in S^*$ such that $\omega^*\in C_i$. It then follows from Equations~\eqref{eq:theo:aanelkaarplakken:firstpartofdomain} and~\eqref{eq:theo:aanelkaarplakken:triplestarpartone} that $\omega^{***}\in C_i\subseteq C_0$. 
Next, assume that $S^*=\emptyset$. In that case, we have that $\omega^*\in C_0$, which implies that there is some $i\in\{1,\dots,n\}$ such that $\omega^*\in C_i$. Since $(A_i,C_i)\in\mathcal{C}$ and $S^*=\emptyset$, there is some $x_u\in\states_u$ such that $(A_i,C_i)\in\mathcal{C}_{x_u}$ and, since Equation~\eqref{eq:theo:aanelkaarplakken:secondpartofdomain} implies that $x_t=\omega^*(t)$ for all $t\in u$, it follows that $x_u=x_u^*$. We conclude from this that $(A_i,C_i)\in\mathcal{C}_{x_u^*}$. Furthermore, since $\omega^*\in C_i\subseteq C_i^*$, we know that $\ind{C_i^*}(\omega^*)=1$. Therefore, it follows from Equation~\eqref{eq:theo:aanelkaarplakken:Sstarstardef} that $S^{**}\neq\emptyset$, which implies that $\omega^{**}\in C^{**}$. Hence, there is some $j\in S^{**}$ such that $\omega^{**}\in C_j^{**}$ and, since $j\in S^{**}$, we also know that $\ind{C_j^*}(\omega^*)=1$, or equivalently, that $\omega^*\in C_j^*$. By combining this with Equation~\eqref{eq:theo:aanelkaarplakken:triplestarequivalence}, it follows that $\omega^{***}\in C_j\subseteq C_0$.
So, in all cases, we find that $\omega^{***}\in C_0$. By combining this with Equation~\eqref{eq:theo:aanelkaarplakken:geqtotal}, it follows that Equation~\eqref{eq:theo:aanelkaarplakken:coh} holds, and therefore, that $\tilde{P}$ is coherent.

Since $\tilde{P}$ is coherent, and because $\mathcal{C}\subseteq\mathcal{C}^\mathrm{SP}$, it now follows from Theorem~\ref{theo:largerdomain} and Definition~\ref{def:stoch_process} that $\tilde{P}$ can be extended to a stochastic process $P$. Furthermore, since $P$ coincides with $\tilde{P}$ on $\mathcal{C}$, it follows from Equation~\eqref{eq:theo:aanelkaarplakken:defPtilde} that $P$ satisfies Equations~\eqref{eq:theo:aanelkaarplakken:equalsfirst} and~\eqref{eq:theo:aanelkaarplakken:equalssecond}. This concludes the first part of this proof.

In the remainder of this proof, we will show that $P\in\wprocesses_{\rateset,\,\mathcal{M}}$, as desired. To this end, let $P^*$ be any full conditional probability that coincides with $P$ on $\smash{\mathcal{C}^{\mathrm{SP}}}$; Corollary~\ref{corol:processiffrestriction} implies that such a full conditional probability always exists.

Now observe that due to Equation~\eqref{eq:theo:aanelkaarplakken:firstpartofdomain}, we have for all $x\in\states$ that $(X_0=y,X_{\emptyset}=x_{\emptyset})\in\mathcal{C}_{\emptyset}$. Therefore, and because of Equation~\eqref{eq:theo:aanelkaarplakken:defPtilde}, we find that
\begin{equation*}
P(X_0=y)=P(X_0=y\,\vert\,X_{\emptyset}=x_{\emptyset})=P_{\emptyset}(X_0=y\,\vert\,X_{\emptyset}=x_{\emptyset})=P_{\emptyset}(X_0=y)
\text{ for all $y\in\states$,}
\end{equation*}
which together with the fact that $P_{\emptyset}\in\wprocesses_{\rateset,\,\mathcal{M}}$, implies that $P(X_0)\in\mathcal{M}$. Hence, in order to prove that $\smash{P\in\wprocesses_{\rateset,\,\mathcal{M}}}$, it remains to show that $P$ is well-behaved as well as consistent with $\rateset$.

In order to do this, we start by establishing an important equality.
Consider any $w\in\mathcal{U}$ and $s\in\reals_{\geq0}$ such that $w<s$ and $u<s$. Then for all $x_w\in\states_w$ and $y\in\states$, we have that
\begin{align}\label{eq:theo:aanelkaarplakken:convexcombo}
P&(X_s=y\vert X_w=x_w)
=P^*(X_s=y\vert X_w=x_w)\notag\\
&=
\sum_{x_{u\setminus w}\in\states_{u\setminus w}}
P^*(X_s=y,
X_{u\setminus w}=x_{u\setminus w}
\vert X_w=x_w)\notag\\
&=
\sum_{x_{u\setminus w}\in\states_{u\setminus w}}
P^*(X_s=y\vert
X_{u\setminus w}=x_{u\setminus w}, X_w=x_w)
P^*(X_{u\setminus w}=x_{u\setminus w}
\vert X_w=x_w)\notag\\
&=
\sum_{x_{u\setminus w}\in\states_{u\setminus w}}
P^*(X_s=y\vert
X_{u}=x_{u}, X_{w\setminus u}=x_{w\setminus u})
P^*(X_{u\setminus w}=x_{u\setminus w}
\vert X_w=x_w)\notag\\
&=
\sum_{x_{u\setminus w}\in\states_{u\setminus w}}
P(X_s=y\vert
X_{u}=x_{u}, X_{w\setminus u}=x_{w\setminus u})
P^*(X_{u\setminus w}=x_{u\setminus w}
\vert X_w=x_w)\notag\\
&=
\sum_{x_{u\setminus w}\in\states_{u\setminus w}}
\tilde{P}(X_s=y\vert
X_{u}=x_{u}, X_{w\setminus u}=x_{w\setminus u})
P^*(X_{u\setminus w}=x_{u\setminus w}
\vert X_w=x_w)\notag\\
&=
\sum_{x_{u\setminus w}\in\states_{u\setminus w}}
P_{x_u}(X_s=y\vert
X_{u}=x_{u}, X_{w\setminus [0,\max u]}=x_{w\setminus [0,\max u]})
P^*(X_{u\setminus w}=x_{u\setminus w}
\vert X_w=x_w)\notag\\[-3mm]
\end{align}

Using this equality, we will next show that for small enough $\Delta\in\realspos$, the transition matrices $T_{t,x_v}^{t+\Delta}$ and $T_{t-\Delta,x_v}^t$ corresponding to $P$ can each be written as a (different) convex combination of transition matrices corresponding to processes in $\smash{\wprocesses_{\rateset,\,\mathcal{M}}}$.

Formally, we will show that for any $t\geq0$, $v\in\mathcal{U}_{<t}$ and $x_v\in\states_v$, there is some finite index set $\mathcal{I}$, some $v^*\in\mathcal{U}_{<t}$, a set of non-negative coefficients $(\lambda_i)_{i\in \mathcal{I}}$ that sum to one, a set of stochastic processes $({}^iP\in\wprocesses_{\rateset,\,\mathcal{M}})_{i\in \mathcal{I}}$ and a set of state instantiations $({}^ix_{v^*}\in\states_{v^*})_{i\in \mathcal{I}}$ such that
\begin{equation}\label{eq:theo:aanelkaarplakken:convexTright}
(\exists \delta>0)~(\forall 0<\Delta<\delta)~~
T_{t,x_v}^{t+\Delta}
=\sum_{i\in \mathcal{I}}\lambda_i
{}^iT_{t,{}^ix_{v^*}}^{t+\Delta}
\end{equation}
and, similarly, that for any $t>0$, $v\in\mathcal{U}_{<t}$ and $x_v\in\states_v$, there is some finite index set $\mathcal{I}$, some $v^*\in\mathcal{U}_{<t}$, a set of non-negative coefficients $(\lambda_i)_{i\in \mathcal{I}}$ that sum to one, a set of stochastic processes $({}^iP\in\wprocesses_{\rateset,\,\mathcal{M}})_{i\in \mathcal{I}}$ and a set of state instantiations $({}^ix_{v^*}\in\states_{v^*})_{i\in \mathcal{I}}$ such that
\begin{equation}\label{eq:theo:aanelkaarplakken:convexTleft}
(\exists \delta>0)~(\forall 0<\Delta<\delta)~~
T_{t-\Delta,x_v}^{t}
=\sum_{i\in \mathcal{I}}\lambda_i
{}^iT_{t-\Delta,{}^ix_{v^*}}^{t}.
\end{equation}

We start by constructing the convex combination that satisfies Equation~\eqref{eq:theo:aanelkaarplakken:convexTright}. So consider any $t\geq0$, $v\in\mathcal{U}_{<t}$ and $x_v\in\states_v$.
We distinguish between two cases: $t<\max u$ and $t\geq\max u$. If $t<\max u$, then for all $\Delta\in(0,\max u-\max v)$ and $x,y\in\states$, we see that $(X_{t+\Delta}=y,(X_t=x, X_v=x_v))\in\mathcal{C}_\emptyset$, and therefore, since $P$ is an extension of $\tilde{P}$, it follows from Equation~\eqref{eq:theo:aanelkaarplakken:defPtilde} that
\begin{equation*}
P(X_{t+\Delta}=y\vert X_t=x, X_v=x_v)
=P_\emptyset(X_{t+\Delta}=y\vert X_t=x, X_v=x_v).
\end{equation*}
Hence, if we let $\mathcal{I}\coloneqq\{i\}$, $v^*\coloneqq v$, $\lambda_i\coloneqq 1$, ${}^iP\coloneqq P_\emptyset$ and ${}^ix_{v^*}\coloneqq x_v$, Equation~\eqref{eq:theo:aanelkaarplakken:convexTright} is satisfied by choosing $\delta\coloneqq \max u-\max v$.
If $t\geq \max u$, then for all $\Delta>0$, it follows from Equation~\eqref{eq:theo:aanelkaarplakken:convexcombo} (with $s\coloneqq t+\Delta$ and $w\coloneqq v\cup t$) that
\begin{align*}
&P(X_{t+\Delta}=y\vert X_t=x, X_v=x_v)\\
&~~~~=
\sum_{x_{u\setminus(v\cup t)}\in\states_{u\setminus(v\cup t)}}
P_{x_u}(X_{t+\Delta}=y\vert X_t=x, 
X_{(u\setminus t)\cup(v\setminus [0,\max u])}= 
x_{(u\setminus t)\cup(v\setminus [0,\max u])})\\[-4mm]
&\quad\quad\quad\quad\quad\quad\quad\quad\quad\quad\quad\quad\quad\quad\quad\quad\quad\quad
P^*(X_{u\setminus(v\cup t)}=x_{u\setminus(v\cup t)}
\vert X_t=x, X_v=x_v).
\end{align*}
Therefore, if we let $\mathcal{I}\coloneqq\states_{u\setminus(v\cup t)}$, $v^*\coloneqq (u\setminus t)\cup(v\setminus [0,\max u])$ and, for all $x_{u\setminus(v\cup t)}\in \mathcal{I}$,
\begin{equation*}
\lambda_{x_{u\setminus(v\cup t)}}
\coloneqq P^*(X_{u\setminus(v\cup t)}=x_{u\setminus(v\cup t)}
\vert X_t=x, X_v=x_v),
\end{equation*}
${}^{x_{u\setminus(v\cup t)}}P=P_{x_u}$ and ${}^{x_{u\setminus(v\cup t)}}x_{v^*}\coloneqq
x_{(u\setminus t)\cup(v\setminus [0,\max u])}$, Equation~\eqref{eq:theo:aanelkaarplakken:convexTright} is satisfied for any $\delta>0$.
Hence, Equation~\eqref{eq:theo:aanelkaarplakken:convexTright} can be satisfied both when $t<\max u$ and when $t\geq \max u$.

We will next construct the convex combination that satisfies Equation~\eqref{eq:theo:aanelkaarplakken:convexTleft}. So, consider any $t>0$, $v\in\mathcal{U}_{<t}$ and $x_v\in\states_v$.
We again distinguish between two cases: $t\leq\max u$ and $t>\max u$. If $t\leq\max u$, then for all $\Delta\in(0,t-\max v)$ and $x,y\in\states$, we see that $(X_{t}=y,(X_{t-\Delta}=x, X_v=x_v))\in\mathcal{C}_\emptyset$, and therefore, since $P$ is an extension of $\tilde{P}$, it follows from Equation~\eqref{eq:theo:aanelkaarplakken:defPtilde} that
\begin{equation*}
P(X_{t}=y\vert X_{t-\Delta}=x, X_v=x_v)
=P_\emptyset(X_{t}=y\vert X_{t-\Delta}=x, X_v=x_v).
\end{equation*}
Hence, if we let $\mathcal{I}\coloneqq\{i\}$, $v^*\coloneqq v$, $\lambda_i\coloneqq 1$, ${}^iP\coloneqq P_\emptyset$ and ${}^ix_{v^*}\coloneqq x_v$, Equation~\eqref{eq:theo:aanelkaarplakken:convexTleft} is satisfied by choosing $\delta\coloneqq t-\max v$.
If $t>\max u$, then for all $\Delta\in(0,t-\max(v\cup u))$, it follows from Equation~\eqref{eq:theo:aanelkaarplakken:convexcombo} (with $s\coloneqq t$ and $w\coloneqq v\cup {t-\Delta}$) that
\begin{align*}
&P(X_{t}=y\vert X_{t-\Delta}=x, X_v=x_v)\\
&~~~~=
\sum_{x_{u\setminus v}\in\states_{u\setminus v}}
P_{x_u}(X_{t}=y\vert X_{t-\Delta}=x, 
X_{u\cup(v\setminus [0,\max u])}= 
x_{u\cup(v\setminus [0,\max u])})\\[-4mm]
&\quad\quad\quad\quad\quad\quad\quad\quad\quad\quad\quad\quad\quad\quad\quad\quad\quad\quad
P^*(X_{u\setminus v}=x_{u\setminus v}
\vert X_{t-\Delta}=x, X_v=x_v).
\end{align*}
Therefore, if we let $\mathcal{I}\coloneqq\states_{u\setminus v}$, $v^*\coloneqq u\cup(v\setminus [0,\max u])$ and, for all $x_{u\setminus v}\in \mathcal{I}$,
\begin{equation*}
\lambda_{x_{u\setminus v}}
\coloneqq P^*(X_{u\setminus v}=x_{u\setminus v}
\vert X_{t-\Delta}=x, X_v=x_v),
\end{equation*}
${}^{x_{u\setminus v}}P=P_{x_u}$ and ${}^{x_{u\setminus v}}x_{v^*}\coloneqq
x_{ u\cup(v\setminus [0,\max u])}$, Equation~\eqref{eq:theo:aanelkaarplakken:convexTleft} is satisfied by choosing $\delta\coloneqq t-\max(v\cup u)$.
Hence, Equation~\eqref{eq:theo:aanelkaarplakken:convexTleft} can be satisfied both when $t\leq \max u$ and when $t>\max u$.

Therefore, indeed, as claimed before, both $T_{t,x_v}^{t+\Delta}$ and $T_{t-\Delta,x_v}^t$ can be written as a convex combination of transition matrices corresponding to elements of $\smash{\wprocesses_{\rateset,\,\mathcal{M}}}$---assuming that $\Delta$ is small enough.
We will now use this fact to prove that $P$ is well-behaved and consistent with $\mathcal{Q}$.

We start by proving that $P$ is well-behaved. First consider any $t\geq0$, $v\in\mathcal{U}_{<t}$ and $x_v\in\states_v$, and consider the indexed set $\{{}^iP\in\wprocesses_{\rateset,\,\mathcal{M}}\}_{i\in \mathcal{I}}$ of stochastic processes whose transition matrices appear in the convex combination that satisfies  Equation~\eqref{eq:theo:aanelkaarplakken:convexTright}. Then,
\begin{equation}\label{eq:theo:aanelkaarplakken:wellbehavedTright}
\begin{aligned}
\limsup_{\Delta\to 0^{+}}\frac{1}{\Delta}\norm{T_{t,x_v}^{t+\Delta}-I} &= \limsup_{\Delta\to 0^{+}}\frac{1}{\Delta}\norm{\sum_{i\in \mathcal{I}}\lambda_i{}^iT_{t,{}^ix_{v^*}}^{t+\Delta}-I} \\
 & \leq\sum_{i\in \mathcal{I}}\lambda_i\limsup_{\Delta\to 0^{+}}\frac{1}{\Delta}\norm{{}^iT_{t,{}^ix_{v^*}}^{t+\Delta}-I}
<+\infty\,,
\end{aligned}
\end{equation}
where the last inequality follows from Proposition~\ref{prop:stochasticprocess:simpleproperties} and the fact that every process ${}^iP$ that appears in this combination is well-behaved.

Similarly, for any $t>0$, $v\in\mathcal{U}_{<t}$ and $x_v\in\states_v$, Equation~\eqref{eq:theo:aanelkaarplakken:convexTleft} can be satisfied with a convex combination of transition matrices corresponding to well-behaved processes, which implies that
\begin{equation}\label{eq:theo:aanelkaarplakken:wellbehavedTleft}
\begin{aligned}
\limsup_{\Delta\to 0^{+}}\frac{1}{\Delta}\norm{T_{t-\Delta,x_v}^t-I} &= \limsup_{\Delta\to 0^{+}}\frac{1}{\Delta}\norm{\sum_{i\in \mathcal{I}}\lambda_i{}^iT_{t-\Delta,{}^ix_{v^*}}^{t}-I} \\
 &\leq\sum_{i\in \mathcal{I}}\lambda_i\limsup_{\Delta\to 0^{+}}\frac{1}{\Delta}\norm{{}^iT_{t-\Delta,{}^ix_{v^*}}^t-I}
<+\infty.
\end{aligned}
\end{equation}
Because the $t\in\realsnonneg$, $v\in\mathcal{U}_{<t}$ and $x_v\in\states_v$ in Equation~\eqref{eq:theo:aanelkaarplakken:wellbehavedTright}, and the $t\in\realspos$, $v\in\mathcal{U}_{<t}$ and $x_v\in\states_v$ in Equation~\eqref{eq:theo:aanelkaarplakken:wellbehavedTleft} are arbitrary, by invoking Proposition~\ref{prop:stochasticprocess:simpleproperties}, it follows that $P$ is well-behaved.

We end by proving that $P$ is consistent with $\rateset$. Assume \emph{ex absurdo} that $P$ is not consistent with $\rateset$, or equivalently, that there are $t\geq0$, $v\in\mathcal{U}_{<t}$, $x_v\in\states_v$ and $Q^*\in\mathcal{R}$ such that $Q^*\in\overline{\partial}
{T^t_{t,\,x_v}}$ and $Q^*\notin\rateset$. We will show that this leads to a contradiction. 
We consider two (possibly overlapping) cases: $\smash{Q^*\in\overline{\partial}_{+}
{T^t_{t,\,x_v}}}$ and $\smash{Q^*\in\overline{\partial}_{-}
{T^t_{t,\,x_v}}}$.

If $\smash{Q^*\in\overline{\partial}_{+}
{T^t_{t,\,x_v}}}$, it follows from Equation~\eqref{eq:rightouterderivative} that there is a sequence $\{\Delta_j\}_{j\in\nats}\to0^+$ such that
\begin{equation}\label{eq:theo:aanelkaarplakken:Qstar}
\lim_{j\to+\infty}
\frac{1}{\Delta_j}
(T^{t+\Delta_j}_{t,\,x_v}-I)
=Q^*
\end{equation}
Consider the $\mathcal{I}$, $v^*\in\mathcal{U}_{<t}$, $\{\lambda_i\}_{i\in \mathcal{I}}$, $\{{}^iP\in\wprocesses_{\rateset,\,\mathcal{M}}\}_{i\in \mathcal{I}}$, $\{{}^ix_{v^*}\}_{i\in \mathcal{I}}$ and $\delta$ that satisfy Equation~\eqref{eq:theo:aanelkaarplakken:convexTright}. Fix any $i\in \mathcal{I}$. Since ${}^iP$ is well-behaved, the sequence
\begin{equation*}
\left\{\frac{1}{\Delta_j}
({}^iT^{t+\Delta_j}_{t,\,{}^ix_{v^*}}-I)\right\}_{j\in\nats}
\end{equation*}
is bounded, and therefore, the Bolzano-Weierstra{\ss} theorem implies that it has a convergent subsequence, of which we denote the limit by $Q_i$. Hence, without loss of generality---simply remove the indexes $j$ that do not correspond to the subsequence---we may assume that
\begin{equation}\label{eq:theo:aanelkaarplakken:fixingsubsequence}
\lim_{j\to+\infty}\frac{1}{\Delta_j}
({}^iT^{t+\Delta_j}_{t,\,{}^ix_{v^*}}-I)=Q_i.
\end{equation}
Since we know from Proposition~\ref{prop:rate_from_stochastic_matrix} that $Q_i$ is a limit of rate matrices, $Q_i$ is also a rate matrix, and therefore, it follows from Equation~\eqref{eq:rightouterderivative} that $Q_i\in\smash{\overline{\partial}_{+}
{{}^iT^t_{t,\,{}^ix_{v^*}}}}$, which, since ${}^iP$ is consistent with $\rateset$, implies that $Q_i\in\rateset$. By repeating this argument for every other $i\in \mathcal{I}$, we obtain a set of rate matrices $\{Q_i\in\rateset\}_{i\in \mathcal{I}}$ such that, without loss of generality, Equation~\eqref{eq:theo:aanelkaarplakken:fixingsubsequence} holds for every $i\in \mathcal{I}$. Additionally, since $\lim_{j\to\infty}\Delta_j=0$, we may assume without loss of generality that $0<\Delta_j<\delta$ for all $j\in\nats$. Equations~\eqref{eq:theo:aanelkaarplakken:convexTright} and~\eqref{eq:theo:aanelkaarplakken:fixingsubsequence} now imply that
\begin{equation*}
\lim_{j\to+\infty}
\frac{1}{\Delta_j}
(T_{t,x_v}^{t+\Delta_j}-I)
=\lim_{j\to+\infty}
\frac{1}{\Delta_j}
(\sum_{i\in \mathcal{I}}\lambda_i
{}^iT_{t,{}^ix_{v^*}}^{t+\Delta_j}-I)
=\sum_{i\in I}\lambda_i Q_i,
\end{equation*}
which, because of Equation~\eqref{eq:theo:aanelkaarplakken:Qstar}, implies that $Q^*=\sum_{i\in \mathcal{I}}\lambda_i Q_i$. Since $\rateset$ is convex, this implies that $Q^*\in\rateset$, a contradiction. Recall that this contradiction was derived from the assumption that $Q^*\in\overline{\partial}_{+}
{T^t_{t,\,x_v}}$.

If instead $\smash{Q^*\in\overline{\partial}_{-}
{T^t_{t,\,x_v}}}$, a completely analogous argument leads to the same contradiction: simply use Equations~\eqref{eq:leftouterderivative} and~\eqref{eq:theo:aanelkaarplakken:convexTleft} instead of Equations~\eqref{eq:rightouterderivative} and~\eqref{eq:theo:aanelkaarplakken:convexTright}, respectively, and adapt the rest of the argument accordingly. Since the two cases $Q^*\in\overline{\partial}_{+}
{T^t_{t,\,x_v}}$ and $\smash{Q^*\in\overline{\partial}_{-}
{T^t_{t,\,x_v}}}$ both lead to a contradiction, we conclude that our ex absurdo assumption must be false, or equivalently, that $P$ is consistent with $\rateset$. 

Hence, since we already know that $P$ is a well-behaved stochastic process such that $P(X_0)\in\mathcal{M}$, it follows that $P\in\wprocesses_{\rateset,\,\mathcal{M}}$.
\end{proof}

\proplowerexpmarkovboundedbynonmarkov*
\begin{proof}
This result is an immediate consequence of Proposition \ref{prop:markov_set_subset_of_nonmarkov_set} and Definition~\ref{def:lower_exp}.
\end{proof}

\theoremdecompositionmultivar*
\begin{proof}
Let $g(X_u,X_v)\coloneqq\underline{\mathbb{E}}^{\mathrm{W}}_{\rateset,\,\mathcal{M}}\left[f(X_u,X_v,X_w)\vert\,X_u,X_v\right]$ and fix any $x_u\in\states_u$.
For any $P\in\wprocesses_{\rateset,\,\mathcal{M}}$, it then follows from the basic properties of expectations (which follow from~\ref{def:coh_prob_2}--\ref{def:coh_prob_6}) and Definition~\ref{def:lower_exp} that
\begin{align*}
\mathbb{E}_P\left[f(X_u,X_v,X_w)\vert\,x_u\right] = \mathbb{E}_P\bigl[\mathbb{E}_P\left[f(X_u,X_v,X_w)\vert\,X_u,X_v\right]\vert\,x_u\bigr] 
 &\geq \mathbb{E}_P\bigl[g(X_u,X_v)\vert\,x_u\bigr] \\
 &\geq \underline{\mathbb{E}}^{\mathrm{W}}_{\rateset,\,\mathcal{M}}\bigl[g(X_u,X_v)\vert\,x_u\bigr].
\end{align*}\\[-20pt]
Since $P\in\wprocesses_{\rateset,\,\mathcal{M}}$ is arbitrary, this implies that
\begin{equation}\label{eq:theorem:decomposition_multivar:easyinequality}
\underline{\mathbb{E}}^{\mathrm{W}}_{\rateset,\,\mathcal{M}}\left[f(X_u,X_v,X_w)\vert\,x_u\right] \geq \underline{\mathbb{E}}^{\mathrm{W}}_{\rateset,\,\mathcal{M}}\bigl[g(X_u,X_v)\vert\,x_u\bigr].
\end{equation}

\noindent
Now fix any $\epsilon\in\realspos$.
Then due to Definition~\ref{def:lower_exp}, there is some $P_\emptyset\in\wprocesses_{\rateset,\,\mathcal{M}}$ such that
\begin{equation}\label{eq:theorem:decomposition_multivar:firsthalfprob}
\mathbb{E}_{P_\emptyset}\bigl[g(x_u,X_v)\vert\,x_u\bigr]
=\mathbb{E}_{P_\emptyset}\bigl[g(X_u,X_v)\vert\,x_u\bigr] < \underline{\mathbb{E}}^{\mathrm{W}}_{\rateset,\,\mathcal{M}}\bigl[g(X_u,X_v)\big\vert\,x_u\bigr]
+\nicefrac{\epsilon}{2},
\end{equation}
where the first equality follows from the basic properties of expectations. Similarly, for all $x_{v}\in\states_v$, there is some $P_{x_v}\in\wprocesses_{\rateset,\,\mathcal{M}}$ such that
\begin{equation}\label{eq:theorem:decomposition_multivar:secondhalfprob}
\mathbb{E}_{P_{x_v}}\left[f(X_u,X_v,X_w)\vert\,x_u,x_v\right] < \underline{\mathbb{E}}^{\mathrm{W}}_{\rateset,\,\mathcal{M}}\left[f(X_u,X_v,X_w)\vert\,x_u,x_v\right]+\nicefrac{\epsilon}{2}=g(x_u,x_v)+\nicefrac{\epsilon}{2}.
\end{equation}
Since $\rateset$ is non-empty, bounded, and convex, Theorem~\ref{theo:aanelkaarplakken} now implies the existence of a process $P\in\wprocesses_{\rateset,\,\mathcal{M}}$ such that for all $x_v\in\states_v$ and $x_w\in\states_w$:
\begin{equation*}
P(X_v=x_v\,\vert\,X_u=x_u) = P_\emptyset(X_v=x_v\,\vert\,X_u=x_u)
\end{equation*}\\[-20pt]
and
\begin{equation*}
P(X_w=x_w\,\vert\,X_u=x_u,X_v=x_v) = P_{x_v}(X_w=x_w\,\vert\,X_u=x_u,X_v=x_v)\vspace{6pt}
\end{equation*}
and therefore also, due to Equations~\eqref{eq:theorem:decomposition_multivar:firsthalfprob} and~\eqref{eq:theorem:decomposition_multivar:secondhalfprob},
\begin{equation*}
\mathbb{E}_P[g(x_u,X_v)\vert\,x_u]
=\mathbb{E}_{P_\emptyset}[g(x_u,X_v)\vert\,x_u]
< \underline{\mathbb{E}}^{\mathrm{W}}_{\rateset,\,\mathcal{M}}\bigl[g(X_u,X_v)\big\vert\,x_u\bigr]
+\nicefrac{\epsilon}{2}
\end{equation*}\\[-20pt]
and
\begin{equation*}
\mathbb{E}_P[f(X_u,X_v,X_w)\vert\,x_u,x_v]
=\mathbb{E}_{P_{x_v}}[f(X_u,X_v,X_w)\vert\, x_u,x_v]
<g(x_u,x_v)+\nicefrac{\epsilon}{2}.\vspace{6pt}
\end{equation*}
Hence, we find that
\begin{align*}
\mathbb{E}_P\bigl[\mathbb{E}_P\left[f(X_u,X_v,X_w)\vert\,x_u,X_v\right]\vert\,x_u\bigr]
\leq
\mathbb{E}_P\bigl[g(x_u,X_v)\,\big\vert\,x_u\bigr]+\nicefrac{\epsilon}{2}
< \underline{\mathbb{E}}^{\mathrm{W}}_{\rateset,\,\mathcal{M}}\bigl[g(X_u,X_v)\big\vert\,x_u\bigr]
+\epsilon.
\end{align*}
Since $\epsilon$ was arbitrary, and because
\begin{equation*}
\underline{\mathbb{E}}^{\mathrm{W}}_{\rateset,\,\mathcal{M}}\left[f(X_u,X_v,X_w)\vert\,x_u\right]
\leq
\mathbb{E}_P\left[f(X_u,X_v,X_w)\vert\,x_u\right] = \mathbb{E}_P\bigl[\mathbb{E}_P\left[f(X_u,X_v,X_w)\vert\,x_u,X_v\right]\vert\,x_u\bigr],
\end{equation*}
this implies that
$\underline{\mathbb{E}}^{\mathrm{W}}_{\rateset,\,\mathcal{M}}\left[f(X_u,X_v,X_w)\vert\,x_u\right]\leq\underline{\mathbb{E}}^{\mathrm{W}}_{\rateset,\,\mathcal{M}}\bigl[g(X_u,X_v)\big\vert\,x_u\bigr]$ and therefore, because of Equation~\eqref{eq:theorem:decomposition_multivar:easyinequality}, that $\underline{\mathbb{E}}^{\mathrm{W}}_{\rateset,\,\mathcal{M}}\left[f(X_u,X_v,X_w)\vert\,x_u\right]=\underline{\mathbb{E}}^{\mathrm{W}}_{\rateset,\,\mathcal{M}}\bigl[g(X_u,X_v)\big\vert\,x_u\bigr]$. Since $x_u$ was arbitrary, this implies that $\underline{\mathbb{E}}^{\mathrm{W}}_{\rateset,\,\mathcal{M}}\left[f(X_u,X_v,X_w)\vert\,X_u\right]=\underline{\mathbb{E}}^{\mathrm{W}}_{\rateset,\,\mathcal{M}}\bigl[g(X_u,X_v)\big\vert\,X_u\bigr]$. The result is now immediate because $g(X_u,X_v)=\underline{\mathbb{E}}^{\mathrm{W}}_{\rateset,\,\mathcal{M}}\left[f(X_u,X_v,X_w)\vert\,X_u,X_v\right]$. 
\end{proof}

\section{Proofs and Lemmas for the results in Section~\ref{sec:lowertrans}}

\lemmacompositioncoherence*
\begin{proof}
Consider any $f,g\in\gamblesX$, $\lambda\in\realsnonneg$ and $x\in\states$. Since $\underline{T}$ and $\underline{S}$ are lower transition operators, they both satisfy~\ref{LT:bounded_min}, and therefore, we find that
\begin{align*}
[\underline{T}\underline{S}f](x)
&\geq
\min\{[\underline{S}f](y)\colon y\in\states\}\\
&\geq
\min\{\min\{f(z)\colon z\in\states\}\colon y\in\states\}
=
\min\{f(z)\colon z\in\states\},
\end{align*}
which implies that $\underline{T}\underline{S}$ satisfies~\ref{LT:bounded_min} as well. 
Similarly, $\underline{T}\underline{S}$ satisfies~\ref{LT:super_additive} because
\begin{align*}
[\underline{T}\underline{S}(f+g)](x)
\geq
[\underline{T}(\underline{S}f+\underline{S}g)](x)
\geq
[\underline{T}\underline{S}f](x)
+
[\underline{T}\underline{S}g](x),
\end{align*}
where the first inequality follows from~\ref{LT:monotonicity} and the fact that $\underline{S}$ satisfies~\ref{LT:super_additive}, and where the second inequality follows from the fact that $\underline{T}$ satisfies~\ref{LT:super_additive}.
Finally, since $\underline{T}$ and $\underline{S}$ both satisfy~\ref{LT:homo}, it follows that $\underline{T}\underline{S}$ also satisfies~\ref{LT:homo}, because
\begin{align*}
[\underline{T}\underline{S}(\lambda f)](x)
=
[\underline{T}(\lambda\underline{S}f)](x)
=
\lambda[\underline{T}\underline{S}f](x).
\end{align*}
We conclude that $\underline{T}\underline{S}$ satisfies~\ref{LT:bounded_min}-\ref{LT:homo}, and therefore, because of Definition~\ref{def:low_trans}, it is a lower transition operator.
\end{proof}

We require the following two results for the proof of Proposition~\ref{lemma:completemetricspace}.

\begin{lemma}\cite[Proposition~1]{DeBock:2016}\label{lemma:lowertrans_if_pointwise_limit}
Consider any sequence $\{\lt_i\}_{i\in\nats}$ of lower transition operators such that $\lt f=\lim_{i\to\infty}\lt_i f$ for all $f\in\gamblesX$. Then $\lt$ is a lower transition operator.
\end{lemma}

\begin{lemma}\cite[Proposition~2]{DeBock:2016}\label{lemma:lowertrans_limit_iff_pointwise_limit}
Let $\lt$ be a lower transition operator, and consider any sequence $\{\lt_i\}_{i\in\nats}$ of lower transition operators. Then, $\lt=\lim_{i\to\infty}\lt_i$ if and only if $\lt f=\lim_{i\to\infty}\lt_i f$ for all $f\in\gamblesX$.
\end{lemma}

\lemmacompletemetricspace*
\begin{proof}
Consider any sequence $\{\lt_i\}_{i\in\nats}$ of lower transition operators that is Cauchy with respect to the operator norm $\norm{\cdot}$. We will prove that $\{\lt_i\}_{i\in\nats}$ converges to a limit $\lt\colon\gamblesX\to\gamblesX$ that is itself a lower transition operator.

Consider any $f\in\gamblesX$ and $x\in\states$. For any $k,\ell\in\nats$, \eqref{N:normAf} then implies that
\begin{equation*}
\abs{[\lt_k f](x)-[\lt_\ell f](x)}
\leq\norm{\lt_k f-\lt_\ell f}
=\norm{(\lt_k-\lt_\ell)f}
\leq\norm{\lt_k-\lt_\ell}\norm{f}.
\end{equation*}
Therefore, and because $\{\lt_i\}_{i\in\nats}$ is Cauchy with respect to the norm $\norm{\cdot}$, it follows that $\{[\lt_i f](x)\}_{i\in\nats}$ is Cauchy with respect to the norm $\abs{\cdot}$. Hence, since $\reals$ is (well known to be) complete with respect to the topology that is induced by $\abs{\cdot}$, we find that $\{[\lt_i f](x)\}_{i\in\nats}$ converges to a limit in $\reals$, which we will denote by $[\lt f](x)$. Let $\lt f$ be the unique function in $\gamblesX$ that has $[\lt f](x)$, $x\in\states$, as its components. Then clearly, $\lt f=\lim_{i\to+\infty}\lt_if$. 

Let $\lt\colon\gamblesX\to\gamblesX$ be the unique operator that maps any $f\in\gamblesX$ to $\lt f$. It then follows from Lemma~\ref{lemma:lowertrans_if_pointwise_limit} that $\lt$ is a lower transition operator. Therefore, and because we already know that $\lt f=\lim_{i\to+\infty}\lt_i f$ for all $f\in\gamblesX$, it now follows from Lemma~\ref{lemma:lowertrans_limit_iff_pointwise_limit} that $\lim_{i\to+\infty}\lt_i=\lt$.
\end{proof}

\proplowerenvelopeislowertrans*
\begin{proof}
Consider any $Q\in\rateset$. It then follows from Definition~\ref{def:rate_matrix} that the matrix $Q$, when regarded as a map from $\gamblesX$ to $\gamblesX$, satisfies \ref{LR:constantzero}--\ref{LR:homo}. Since each of these properties is preserved under taking lower envelopes, it follows that $\lrate$ satisfies \ref{LR:constantzero}--\ref{LR:homo}, which means that $\lrate$ is a lower transition rate operator.
\end{proof}

\propdominatingnonemptybounded*
\begin{proof}
We will start by showing that, for every $f\in\gamblesX$, there is some $\smash{Q\in\rateset_{\lrate}}$ such that $\lrate f=Qf$. To this end, fix any $f\in\gamblesX$. Now choose $\Delta>0$ small enough such that $\smash{0\leq\Delta\norm{\lrate}\leq 1}$ (this always possible because of property~\ref{LR:normlratefinite}) and define $\lt\coloneqq I+\Delta\lrate$. Since $\lrate$ is a lower transition rate operator, it then follows from Proposition~\ref{lemma:normQsmallenough} that $\lt$ is a lower transition operator. For any $x\in\states$, we now consider the operator $\lt_x\colon\gamblesX\to\reals$, as defined by Equation~\eqref{eq:lowerprevisionfromlt}, which is a coherent lower prevision. Because of \cite[Theorem~3.3.3(b)]{Walley:1991vk}, this implies the existence of an expectation operator $\mathbb{E}_x$ on $\gamblesX$---Reference~\cite{Walley:1991vk} calls this a linear prevision on $\gamblesX$---such that $\mathbb{E}_xg\geq\lt_xg$ for all $g\in\gamblesX$ and $\mathbb{E}_xf=\lt_xf$. Let $P_x$ be the unique probability mass function that corresponds to $\mathbb{E}_x$. For all $x,y\in\states$, we now let $T(x,y)\coloneqq P_x(y)\coloneqq \mathbb{E}_x(\ind{y})$. Then $T$ is clearly a transition matrix. Furthermore, for every $x\in\states$ and $g\in\gamblesX$, we have that $(Tg)(x)=\mathbb{E}_xg$. Hence, it follows that $Tg\geq\lt g$ for all $g\in\gamblesX$ and that $Tf=\lt f$. Now let $Q\coloneqq\nicefrac{1}{\Delta}(T-I)$, which, because of Proposition~\ref{prop:rate_from_stochastic_matrix}, is a rate matrix. Since $Tf=\lt f$, it then follows that
\begin{equation*}
Qf=\frac{1}{\Delta}(Tf-f)=\frac{1}{\Delta}(\lt f-f)=\lrate f.
\end{equation*}
Similarly, since $Tg\geq\lt g$ for all $g\in\gamblesX$, it follows that $Qg\geq\lrate g$, or equivalently, since $Q$ is a rate matrix, that $Q\in\rateset_{\lrate}$. Since $f$ was arbitrary, this proves that, for all $f\in\gamblesX$, there is some $Q\in\rateset_{\lrate}$ such that $Qf=\lrate f$. Since $\gamblesX$ is non-empty, this clearly also implies that $\rateset_{\lrate}$ is non-empty.

We end this proof by showing that $\rateset_{\lrate}$ is bounded. Consider any $x\in\states$. Then for all $Q\in\rateset_{\lrate}$, we have that $Q(x,x)=[Q\ind{x}](x)\geq[\lrate\ind{x}](x)$, which implies that
\begin{equation*}
\inf\left\{Q(x,x)\colon Q\in\rateset_{\lrate}\right\}\geq[\lrate\ind{x}](x)>-\infty.
\end{equation*}
Since $x\in\states$ is arbitrary, Proposition~\ref{prop:alternativedefforbounded} now guarantees that $\rateset_{\lrate}$ is bounded. 
\end{proof}

\propdominatingproperties*
\begin{proof}
We start by showing that $\rateset_{\lrate}$ is closed, or equivalently, that for any converging sequence $\{Q_i\}_{i\in\nats}$ in $\rateset_{\lrate}$, the limit $Q\coloneqq\lim_{i\to+\infty}Q_i$ is again an element of $\rateset_{\lrate}$. Since $\{Q_i\}_{i\in\nats}$ belongs to the bounded set of rate matrices $\rateset_{\lrate}$---see Proposition~\ref{prop:dominating_nonempty_bounded}---we know that the limit $Q$ is a real-valued matrix, and therefore, since~\ref{def:Q:sumzero} and~\ref{def:Q:nonnegoffdiagonal} are clearly both preserved under taking limits, $Q$ is a rate matrix. Now, assume \emph{ex absurdo} that $Q\notin\rateset_{\lrate}$. Then, by Equation~\eqref{eq:dominatingratematrices}, there is some $f\in\gamblesX$ and some $x\in\states$ such that $\smash{\left[Qf\right](x) < \left[\lrate f\right](x)}$, and therefore some $\epsilon\in\realspos$ such that $[Qf](x) + \epsilon < [\lrate f](x)$. Hence, since $\lim_{i\to+\infty}Q_i=Q$, there is some $i^*\in\nats$ such that $[Q_{i^*}f](x) <[Qf](x) + \epsilon<[\lrate f](x)$. Since $Q_{i^*}\in\rateset_{\lrate}$, this is a contradiction. Therefore, $Q\in\rateset_{\lrate}$, and because the converging sequence $\{Q_i\}_{i\in\nats}$ was arbitrary, this proves that $\rateset_{\lrate}$ is closed.

Next, we show that $\rateset_{\lrate}$ is convex, or equivalently, that for any two rate matrices $Q_1,Q_2\in\rateset_{\lrate}$, and any $\lambda\in[0,1]$, the matrix $Q_\lambda\coloneqq\lambda Q_1 + (1-\lambda)Q_2$ is again an element of $\rateset_{\lrate}$. It is easily verified from Definition~\ref{def:rate_matrix} that $Q_\lambda$ is a rate matrix. Furthermore, for any $f\in\gamblesX$, we find that
\begin{equation*}
Q_\lambda f=\lambda Q_1f+(1-\lambda)Q_2f\geq\lambda \lrate f+(1-\lambda)\lrate f=\lrate f,
\end{equation*}
where the inequality holds because $Q_1$ and $Q_2$ belong to $\rateset_{\lrate}$. Hence, it follows from Equation~\eqref{eq:dominatingratematrices} that $Q_\lambda\in\rateset_{\lrate}$.

We finally show that $\rateset_{\lrate}$ has separately specified rows. For all $x\in\states$, let $\rateset_x\coloneqq\{Q(x,\cdot)\,:\,Q\in\rateset_{\lrate}\}$. Consider now any rate matrix $Q$ such that, for all $x\in\states$, $Q(x,\cdot)\in\rateset_x$ and assume \emph{ex absurdo} that $\smash{Q\notin\rateset_{\lrate}}$. Equation~\eqref{eq:dominatingratematrices} then implies the existence of some $f\in\gamblesX$ and $x\in\states$ such that $\smash{\left[Qf\right](x) < \left[\lrate f\right](x)}$. Since $Q(x,\cdot)\in\rateset_x$, this in turn implies that there is some $Q'\in\rateset_{\lrate}$ such that $\left[Q'f\right](x) < \left[\lrate f\right](x)$, a contradiction. Hence we find that $Q\in\rateset_{\lrate}$.
\end{proof}

\begin{lemma}\label{lemma:rows_nonempty_bounded_closed_convex}
Let $\rateset$ be a non-empty, bounded, closed and convex set of rate matrices. Then, for all $x\in\states$, $\rateset_x\coloneqq\{Q(x,\cdot)\,:\,Q\in\rateset\}$ is a non-empty, bounded, closed, and convex subset of $\gamblesX$.
\end{lemma}
\begin{proof}
Fix any $x\in\states$. The non-emptiness, boundedness and convexity of $\rateset_x$ then follows trivially from the fact that $\rateset$ is non-empty, bounded and convex. It remains to show that $\rateset_x$ is closed. Consider therefore any sequence $\{Q_i\}_{i\in\nats}$ in $\rateset$ such that $\{Q_i(x,\cdot)\}_{i\in\nats}$ converges to a limit $Q_{\infty}(x,\cdot)$. We need to prove that $Q_{\infty}(x,\cdot)\in\rateset_x$.

Since $\{Q_i\}_{i\in\nats}$ belongs to the---closed and bounded and therefore---compact set $\rateset$, it has a convergent subsequence $\{Q_{i_k}\}_{k\in\nats}$ whose limit $Q^*$ belongs to $\rateset$. Since $\{Q_i(x,\cdot)\}_{i\in\nats}$ converges to $Q_{\infty}(x,\cdot)$, it follows that $Q^*(x,\cdot)=Q_{\infty}(x,\cdot)$, which, since $Q^*\in\rateset$, implies that $Q_{\infty}(x,\cdot)\in\rateset_x$.
\end{proof}

\propdominatinguniquecharacterization*
\begin{proof}
Since $\rateset$ has $\lrate$ as its lower envelope, it follows from Equation~\eqref{eq:dominatingratematrices} that $\rateset\subseteq\rateset_{\lrate}$. Assume \emph{ex absurdo} that $\rateset\subset\rateset_{\lrate}$, or equivalently, that there is some $Q\in\rateset_{\lrate}$ such that $Q\notin\rateset$. Consider now any such $Q$.

Because $\rateset$ has separately specified rows, it follows from $Q\notin\rateset$ that there is some $x\in\states$ such that $Q(x,\cdot)\notin\rateset_x\coloneqq\{Q'(x,\cdot)\,:\,Q'\in\rateset\}$. Furthermore, since $\rateset$ is non-empty, bounded, closed and convex, it follows from Lemma~\ref{lemma:rows_nonempty_bounded_closed_convex} that $\rateset_x$ is a non-empty, bounded, closed and convex subset of $\gamblesX$, and therefore, since $\gamblesX$ is a normed linear space and $Q(x,\cdot)\notin\rateset_x$, it follows from the separating hyperplane theorem~\cite[Chapter 14, Corollary 25]{Royden:2010vn} that there is a linear operator $\psi\colon\gamblesX\to\reals$ such that
\begin{equation}\label{eq:prop:dominating_unique_characterization:C}
\psi(Q(x,\cdot))<\inf\{\psi(q)\colon q\in\rateset_x\}=\colon C.
\end{equation}
Now let $f\in\gamblesX$ be defined by $f(y)\coloneqq\psi(\ind{y})$ for all $y\in\states$. Then for any $\epsilon>0$, since $\rateset$ has $\lrate$ as its lower envelope, and because of Equation~\eqref{eq:correspondinglowertrans}, there is some $Q^*\in\rateset$ such that $[Q^*f](x)\leq[\lrate f](x)+\epsilon$, and therefore also
\begin{equation*}
C
\leq\psi(Q^*(x,\cdot))
=\sum_{y\in\states}Q^*(x,y)\psi(\ind{y})
=\sum_{y\in\states}Q^*(x,y)f(y)
=[Q^* f](x)\leq[\lrate f](x)+\epsilon.
\end{equation*}
Since $\epsilon>0$ is arbitrary, this implies that $C\leq[\lrate f](x)$ and, by combining this with Equation~\eqref{eq:prop:dominating_unique_characterization:C}, it follows that
\begin{equation*}
[Qf](x)
=\sum_{y\in\states}Q(x,y)f(y)
=\sum_{y\in\states}Q(x,y)\psi(\ind{y})
=\psi(Q(x,\cdot))
<C\leq[\lrate f](x).
\end{equation*}
This implies that $Qf\not\geq\lrate f$, and therefore, that $Q\notin\rateset_{\lrate}$, a contradiction.
\end{proof}

In order to prove the norm inequality stated by Proposition~\ref{prop:differencebetweenu}, we first give a number of norm inequalities that are more conveniently proved separately. These inequalities are stated by Lemmas~\ref{lemma:recursive_lower_trans}-\ref{lemma:differencebetweennested} and Corollary \ref{corol:differencebetweennestedintermsofu} below.

We start with the following result, which states that the distance between two composed lower transition operators $\lt_1\lt_2\cdots\lt_n$ and $\underline{S}_1\underline{S}_2\cdots\underline{S}_n$ is bounded from above by the sum of the distances $\norm{\lt_i - \underline{S}_i}$ between their components. Note that this result is a generalized version of Lemma~\ref{lemma:recursive}, which until this point has remained unproven. To see that the result below indeed implies Lemma~\ref{lemma:recursive}, simply recall from Section~\ref{subsec:lowertrans_rate} that every transition matrix $T$ is also a lower transition operator.

\begin{lemma}\label{lemma:recursive_lower_trans}
Consider two finite sequences $\underline{T}_1,\ldots,\underline{T}_n$ and $\underline{S}_1,\ldots,\underline{S}_n$ of lower transition operators. Then
\begin{equation}\label{eq:lemma_recursive_inequality}
\norm{\prod_{i=1}^n\underline{T}_i - \prod_{i=1}^n\underline{S}_i} \leq \sum_{i=1}^n \norm{\underline{T}_i - \underline{S}_i}.
\end{equation}
\end{lemma}
\begin{proof}
We provide a proof by induction. Clearly, Equation~\eqref{eq:lemma_recursive_inequality} holds for $n=1$. Suppose that it holds for $n-1$. We show that it then also holds for $n$.
\begin{align*}
\norm{\prod_{i=1}^n\underline{T}_i - \prod_{i=1}^n\underline{S}_i} &= \norm{\prod_{i=1}^n\underline{T}_i - \left(\prod_{i=1}^{n-1}\underline{T}_i\right)\underline{S}_n + \left(\prod_{i=1}^{n-1}\underline{T}_i\right)\underline{S}_n - \prod_{i=1}^n\underline{S}_i} \\
 &\leq \norm{\prod_{i=1}^n\underline{T}_i - \left(\prod_{i=1}^{n-1}\underline{T}_i\right)\underline{S}_n} + \norm{\left(\prod_{i=1}^{n-1}\underline{T}_i\right)\underline{S}_n - \prod_{i=1}^n\underline{S}_i} \\
 &= \norm{\left(\prod_{i=1}^{n-1}\underline{T}_i\right)\underline{T}_n - \left(\prod_{i=1}^{n-1}\underline{T}_i\right)\underline{S}_n} + \norm{\left(\prod_{i=1}^{n-1}\underline{T}_i - \prod_{i=1}^{n-1}\underline{S}_i\right)\underline{S}_n} \\
 &\leq \norm{\underline{T}_n - \underline{S}_n} + \norm{\prod_{i=1}^{n-1}\underline{T}_i - \prod_{i=1}^{n-1}\underline{S}_i}\norm{\underline{S}_n} \\
 &\leq \norm{\underline{T}_n - \underline{S}_n} + \norm{\prod_{i=1}^{n-1}\underline{T}_i - \prod_{i=1}^{n-1}\underline{S}_i} \\
 &\leq \norm{\underline{T}_n - \underline{S}_n} + \sum_{i=1}^{n-1}\norm{\underline{T}_i - \underline{S}_i} = \sum_{i=1}^{n}\norm{\underline{T}_i - \underline{S}_i}\,.
\end{align*}
Here, in the second inequality, we applied Proposition~\ref{lemma:compositioncoherence} and properties~\ref{N:normAB} and~\ref{LT:differencenorm}. In the third inequality, we used property~\ref{LT:norm_at_most_one}. In the final inequality, we used the induction hypothesis.
\end{proof}

\begin{lemma}\label{lemma:justthelinearpart}
Consider a lower transition rate operator $\lrate$ and, for all $i\in\{1,\ldots,n\}$, some $\Delta_i\geq0$ such that $\Delta_i\norm{\lrate}\leq1$. Let $\Delta\coloneqq\sum_{i=1}^n\Delta_i$. Then
\begin{equation*}
\norm{\prod_{i=1}^n(I+\Delta_i\lrate)-(I+\Delta\lrate)}\leq\Delta^2\norm{\lrate}^2.
\end{equation*}
\end{lemma}
\begin{proof}
We provide a proof by induction. For $n=1$, the result is trivial. So consider the case $n\geq2$ and assume that the result is true for $n-1$. 

For all $i\in\{2,\dots,n\}$, since $\Delta_i\norm{\lrate}\leq1$, it follows from Proposition~\ref{lemma:normQsmallenough} that $I$ and $(I+\Delta_i\lrate)$ are lower transition operators. Therefore,
\begin{multline*}
\norm{\prod_{i=1}^n(I+\Delta_i\lrate)-(I+\Delta\lrate)}\\
\begin{aligned}
&=\norm{\prod_{i=2}^n(I+\Delta_i\lrate)+\Delta_1\lrate\prod_{i=2}^n(I+\Delta_i\lrate)-(I+\sum_{i=2}^n\Delta_i\lrate)-\Delta_1\lrate}\\
&\leq\norm{\prod_{i=2}^n(I+\Delta_i\lrate)-(I+\sum_{i=2}^n\Delta_i\lrate)}+\norm{\Delta_1\lrate\prod_{i=2}^n(I+\Delta_i\lrate)-\Delta_1\lrate}\\
&\leq(\sum_{i=2}^n\Delta_i)^2\norm{\lrate}^2
+2\Delta_1\norm{\lrate}
\norm{\prod_{i=2}^n(I+\Delta_i\lrate)-I}\\
&\leq(\sum_{i=2}^n\Delta_i)^2\norm{\lrate}^2
+2\Delta_1\norm{\lrate}
\sum_{i=2}^n
\norm{(I+\Delta_i\lrate)-I}\\
&=(\sum_{i=2}^n\Delta_i)^2\norm{\lrate}^2
+\Big(2\Delta_1
\sum_{i=2}^n
\Delta_i\Big)\norm{\lrate}^2
\leq\Big(
\Delta_1+\sum_{i=2}^n\Delta_i
\Big)^2\norm{\lrate}^2
=\Delta^2\norm{\lrate}^2,
\end{aligned}
\end{multline*}\\[3pt]
where the second inequality follows from the induction hypothesis and property~\ref{LR:differenceofnorm}, and the third inequality follows from Lemma~\ref{lemma:recursive_lower_trans}.
\end{proof}

\begin{lemma}\label{lemma:differencebetweennested}
For any $k\in\{1,\dots,n\}$, consider a sequence $\Delta_{k,i}\geq 0$, $i\in\{1,\dots,n_k\}$, and let $\Delta_k\coloneqq\sum_{i=1}^{n_k}\Delta_{k,i}$. Let $C\coloneqq\sum_{k=1}^n\Delta_k$ and let $\delta\coloneqq\max_{k=1}^n\Delta_k$. Then for any lower transition rate operator $\lrate$ such that $\delta\norm{\lrate}\leq1$:
\begin{equation*}
\norm{\prod_{k=1}^n\left(\prod_{i=1}^{n_k}(I+\Delta_{k,i}\lrate)\right)
-
\prod_{k=1}^n(I+\Delta_k\lrate)
}
\leq\delta C\norm{\lrate}^2.
\end{equation*}
\end{lemma}
\begin{proof}
For any $k\in\{1,\dots,n\}$, we know that $\Delta_{k,i}\norm{\lrate}\leq\Delta_k\norm{\lrate}\leq\delta\norm{\lrate}\leq1$ for all $i\in\{1,\dots,n_k\}$, and therefore, it follows from Propositions~\ref{lemma:normQsmallenough} and~\ref{lemma:compositioncoherence} that $\prod_{i=1}^{n_k}(I+\Delta_{k,i}\lrate)$ and $(I+\Delta_k\lrate)$ are lower transition operators. Hence,
\begin{align*}
\norm{\prod_{k=1}^n\left(\prod_{i=1}^{n_k}(I+\Delta_{k,i}\lrate)\right)
-
\prod_{k=1}^n(I+\Delta_k\lrate)
}
&\leq
\sum_{k=1}^n
\norm{
	\left(\prod_{i=1}^{n_k}(I+\Delta_{k,i}\lrate)\right)
	-
	(I+\Delta_k\lrate)
}\\
&\leq
\sum_{k=1}^n
\Delta_k^2\norm{\lrate}^2
\leq
\sum_{k=1}^n
\delta\Delta_k\norm{\lrate}^2
=\delta C\norm{\lrate}^2,
\end{align*}
where the first inequality follows from Lemma~\ref{lemma:recursive_lower_trans} and the second inequality follows from Lemma~\ref{lemma:justthelinearpart}.
\end{proof}

\begin{corollary}\label{corol:differencebetweennestedintermsofu}
Let $\lrate$ be a lower transition rate operator. Consider any $t,s\in\realsnonneg$ such that $t\leq s$ and any $u\in\mathcal{U}_{[t,s]}$ such that $\sigma(u)\norm{\lrate}\leq 1$. Then for all $u'\in\mathcal{U}_{[t,s]}$ such that $u\subseteq u'$:
\begin{equation*}
\norm{\Phi_u-\Phi_{u'}}\leq\sigma(u)(s-t)\norm{\lrate}^2.\vspace{5pt}
\end{equation*}
\end{corollary}
\begin{proof}
This result is trivial if $s=t$, because then $u=u'=\{t\}=\{s\}$. Hence, without loss of generality, we assume that $s>t$, which implies that $u=(t_0,\dots,t_n)$, with $n\in\nats$, $t_0=t$ and $t_n=s$.
Since $u\subseteq u'$, we know that, for all $k\in\{1,\dots,n\}$, there is some sequence $\Delta_{k,i}> 0$, $i\in\{1,\dots,n_k\}$, with $n_k\in\nats$ and $\Delta_k=\sum_{i=1}^{n_k}\Delta_{k,i}\leq\sigma(u)$, such that $\sum_{k=1}^n\Delta_k=s-t$,
\begin{equation*}
\Phi_{u'}\coloneqq\prod_{k=1}^n\left(\prod_{i=1}^{n_k}(I+\Delta_{k,i}\lrate)\right)
\text{~~and~~}
\Phi_{u}\coloneqq\prod_{k=1}^n(I+\Delta_{k}\lrate).
\end{equation*}
Because of Lemma~\ref{lemma:differencebetweennested}, this implies that $\norm{\Phi_{u'}-\Phi_u}\leq\sigma(u)(s-t)\norm{\lrate}^2$. 
\end{proof}

\propdifferencebetweenu*
\begin{proof}
Consider any $u'\in\mathcal{U}_{[t,s]}$ that is finer than $u$ and $u^*$, meaning that the timepoints it consists of contain the timepoints in $u$ and the timepoints in $u^*$. For example, let $u'$ be the ordered union of the timepoints in $u$ and $u^*$. Corollary~\ref{corol:differencebetweennestedintermsofu} then implies that $\smash{\norm{\Phi_{u}-\Phi_{u'}}\leq\delta C\norm{\lrate}^2}$ and $\smash{\norm{\Phi_{u^*}-\Phi_{u'}}\leq\delta C\norm{\lrate}^2}$, and therefore, it follows that
\begin{equation*}
\norm{\Phi_{u}-\Phi_{u^*}}
\leq
\norm{\Phi_{u}-\Phi_{u'}}
+
\norm{\Phi_{u'}-\Phi_{u^*}}
\leq2\delta C\norm{\lrate}^2.\vspace{-10pt}
\end{equation*}
\end{proof}

\corolcauchy*
\begin{proof}
By definition of a Cauchy sequence, we need to show that
\begin{equation*}
(\forall \epsilon>0)(\exists n\in\nats)(\forall i,j\geq n)
\norm{\Phi_{u_i}-\Phi_{u_j}}<\epsilon.
\end{equation*}
Therefore, fix any $\epsilon\in\realspos$, and choose $\delta\in\realspos$ such that $\smash{2\delta(s-t)\norm{\lrate}^2<\epsilon}$ and $\delta\norm{\lrate}\leq 1$ [this is clearly always possible]. Because $\lim_{i\to\infty}\sigma(u_i)=0$, there is some $n\in\nats$ such that, for all $i\geq n$, $\sigma(u_i)\leq\delta$. Consider any $i,j\geq n$. It then follows from Proposition~\ref{prop:differencebetweenu} that
\begin{equation*}
\norm{\Phi_{u_i}-\Phi_{u_j}} \leq 2\delta(s-t)\norm{\lrate}^2 < \epsilon\,.
\end{equation*}
\end{proof}

\corollimitexistsandiscoherent*
\begin{proof}
Since $\lim_{i\to\infty}\sigma(u_i)=0$, and due to Propositions~\ref{lemma:normQsmallenough} and~\ref{lemma:compositioncoherence} and property~\ref{LR:normlratefinite}, there is some index $n$ such that the sequence $\Phi_{u_n},\Phi_{u_{n+1}},\ldots$ consists of lower transition operators. Due to Corollary~\ref{corol:cauchy}, this sequence is Cauchy and therefore, because of Proposition~\ref{lemma:completemetricspace}, this sequence has a limit that is also a lower transition operator. Since the limit starting from $n$ and the limit starting from $1$ are identical (initial elements do not influence the limit), we find that the sequence $\left\{\Phi_{u_i}\right\}_{i\in\nats}$ has a limit, and that this limit is a lower transition operator.
\end{proof}

\theoconvergencelowerbound*
\begin{proof}
Let $\{u_i\}_{i\in\nats}$ be any sequence in $\mathcal{U}_{[t,s]}$ such that $\lim_{i\to\infty}\sigma(u_i)=0$. Because of Corollary~\ref{corol:limitexistsandiscoherent}, the sequence $\{\Phi_{u_i}\}_{i\in\nats}$ then converges to a lower transition operator, which we denote by $\lt$. 

Let $C\coloneqq s-t$, fix any $\epsilon>0$, and choose any $\delta>0$ such that $\smash{4\delta C\norm{\lrate}^2<\epsilon}$ and $\delta\norm{\lrate}\leq1$ [this is clearly always possible].
Since $\lim_{i\to+\infty}\Phi_{u_i}=\lt$ and $\lim_{i\to\infty}\sigma(u_i)=0$, there is some $i^*\in\nats$ such that
\begin{equation}\label{eq:theo:convergencelowerbound:proof}
\sigma(u_{i^*})\leq\delta\text{ and }\norm{\lt - \Phi_{u_{i^*}}}\leq\frac{\epsilon}{2}.
\end{equation}
Consider now any $u\in\mathcal{U}_{[t,s]}$ such that $\sigma(u)\leq\delta$. Then
\begin{equation*}
\norm{\lt - \Phi_u}\leq\norm{\lt-\Phi_{u_{i^*}}}
+\norm{\Phi_{u_{i^*}}-\Phi_u}
\leq\frac{\epsilon}{2}+2\delta C\norm{\lrate}^2\leq\epsilon,
\end{equation*}
where the second inequality follows from Equation~\eqref{eq:theo:convergencelowerbound:proof} and Proposition~\ref{prop:differencebetweenu}.

It remains to show that $\lt$ is unique. Therefore, let $\lt'$ be any lower transition operator that satisfies Equation~\eqref{eq:theo:convergencelowerbound}. Then clearly, for any $\epsilon>0$, there is some $u\in\mathcal{U}_{[t,s]}$ such that $\norm{\lt-\Phi_u}\leq\epsilon$ and $\norm{\lt'-\Phi_u}\leq\epsilon$, and therefore also $\norm{\lt-\lt'}\leq2\epsilon$. Since $\epsilon>0$ is arbitrary, this implies that $\norm{\lt-\lt'}=0$, or equivalently, that $\lt=\lt'$. Because this holds for every $\lt'$ that satisfies Equation~\eqref{eq:theo:convergencelowerbound}, it follows that $\lt$ is the unique operator satisfying this equation.
\end{proof}

We following lemma restates the norm inequality given by Corollary~\ref{corol:differencebetweennestedintermsofu}, so that it can be used with lower transition operators $L_t^s$. In effect, it therefore provides a bound on how well we can approximate the operator $L_t^s$ using an operator $\Phi_u$ constructed using a finite partition $u\in\mathcal{U}_{[t,s]}$ of the interval $[t,s]$. In particular, this approximation improves as we take $u$ to be increasingly finer, that is, as we decrease $\sigma(u)$.

\begin{lemma}\label{lemma:limitboundonL}
Let $\lrate$ be a lower transition rate operator. Then for any $t,s\in\realsnonneg$ such that $t\leq s$ and any $u\in\mathcal{U}_{[t,s]}$ such that $\sigma(u)\norm{\lrate}\leq 1$:
\begin{equation*}
\norm{L_t^s-\Phi_{u}}\leq\sigma(u)(s-t)\norm{\lrate}^2.\vspace{5pt}
\end{equation*}
\end{lemma}
\begin{proof}
Fix any $\epsilon>0$. Because of Definition~\ref{def:low_trans}, there is some $u'\in\mathcal{U}_{[t,s]}$ such that $u\subseteq u'$ and $\norm{L_t^s-\Phi_{u'}}\leq\epsilon$. By combining this with Corollary~\ref{corol:differencebetweennestedintermsofu}, it follows that
\begin{equation*}
\norm{L_t^s-\Phi_{u}}
\leq
\norm{L_t^s-\Phi_{u'}}
+
\norm{\Phi_{u'}-\Phi_{u}}
\leq\epsilon+
\sigma(u)(s-t)\norm{\lrate}^2.
\end{equation*}
Since $\epsilon>0$ is arbitrary, the result is now immediate.
\end{proof}

\proplowertranssystemissystem*
\begin{proof}
Consider any $t\in\realsnonneg$. Then $\mathcal{U}_{[t,t]}$ contains only a single degenerate sequence of time points $u$, which consists of the single time point $t$, and for this sequence $u$, we trivially have that $\Phi_u=I$ and $\sigma(u)=0$. Therefore, it follows from Definition~\ref{def:low_trans} that $L_t^t=I$.

Consider now any $t,r,s,\in\realsnonneg$ such that $t\leq r\leq s$. It remains to show that $L_t^s=L_t^rL_r^s$. If $t=r$ or $r=s$, this follows trivially from the first part of this proof. Hence, without loss of generality, we may assume that $t<r<s$.

Fix any $\delta>0$ such that $\delta\norm{\lrate}\leq1$. Since $r>t$, there is some $0<\Delta_1<\delta$ and $n_1\in\nats$ such that $\Delta_1 n_1=r-t$. Similarly, since $s>r$, there is some $0<\Delta_2<\delta$ and $n_2\in\nats$ such that $\Delta_2 n_2=s-r$. Furthermore, because of Propositions~\ref{lemma:normQsmallenough} and~\ref{lemma:compositioncoherence}, $(I+\Delta_1\lrate)^{n_1}$ and $(I+\Delta_1\lrate)^{n_1}$ are lower transition operators. Hence, we find that
\begin{align*}
&\norm{L_t^s-L_t^rL_r^s}
\leq
\norm{L_t^s-(I+\Delta_1\lrate)^{n_1}(I+\Delta_2\lrate)^{n_2}}
+
\norm{(I+\Delta_1\lrate)^{n_1}(I+\Delta_2\lrate)^{n_2}-L_t^rL_r^s}\\
&~~~~~~~~\leq
\norm{L_t^s-(I+\Delta_1\lrate)^{n_1}(I+\Delta_2\lrate)^{n_2}}
+
\norm{(I+\Delta_1\lrate)^{n_1}-L_t^r}
+
\norm{(I+\Delta_2\lrate)^{n_2}-L_r^s}\\
&~~~~~~~~\leq
\max\{\Delta_1,\Delta_2\}(s-t)\norm{\lrate}^2
+
\Delta_1(r-t)\norm{\lrate}^2
+
\Delta_2(s-r)\norm{\lrate}^2\\
&~~~~~~~~\leq
\delta(s-t)\norm{\lrate}^2
+
\delta(r-t)\norm{\lrate}^2
+
\delta(s-r)\norm{\lrate}^2
=
2\delta(s-t)\norm{\lrate}^2,
\end{align*}
where the second inequality follows from Lemma~\ref{lemma:recursive_lower_trans} and the third inequality follows from Lemma~\ref{lemma:limitboundonL}. Since $\delta>0$ can be chosen arbitrarily small, this implies that $\norm{L_t^s-L_t^rL_r^s}=0$, or equivalently, that $L_t^s=L_t^rL_r^s$.
\end{proof}

\proplowertransitionishomogeneous*
\begin{proof}
Consider any $t,s\in\realsnonneg$ such that $t\leq s$, any $\Delta\in\realsnonneg$, and any sequence $\{u_i\}_{i\in\nats}$ in $\mathcal{U}_{[t,s]}$ such that $\lim_{i\to\infty}\sigma(u_i)=0$. For any $i\in\nats$, with $u_i=t_0,\ldots,t_n$, we now define the sequence $u_i^*=(t_0+\Delta),(t_1+\Delta),\ldots,(t_n+\Delta)$. Then clearly, $\lim_{i\to\infty}\sigma(u_i^*)=0$ and, for all $i\in\nats$, $u_i^*\in\mathcal{U}_{[t+\Delta,s+\Delta]}$ and, due to Equation~\eqref{eq:aux_lower_trans}, $\Phi_{u_i}=\Phi_{u_i^*}$.
We now have that
\begin{equation*}
L_t^s=\lim_{i\to\infty}\Phi_{u_i}=\lim_{i\to\infty}\Phi_{u_i^*}=L_{t+\Delta}^{s+\Delta},
\end{equation*}
where the first and last equality follow from Theorem~\ref{theo:convergencelowerbound}.
\end{proof}

The following lemma provides a bound on how well the operator $L_t^{t+\Delta}$ can be approximated using the much simpler operator $(I+\Delta\lrate)$. Note that this lemma states the general version of Lemma~\ref{lemma:linearpartofexponential}, which until this point has remained unproven. To see that the result below indeed implies Lemma~\ref{lemma:linearpartofexponential}, simply recall from Section~\ref{sec:connections_rate} that any rate matrix $Q$ is also a lower transition rate operator. Since it follows from our comments in Section~\ref{sec:properties_lower_trans} that the lower transition operator $L_t^{t+\Delta}$ corresponding to such a $\lrate=Q$ satisfies $L_t^{t+\Delta}=e^{Q\Delta}$, Lemma~\ref{lemma:linearpartofexponential} clearly follows from the following result.

\begin{lemma}\label{lemma:quadraticboundonL}
Let $\lrate$ be a lower transition rate operator. Then for any $t,\Delta\in\realsnonneg$:
\begin{equation*}
\norm{L_t^{t+\Delta}-(I+\Delta\lrate)}\leq\Delta^2\norm{\lrate}^2.\vspace{5pt}
\end{equation*}
\end{lemma}
\begin{proof}
Fix any $\epsilon>0$. Because of Definition~\ref{def:low_trans}, there is some $\smash{u\in\mathcal{U}_{[t,t+\Delta]}}$ such that $\sigma(u)\norm{\lrate}\leq1$ and $\norm{L_t^s-\Phi_{u}}\leq\epsilon$. By combining this with Lemma~\ref{lemma:justthelinearpart}, it follows that
\begin{equation*}
\norm{L_t^s-(I+\Delta\lrate)}
\leq
\norm{L_t^s-\Phi_{u}}
+
\norm{\Phi_{u}-(I+\Delta\lrate)}
\leq\epsilon+
\Delta^2\norm{\lrate}^2.
\end{equation*}
Since $\epsilon>0$ is arbitrary, the result is now immediate.
\end{proof}

The final lemma of this section states how well we can approximate the operator $L_t^{t+\Delta}$ using the identity matrix $I$. Put differently, since we know from Proposition~\ref{prop:lower_trans_system_is_system} that $L_t^t=I$, this tells us how quickly $L_t^{t+\Delta}$ changes as we increase $\Delta$. This result is a generalised version of Lemma~\ref{lemma:linearboundonexponential}, which has so far remained unproven. The reason why Lemma~\ref{lemma:linearboundonexponential} is a special case is again because for $\lrate=Q$, we have that $L_t^{t+\Delta}=e^{Q\Delta}$.

\begin{lemma}\label{lemma:linearboundonL}
Let $\lrate$ be a lower transition rate operator. Then for any $t,\Delta\in\realsnonneg$:
\begin{equation*}
\norm{L_t^{t+\Delta}-I}\leq\Delta\norm{\lrate}.\vspace{5pt}
\end{equation*}
\end{lemma}
\begin{proof}
Fix any $n\in\nats$ and let $\Delta_n\coloneqq\nicefrac{\Delta}{n}$. Propositions~\ref{prop:lower_trans_system_is_system} and~\ref{prop:lower_transition_is_homogeneous} then imply that
\begin{equation*}
L_t^{t+\Delta}
=L_0^{\Delta}
=L_0^{n\Delta_n}
=L_0^{\Delta_n}
L_{\Delta_n}^{2\Delta_n}
\cdots
L_{(n-2)\Delta_n}^{(n-1)\Delta_n}
L_{(n-1)\Delta_n}^{n\Delta_n}
=L_0^{\Delta_n}
L_0^{\Delta_n}
\cdots
L_0^{\Delta_n}
L_0^{\Delta_n}
=(L_0^{\Delta_n})^n,
\end{equation*}
and therefore, it follows from Lemma~\ref{lemma:recursive_lower_trans} that
\begin{equation*}
\norm{L_t^{t+\Delta}-I}
=
\norm{(L_0^{\Delta_n})^n-I^n}
\leq n\norm{L_0^{\Delta_n}-I}
\leq n\norm{L_0^{\Delta_n}-(I+\Delta_n\lrate)}+n\Delta_n\norm{\lrate},
\end{equation*}
which, when combined with Lemma~\ref{lemma:quadraticboundonL}, implies that
\begin{equation*}
\norm{L_t^{t+\Delta}-I}
\leq
n\Delta_n^2\norm{\lrate}^2+n\Delta_n\norm{\lrate}
=\frac{1}{n}\Delta^2\norm{\lrate}^2+\Delta\norm{\lrate}.
\end{equation*}
Since $n\in\nats$ is arbitrary, the result is now immediate.
\end{proof}

\proplowertransitionhasderiv*
\begin{proof}
We first prove that $\frac{\partial}{\partial t}L_t^s=-\lrate L_t^s$, or equivalently, that
\begin{equation}\label{eq:lower_deriv_backward}
(\forall\epsilon>0)\,
(\exists\delta>0)\,
(\forall\Delta\,:\,0<\lvert\Delta\rvert <\delta,\,0\leq t+\Delta\leq s)~
\norm{\frac{L_{t+\Delta}^s-L_t^s}{\Delta}+\lrate L_t^s}<\epsilon.
\end{equation}
Fix any $\epsilon\in\realspos$, and choose any $\delta>0$ such that $3\delta\norm{\lrate}^2<\epsilon$. Consider now any $\Delta\in\reals$ such that $0<\abs{\Delta}<\delta$ and $0\leq t+\Delta\leq s$.
We will show that the inequality in Equation~\eqref{eq:lower_deriv_backward} holds.
Let $t'\coloneqq\max\{t,t+\Delta\}$. We then find that
\begin{align*}
\norm{L_{t+\Delta}^s-L_t^s+\Delta\lrate L_t^s}
=\norm{L_{t'}^s-L_{t'-\abs{\Delta}}^s+\abs{\Delta}\lrate L_{t}^s}
=\norm{L_{t'}^s-L_{t'-\abs{\Delta}}^{t'}L_{t'}^s+\abs{\Delta}\lrate L_{t}^{t'}L_{t'}^s},
\end{align*}
where the last equality follows from Proposition~\ref{prop:lower_trans_system_is_system} and because $t\leq t'\leq s$. Therefore, and because of \ref{N:normAB} and~\ref{LT:norm_at_most_one}, we find that
\begin{align*}
\norm{L_{t+\Delta}^s-L_t^s+\Delta\lrate L_t^s}
\leq\norm{I-L_{t'-\abs{\Delta}}^{t'}+\abs{\Delta}\lrate L_{t}^{t'}}\norm{L_{t'}^s}
\leq\norm{I-L_{t'-\abs{\Delta}}^{t'}+\abs{\Delta}\lrate L_{t}^{t'}},
\end{align*}
which implies that
\begin{align*}
\norm{L_{t+\Delta}^s-L_t^s+\Delta\lrate L_t^s}
&\leq\norm{I+\abs{\Delta}\lrate-L_{t'-\abs{\Delta}}^{t'}}
+
\abs{\Delta}\norm{\lrate L_{t}^{t'}-\lrate}\\
&\leq
\Delta^2\norm{\lrate}^2+\abs{\Delta}2\norm{\lrate}\norm{L_{t}^{t'}-I}\leq\Delta^2\norm{\lrate}^2+2\abs{\Delta}(t'-t)\norm{\lrate}^2,
\end{align*}
using Lemma~\ref{lemma:quadraticboundonL} and~\ref{LR:differenceofnorm} for the second inequality and Lemma~\ref{lemma:linearboundonL} for the third inequality. Hence, since $0\leq(t'-t)\leq\abs{\Delta}$, we find that $\smash{\norm{L_{t+\Delta}^s-L_t^s+\Delta\lrate L_t^s}\leq3\Delta^2\norm{\lrate}^2}$, which implies that indeed, as required,
\begin{equation*}
\norm{\frac{L_{t+\Delta}^s-L_t^s}{\Delta}+\lrate L_t^s}
=\frac{1}{\abs{\Delta}}\norm{L_{t+\Delta}^s-L_t^s+\Delta\lrate L_t^s}
\leq3\abs{\Delta}\norm{\lrate}^2
\leq3\delta\norm{\lrate}^2
<\epsilon.
\end{equation*}

Next, we prove that $\frac{\partial}{\partial s}L_t^s=\lrate L_t^s$, or equivalently, that
\begin{equation}\label{eq:lower_deriv_forward}
(\forall\epsilon>0)\,
(\exists\delta>0)\,
(\forall\Delta\,:\,0<\lvert\Delta\rvert<\delta,\,t\leq s+\Delta)~
\norm{\frac{L_{t}^{s+\Delta}-L_t^s}{\Delta}-\lrate\lbound_t^s }<\epsilon.
\end{equation}
Fix any $\epsilon>0$ and $\alpha>0$ and let $t^*\coloneqq t+\alpha$ and $s^*\coloneqq s+\alpha$. It then follows from Equation~\eqref{eq:lower_deriv_backward} that there is some $\delta^*>0$ such that
\begin{equation}\label{eq:lower_deriv_backward:star}
(\forall\Delta^*\,:\,0<\lvert\Delta^*\rvert <\delta^*,\,0\leq t^*+\Delta^*\leq s^*)~
\norm{\frac{L_{t^*+\Delta^*}^{s^*}-L_{t^*}^{s^*}}{\Delta^*}+\lrate L_{t^*}^{s^*}}<\epsilon.
\end{equation}
Let $\delta\coloneqq\min\{\delta^*,\alpha\}$. Consider now any $\Delta\in\reals$ such that $0<\abs{\Delta}<\delta$ and $t\leq s+\Delta$. We will prove that the inequality in Equation~\eqref{eq:lower_deriv_forward} holds. Let $\Delta^*\coloneqq-\Delta$. We then have that $0<\abs{\Delta^*}<\delta\leq\delta^*$ and that $t^*+\Delta^*=t+\alpha-\Delta\geq t+\alpha-\delta\geq t\geq 0$ and $t^*+\Delta^*=t+\alpha-\Delta\leq s+\Delta+\alpha-\Delta=s^*$, and therefore, we find that indeed, as required,
\begin{align*}
\norm{\frac{L_{t}^{s+\Delta}-L_t^s}{\Delta}-\lrate\lbound_t^s}
= \norm{\frac{L_{t^*+\Delta^*}^{s^*}-L_{t^*}^{s^*}}{-\Delta^*}-\lrate\lbound_{t^*}^{s^*}}
= \norm{\frac{L_{t^*+\Delta^*}^{s^*}-L_{t^*}^{s^*}}{\Delta^*}+\lrate\lbound_{t^*}^{s^*}}<\epsilon,
\end{align*}
where the first equality follows from Proposition~\ref{prop:lower_transition_is_homogeneous} and the final inequality follows from Equation~\eqref{eq:lower_deriv_backward:star}.
\end{proof}

\section{Proofs and Lemmas for the results in Section~\ref{sec:connections}}

Before giving the proof of Proposition~\ref{theorem:nonmarkov_single_var_lower_bounded}, we first prove some crucial steps of the proof separately, in the lemmas below.

First of all, recall from Proposition~\ref{prop:outerderivativebehaveslikelimit} that, for any $\epsilon\in\realspos$, there is some $\delta>0$ such that, for all $0<\Delta<\delta$, the transition matrix $\smash{T_{t,x_u}^{t+\Delta}}$ of a fixed stochastic process $\smash{P\in\wprocesses_{\rateset,\mathcal{M}}}$ can be approximated by $(I+\Delta Q)$ with an error of at most $\Delta \epsilon$, using a rate matrix $Q\in\overline{\partial}T_{t,x_u}^t\subseteq \mathcal{Q}$. Quite similarly, the following lemma states that, for any given time interval $[t,s]$, there exists a finite partition $u\in\mathcal{U}_{[t,s]}$, $u=t_0,\ldots,t_n$, such that the transition matrices $\smash{T_{t_i,x_u}^{t_{i+1}}}$ can all be approximated by $(I+\Delta_{i+1}Q_{i+1})$, for some $Q_{i+1}\in\rateset$ and with $\Delta_{i+1}=t_{i+1}-t_i$. 

The reason that this result does not follow trivially from Proposition~\ref{prop:outerderivativebehaveslikelimit} is because the $\delta$---and hence also $\Delta$---in Proposition~\ref{prop:outerderivativebehaveslikelimit} depends on the particular time point that is considered. For this reason, the intuitive idea of using Proposition~\ref{prop:outerderivativebehaveslikelimit} to first find some $\Delta_0$ and $Q_0$ such that $T_{t_0,x_u}^{t_1}=T_{t_0,x_u}^{t_0+\Delta}$ can be approximated by $I+\Delta_1Q_1$, and to then continue in this way to find some $\Delta_2$ and $Q_2$, and then some $\Delta_3$ and $Q_3$, and so on, is not feasible, because this process may continue indefinitely if $\sum_{i=1}^\infty\Delta_i$ is finite. In order to make this work, we need some kind of guarantee that it suffices to consider a finite number of $\Delta_i$, and this is exactly what the following lemma establishes.

\begin{lemma}\label{lemma:bound_on_linear_approx_partition}
Consider any $P\in\wprocesses_{\rateset,\,\mathcal{M}}$, any $0\leq t<s$, any $u\in\mathcal{U}_{<t}$ and any $x_u\in\states_u$. Then for all $\epsilon>0$ and $\delta>0$, there is some $v\in\mathcal{U}_{[t,s]}$ such that $\sigma(v)<\delta$ and, for all $i\in\{0,\dots,n-1\}$:
\begin{equation*}
(\exists Q\in\rateset)
~
\norm{
T^{t_{i+1}}_{t_i,\,x_u}-(I+\Delta_{i+1}Q)
}<\Delta_{i+1}\epsilon
\end{equation*}
\end{lemma}
\begin{proof}
Fix any $\epsilon>0$ and $\delta>0$. It then follows from Proposition~\ref{prop:outerderivativebehaveslikelimit} and Definition~\ref{def:consistent_process} that there is some $0<\delta^*<\min\{\delta,\nicefrac{1}{2}(s-t)\}$ such that, for all $0<\Delta<\delta^*$:
\begin{equation}\label{eq:epsilonboundsforboundsinlemma}
(\exists Q\in\rateset)
\norm{\frac{1}{\Delta}
(T^{t+\Delta}_{t,\,x_u}-I)-Q}<\epsilon
\text{~~and~~}
(\exists Q\in\rateset)
\norm{\frac{1}{\Delta}
(T^{s}_{s-\Delta,\,x_u}-I)-Q}<\epsilon.
\end{equation}
Let $t^*\coloneqq t+\delta^*$ and $s^*\coloneqq s-\delta^*$. Then clearly, $t<t^*<s^*<s$.
For any $r\in[t^*,s^*]$, it follows from Proposition~\ref{prop:outerderivativebehaveslikelimit} and Definition~\ref{def:consistent_process} that there is some $0<\delta_r<\delta^*$ such that, for all $0<\Delta<\delta_r$:
\begin{equation}\label{eq:epsilonboundsforlemma}
(\exists Q\in\rateset)
\norm{\frac{1}{\Delta}
(T^{r+\Delta}_{r,\,x_u}-I)-Q}<\epsilon
\text{~~and~~}
(\exists Q\in\rateset)
\norm{\frac{1}{\Delta}
(T^{r}_{r-\Delta,\,x_u}-I)-Q}<\epsilon.
\end{equation}
Let $U_r\coloneqq(r-\delta_r,r+\delta_r)$. Then the set $C\coloneqq\{U_r\colon r\in[t^*,s^*]\}$ is an open cover of $[t^*,s^*]$. By the Heine-Borel theorem, $C$ contains a finite subcover $C^*$ of $[t^*,s^*]$. Without loss of generality, we can take this subcover to be minimal, in the sense that if we remove any of its elements, it is no longer a cover. Let $m$ be the cardinality of $C^*$ and let $r_1<r_2<\dots<r_m$ be the ordered sequence of the midpoints of the intervals in $C^*$.

We will now prove that
\begin{equation}\label{eq:orderingofbounds}
r_i-\delta_{r_i}<r_j-\delta_{r_j}
\text{~~and~~}
r_i+\delta_{r_i}<r_j+\delta_{r_j}
\text{~~for all $1\leq i<j\leq m$.}
\end{equation}
Assume \emph{ex absurdo} that this statement is not true. Then this implies that there are $1\leq i<j\leq m$ such that either $r_i-\delta_{r_i}\geq r_j-\delta_{r_j}$ or $r_i+\delta_{r_i}\geq r_j+\delta_{r_j}$. If $r_i-\delta_{r_i}\geq r_j-\delta_{r_j}$, then since $i<j$ implies that $r_i<r_j$, it follows that $\delta_{r_j}\geq\delta_{r_i}+r_j-r_i>\delta_{r_i}$ and therefore, that $r_j+\delta_{r_j}>r_i+\delta_{r_i}$. 
Hence, we find that $U_{r_i}\subseteq U_{r_j}$. Since $C^*$ was taken to be a minimal cover, this is a contradiction.
Similarly, if $r_i+\delta_{r_i}\geq r_j+\delta_{r_j}$, then since $i<j$ implies that $r_i<r_j$, it follows that $\delta_{r_i}\geq\delta_{r_j}+r_j-r_i>\delta_{r_j}$ and therefore, that $r_i-\delta_{r_i}<r_j-\delta_{r_i}$. Hence, we find that $U_{r_j}\subseteq U_{r_i}$. Since $C^*$ was taken to be a minimal cover, this is again a contradiction. From these two contradictions, it follows that Equation~\eqref{eq:orderingofbounds} is indeed true.

Next, we prove that
\begin{equation}\label{eq:overlapasyouwantit}
r_{k+1}-\delta_{r_{k+1}}<r_k+\delta_{r_k}
\text{~~for all $k\in\{1,\dots,m-1\}$.}
\end{equation}
Assume \emph{ex absurdo} that this statement is not true or, equivalently, that there is some $k\in\{1,\dots,m-1\}$ such that $r_k+\delta_{r_k}\leq r_{k+1}-\delta_{r_{k+1}}$. For all $i\in\{k+1,\dots, m\}$, it then follows from Equation~\eqref{eq:orderingofbounds} that $r_k+\delta_{r_k}\leq r_i-\delta_{r_i}$, which implies that $r_k+\delta_{r_k}\notin U_{r_i}$. Furthermore, for all $i\in\{1,\dots,k\}$, it follows from Equation~\eqref{eq:orderingofbounds} that $r_i+\delta_{r_i}\leq r_k+\delta_{r_k}$, which again implies that $r_k+\delta_{r_k}\notin U_{r_i}$. 
Hence, for all $i\in\{1,\dots,m\}$, we have found that $r_k+\delta_{r_k}\notin U_{r_i}$. 
Since $C^*$ is a cover of $[t^*,s^*]$, this implies that $r_k+\delta_{r_k}\notin[t^*,s^*]$, which, since $r_k\in[t^*,s^*]$, implies that $r_k+\delta_{r_k}>s^*$. Hence, since we know from Equation~\eqref{eq:orderingofbounds} that $r_k-\delta_{r_k}<r_m-\delta_{r_m}$, it follows that $U_{r_m}\cap[t^*,s^*]\subseteq U_{r_k}\cap[t^*,s^*]$. This contradicts the fact that $C^*$ was taken to be a minimal cover, and therefore, Equation~\eqref{eq:overlapasyouwantit} must indeed be true.

For all $k\in\{1,\dots,m-1\}$, we now define $q_k\coloneqq\nicefrac{1}{2}(r_k+\delta_{r_k}+r_{k+1}-\delta_{r_{k+1}})$.
Using Equation~\eqref{eq:orderingofbounds}, it then follows that
\begin{equation*}
q_k<\frac{r_{k+1}+\delta_{r_{k+1}}+r_{k+1}-\delta_{r_{k+1}}}{2}=r_{k+1}
\text{~~and~~}
q_k>\frac{r_{k}+\delta_{r_{k}}+r_{k}-\delta_{r_{k}}}{2}=r_{k},
\end{equation*}
and Equation~\eqref{eq:overlapasyouwantit} trivially implies that $r_{k+1}-\delta_{r_{k+1}}<q_k<r_k+\delta_{r_k}$. Hence,
\begin{equation*}
r_k<q_k<r_k+\delta_{r_k}
\text{~~and~~}
r_{k+1}-\delta_{r_{k+1}}<q_k<r_{k+1}.
\end{equation*}
Due to Equation~\eqref{eq:epsilonboundsforlemma}, and with $\Delta^*_k\coloneqq q_k-r_k$ and $\Delta^{**}_k\coloneqq r_{k+1}-q_k$, this implies that
\begin{equation}\label{eq:epsilonboundsforqk}
(\exists Q\in\rateset)
\norm{\frac{1}{\Delta^*_k}
(T^{q_k}_{r_k,\,x_u}-I)-Q}<\epsilon
\text{~~and~~}
(\exists Q\in\rateset)
\norm{\frac{1}{\Delta^{**}_k}
(T^{r_{k+1}}_{q_k,\,x_u}-I)-Q}<\epsilon.
\end{equation}

For all $k\in\{1,\dots,m\}$, we now let $t_{2k}\coloneqq r_k$ and, for all $k\in\{1,\dots,m-1\}$, we let $t_{2k+1}\coloneqq q_k$. For the resulting sequence $t_2<t_3<\dots<t_{2m-1}<t_{2m}$, it then follows from Equation~\eqref{eq:epsilonboundsforqk} and~\ref{N:homogeneous} that, for all $i\in\{2,\dots,2m-1\}$:
\begin{equation}\label{eq:opencoverproofresult} 
(\exists Q\in\rateset)
~
\norm{
T^{t_{i+1}}_{t_i,\,x_u}-(I+\Delta_{i+1}Q)
}<\Delta_{i+1}\epsilon,
\end{equation}
with $\Delta_{i+1}\coloneqq t_{i+1}-t_i<\delta$.  

Next, since $C^*$ is a minimal cover, and because of Equation~\eqref{eq:orderingofbounds}, we know that $r_1-\delta_{r_1}<t^*\leq r_1=t_2$ and, since $\delta_{r_1}<\delta^*$, we also know that $r_1-\delta_{r_1}>t$. Therefore, it follows that there is some $t_1\in\reals$ such that $t<r_1-\delta_{r_1}<t_1<t^*\leq r_1$. If we now let $t_0\coloneqq t$, then $\Delta_1\coloneqq t_1-t_0<\delta^*$ and $\Delta_2\coloneqq t_2-t_1=r_1-t_1<\delta_{r_1}$, and therefore, it follows from Equations~\eqref{eq:epsilonboundsforboundsinlemma} and~\eqref{eq:epsilonboundsforlemma} and~\ref{N:homogeneous} that Equation~\eqref{eq:opencoverproofresult} is also true for $i=0$ and $i=1$.

Finally, again since $C^*$ is a minimal cover and because of Equation~\eqref{eq:orderingofbounds}, we know that $t_{2m}=r_m\leq s^*<r_m+\delta_{r_m}$ and, since $\delta_{r_m}<\delta^*$, we also know that $r_m+\delta_{r_m}<s$. Therefore, it follows that there is some $t_{2m+1}\in\reals$ such that $t_{2m}=r_m\leq s^*<t_{2m+1}<r_m+\delta_{r_m}<s$. If we now let $t_{2m+2}\coloneqq s$, then $\Delta_{2m+2}\coloneqq t_{2m+2}-t_{2m+1}<\delta^*$ and $\Delta_{2m+1}\coloneqq t_{2m+1}-t_{2m}=t_{2m+1}-r_m<\delta_{r_m}$, and therefore, it follows from Equations~\eqref{eq:epsilonboundsforboundsinlemma} and~\eqref{eq:epsilonboundsforlemma} and~\ref{N:homogeneous} that Equation~\eqref{eq:opencoverproofresult} is also true for $i=2m$ and $i=2m+1$.

Hence, we conclude that Equation~\eqref{eq:opencoverproofresult} holds for all $i\in\{0,1,\dots,2m,2m+1\}$. The result now follows by letting $n\coloneqq2m+2$.
\end{proof}

The following lemma provides a decomposition property for the expectation operator $\mathbb{E}_P$ that corresponds to a stochastic processes $P\in\processes$. For notational convenience, we express it in terms of transition matrices $T_{t,x_u}^s$, using Remark~\ref{remark:expectationT}. This result will end up being useful in the proof of Proposition~\ref{theorem:nonmarkov_single_var_lower_bounded}.

Roughly speaking, the result establishes that the (history-dependent) transition matrix $T_{t,x_u}^s$ of a stochastic process $P\in\processes$ can be decomposed into a number of other transition matrices corresponding to this $P$, such that---crucially---these individual transition matrices do not depend on the exact way that the decomposition was performed---that is, the other transition matrices in the decomposition---despite $P$ not necessarily being a Markov process.

\begin{lemma}\label{lemma:weirddecomposition}
Consider any $P\in\wprocesses$, any $t,s\in\realsnonneg$ such that $t<s$, any $u\in\mathcal{U}_{<t}$ and $x_u\in\states_u$ and any sequence $t=t_0<t_1<\cdots<t_n=s$, with $n\in\nats$.
Then for any $f\in\gamblesX$ and $x_t\in\states$:
\begin{equation*}
[T_{t,\,x_u}^sf](x_t)
=\Big[T_{t_0,\,x_u}^{t_1}\Bigg(\prod_{i=2}^{n}T_{t_{i-1},\,x_{u\cup\{t\}}}^{t_{i}}\Bigg)f\Big](x_t)
\end{equation*}
\end{lemma}
\begin{proof}
We provide a proof by induction. For $n=1$, the result holds trivially. So consider now any $n>1$ and assume that the result is true for $n-1$.
For any $g\in\gamblesX$, Remark~\ref{remark:expectationT} then implies that
\begin{align*}
[T_{t_0,\,x_u}^{t_2}g](x_t)
&=\mathbb{E}_P[g(X_{t_2})\vert\,X_{t}=x_t,X_u=x_u]\\
&=\sum_{x_{t_1}\in\states}\mathbb{E}_P[g(X_{t_2})\vert\,X_{t_1}=x_{t_1},X_{t}=x_t,X_u=x_u]P(X_{t_1}=x_{t_1}\vert\,X_{t}=x_t,X_u=x_u)\\[-8pt]
&\quad\quad\quad\quad\quad=\sum_{x_{t_1}\in\states}
[T_{t_1,\,x_{u\cup\{t\}}}^{t_{2}}g](x_{t_1})P(X_{t_1}=x_{t_1}\vert\,X_{t}=x_t,X_u=x_u)\\[-6pt]
&\quad\quad\quad\quad\quad\quad\quad\quad\quad\quad=[T_{t_0,\,x_u}^{t_1}
T_{t_1,\,x_{u\cup\{t\}}}^{t_{2}}g](x_t),
\end{align*}
and therefore, it follows that
\begin{align*}
[T_{t,\,x_u}^sf](x_t)
&=\Big[T_{t_0,\,x_u}^{t_2}\Bigg(\prod_{i=3}^{n}T_{t_{i-1},\,x_{u\cup\{t\}}}^{t_{i}}\Bigg)f\Big](x_t)
\\
 &= \Big[T_{t_0,\,x_u}^{t_1}
T_{t_1,\,x_{u\cup\{t\}}}^{t_{2}}
\Bigg(\prod_{i=3}^{n}T_{t_{i-1},\,x_{u\cup\{t\}}}^{t_{i}}\Bigg)f\Big](x_t) =\Big[T_{t_0,\,x_u}^{t_1}\Bigg(\prod_{i=2}^{n}T_{t_{i-1},\,x_{u\cup\{t\}}}^{t_{i}}\Bigg)f\Big](x_t),
\end{align*}
using the induction hypothesis for the first equality.
\end{proof}

The final two results that we require have to do with the connection between $\rateset$ and $\lrate$. Due to the definition of the lower transition rate operator $\lrate$ corresponding to a given $\rateset$---see Equation~\eqref{eq:correspondinglowertrans}---it holds for all $\Delta\in\realsnonneg$, all $Q\in\rateset$, and all $f\in\gamblesX$ that $(I+\Delta Q)f \geq (I+\Delta\lrate)f$. Lemma~\ref{lemma:productofQsdominatesproductoflrates} establishes that this inequality extends to compositions of such operators $(I+\Delta Q)$ and $(I+\Delta\lrate)$, provided that we impose an upper bound on $\Delta$. In orde to prove this Lemma~\ref{lemma:productofQsdominatesproductoflrates}, we also require Lemma~\ref{lemma:normQboundedbynormLQ}, which states that $\norm{Q}$ is bounded above by $\norm{\lrate}$.

\begin{lemma}\label{lemma:normQboundedbynormLQ}
Consider a non-empty bounded set $\rateset$ of rate matrices and let $\lrate$ be the corresponding lower transition rate operator. Then for all $Q\in\rateset$, we have that $\norm{Q}\leq\norm{\lrate}$.
\end{lemma}
\begin{proof}
Consider any $f\in\gamblesX$ such that $\norm{f}=1$. It then follows from Equation~\eqref{eq:correspondinglowertrans} that $Qf\geq\lrate f$ and, due to the linearity of $Q$, also that $Qf=-Q(-f)\leq-\lrate(-f)$. Hence, we find that $\lrate f\leq Q f\leq-\lrate(-f)$, which implies that
\begin{equation}\label{eq:lemma:normQboundedbynormLQ}
\norm{Qf}\leq\max\{\norm{\lrate f},\norm{-\lrate(-f)}\}=\max\{\norm{\lrate f},\norm{\lrate(-f)}\}.
\end{equation}
Since $\norm{f}=1$, it follows from Equation~\eqref{eq:operatornorm} that $\norm{\lrate f}\leq\norm{\lrate}$, and similarly, since $\norm{-f}=\norm{f}=1$, we also find that $\norm{\lrate(-f)}\leq\norm{\lrate}$. By combining this with Equation~\eqref{eq:lemma:normQboundedbynormLQ}, we find that $\norm{Qf}\leq\norm{\lrate}$, and since this is true for every $f\in\gamblesX$ such that $\norm{f}=1$, it now follows from Equation~\eqref{eq:operatornorm} that $\norm{Q}\leq\norm{\lrate}$.
\end{proof}

\begin{lemma}\label{lemma:productofQsdominatesproductoflrates}
Consider a non-empty bounded set $\rateset$ of rate matrices, let $\smash{\lrate}$ be the corresponding lower transition rate operator, and consider any $\delta\in\realspos$ such that $\delta\norm{\lrate}\leq1$. Now fix any $n\in\nats$ and, for all $i\in\{1,\dots,n\}$, consider some $0\leq\Delta_i\leq\delta$ and $Q_i\in\rateset$. Then for any $f\in\gamblesX$:
\begin{equation*}
\prod_{i=1}^n(I+\Delta_iQ_i)f
\geq
\prod_{i=1}^n(I+\Delta_i\lrate)f.
\end{equation*}
\end{lemma}
\begin{proof}
We provide a proof by induction. For $n=1$, the result follows trivially from Equation~\eqref{eq:correspondinglowertrans}. Consider now any $n>1$ and assume that the result is true for $n-1$. Since Lemma~\ref{lemma:normQboundedbynormLQ} implies that $\Delta_1\norm{Q_1}\leq\Delta_1\norm{\lrate}\leq\delta\norm{\lrate}\leq1$, it then follows from Proposition~\ref{prop:stochastic_from_rate_matrix} that $I+\Delta_1Q_1$ is a transition matrix, and therefore, as noted in Section~\ref{subsec:lowertrans_rate}, also a lower transition matrix, which therefore satisfies~\ref{LT:monotonicity}. We now find that
\begin{equation*}
\prod_{i=1}^n(I+\Delta_iQ_i)f
\geq
(I+\Delta_1Q_1)\prod_{i=2}^n(I+\Delta_i\lrate)f
\geq
\prod_{i=1}^n(I+\Delta_i\lrate)f,
\end{equation*}
where the first inequality follows from the induction hypothesis and~\ref{LT:monotonicity}, and the second inequality follows from Equation~\eqref{eq:correspondinglowertrans}.
\end{proof}

\theoremnonmarkovsinglevarlowerbounded*
\begin{proof}
This result is trivial if $t=s$. Hence, without loss of generality, we may assume that $t<s$. Fix any $\epsilon>0$ and let $C\coloneqq(s-t)$. Choose any $\epsilon_1>0$ such that $\epsilon_1\norm{f}<\nicefrac{\epsilon}{2}$ and any $\epsilon_2>0$ such that $\epsilon_2 C\norm{f}<\nicefrac{\epsilon}{2}$.

Due to Theorem~\ref{theo:convergencelowerbound}, there is some $\delta\in\realspos$ such that $\delta\norm{\lrate}\leq1$ and
\begin{equation}\label{eq:theorem:nonmarkov_single_var_lower_bounded}
(\forall v\in\mathcal{U}_{[t,s]}\,:\,\sigma(v)\leq\delta) \norm{L_{t}^s - \Phi_v} \leq \epsilon_1,
\end{equation}
with $\Phi_v$ as in Equation~\eqref{eq:aux_lower_trans}.
Since $\smash{P\in\wprocesses_\rateset}$, it now follows from Proposition~\ref{prop:outerderivativebehaveslikelimit} that there is some $0<\Delta_1<\min\{\delta,C\}$ such that
\begin{equation*}
(\exists Q_1\in\overline{\partial}_+T_{t,\,x_u}^{t}\subseteq\rateset)~
\norm{T_{t_0,\,x_u}^{t_1} - (I+\Delta_1 Q_1)}
=
\norm{T_{t,\,x_u}^{t+\Delta_1} - (I+\Delta_1 Q_1)} < \Delta_1\epsilon_2,
\end{equation*}
with $t_0\coloneqq t$ and $t_1\coloneqq t+\Delta_1$.
Furthermore, since $P\in\wprocesses_\rateset$, and because $\Delta_1<C$ implies that $t_1t+\Delta_1<s$, it follows from Lemma~\ref{lemma:bound_on_linear_approx_partition} that there is some $v\in\mathcal{U}_{[t_1,s]}$ such that $\sigma(v)<\delta$, with $v=t_1,\ldots,t_n$ and $t_n=s$, and such that for all $i\in\{2,\ldots,n\}$, with $\Delta_i\coloneqq t_i-t_{i-1}$:
\begin{equation*}
(\exists Q_i\in\rateset)\norm{T_{t_{i-1},x_{u\cup \{t\}}}^{t_{i}} - (I+\Delta_{i}Q_i)} < \Delta_{i}\epsilon_2.\vspace{5pt}
\end{equation*}
Since $\Delta_1<\delta$ and, for all $i\in\{2,\dots,n\}$, $\Delta_i\leq\sigma(v)<\delta$, we know that, for all $i\in\{1,\dots,n\}$, $\Delta_i<\delta$ and therefore also, using Lemma~\ref{lemma:normQboundedbynormLQ}, that $\Delta_i\norm{Q_i}\leq\delta\norm{\lrate}\leq1$.
Therefore, we find that
\begin{multline*}
\abs{[T_{t,\,x_u}^sf](x_t)
-\left[\left(\prod_{i=1}^n(I+\Delta_iQ_i)\right)f\right](x_t)}\\
\begin{aligned}
&=\abs{\Big[T_{t_0,\,x_u}^{t_1}\Bigg(\prod_{i=2}^{n}T_{t_{i-1},\,x_{u\cup\{t\}}}^{t_{i}}\Bigg)f\Big](x_t)
-\left[\left(\prod_{i=1}^n(I+\Delta_iQ_i)\right)f\right](x_t)}\\
&\leq\norm{T_{t_0,\,x_u}^{t_1}\Bigg(\prod_{i=2}^{n}T_{t_{i-1},\,x_{u\cup\{t\}}}^{t_{i}}\Bigg)
-\prod_{i=1}^n(I+\Delta_iQ_i)}\norm{f}\\
&\leq
\norm{T_{t_0,\,x_u}^{t_1}-(I+\Delta_1Q_1)}\norm{f}
+\sum_{i=2}^{n}
\norm{T_{t_{i-1},\,x_{u\cup\{t\}}}^{t_{i}}-(I+\Delta_iQ_i)}\norm{f}\\
&<\sum_{i=1}^n\Delta_i\epsilon_2\norm{f}=C\epsilon_2\norm{f}<\frac{\epsilon}{2},
\end{aligned}
\end{multline*}
where the equality follows from Lemma~\ref{lemma:weirddecomposition}, the first inequality follows from the properties of $\norm{\cdot}$, and the second inequality follows from Lemma~\ref{lemma:recursive} and Proposition~\ref{prop:stochastic_from_rate_matrix} and, we also find that
\begin{align*}
\abs{[L_t^sf](x_t)-\left[\left(\prod_{i=1}^n(I+\Delta_i\lrate)\right)f\right](x_t)}
&\leq
\norm{L_t^s-\prod_{i=1}^n(I+\Delta_i\lrate)}\norm{f}\leq\epsilon_1\norm{f}<\frac{\epsilon}{2},
\end{align*}
using Equation~\eqref{eq:theorem:nonmarkov_single_var_lower_bounded} to establish the second inequality. Hence, Lemma~\ref{lemma:productofQsdominatesproductoflrates} implies that
\begin{multline*}
[L_{t}^s f](x_t)
<\left[\left(\prod_{i=1}^n(I+\Delta_i\lrate)\right)f\right](x_t)+\frac{\epsilon}{2}\\
\leq
\left[\left(\prod_{i=1}^n(I+\Delta_i Q_i)\right)f\right](x_t)+\frac{\epsilon}{2}
<[T_{t,\,x_u}^sf](x_t)+\epsilon.
\end{multline*}\\[-0pt]
Since $\epsilon>0$ was arbitrary, this allows us to infer that $[L_{t}^s f](x_t)\leq [T_{t,\,x_u}^sf](x_t)$. The result now follows because $[T_{t,x_u}^sf](x_t)=\mathbb{E}_P[f(X_s)\vert X_t=x_t,X_u=x_u]$, as noted in Remark~\ref{remark:expectationT}.
\end{proof}

The following lemma is required for the proof of Proposition~\ref{theorem:lower_markov_bound_is_tight}. Recall from the definition of the lower envelope $\lrate$ of a given set $\rateset$ of rate matrices, that for any $f\in\gamblesX$, and $x\in\states$ and any $\epsilon>0$, there is some $Q\in\rateset$ such that $[Qf](x)<[\lrate f](x)+\epsilon$. In other words, $[\lrate f](x)$ can be approximated arbitrarily closely using the elements of $\rateset$. The following lemma establishes that the \emph{entire} function $\lrate f$ can be approximated arbitrarily closely by a single $Q\in\rateset$, whenever $\rateset$ has separately specified rows.

\begin{lemma}\label{lemma:rateset_has_arginf}
Let $\rateset$ be an arbitrary non-empty bounded set of rate matrices that has separately specified rows, with corresponding lower transition rate operator $\lrate$. Then for any $f\in\gamblesX$ and $\epsilon\in\realspos$, there exists a $Q\in\rateset$ such that
\begin{equation*}
\norm{\lrate f - Qf} < \epsilon\,.
\end{equation*}
\end{lemma}
\begin{proof}
This is immediate from the definition of the lower envelope of $\rateset$, as given by Equation~\eqref{eq:correspondinglowertrans}, and the fact that $\rateset$ has separately specified rows.
\end{proof}

\theoremlowermarkovboundistight*
\begin{proof}
Fix any $t,s\in\realsnonneg$ such that $t\leq s$, any $f\in\gamblesX$ and any $\epsilon\in\realspos$. If $t=s$, the result is trivial. Hence, without loss of generality, we may assume that $t<s$. Let $C\coloneqq (s-t)$, choose any $\epsilon^*>0$ such that $\epsilon^*C<\epsilon$, choose any $\epsilon_1,\epsilon_2,\epsilon_3>0$ such that $\epsilon_1+\epsilon_2+\epsilon_3<\epsilon^*$, choose any $\delta>0$ such that $\smash{\delta\norm{\rateset}^2\norm{f}<\epsilon_3}$ and $\smash{\delta\norm{\lrate}^2\norm{f}<\epsilon_1}$ (this is always possible because of~\ref{LR:normlratefinite}), and consider any $u\in\mathcal{U}_{[t,s]}$ such that $\sigma(u)<\delta$, with $u=t_0,\ldots,t_n$ and $n\in\nats$. 

Fix any $i\in\{1,\dots,n\}$ and let $g_i\coloneqq L_{t_i}^{t_n}f$. It then follows from Lemma~\ref{lemma:rateset_has_arginf} that there is some $Q_i\in\rateset$ such that $\norm{\lrate g_i-Q_i g_i}<\epsilon_2$ and, due to~\ref{N:normAf} and~\ref{LT:norm_at_most_one}, we also know that $\norm{g_i}=\norm{L_{t_i}^{t_n}f}\leq\norm{L_{t_i}^{t_n}}\norm{f}\leq\norm{f}$.
Hence, we find that 
\begin{align}
&\norm{L_{t_{i-1}}^{t_i}g_i - e^{Q_i\Delta_i}g_i}\notag\\
&\quad\quad\quad\leq \norm{L_{t_{i-1}}^{t_i}g_i - \left[I+\Delta_i\lrate\right]g_i} 
+\norm{\Delta_i\lrate g_i-\Delta_iQ_ig_i}
+ \norm{\left[I+\Delta_iQ_i\right]g_i - e^{Q_i\Delta_i}g_i} \notag\\
&\quad\quad\quad\leq \norm{L_{t_{i-1}}^{t_i} - \left[I+\Delta_i\lrate\right]}\norm{g_i} 
+\Delta_i\norm{\lrate g_i-Q_ig_i}
+ \norm{I+\Delta_iQ_i - e^{Q_i\Delta_i}}\norm{g_i} \notag\\
&\quad\quad\quad< \norm{L_{t_{i-1}}^{t_i} - \left[I+\Delta_i\lrate\right]}\norm{f} 
+\Delta_i\epsilon_2
+ \norm{I+\Delta_iQ_i - e^{Q_i\Delta_i}}\norm{f} \notag\\
&\quad\quad\quad\leq
\Delta_i^2\norm{\lrate}^2\norm{f}
+\Delta_i\epsilon_2
+
\Delta_i^2\norm{Q_i}^2\norm{f}\notag\\
&\quad\quad\quad\leq
\Delta_i(
\delta\norm{\lrate}^2\norm{f}
+\epsilon_2
+
\delta\norm{\rateset}^2\norm{f})
<\Delta_i(\epsilon_1+\epsilon_2+\epsilon_3)<\Delta_i\epsilon^*,\label{eq:theorem:lower_markov_bound_is_tight:localbound}
\end{align}\\[-8pt]
where the second inequality holds because of~\ref{N:normAf}, the third inequality holds because $\norm{g_i}\leq\norm{f}$ and $\norm{\lrate g_i-Q_i g_i}<\epsilon_2$, the fourth inequality holds because of Lemmas~\ref{lemma:quadraticboundonL} and~\ref{lemma:linearpartofexponential} and where the fifth inequality holds because $\Delta_i\leq\sigma(u)<\delta$.

Let $Q_0$ and $Q_{n+1}$ be two arbitrary elements of $\rateset$ and, for all $i\in\{0,\ldots,n+1\}$, let $\mathcal{T}_{Q_i}$ denote the transition matrix system corresponding to $Q_i$, as in Definition~\ref{def:systemfromQ}. Then, by Proposition~\ref{prop:nonhomogeneous_in_process_set}, there is some $P\in\wmprocesses_{\rateset,\,\mathcal{M}}$ with transition matrix system $\mathcal{T}_P$, such that
\begin{equation*}
\mathcal{T}_P = \mathcal{T}_{Q_0}^{[0,t_0]}\otimes \mathcal{T}_{Q_1}^{[t_0,t_1]}\otimes \cdots \otimes \mathcal{T}_{Q_n}^{[t_{n-1},t_n]} \otimes \mathcal{T}_{Q_n}^{[t_{n},\infty)}\,.
\end{equation*}

Due to Equation~\eqref{eq:theorem:lower_markov_bound_is_tight:localbound}, we know that the transition matrices of this process $P$ satisfy
\begin{equation}\label{eq:theorem:lower_markov_bound_is_tight:localbounds}
(\forall i\in\{1,\ldots,n\})~
\norm{L_{t_{i-1}}^{t_i}g_i - T_{t_{i-1}}^{t_i}g_i}= \norm{L_{t_{i-1}}^{t_i}g_i - e^{Q_i\Delta_i}g_i}<\Delta_i\epsilon^*.
\end{equation}
Furthermore, we also know that
\begin{align*}
\norm{L_{t_0}^{t_n}g_n - T_{t_0}^{t_n}g_n}
&=\norm{L_{t_0}^{t_{n-1}}L_{t_{n-1}}^{t_n}g_n - T_{t_0}^{t_{n-1}}L_{t_{n-1}}^{t_n}g_n+T_{t_0}^{t_{n-1}}(L_{t_{n-1}}^{t_n}g_n - T_{t_{n-1}}^{t_n}g_n)} \\ 
 &\leq \norm{L_{t_0}^{t_{n-1}}L_{t_{n-1}}^{t_n}g_n - T_{t_0}^{t_{n-1}}L_{t_{n-1}}^{t_n}g_n} + \norm{T_{t_0}^{t_{n-1}}}\norm{L_{t_{n-1}}^{t_n}g_n - T_{t_{n-1}}^{t_n}g_n} \\
 &\leq \norm{L_{t_0}^{t_{n-1}}g_{n-1} - T_{t_0}^{t_{n-1}}g_{n-1}} + \norm{L_{t_{n-1}}^{t_n}g_n - T_{t_{n-1}}^{t_n}g_n},
\end{align*}
using Proposition~\ref{prop:lower_trans_system_is_system} and Equation~\eqref{eq:markovintermsofmatrices} for the first inequality, and~\ref{LT:norm_at_most_one} for the second one.
Similarly, we also find that
\begin{equation*}
\norm{L_{t_0}^{t_{n-1}}g_{n-1} - T_{t_0}^{t_{n-1}}g_{n-1}}
\leq \norm{L_{t_0}^{t_{n-2}}g_{n-2} - T_{t_0}^{t_{n-2}}g_{n-2}} + \norm{L_{t_{n-2}}^{t_{n-1}}g_{n-1} - T_{t_{n-2}}^{t_{n-1}}g_{n-1}}.
\end{equation*}
By continuing in this way---applying induction---we eventually find that
\begin{equation*}
\norm{L_t^sf - T_t^sf}=\norm{L_{t_0}^{t_n}g_n - T_{t_0}^{t_n}g_n} 
\leq \sum_{i=1}^{n} \norm{L_{t_{i-1}}^{t_i}g_i - T_{t_{i-1}}^{t_i}g_i}\leq\sum_{i=1}^n\Delta_i\epsilon^*=C\epsilon^*<\epsilon,
\end{equation*}
using Equation~\eqref{eq:theorem:lower_markov_bound_is_tight:localbounds} to establish the second inequality. The result now follows directly from Remark~\ref{remark:expectationT}, which states that $[T_t^sf](x_t)=\mathbb{E}_P[f(X_s)\vert X_t=x_t]$ for all $x_t\in\states$.
\end{proof}

\corloweroperatorisinfimum*
\begin{proof}
Fix any $\epsilon>0$.
It then follows from Proposition~\ref{theorem:lower_markov_bound_is_tight} that there is some $P\in\wmprocesses_{\rateset,\,\mathcal{M}}$ such that $\mathbb{E}_P[f(X_s)\vert\,X_t=x_t] < [L_t^sf](x_t)+\epsilon$, which implies that
\begin{align*}
\underline{\mathbb{E}}^{\mathrm{WM}}_{\rateset,\,\mathcal{M}}\left[f(X_s)\,\vert\,X_t=x_t,X_u=x_u\right]
&\leq
\mathbb{E}_P[f(X_s)\vert\,X_t=x_t,X_u=x_u]\\
&=\mathbb{E}_P[f(X_s)\vert\,X_t=x_t]
< [L_t^sf](x_t)+\epsilon,
\end{align*}
using Equation~\eqref{eq:lowerexp3} for the first inequality and the Markov property of $P$ for the equality. Since $\epsilon>0$ is arbitrary, it follows that  $\underline{\mathbb{E}}^{\mathrm{WM}}_{\rateset,\,\mathcal{M}}\left[f(X_s)\,\vert\,X_t=x_t,X_u=x_u\right]\leq[L_t^sf](x_t)$.
This implies the result because we also have that
\begin{align*}
\left[L_t^sf\right](x_t)
&\leq\inf\left\{\mathbb{E}_P[f(X_s)\,\vert\,X_t=x_t,X_u=x_u]\colon P\in\wprocesses_{\rateset,\,\mathcal{M}}\right\}\\
&=\underline{\mathbb{E}}^{\mathrm{W}}_{\,\rateset,\,\mathcal{M}}[f(X_s)\,\vert\,X_t=x_t,X_u=x_u] 
\leq
\underline{\mathbb{E}}^{\mathrm{WM}}_{\,\rateset,\,\mathcal{M}}[f(X_s)\,\vert\,X_t=x_t,X_u=x_u],
\end{align*}
where the first inequality follows from Proposition~\ref{theorem:nonmarkov_single_var_lower_bounded} and the last inequality follows from Proposition~\ref{prop:lower_exp_markov_bounded_by_nonmarkov}.
\end{proof}

\theodominatingrateprocessesmaxset*
\begin{proof}
Clearly, it suffices to prove that for any $P\in\processes$ that is not well-behaved or not consistent with $\rateset_{\lrate}$, there are $t,s\in\realsnonneg$ with $t\leq s$, $u\in\mathcal{U}_{<t}$, $x_u\in\states_u$, $x_t\in\states$ and $f\in\gamblesX$ such that
\begin{equation}\label{eq:theo:dominating_rate_processes_max_set:toprove}
\mathbb{E}_P[f(X_s)\,\vert\,X_t=x_t,X_u=x_u] < [L_t^sf](x_t).\vspace{5pt}
\end{equation}

We start with the case that $P$ is not well-behaved. Fix any $\epsilon>0$ and let $C\coloneqq\norm{\lrate}+\epsilon$. It then follows from Proposition~\ref{prop:stochasticprocess:simpleproperties} that there are $t,s\in\realsnonneg$ with $t< s$, $u\in\mathcal{U}_{<t}$ and $x_u\in\states_u$ such that $\nicefrac{1}{\Delta}\norm{T_{t,\,x_u}^s-I}>C$ and $\smash{\Delta\norm{\lrate}^2<\epsilon}$, with $\Delta\coloneqq s-t>0$. Let $Q\coloneqq\nicefrac{1}{\Delta}(T_{t,\,x_u}^s-I)$. Since $\norm{Q}>C$, it then follows that there is some $f'\in\gamblesX$ such that $\norm{f'}=1$ and $\norm{Qf'}>C$, which in turn implies that there is some $x_t\in\states$ such that $\abs{[Qf'](x_t)}>C$. If $[Qf'](x_t)<0$, we let $f\coloneqq f'$, and if $[Qf'](x_t)>0$, we let $f\coloneqq -f'$. Clearly, this implies that $\norm{f}=1$ and $[Qf](x_t)<-C$. From $\norm{f}=1$, it furthermore follows that $[\lrate f](x_t)\geq-\norm{\lrate f}\geq-\norm{\lrate}$, and therefore, we we find that $[Qf](x_t)<-\norm{\lrate}-\epsilon\leq[\lrate f](x_t)-\epsilon$, which implies that $[(I+\Delta Q)f](x_t)\leq [(I+\Delta\lrate)f](x_t)-\Delta\epsilon$. Hence, since we also know that
\begin{align*}
\abs{[L_t^sf](x_t)-[(I+\Delta\lrate)f](x_t)}
\leq
\norm{L_t^s-(I+\Delta\lrate)}\leq\Delta^2\norm{\lrate}^2<\Delta\epsilon,
\end{align*}
where we use $\norm{f}=1$ for the first inequality and Lemma~\ref{lemma:quadraticboundonL} for the second inequality, it follows that
\begin{align*}
[T_{t,\,x_u}^sf](x_t)
=[(I+\Delta Q)f](x_t)
\leq[(I+\Delta \lrate)f](x_t)-\Delta\epsilon
<[L_t^sf](x_t),
\end{align*}
which, because of Remark~\ref{remark:expectationT}, implies that Equation~\eqref{eq:theo:dominating_rate_processes_max_set:toprove} holds.

Next, we consider the case that $P$ is well behaved, but not consistent with $\rateset_{\lrate}$. In that case, it follows from Definition~\ref{def:consistent_process} that there are $t^*\in\realsnonneg$, $u\in\mathcal{U}_{<t}$ and $x_u\in\states_u$ such that $\smash{\overline{\partial}T_{t^*,\,x_u}^{t^*} \not\subseteq \rateset_{\lrate}}$. Since we know that $\smash{\overline{\partial}T_{t^*,\,x_u}^{t^*}}$ is a non-empty set of rate matrices because of Proposition~\ref{prop:boundednon-emptyandclosed}, this implies the existence of a rate matrix $Q^*\in\smash{\overline{\partial}}T_{t^*,x_u}^{t^*}$ such that $Q^*\notin\rateset_{\lrate}$. Furthermore, since $Q^*\notin\rateset_{\lrate}$, Equation~\eqref{eq:dominatingratematrices} implies that there are $f'\in\gamblesX$ and $x_t\in\states$ such that $[Q^*f'](x_t)<[\lrate f'](x_t)$. Clearly, this implies that $f'\neq0$, and therefore, that $\norm{f'}>0$. If we now let $f\coloneqq\nicefrac{1}{\norm{f'}}f'$, then $\norm{f}=1$, and furthermore, because of the linearity of $Q^*$ and the non-negative homogeneity of $\lrate$, it follows that $[Q^*f](x_t)=\nicefrac{1}{\norm{f'}}[Q^*f'](x_t)<\nicefrac{1}{\norm{f'}}[\lrate f'](x_t)=[\lrate f](x_t)$. Consider now any $\epsilon>0$ such that $\smash{[Q^*f](x_t)\leq[\lrate f](x_t)-2\epsilon}$ (this is now clearly possible). Since $\smash{Q^*\in\smash{\overline{\partial}}T_{t^*,\,x_u}^{t^*}}$, it then follows from Definition~\ref{def:outerpartialderivatives} that there are $t,s\in\realsnonneg$ such that $u<t<s$, $\norm{\nicefrac{1}{\Delta}(T_{t,\,x_u}^s-I)-Q^*}\leq\epsilon$ and $\smash{\Delta\norm{\lrate}^2<\epsilon}$, with $\Delta\coloneqq s-t>0$. Let $Q\coloneqq\nicefrac{1}{\Delta}(T_{t,\,x_u}^s-I)$. Since $\norm{Q-Q^*}\leq\epsilon$ and $\norm{f}=1$, it then follows that $[Qf](x_t)\leq[Q^*f](x_t)+\epsilon\leq[\lrate f](x_t)-\epsilon$. Hence, using the exact same argument as in the first part of this proof, we find that Equation~\eqref{eq:theo:dominating_rate_processes_max_set:toprove} holds.
\end{proof}

\propapproximationerrorbound*
\begin{proof}
Define $h\coloneqq f-\min f-\nicefrac{1}{2}\norm{f}_\mathrm{v}$. Then
\begin{equation*}
\max h
=\max f-\min f-\nicefrac{1}{2}\norm{f}_\mathrm{v}
=\norm{f}_\mathrm{v}-\nicefrac{1}{2}\norm{f}_\mathrm{v}
=\nicefrac{1}{2}\norm{f}_\mathrm{v}
\end{equation*}
and
\begin{equation*}
\min h
=\min f-\min f-\nicefrac{1}{2}\norm{f}_\mathrm{v}
=-\nicefrac{1}{2}\norm{f}_\mathrm{v},
\end{equation*}
and therefore, 
\begin{equation*}
\norm{h}
\coloneqq
\max\{\abs{h(x)}\colon x\in\states\}
=\nicefrac{1}{2}\norm{f}_\mathrm{v}.
\end{equation*}
Now let $u\in\mathcal{U}_{[t,s]}$ be such that $u=t_0,t_1,\ldots,t_n$, where, for all $i\in\{0,1,\ldots,n\}$, $t_i\coloneqq t + i\Delta$. Since $\Delta=\nicefrac{(s-t)}{n}$, we then have that $t_0=t$, $t_n=s$, $\sigma(u)=\Delta$ and $\Phi_u= \prod_{i=1}^n(I+\Delta\lrate)$.
Furthermore, since $n\geq (s-t)\norm{\lrate}$, we also know that $\Delta\norm{\lrate}=\nicefrac{(s-t)}{n}\norm{\lrate}\leq 1$. Hence, we find that
\begin{align*}
\norm{L_t^sh - \prod_{i=1}^n(I+\Delta\lrate)h} = \norm{L_t^sh - \Phi_uh} &\leq \norm{L_t^s - \Phi_u}\norm{h} \\[-8pt]
 &\leq \sigma(u)(s-t)\norm{\lrate}^2\norm{h} = \Delta(s-t)\norm{\lrate}^2\frac{\norm{f}_\mathrm{v}}{2}\\
 &= 
\frac{\epsilon}{n}\frac{1}{2\epsilon}(s-t)^2\norm{\lrate}^2\norm{f}_\mathrm{v}
 \leq \epsilon.
\end{align*}
where the first inequality follows from Property~\ref{N:normAf}, the second inequality follows from Lemma~\ref{lemma:limitboundonL} and the final inequality follows from our lower bound on $n$. The result is now immediate because $L_t^s$ and $\prod_{i=1}^n(I+\Delta\lrate)$ are both lower transition operators---for the latter, this follows from Propositions~\ref{lemma:normQsmallenough} and~\ref{lemma:compositioncoherence}---which implies that
\begin{equation*}
\norm{L_t^sh - \prod_{i=1}^n(I+\Delta\lrate)h}
=
\norm{L_t^sf - \prod_{i=1}^n(I+\Delta\lrate)f}
\end{equation*}
because of~\ref{LT:constantadditivity}.
\end{proof}

\section{Proofs for the results in Section~\ref{sec:funcs_multi_time_points}}

\corinfworksforsinglefuturevar*
\begin{proof}
Because of Equations~\eqref{eq:fixxu} and~\eqref{eq:applyLtolargerfunctions}, this result is an immediate consequence of Corollary~\ref{cor:lower_operator_is_infimum}. 
\end{proof}

\corcompositionlowertrans*
\begin{proof}
As explained in the main text of Section~\ref{sec:decomposition}, this result is an immediate consequence of Theorem~\ref{theorem:decomposition_multivar} and Corollary~\ref{cor:inf_works_for_single_future_var}.
\end{proof}

\propcomputeinitial*
\begin{proof}
Fix $f\in\gamblesX$ and $\epsilon>0$. It then follows from Equation~\eqref{eq:lowerexp3} that there is a stochastic process $\smash{P\in\wprocesses_{\rateset,\mathcal{M}}}$ such that $\mathbb{E}_P[f(X_0)]\leq\underline{\mathbb{E}}^\mathrm{W}_{\rateset,\mathcal{M}}[f(X_0)]+\epsilon$. Furthermore, since $P\in\wprocesses_{\rateset,\mathcal{M}}$, it follows from Definitions~\ref{def:consistent_process_initialdistribution} and~\ref{def:process_sets} that $P(X_0)\in\mathcal{M}$, and therefore, that
\begin{equation*}
\mathbb{E}_P[f(X_0)]=\sum_{x\in\states}P(X_0=x)f(x)\geq\underline{\mathbb{E}}_{\mathcal{M}}[f].
\end{equation*}
Because of Equation~\eqref{eq:initiallowerexp}, we also know that there is a probability mass function $p\in\mathcal{M}$ such that $\sum_{x\in\states}p(x)f(x)\leq\underline{\mathbb{E}}_{\mathcal{M}}[f]+\epsilon$. Consider now any $Q\in\rateset$. It then follows from Corollary~\ref{cor:rate_has_unique_homogen_markov_process} that there is a unique well-behaved homogeneous Markov chain $P'\in\mprocesses$, with transition matrix $Q$, and with $P'(X_0=x)=p(x)$ for all $x\in\states$. Since $P'$ clearly belongs to $\whmprocesses_{\rateset,\mathcal{M}}$, this implies that
\begin{equation*}
\sum_{x\in\states}p(x)f(x)
=\sum_{x\in\states}P'(X_0=x)f(x)
=\mathbb{E}_{P'}[f(X_0)]
\geq\underline{\mathbb{E}}_{\rateset,\mathcal{M}}^{\mathrm{W
HM}}[f(X_0)].
\end{equation*}
Hence, we conclude that
\begin{equation*}
\underline{\mathbb{E}}^\mathrm{W}_{\rateset,\mathcal{M}}[f(X_0)]+\epsilon\geq\mathbb{E}_P[f(X_0)]
\geq
\underline{\mathbb{E}}_{\mathcal{M}}[f]
\geq\sum_{x\in\states}p(x)f(x)-\epsilon
\geq\underline{\mathbb{E}}_{\rateset,\mathcal{M}}^{\mathrm{W
HM}}[f(X_0)]-\epsilon.
\end{equation*}
Since $\epsilon>0$ is arbitrary, this implies that $
\underline{\mathbb{E}}^\mathrm{W}_{\rateset,\mathcal{M}}[f(X_0)]
\geq\underline{\mathbb{E}}_{\mathcal{M}}[f]\geq\underline{\mathbb{E}}_{\rateset,\mathcal{M}}^{\mathrm{W
HM}}[f(X_0)]$, and therefore, the result follows from Proposition~\ref{prop:lower_exp_markov_bounded_by_nonmarkov}.
\end{proof}

\propcomputemarginal*
\begin{proof}
Fix $s\in\realsnonneg$, $f\in\gamblesX$ and $\epsilon>0$. 
It then follows from Equation~\eqref{eq:lowerexp3} that there is some $\smash{P\in\wprocesses_{\rateset,\mathcal{M}}}$ such that $\mathbb{E}_P[f(X_s)]\leq\underline{\mathbb{E}}^\mathrm{W}_{\rateset,\mathcal{M}}[f(X_s)]+\epsilon$. Since $\smash{P\in\wprocesses_{\rateset,\mathcal{M}}}$, Equation~\eqref{eq:lowerexp3} now implies that, for all $x\in\states$,
\begin{equation}\label{eq:prop:computemarginal}
\mathbb{E}_P[f(X_s)\vert X_0=x]
\geq
\underline{\mathbb{E}}^\mathrm{W}_{\rateset,\mathcal{M}}[f(X_s)\vert X_0=x]
=[L_0^sf](x),
\end{equation}
where the last equality follows from Corollary~\ref{cor:lower_operator_is_infimum}. Hence, we find that
\begin{align*}
\underline{\mathbb{E}}^\mathrm{W}_{\rateset,\mathcal{M}}[f(X_s)]+\epsilon
\geq
\mathbb{E}_P[f(X_s)]
&=\mathbb{E}_P[\mathbb{E}_P[f(X_s)\vert X_0]]\notag\\
&\geq\mathbb{E}_P[[L_0^sf](X_0)]
\geq
\underline{\mathbb{E}}^\mathrm{W}_{\rateset,\mathcal{M}}[[L_0^sf](X_0)]
=\underline{\mathbb{E}}_{\mathcal{M}}[L_0^sf],
\end{align*}
where the first equality follows from Equation~\eqref{eq:lawofiteratedexpectation}, the second inequality follows from Equation~\eqref{eq:prop:computemarginal}, the third inequality follows from $P\in\wprocesses_{\rateset,\mathcal{M}}$ and Equation~\eqref{eq:lowerexp3}, and the last equality follows from Proposition~\ref{prop:computeinitial}.
Next, because of Equation~\eqref{eq:initiallowerexp}, we know that there is a probability mass function $p\in\mathcal{M}$ such that $\sum_{x\in\states}p(x)[L_0^sf](x)\leq\underline{\mathbb{E}}_{\mathcal{M}}[L_0^sf]+\epsilon$. Let $\mathcal{M}^*\coloneqq\{p\}$. Proposition~\ref{theorem:lower_markov_bound_is_tight} then implies that there is a well-behaved Markov chain $P^*\in\wmprocesses_{\rateset,\mathcal{M}^*}\subseteq\wmprocesses_{\rateset,\mathcal{M}}$ such that
\begin{equation*}
\mathbb{E}_{P^*}[f(X_s)\,\vert\,X_0]\leq[\lbound_0^sf](X_0)+\epsilon.
\end{equation*}
It now follows from the definition of $\underline{\mathbb{E}}^\mathrm{WM}_{\rateset,\mathcal{M}}$ that
\begin{equation*}
\underline{\mathbb{E}}^\mathrm{WM}_{\rateset,\mathcal{M}}[f(X_s)]
\leq
\mathbb{E}_{P^*}[f(X_s)]
=\mathbb{E}_{P^*}[\mathbb{E}_{P^*}[f(X_s)\vert X_0]]
\leq\mathbb{E}_{P^*}[[L_0^sf](X_0)]+\epsilon,
\end{equation*}
using Equation~\eqref{eq:lawofiteratedexpectation} for the equality. Furthermore, since $\smash{P^*\in\wmprocesses_{\rateset,\mathcal{M}^*}}$ implies that $P^*(X_0=x)=p(x)$ for all $x\in\states$, we also know that
\begin{equation*}
\mathbb{E}_{P^*}[[L_0^sf](X_0)]
=
\sum_{x\in\states}P^*(X_0=x)[L_0^sf](x)
=\sum_{x\in\states}p(x)[L_0^sf](x)
\leq\underline{\mathbb{E}}_{\mathcal{M}}[L_0^sf]+\epsilon.
\end{equation*}
Combining all of the above inequalities, we find that
\begin{equation*}
\underline{\mathbb{E}}^\mathrm{WM}_{\rateset,\mathcal{M}}[f(X_s)]-2\epsilon
\leq
\mathbb{E}_{P^*}[[L_0^sf](X_0)]-\epsilon\leq\underline{\mathbb{E}}_{\mathcal{M}}[L_0^sf]\leq\underline{\mathbb{E}}^\mathrm{W}_{\rateset,\mathcal{M}}[f(X_s)]+\epsilon.
\end{equation*}
Since $\epsilon>0$ is arbitrary, this implies that $\underline{\mathbb{E}}^\mathrm{WM}_{\rateset,\mathcal{M}}[f(X_s)]\leq\underline{\mathbb{E}}_{\mathcal{M}}[L_0^sf]\leq\underline{\mathbb{E}}^\mathrm{W}_{\rateset,\mathcal{M}}[f(X_s)]$. The result now follows from Proposition~\ref{prop:lower_exp_markov_bounded_by_nonmarkov}.
\end{proof}

\propcomputeunconditional*
\begin{proof}
Consider any $u=t_0,\dots,t_n$ in $\mathcal{U}_{\emptyset}$ such that $t_0=0$ and any $f\in\gambles(\states_u)$. Let $\hat{f}\in\gamblesX$ be defined by
\begin{equation}\label{eq:prop:computeunconditional}
\hat{f}(x)\coloneqq
\underline{\mathbb{E}}_{\rateset,\,\mathcal{M}}^\mathrm{W}[f(X_u)\,\vert\,X_0=x]
=\left[L_{t_0}^{t_1}L_{t_1}^{t_2}\cdots L_{t_{n-1}}^{t_n}f\right](x)
~\text{ for all $x\in\states$},
\end{equation}
using Corollary~\ref{cor:composition_lower_trans} to establish the last equality.
It then follows from Theorem~\ref{theorem:decomposition_multivar} and Proposition~\ref{prop:computeinitial} that
\begin{equation*}
\underline{\mathbb{E}}_{\rateset,\,\mathcal{M}}^\mathrm{W}[f(X_u)] = \underline{\mathbb{E}}_{\rateset,\,\mathcal{M}}^\mathrm{W}[\underline{\mathbb{E}}_{\rateset,\,\mathcal{M}}^\mathrm{W}[f(X_u)\,\vert\,X_0]]
=\underline{\mathbb{E}}_{\rateset,\,\mathcal{M}}^\mathrm{W}[\hat{f}(X_0)]
=\underline{\mathbb{E}}_{\mathcal{M}}[\hat{f}].
\end{equation*}
Combined with Equation~\eqref{eq:prop:computeunconditional}, this implies the result.
\end{proof}

\section{A gambling interpretation for coherence}\label{app:coherence}

This appendix aims to provide a basic exposition of the gambling interpretation for coherent conditional probabilities. A more extensive discussion, which also provides some historic context, can be found in, among others, References~\cite{regazzini1985finitely,williams1975,Williams:2007eu, Vicig:2007gs,berti1991coherent, berti2002coherent}.

Basically, the idea is to interpret $P$ as a set of gambles on the actual---but unknown---value of $X$ in $\Omega$, which some bettor is willing to either buy or sell, and to impose a rationality criterion on this set of gambles.

Concretely, for every pair $(A,C)\in\mathcal{C}$, $P(A\vert C)$ is interpreted as a bettor's fair price for a ticket that yields a reward of one currency unit to its holder if the event $A$ occurs, conditional on the fact that $C$ happens. In other words, the bettor is willing to either sell or buy such ticket at this price, provided that she will be refunded should event $C$ not happen. Furthermore, it also assumed that the bettor's utility is linear, which implies that she is willing to vary the stakes of her bets arbitrarily.

Suppose for example that the actual value of $X$ ends up being $\omega$. For each ticket that the bettor sold, she has then received $P(A\vert C)$ currency units in advance, but after the value of $X$ is revealed, she loses one currency unit if $A$ has happened, that is, she loses $\ind{A}(\omega)$ currency unit. Because all of this is conditional on $C$ happening, her net profit is $\ind{C}(\omega)(P(A\vert C) - \ind{A}(\omega))$, with negative profit being loss. Note that if $C$ does not happen, that is, if $\ind{C}(\omega)=0$, she neither gains nor loses anything. Since we also allow for arbitrary stakes, we conclude that for any $\lambda\in\realsnonneg$, the bettor is willing to accept the uncertain net profit $\lambda \ind{C}(\omega)(P(A\vert C) - \ind{A}(\omega))$.

Similarly, for each ticket that that the bettor buys, she first has to pay $P(A\vert C)$ to buy the ticket, but will then receive one unit of currency if $A$ happens. Her profit is then $\ind{A}(\omega) - P(A\vert C)$ per ticket. However, she only receives this profit if event $C$ also came to pass, and otherwise gets refunded. Hence, if we again take into account that the stake can be chosen arbitrarily, we find that for any $\lambda\in\realsnonneg$ the bettor is willing to accept the uncertain net profit $\lambda\ind{C}(\omega)(\ind{A}(\omega) - P(A\vert C))$. 

By combining the arguments for selling and buying, we conclude from the above that for any $\lambda\in\reals$, the bettor is willing to accept a bet in which she receives the uncertain net profit $\lambda\ind{C}(\omega)(P(A\vert C)-\ind{A}(\omega))$, with negative profit being loss.

The final assumption is now that the bettor is willing to combine any finite number of such transactions. That is, if we consider any $n\in\nats$ and, for every $i\in\{1,\ldots,n\}$, some $\lambda_i\in\reals$ and $(A_i,C_i)\in\mathcal{C}$, then the bettor is willing to accept a bet in which her net profit is equal to
\begin{equation*}
\sum_{i=1}^n\lambda_i\ind{C_i}(\omega)\left(P(A_i\vert C_i) - \ind{A_i}(\omega)\right)\,.
\end{equation*}
The coherence of $P$ is now equivalent to requiring that any such bet \emph{avoids sure loss}, in the sense that there exists at least one ``non-trivial'' outcome $\omega$ for which her total profit is non-negative, the trivial case being when none of the events $C_i$ happen---$\omega\notin\cup_{i=1}^n C_i$---because she then gets refunded completely. 

\section{Proofs of the examples}\label{app:example_proofs}

\begin{proof}[Proof of Example~\ref{exmp:limit_trans_mat_system}]
We will show that the sequence $\{\mathcal{T}_i\}_{i\in\nats_0}$ defined in Equation~\eqref{eq:def:sequenceinexample3} is Cauchy, or in other words, that
\begin{equation}\label{eq:cauchyexample}
(\forall\epsilon\in\realspos)\,(\exists n_\epsilon\in\nats)\,(\forall k,\ell > n_\epsilon)~d(\mathcal{T}_{k},\mathcal{T}_\ell)<\epsilon.
\end{equation}
In order to prove this, the first step is to notice that for any $i\in\nats$, the difference between $\mathcal{T}_{i}$ and $\mathcal{T}_{i-1}$ is essentially situated on the interval $[0,\delta_i]$. It should therefore be intuitively clear that $d(\mathcal{T}_i,\mathcal{T}_{i-1})$ is proportional to $\delta_i$. 

In fact, it holds that $d(\mathcal{T}_i,\mathcal{T}_{i-1})\leq\delta_i\norm{Q_i-Q_{i-1}}$; we will start by proving this inequality.
So, fix any $i\in\nats$. 
Fix any $t,s\in\realsnonneg$ such that $t\leq s$, and let $\presuper{i}T_t^s$ and $\presuper{i-1}T_t^s$ be the transition matrices that correspond to $\mathcal{T}_i$ and $\mathcal{T}_{i-1}$, respectively. 
We now consider three cases. The first case is $t\geq\delta_i$. It then follows from Equation~\eqref{eq:def:sequenceinexample3} that $\norm{\presuper{i}T_t^s-\presuper{i-1}T_t^s}=0$. The second case is $s\leq\delta_i$. It then follows from Equation~\eqref{eq:def:sequenceinexample3} and Lemma~\ref{lemma:differencebetweenexponentials} that
\begin{equation*}
\norm{\presuper{i}T_t^s-\presuper{i-1}T_t^s}
=\norm{e^{Q_i(s-t)}-e^{Q_{i-1}(s-t)}}
\leq (s-t)\norm{Q_i-Q_{i-1}}\leq\delta_i\norm{Q_i-Q_{i-1}}.
\end{equation*}
The third case is $t\leq\delta_i\leq s$. We then find that
\begin{align*}
\norm{\presuper{i}T_t^s-\presuper{i-1}T_t^s}
=\norm{\presuper{i}T_t^{\delta_i}\presuper{i}T_{\delta_i}^s-\presuper{i-1}T_t^{\delta_i}\presuper{i-1}T_{\delta_i}^s}
&\leq
\norm{\presuper{i}T_t^{\delta_i}-\presuper{i-1}T_t^{\delta_i}}
+
\norm{\presuper{i}T_{\delta_i}^s-\presuper{i-1}T_{\delta_i}^s}\\
&=\norm{\presuper{i}T_t^{\delta_i}-\presuper{i-1}T_t^{\delta_i}}
\leq\delta_i\norm{Q_i-Q_{i-1}},
\end{align*}
where the first inequality follows from Lemma~\ref{lemma:recursive}, the second equality follows from the first case above, and the last inequality follows from the second case above. Hence, in all three cases, we find that $\norm{\presuper{i}T_t^s-\presuper{i-1}T_t^s}\leq\delta_i\norm{Q_i-Q_{i-1}}$. Since this inequality holds for any $t,s\in\realsnonneg$ such that $s\geq t$, it now follows from Equation~\eqref{eq:trans_mat_system_metric} that $d(\mathcal{T}_i,\mathcal{T}_{i-1})\leq\delta_i\norm{Q_i-Q_{i-1}}$.

Using this inequality, and because $d$ is a metric, Equation~\eqref{eq:cauchyexample} can now easily be proven. It suffices to choose $n_\epsilon$ in such a way that $c2^{-n_\epsilon}<\epsilon$. Indeed, in that case, for any $k,\ell>n_\epsilon$, if we assume---without loss of generality---that $k\leq\ell$, it follows that
 \begin{equation*}
d(\mathcal{T}_{k},\mathcal{T}_\ell)
\leq\sum_{i=k+1}^{\ell}d(\mathcal{T}_{i-1},\mathcal{T}_i)
\leq\sum_{i=k+1}^{\ell}\delta_i\norm{Q_i-Q_{i-1}}
\leq c\sum_{i=k+1}^{\ell}\delta_i
 \end{equation*}
and therefore, since we also know that
\begin{equation*}
\sum_{i=k+1}^{\ell}\delta_i
\leq
\sum_{i=k+1}^{+\infty}\delta_i
=\sum_{i=k+1}^{+\infty}2^{-i}
=2^{-k},
\end{equation*}
it follows that $\smash{d(\mathcal{T}_{k},\mathcal{T}_\ell)\leq c 2^{-k}\leq c 2^{-n_\epsilon}}<\epsilon$, as required.

As a side result, we also obtain a similar bound on the distance between $\mathcal{T}_k$ and $\mathcal{T}$. For any $\ell\geq k$, we know that
\begin{equation*}
d(\mathcal{T}_k,\mathcal{T})
\leq d(\mathcal{T}_k,\mathcal{T}_\ell)+d(\mathcal{T}_\ell,\mathcal{T})\leq c 2^{-k}+d(\mathcal{T}_\ell,\mathcal{T}),
\end{equation*}
and therefore, since $\mathcal{T}=\lim_{\ell\to+\infty}\mathcal{T}_\ell$, it follows that $d(\mathcal{T}_k,\mathcal{T})\leq c 2^{-k}$.
\end{proof}

\begin{proof}[Proof of Example~\ref{exmp:limit_trans_mat_system_matrices}]
We start by proving that Equation~\eqref{eq:ex4:def} gives us the transition matrix from $0$ to $t$ that corresponds to the transition matrix system $\mathcal{T}$. To this end, we first establish some properties. 
Consider any $t\in(0,1]$ and let $j$ be the unique element of $\nats_0$ such that $\delta_{j+1}< t\leq\delta_j$. If $j$ is odd, it follows from Equations~\eqref{eq:ex4:phi1} and~\eqref{eq:ex4:phi2} that
\begin{equation*}
\varphi_1(t)=t-\nicefrac{2}{3}\delta_{j+1}=(t-\delta_{j+1})+\nicefrac{1}{3}\delta_{j+1}=(t-\delta_{j+1})+\varphi_1(\delta_{j+1})
\end{equation*}
and $\varphi_2(t)=\nicefrac{2}{3}\delta_{j+1}=\varphi_2(\delta_{j+1})$. Similarly, if $j$ is even, it follows that $\varphi_1(t)=\varphi_1(\delta_{j+1})$ and $\varphi_2(t)=\varphi_2(\delta_{j+1})+(t-\delta_{j+1})$. Hence, in both cases, we have that
\begin{equation*}
Q_1\varphi_1(t)+Q_2\varphi_2(t)=Q_1\varphi_1(\delta_{j+1})+Q_2\varphi_2(\delta_{j+1})+Q_j(t-\delta_{j+1}).
\end{equation*}
Therefore, and since $Q_1$ and $Q_2$ commute, it now follows from Equation~\eqref{eq:ex4:def} that
\vspace{3pt}
\begin{align}\label{eq:ex4:recursestar}
T_0^{t}
=e^{Q_1\varphi_1(t)+Q_2\varphi_2(t)}
&=e^{Q_1\varphi_1(\delta_{j+1})+Q_2\varphi_2(\delta_{j+1})+Q_j(t-\delta_{j+1})}\notag\\
&=e^{Q_1\varphi_1(\delta_{j+1})+Q_2\varphi_2(\delta_{j+1})}e^{Q_j(t-\delta_{j+1})}
=T_0^{\delta_{j+1}}e^{Q_j(t-\delta_{j+1})}.
\end{align}
For large enough $k\in\nats$, a similar statement holds for the transition matrix $\presuper{k}T_0^{t}$ that corresponds to $\mathcal{T}_k$. In particular, Equation~\eqref{eq:def:sequenceinexample3} implies that
\begin{equation}\label{eq:ex4:recursek}
\presuper{k}T_0^{t}
=\presuper{k}T_0^{\delta_{j+1}}e^{Q_j(t-\delta_{j+1})}
~~\text{ for all $k\geq j$}
\end{equation}
Hence, for all $j\in\nats_0$, by choosing $t=\delta_j$, and because $\delta_j-\delta_{j+1}=\delta_{j+1}$, it follows that
\begin{equation}\label{eq:ex4:recursekandstar}
T_0^{\delta_j}
=T_0^{\delta_{j+1}}e^{Q_j\delta_{j+1}}
~\text{ and }~
\presuper{k}T_0^{\delta_j}
=\presuper{k}T_0^{\delta_{j+1}}e^{Q_j\delta_{j+1}}
~\text{ for all $k\geq j$.}
\vspace{3pt}
\end{equation}
Finally, for all $k\in\nats_0$, it follows from Equations~\eqref{eq:ex4:def}--\eqref{eq:ex4:phi2} and~\eqref{eq:def:sequenceinexample3} that
\begin{equation}\label{eq:ex4:finalstep}
T_0^{\delta_{k}}
=e^{Q_k\nicefrac{2}{3}\delta_k+Q_{k+1}\nicefrac{1}{3}\delta_k}
~\text{ and }~
\presuper{k}T_0^{\delta_{k}}
=e^{Q_k\delta_k}
~\text{ for all $k\in\nats_0$}.
\end{equation}

Using these properties, proving that Equation~\eqref{eq:ex4:def} gives us the transition matrix from $0$ to $t$ that corresponds to the transition matrix system $\mathcal{T}$, is now relatively easy. Consider any $t\in(0,1]$, let $i$ be the unique element of $\nats_0$ such that $\delta_{i+1}< t\leq\delta_i$, and fix any $k\geq i$. Then on the one hand, we find that
\vspace{5pt}
\begin{align*}
\norm{T_0^t-\presuper{k}T_0^t}
&=
\norm{T_0^{\delta_{i+1}}e^{Q_i(t-\delta_{i+1})}-\presuper{k}T_0^{\delta_{i+1}}e^{Q_i(t-\delta_{i+1})}}\\
&\leq
\norm{T_0^{\delta_{i+1}}-\presuper{k}T_0^{\delta_{i+1}}}
+
\norm{e^{Q_i(t-\delta_{i+1})}-e^{Q_i(t-\delta_{i+1})}}\\
&=
\norm{T_0^{\delta_{i+1}}-\presuper{k}T_0^{\delta_{i+1}}}
=
\norm{T_0^{\delta_{i+2}}e^{Q_{i+1}\delta_{i+2}}-\presuper{k}T_0^{\delta_{i+2}}e^{Q_{i+1}\delta_{i+2}}}\\
&\quad\quad\quad\quad\quad\quad\quad\quad~\,\,\leq
\norm{T_0^{\delta_{i+2}}-\presuper{k}T_0^{\delta_{i+2}}}
+
\norm{e^{Q_{i+1}\delta_{i+2}}-e^{Q_{i+1}\delta_{i+2}}}\\
&\quad\quad\quad\quad\quad\quad\quad\quad~\,\,=
\norm{T_0^{\delta_{i+2}}-\presuper{k}T_0^{\delta_{i+2}}}
\leq\dots\leq
\norm{T_0^{\delta_{k}}-\presuper{k}T_0^{\delta_{k}}},\\[-10pt]
\end{align*}
where the first equality follows from Equations~\eqref{eq:ex4:recursestar} and~\eqref{eq:ex4:recursek}, the first inequality follows from Lemma~\ref{lemma:recursive} in Appendix~\ref{app:systems}, the third equality follows from Equation~\eqref{eq:ex4:recursekandstar}, the second equality is again due to Lemma~\ref{lemma:recursive}, and the remaining steps consist in repeating the last steps over and over again.
On the other hand, we also know that
\vspace{5pt}
\begin{align*}
\norm{T_0^{\delta_{k}}-\presuper{k}T_0^{\delta_{k}}}
&=
\norm{e^{Q_k\nicefrac{2}{3}\delta_k+Q_{k+1}\nicefrac{1}{3}\delta_k}
-e^{Q_k\delta_k}}\\
&=
\norm{e^{Q_k\nicefrac{2}{3}\delta_k}e^{Q_{k+1}\nicefrac{1}{3}\delta_k}
-e^{Q_k\nicefrac{2}{3}\delta_k}e^{Q_k\nicefrac{1}{3}\delta_k}}\\
&\leq
\norm{e^{Q_k\nicefrac{2}{3}\delta_k}
-e^{Q_k\nicefrac{2}{3}\delta_k}}
+
\norm{e^{Q_{k+1}\nicefrac{1}{3}\delta_k}
-e^{Q_k\nicefrac{1}{3}\delta_k}}\\
&=\norm{e^{Q_{k+1}\nicefrac{1}{3}\delta_k}
-e^{Q_k\nicefrac{1}{3}\delta_k}}
\leq\delta_k\norm{Q_k-Q_{k+1}}
=\delta_k\norm{Q_1-Q_2},\\[-10pt]
\end{align*}
where the first equality follows from Equation~\eqref{eq:ex4:recursekandstar}, the second equality holds because $Q_1$ and $Q_2$ commute, and the two inequalities follow from Lemmas~\ref{lemma:recursive} and~\ref{lemma:differencebetweenexponentials} in Appendix~\ref{app:systems}. Hence, we find that $\norm{T_0^t-\presuper{k}T_0^t}\leq\delta_k\norm{Q_1-Q_2}$. Since this is true for any $k\geq i$, it follows that $\lim_{k\to+\infty}\presuper{k}T_0^t=T_0^t$. Therefore, because $\mathcal{T}\coloneqq \lim_{i\to\infty}\mathcal{T}_i$, we can conclude that $T_0^t$ indeed corresponds to $\mathcal{T}$.

We end this proof by showing that the transition matrix system $\mathcal{T}$ is well-behaved. Let $M\coloneqq\max\{\norm{Q_1},\norm{Q_2}\}$. We will prove that
\begin{equation*}
\frac{1}{\Delta}\norm{T_t^s-I}\leq M
~~\text{for all $t,s,\Delta\in\reals_{\geq0}$ such that $s=t+\Delta$.}
\end{equation*}
According to Definition~\ref{def:well_behaved_trans_mat_system}, this clearly implies that $\mathcal{T}$ is well-behaved.

So fix any $t,s,\Delta\in\reals_{\geq0}$ such that $s=t+\Delta$ and consider any $\epsilon>0$. Then since $\lim_{i\to+\infty}\presuper{i}T_t^s=T_t^s$, there is some $j\in\nats$ such that $\norm{T_t^s-\presuper{j}T_t^s}\leq\epsilon$. Furthermore, since $Q_1$ and $Q_2$ commute, it follows from Equation~\eqref{eq:def:sequenceinexample3} that there are $\Delta_1,\Delta_2\in\reals_{\geq0}$ such that $\Delta_1+\Delta_2=\Delta$ and $\presuper{j}T_t^s=e^{Q_1\Delta_1}e^{Q_2\Delta_2}$. Therefore, we find that
\begin{align*}
\norm{T_t^s-I}
\leq
\norm{T_t^s-\presuper{j}T_t^s}
+
\norm{e^{Q_1\Delta_1}e^{Q_2\Delta_2}-I}
&\leq
\epsilon
+
\norm{e^{Q_1\Delta_1}-I}
+\norm{e^{Q_2\Delta_2}-I}\\
&\leq
\epsilon+
\Delta_1\norm{Q_1}
+
\Delta_2\norm{Q_2}
\leq\Delta M,
\end{align*}
where the second and third inequalities follow from Lemma~\ref{lemma:recursive} and~\ref{lemma:linearboundonexponential} in Appendix~\ref{app:systems}, respectively.
\end{proof}

\begin{proof}[Proof of Example~\ref{exmp:notwellbehavedMarkov}]
We will prove that the transition matrix system $\mathcal{T}$ from Example~\ref{exmp:notwellbehavedMarkov} is not well-behaved.
In order to do this, we first fix any $n\in\nats_0$ and any $\Delta\in(0,\delta_n]$, and we let $i\geq n$ be the unique element of $\nats_0$ such that $\delta_{i+1}<\Delta\leq\delta_i$. It then follows from Equation~\eqref{eq:def:sequenceinexample3} that $\presuper{i}T_0^{\Delta}=e^{Q_i\Delta}$, and therefore, we find that
\begin{equation*}
\norm{\Delta Q_i}
\leq
\norm{e^{Q_i\Delta}-(I+\Delta Q_i)}
+
\norm{T_0^\Delta-\presuper{i}T_0^\Delta}
+
\norm{T_0^\Delta-I}
\leq\Delta^2\norm{Q_i}^2+2^{-i}+\norm{T_0^\Delta-I},
\end{equation*}
where the first inequality holds because $\norm{\cdot}$ is a norm, and where the second inequality holds because of Lemma~\ref{lemma:linearpartofexponential} and because---as proved at the end of the proof of Example~\ref{exmp:limit_trans_mat_system}---$d(\mathcal{T}_i,\mathcal{T})\leq c2^{-i}=2^{-i}$. Hence, since $\norm{Q_i}=\norm{iQ}=i\norm{Q}=i$, we find that
\begin{equation*}
\frac{1}{\Delta}\norm{T_0^\Delta-I}
\geq
\norm{Q_i}-\Delta\norm{Q_i}^2-\frac{1}{\Delta}2^{-i}
=
i-\Delta i^2-\frac{1}{\Delta}2^{-i}
\geq
i-\delta_i i^2-\frac{1}{\delta_{i+1}}2^{-i}
\end{equation*}
and therefore, because $\delta_i=2^{-i}$, $\delta_{i+1}=2^{-i-1}$, $2^i\geq i$ and $i\geq n$, it follows that
\begin{equation}\label{eq:exmp:notwellbehavedMarkov}
\frac{1}{\Delta}\norm{T_0^\Delta-I}
\geq
i-2^{-i}i^2-2
\geq
i-\frac{1}{2}i-2
=\frac{i}{2}-2
\geq\frac{n}{2}-2.
\end{equation}
Since $\Delta\in(0,\delta_n]$ is arbitrary, this inequality immediately implies that
\begin{equation*}
\limsup_{\Delta\to 0^{+}}\frac{1}{\Delta}\norm{T_{0}^{\Delta}-I}\geq\frac{n}{2}-2.
\end{equation*}
and therefore, since this is true for every $n\in\nats_0$, we infer from Definition~\ref{def:well_behaved_trans_mat_system} that $\mathcal{T}$ is not well-behaved, because the definition clearly fails for $t=0$.
\end{proof}

\begin{proof}[Proof of Example~\ref{exmp:well-behaved-no-deriv}]
We will prove that Equation~\eqref{eq:exmp:well-behaved-no-deriv} indeed holds. So fix any $\lambda\in[\nicefrac{1}{3},\nicefrac{2}{3}]$ and consider the sequence $\{\Delta_i\}_{i\in\nats_0}\to0^+$ whose elements are defined by \mbox{$\Delta_i\coloneqq\nicefrac{(2\delta_{2i+1})}{(3\lambda)}$}. For all $i\in\nats_0$, we then find that
\begin{multline*}
\varphi_1(\Delta_i)Q_1+\varphi_2(\Delta_i)Q_2
=\frac{2}{3}\delta_{2i+1}Q_1+(\Delta_i-\frac{2}{3}\delta_{2i+1})Q_2\\
=\lambda\Delta_i Q_1+(1-\lambda)\Delta_i Q_2
=Q_\lambda \Delta_i,
\end{multline*}
where the first equality follows from Equations~\eqref{eq:ex4:phi1} and~\eqref{eq:ex4:phi2}---because $\nicefrac{1}{(3\lambda)}\in[\nicefrac{1}{2},1]$ implies that $\delta_{2i+1}\leq\Delta_i\leq\delta_{2i}$.
Hence, for all $i\in\nats_0$, Equation~\eqref{eq:ex4:def} now tells us that $T_0^{\Delta_i}=e^{Q_\lambda\Delta_i}$.
Therefore, and because $\{\Delta_{i}\}_{i\in\nats_0}\to0^+$, we find that indeed, as required,
\begin{equation}\label{eq:exmp:well-behaved-no-deriv2}
\lim_{i\to+\infty}
\frac{1}{\Delta_i}
(T^{\Delta_i}_{0}-I)
=
\lim_{i\to+\infty}
\frac{1}{\Delta_{i}}
(e^{Q_\lambda\Delta_{i}}-I)
={\frac{d}{dt}e^{Q_\lambda t}}\big\vert_{t=0}
=Q_\lambda,
\end{equation}
where we use Lemma~\ref{lemma:deriv_exponential_trans} to establish the last equality. 
\end{proof}

\begin{proof}[Proof of Example~\ref{exmp:outerderivative}]
Showing that $\{Q_\lambda\colon \lambda\in[\nicefrac{1}{3},\nicefrac{2}{3}]\}$ is a subset of $\smash{\overline{\partial}}_+T_0^0$ was, essentially, already done in Example~\ref{exmp:well-behaved-no-deriv}, because for every $\lambda\in[\nicefrac{1}{3},\nicefrac{2}{3}]$, it follows from Equation~\eqref{eq:exmp:well-behaved-no-deriv2} and Definition~\ref{def:direc_partial_deriv} that $Q_\lambda\in\smash{\overline{\partial}}_+T_0^0$. Therefore, we only need to show that $\smash{\overline{\partial}}_+T_0^0$ is a subset of $\{Q_\lambda\colon \lambda\in[\nicefrac{1}{3},\nicefrac{2}{3}]\}$, or equivalently, we need to show that for every $Q\in\smash{\overline{\partial}}_+T_0^0$, there is some $\lambda\in[\nicefrac{1}{3},\nicefrac{2}{3}]$ such that $Q=Q_\lambda$.

So consider any $Q\in\overline{\partial}_+T_0^0$. Definition~\ref{def:outerpartialderivatives} then implies the existence of a sequence $\{\Delta_i\}_{i\in\nats_0}\to 0^+$ such that 
\begin{equation}\label{eq:exmp:outerderivative}
\lim_{i\to+\infty}\nicefrac{1}{\Delta_i}(T_0^{\Delta_i}-I)=Q.
\end{equation}

The first step of the proof is to observe that, for every $t\in(0,1]$, $\varphi_1(t)+\varphi_2(t)=t$ and $\nicefrac{1}{3}t\leq\varphi_1(t)\leq\nicefrac{2}{3}t$; we leave this as an exercise. Given this observation, it follows that for all $t\in(0,1]$, there is some $\lambda_t\in[\nicefrac{1}{3},\nicefrac{2}{3}]$ such that $\varphi_1(t)=\lambda_t t$ and $\varphi_2(t)=(1-\lambda_t)t$. Hence, in particular, for every $i\in\nats$, there is some $\lambda_i\in[\nicefrac{1}{3},\nicefrac{2}{3}]$ such that $\varphi_1(\Delta_i)=\lambda_i \Delta_i$ and $\varphi_2(\Delta_i)=(1-\lambda_i)\Delta_i$ and therefore also, due to Equation~\eqref{eq:ex4:def}, $\smash{T_0^{\Delta_i}=e^{Q_{\lambda_i}\Delta_i}}$.
Furthermore, because of the Bolzano-Weierstrass theorem, the sequence $\{\lambda_i\}_{i\in\nats_0}$ contains a convergent subsequence $\{\lambda_{i_k}\}_{k\in\nats_0}$ whose limit $\lambda\coloneqq\lim_{k\to+\infty}\lambda_{i_k}$ clearly belongs to $[\nicefrac{1}{3},\nicefrac{2}{3}]$. 

In the remainder of this proof, we will show that $Q=Q_\lambda$. To this end, let us fix any $\epsilon>0$ and prove that $\norm{Q-Q_\lambda}<\epsilon$. First of all, since $\lambda=\lim_{k\to+\infty}\lambda_{i_k}$, there is some $n_1\in\nats_0$ such that, for all $k\geq n_1$, $\abs{\lambda-\lambda_{i_k}}\norm{Q_1-Q_2}<\nicefrac{\epsilon}{3}$ and therefore also
\begin{equation}\label{eq:exmp:outerderivative1}
\norm{Q_\lambda-Q_{\lambda_{i_k}}}
=\norm{(\lambda-\lambda_{i_k})(Q_1-Q_2)}
=\abs{\lambda-\lambda_{i_k}}\norm{Q_1-Q_2}\leq\frac{\epsilon}{3}.
\end{equation}
Secondly, Equation~\eqref{eq:exmp:outerderivative} implies that there is some $n_2\in\nats_0$ such that
\begin{equation}\label{eq:exmp:outerderivative2}
\norm{\nicefrac{1}{\Delta_{i_k}}(T_0^{\Delta_{i_k}}-I)-Q}<\frac{\epsilon}{3}
~\text{ for all $k\geq n_2$}.
\end{equation}
Thirdly, Lemma~\ref{lemma:deriv_exponential_trans} implies that there is some $n_3\in\nats_0$ such that
\begin{equation}\label{eq:exmp:outerderivative3}
\norm{\nicefrac{1}{\Delta_{i_k}}(e^{Q_{\lambda_{i_k}}\Delta_{i_k}}-I)-Q_{\lambda_{i_k}}} < \frac{\epsilon}{3}
~\text{ for all $k\geq n_3$.}
\end{equation}
Consider now any $k\geq\max\{n_1,n_2,n_3\}$. Then since $T_0^{\Delta_{i_k}}=e^{Q_{\lambda_{i_k}}\Delta_{i_k}}$, it follows from Equations~\eqref{eq:exmp:outerderivative1}--\eqref{eq:exmp:outerderivative3} that
\begin{align*}
\norm{Q - Q_{\lambda}} \leq 
\norm{Q-\nicefrac{1}{\Delta_{i_k}}(T_0^{\Delta_{i_k}}-I)}
+
\norm{\nicefrac{1}{\Delta_{i_k}}(e^{Q_{\lambda_{i_k}}\Delta_{i_k}}-I)-Q_{\lambda_{i_k}}}
+
\norm{Q_{\lambda_{i_k}}-Q_\lambda}
\leq\epsilon.
\end{align*}
Since this is true for every $\epsilon\in\realspos$, it follows that $\norm{Q-Q_{\lambda}}=0$, and therefore also, that $Q=Q_\lambda$, as desired.
\end{proof}

\begin{proof}[Proof of Example~\ref{exmp:emptyouterderivative}]
We will prove that the right-sided outer partial derivative $\smash{\overline{\partial}}_+T_0^0$ is empty. To this end, assume \emph{ex absurdo} that it is not empty, and consider any $Q\in\smash{\overline{\partial}}_+T_0^0$. It then follows from Equation~\eqref{eq:rightouterderivative} that there is a sequence $\{\Delta_i\}_{i\in\nats}\to0^+$ such that $\smash{\lim_{i\to+\infty}\nicefrac{1}{\Delta_i}
(T^{\Delta_i}_{0}-I)
=Q}$. Consider now any $\epsilon>0$, any $n\in\nats_0$ such that $n\geq4+2\norm{Q}+2\epsilon$, let $\delta_n\coloneqq 2^{-n}$ as in Example~\ref{exmp:notwellbehavedMarkov}, and consider any $i^*\in\nats$ such that, for all $i\geq i^*$, $\Delta_i\in(0,\delta_n)$---such an $i^*$ always exists because $\{\Delta_i\}_{i\in\nats}\to0^+$. For all $i\geq i^*$, using Equation~\eqref{eq:exmp:notwellbehavedMarkov} from the proof of Example~\ref{exmp:notwellbehavedMarkov}, we then find that
\begin{equation*}
\norm{\frac{1}{\Delta_i}(T_0^{\Delta_i}-I)-Q}
\geq
\norm{\frac{1}{\Delta_i}(T_0^{\Delta_i}-I)}-\norm{Q}
\geq
\frac{n}{2}-2-\norm{Q}\geq\epsilon,
\end{equation*}
which, since $\epsilon>0$, contradicts the fact that $\smash{\lim_{i\to+\infty}\nicefrac{1}{\Delta_i}
(T^{\Delta_i}_{0}-I)
=Q}$. Hence, our assumption must be wrong, and it follows that $\smash{\overline{\partial}}_+T_0^0$ is indeed empty.
\end{proof}

\begin{proof}[Proof of Example~\ref{example:different_sets_same_lower_rate}]
Let $\lrate_1$ and $\lrate_2$ be the lower transition rate operators that correspond to $\rateset_1$ and $\rateset_2$, respectively, 
and consider any $f\in\gamblesX$. Equation~\eqref{eq:correspondinglowertrans} then implies that $\lrate_1f\leq Af$ and $\lrate_1f\leq Bf$, which in turn implies that $\lrate_1f\leq Cf$. Therefore, by applying Equation~\eqref{eq:correspondinglowertrans} once more, we find that $\lrate_1 f\leq\lrate_2 f$. Hence, since Equation~\eqref{eq:correspondinglowertrans} also trivially implies that $\lrate_2 f\leq\lrate_1 f$, we find that $\lrate_1 f=\lrate_2 f$. Since this is true for any $f\in\gamblesX$, we conclude that $\lrate_1=\lrate_2$.
\end{proof}

\begin{proof}[Proof of Example~\ref{ex:dominatingset}]
We will prove that $\lrate$ is also the lower transition rate operator corresponding to $\rateset^*$. First of all, for any $Q\in\rateset^*$, $f\in\gamblesX$ and $x\in\states$, it follows from Equations~\eqref{eq:ex:dominatingset:globaldef} and~\eqref{eq:ex:dominatingset:localdef} that there is some $\lambda\in[0,1]$ such that
\begin{equation*}
[Qf](x)=\lambda[Af](x)+(1-\lambda)[Bf](x)
\geq\lambda[\lrate f](x)+(1-\lambda)[\lrate f](x)
=[\lrate f](x),
\end{equation*}
where the inequality holds because $\lrate$ is the lower envelope of $\rateset_1$. Since this is true for all $x\in\states$, we find that $Qf\geq\lrate f$. Since this is true for all $f\in\gamblesX$ and all $Q\in\rateset^*$, it follows that $\rateset^*\subseteq\rateset_{\lrate}$.

Let now $\lrate^*$ be the lower transition rate operator that corresponds to $\rateset^*$ and consider any $f\in\gamblesX$. Then, because $\rateset^*\subseteq\rateset_{\lrate}$, we find that $\lrate^*f\geq\lrate f$. Similarly, since $\rateset_1$ is clearly a subset of $\rateset^*$, we find that $\lrate^*f\leq\lrate_1f$. Since $\lrate_1=\lrate$, it follows that $\lrate f\leq \lrate^*f\leq \lrate f$, and because this holds for any $f\in\gamblesX$, we conclude that $\lrate^*=\lrate$.
\end{proof}

\end{document}